\documentclass[10pt]{amsart}

\pdfoutput=1

\usepackage[text={420pt,660pt},centering]{geometry}
\usepackage{cancel}
\usepackage{graphicx}
\usepackage{MnSymbol}
\usepackage{mathtools}
\usepackage[colorlinks=true, pdfstartview=FitV, linkcolor=blue, citecolor=blue, urlcolor=blue,pagebackref=false]{hyperref}
\hypersetup{
    pdftitle={Free energy of non-convex multi-species spherical spin glasses},
    pdfauthor={Hong-Bin Chen and Jean-Christophe Mourrat},
    pdfkeywords={spin glasses, multi-species models, Parisi formula, Hamilton--Jacobi equations, Crisanti--Sommers formula}
}
\usepackage{orcidlink}
\usepackage[expansion=false]{microtype}

\usepackage{bm}
\usepackage{dsfont}
\usepackage{mathrsfs}
\usepackage{xcolor}
\usepackage{accents} 
\usepackage{tikz}

\usepackage[normalem]{ulem}

\newtheorem{proposition}{Proposition}
\newtheorem{theorem}[proposition]{Theorem}
\newtheorem{lemma}[proposition]{Lemma}
\newtheorem{corollary}[proposition]{Corollary}

\theoremstyle{remark}
\newtheorem{remark}[proposition]{Remark}

\theoremstyle{definition}
\newtheorem{definition}[proposition]{Definition}

\numberwithin{equation}{section}
\numberwithin{proposition}{section}
\numberwithin{figure}{section}
\numberwithin{table}{section}

\newcommand{\N}{\mathbb{N}}
\newcommand{\Q}{\mathbb{Q}}
\newcommand{\R}{\mathbb{R}}

\newcommand{\E}{\mathbb{E}}
\renewcommand{\P}{\mathbb{P}}

\newcommand{\eps}{\varepsilon}

\renewcommand{\le}{\leqslant}
\renewcommand{\ge}{\geqslant}
\renewcommand{\leq}{\leqslant}
\renewcommand{\geq}{\geqslant}

\renewcommand{\subset}{\subseteq}
\renewcommand{\bar}{\overline}
\newcommand{\Ll}{\left}
\newcommand{\Rr}{\right}
\renewcommand{\d}{\mathrm{d}}
\newcommand{\dr}{\partial}
\newcommand{\mcl}{\mathcal}
\newcommand{\msf}{\mathsf}
\newcommand{\mfk}{\mathfrak}
\newcommand{\bgamma}{\boldsymbol{\gamma}}

\newcommand{\mt}{\mathsf{t}}

\DeclareMathOperator{\supp}{supp}

\newcommand{\sP}{\mathscr{P}}

\newcommand{\la}{\left\langle}
\newcommand{\ra}{\right\rangle}

\renewcommand{\H}{\mathsf{H}}
\newcommand{\cH}{\mathcal H}

\DeclareMathOperator{\diag}{\mathsf{diag}}

\renewcommand*{\dot}[1]{%
\accentset{\mbox{\large\bfseries .}}{#1}}
\newcommand{\one}{\mathds{1}}

\newcommand{\fR}{\mathfrak{R}}

\newcommand{\sS}{\mathscr{S}}
\newcommand{\vecone}{\vec{\mathbf{1}}}

\newcommand{\bff}{\mathbf{f}}

\usepackage{scalerel}
\newcommand{\lapprox}{\mathrel{\hstretch{1.8}{\approx}}}

\newcommand{\bkappa}{\boldsymbol{\kappa}}

\DeclareMathOperator{\Var}{Var}
\newcommand{\Sph}{\mathbb{S}}

\begin{document}

\author[H.-B.\ Chen]{Hong-Bin Chen\,\orcidlink{0000-0001-6412-0800}}
\address[Hong-Bin Chen]{NYU-ECNU Institute of Mathematical Sciences, NYU Shanghai, China}
\email{\href{mailto:hongbin.chen@nyu.edu}{hongbin.chen@nyu.edu}}

\author[J.-C.\ Mourrat]{Jean-Christophe Mourrat\,\orcidlink{0000-0002-2980-725X}}
\address[Jean-Christophe Mourrat]{Department of Mathematics, ENS Lyon and CNRS, Lyon, France}
\email{\href{mailto:jean-christophe.mourrat@ens-lyon.fr}{jean-christophe.mourrat@ens-lyon.fr}}


\title[Free energy of non-convex multi-species spherical spin glasses]{Free energy of non-convex \\ multi-species spherical spin glasses}

\begin{abstract}
We identify the limit free energy of all mean-field multi-species spin glasses defined on products of spheres. In particular, we do not make any convexity assumption on the covariance function of the Hamiltonian, and we allow for the presence of external fields. We also obtain a one-species reduction of the formula in the special case of balanced models.
\end{abstract}

\maketitle


\section{Introduction}
\label{s.introduction}

The aim of this paper is to identify the limit free energy of all multi-species spherical spin glasses, without imposing any convexity assumption on the covariance function of the Hamiltonian. We first describe the model. Let $\sS$ be a finite set representing species labels. For each $N \in\N$, we let $(I_{N,s})_{s\in\sS}$ be a partition of $[N]:=\{1,\ldots,N\}$, and we set
\begin{equation}
\label{e.def.lambda_Nd}
\lambda_N:=(\lambda_{N,s})_{s\in\sS}, \quad \text{ with } \quad  \lambda_{N,s}:=\frac{|I_{N,s}|}{N}.
\end{equation}
We assume that there exists $\lambda_\infty\in(0,1]^{\sS}$ such that
\begin{equation}
\label{e.lambda_infty}
\lim_{N\to\infty}\lambda_N=\lambda_\infty.
\end{equation}
For every $\sigma,\sigma'\in\R^N$, we define the species overlaps by
\begin{equation}
\label{e.R_Ns_def}
R_N(\sigma,\sigma'):=(R_{N,s}(\sigma,\sigma'))_{s\in\sS}, \quad \text{ with } \quad R_{N,s}(\sigma,\sigma'):=\frac1N\sum_{i\in I_{N,s}}\sigma_i\sigma_i'.
\end{equation}
Let $\xi:\R^{\sS}\to\R$ be a function such that, for some coefficients $(\beta_{s_1,\ldots, s_k})$, we have the absolutely convergent expansion, for every $x \in \R^\sS$,
\begin{equation}
\label{e.xi_power_series_def}
\xi(x) = \sum_{k = 1}^{+\infty} \sum_{s_1, \ldots, s_k \in \sS} \beta^2_{s_1, \ldots, s_k} \prod_{j = 1}^k x_{s_j},
\end{equation}
and let $(H_N(\sigma))_{\sigma\in\R^N}$ be a centered Gaussian field with covariance
\begin{equation}
\label{e.def_H_N}
\E\Ll[H_N(\sigma)H_N(\sigma')\Rr]=N\xi\Ll(R_N(\sigma,\sigma')\Rr).
\end{equation}
As discussed in \cite[Section~6]{mourrat2023free}, under the stated assumption on $\xi$, such a Gaussian field exists; and conversely, if such a Gaussian field exists and $\xi$ admits an absolutely convergent power-series expansion, then this expansion must be of the form in \eqref{e.xi_power_series_def}. We stress that we do not assume any convexity property on $\xi$. 
For every $n \in \N$, let
\begin{equation}
\label{e.surface_measure_def}
\Sph_n:=\{u\in\R^n:|u|^2=n\},
\quad
\text{ and } \quad 
\mu_n:=\text{the normalized surface measure on }\Sph_n.
\end{equation}
We define the reference probability measure $P_N$ on $\R^N$ as
\begin{equation}
\label{e.reference_measure_PN}
\d P_N(\sigma):=\bigotimes_{s\in\sS}\d\mu_{|I_{N,s}|}\Ll((\sigma_i)_{i \in I_{N,s}}\Rr).
\end{equation}
We also fix a deterministic external field vector
\begin{equation}
\label{e.external_field_h}
h=(h^s)_{s\in\sS}\in\R^{\sS},
\end{equation}
so that every coordinate in species $s$ feels the field $h^s$. The goal of the paper is to determine, for each $t \ge 0$, the limit as $N$ tends to infinity of
\begin{equation}
\label{e.def.FN.delta0}
\bar F_N(t,0):=-\frac1N\E\log\int\exp\Ll(\sqrt{2t}\,H_N(\sigma)-Nt\xi(\lambda_N)+\sum_{s\in\sS}h^s\sum_{i\in I_{N,s}}\sigma_i\Rr)P_N(\d\sigma).
\end{equation}
The compensation term $-Nt\xi(\lambda_N)$ is harmless and is here to simplify the expression of the derivative in $t$ of $\bar F_N$, and thus of the resulting partial differential equations.

Let $\mcl Q$ be the set of right-continuous nondecreasing functions $q:[0,1)\to\R_+$, with $\R_+ = [0,+\infty)$. For every $r\in[1,\infty]$, we set $\mcl Q_r:=\mcl Q\cap L^r[0,1)$, and for every $a \ge 0$, we write $\mcl Q_{\infty, \le a}$ to denote the set of functions in $\mcl Q_\infty$ that remain bounded by $a$. We will in fact consider a larger family of quantities than those in \eqref{e.def.FN.delta0}, and the second argument of $\bar F_N$ can be any element of $\mcl Q_1^\sS$. This additional path variable encodes the ultrametric structure of a random magnetic field that we add to the model, see \eqref{e.WNq_def}--\eqref{e.F_N_spherical}; the value displayed in \eqref{e.def.FN.delta0} is that obtained for the choice of the null path for each of the species. As we set $t = 0$ in this extended free energy, the Hamiltonian~$H_N$ disappears, and we can determine, for each $q \in \mcl Q_1^\sS$, the limit 
\begin{equation*}  
\psi(q) := \lim_{N \to +\infty} \bar F_N(0,q).
\end{equation*}
We point out that the function $\psi$ depends on $h$, even though we keep it implicit in the notation. For every $t\ge0$, $q,q'\in\mcl Q_1^{\sS}$, and $p\in\mcl Q_\infty^{\sS}$, we define
\begin{equation}
\label{e.mcJ}
\mcl J_{t,q}(q',p):=\psi(q')+\la q-q',p\ra_{L^2}+t\int_0^1\xi(p(r))\d r,
\end{equation}
where $\la\kappa,p\ra_{L^2}:=\int_0^1\kappa(u)\cdot p(u)\d u$ denotes the pairing of $\R^{\sS}$-valued paths. For every $\lambda\in\R_+^{\sS}$, we write $\mcl Q_{\infty, \le \lambda}^\sS := \prod_{s \in \sS} \mcl Q_{\infty, \le \lambda_s}$. Here is our main result. 
\begin{theorem}
\label{t.main}
For every $t\ge0$ and $q\in\mcl Q_2^{\sS}$, we have 
\begin{equation}
\label{e.main.1}
\lim_{N\to\infty}\bar F_N(t,q)=\sup_{p\in\mcl Q_{\infty,\le\lambda_\infty}^{\sS}}\ \inf_{q'\in\mcl Q_\infty^{\sS}}\ \mcl J_{t,q}(q',p).
\end{equation}
Moreover, denoting this limit by $f(t,q)$, the function $f$ is the Lipschitz viscosity solution to
\begin{equation}
\label{e.main.hj}
\begin{cases}
\partial_t f-\displaystyle\int_0^1\xi\Ll(\partial_q f\Rr)=0, &\text{on }\R_+\times\mcl Q_2^{\sS},\\
f(0,\cdot)=\psi, &\text{on }\mcl Q_2^{\sS}.
\end{cases}
\end{equation}
In particular, $f(t,0)$ is the limit of the quantity in~\eqref{e.def.FN.delta0}.
\end{theorem}

It will be shown in forthcoming work \cite{chen2026convex} that the variational formula on the right side of \eqref{e.main.1} is related to the Crisanti--Sommers formula \cite{crisanti1992spherical} by a simple change of variables. When $\xi$ is convex over $\R_+^\sS$, the limit free energy (with $q = 0$) was already identified in \cite{bates2022crisanti, bates2022free, chen2013aizenman, tal.sph}. The approach pursued there is to first identify the limit free energy as a Parisi formula, and then to observe that this formula coincides with the Crisanti--Sommers formula. With our notation, the Parisi formula takes the form $\sup_{q'} \inf_p \mcl J_{t,q}(q',p)$, and one can check that this expression becomes invalid for non-convex $\xi$, as was already observed in the context of $\pm 1$ spins in \cite[Section~6]{mourrat2021nonconvex}. Here, we instead prove the formula \eqref{e.main.1} directly, as is required. 

Prior to the present paper, the case when $\xi$ is not convex was not very well understood. In the case when $|\sS| = 2$, $h = 0$, and $\xi(x,y) = x^p y^q$ with $p, q \ge 96$, the limit of the ground-state energy was determined in \cite{kivimae2023ground}. Still with $h=0$ but with $\xi$ an arbitrary monomial, the limit free energy was determined in \cite{subag2023tap2, subag2025tap1} under the assumption that such a limit exists for a family of auxiliary models derived from $\xi$. Since Theorem~\ref{t.main} applies to every such model, the results of \cite{subag2023tap2, subag2025tap1} are now unconditional. The case of $|\sS| = 2$, $h = 0$, and $\xi(x,y) = xy$ has been obtained unconditionally in \cite{aufchebi, bailee}. For $h = 0$ and special choices of the parameter $\lambda_\infty$, the ground-state energy of models with  $\xi(x,y) = x^p y^q$ or with $\xi(x_1, \ldots, x_d) = x_1 \cdots x_d$ has also been obtained unconditionally in \cite{bates2025balanced, dartois2024injective}. More precisely, the authors of \cite{bates2025balanced} introduce a class of models called \emph{balanced}, where particular symmetries suggest a possible reduction to an equivalent model with a single species, and they show the validity of this reduction for convex $\xi$, for the ground-state energy of the cases mentioned above, and at high temperature for all $\xi$; see also \cite{bates2023parisi, issa2024existence}. We prove that this reduction is indeed valid for all balanced spherical models in Section~\ref{s.balanced}. All the non-convex cases in which these earlier works identified the limit free energy or ground-state energy display at most one step of replica symmetry breaking, whereas we handle models with a potentially infinite number of steps of replica symmetry breaking here.

The proof of Theorem~\ref{t.main} follows the strategy introduced in \cite{chen2026ising} for multi-species models with $\pm 1$ spins and no external field. In contrast to the situation with $\pm 1$ spins, for the spherical model the initial condition $\psi$ is convex for all values of the external field, and this is the reason why we can cover all models in this case. Using this convexity property and the results of \cite{chen2022hamilton, mourrat2021nonconvex, mourrat2023free} based on Hamilton--Jacobi equations, we can readily obtain the bound
\begin{equation*}  
\liminf_{N \to +\infty} \bar F_N(t,q) \ge \sup_{p\in\mcl Q_{\infty,\le\lambda_\infty}^{\sS}}\ \inf_{q'\in\mcl Q_\infty^{\sS}}\ \mcl J_{t,q}(q',p),
\end{equation*}
as will be explained in more detail in Section~\ref{s.hj}. This is our replacement for Guerra's interpolation bound \cite{gue03} used for models with convex $\xi$. To obtain the converse bound, the cavity argument of \cite{bates2022free, chen2013aizenman}  guarantees the existence of a path $p \in \mcl Q_\infty^\sS$ such that
\begin{equation}  
\label{e.FN.upper}
\limsup_{N \to +\infty} \bar F_N(t,q) \le \mcl J_{t,q}(q + t \nabla \xi(p), p). 
\end{equation}
This argument remains valid for non-convex $\xi$, but is no longer sufficient to conclude. In order to complete the argument, we will show that we can realize \eqref{e.FN.upper} for a path $p \in \mcl Q_\infty^\sS$ that satisfies the relation
\begin{equation}
\label{e.critical.p}
p = \dr_q \psi(q + t \nabla \xi(p)).
\end{equation}
Cavity arguments are more difficult to implement for spherical models than for models with $\pm 1$ spins, since the uniform measure on a sphere is clearly not a product measure. For the purpose of proving \eqref{e.FN.upper} only, the inequality goes in the right direction for us to simply restrict integrals to ``good sets'' where we expect certain variables to concentrate. In order to establish \eqref{e.critical.p}, we also rely on certain cavity computations, however here we need to establish an equality, so the derivation of \eqref{e.critical.p} will necessarily be more delicate. We give a detailed roadmap of our approach to these cavity calculations in the beginning of Section~\ref{s.cavity}.

\medskip

\noindent \textbf{Organization of the paper.} Section~\ref{s.enriched_properties} defines the enriched free energy, identifies its initial condition~$\psi$, and records their regularity properties. It ends with the precise statement of the claim \eqref{e.FN.upper}--\eqref{e.critical.p}, in Proposition~\ref{p.crit_pt_bdd_spherical}, and with the short proof of Lemma~\ref{l.inner_infimum} explaining why this would indeed provide the upper bound for Theorem~\ref{t.main}. Section~\ref{s.cavity} contains the cavity calculations, whose outcome is the general cavity estimate of Theorem~\ref{t.spherical_cavity_lim}. Section~\ref{s.crit_pt_bdd} completes the proof of the claim \eqref{e.FN.upper}--\eqref{e.critical.p}, under a number of restrictions including that $\lambda_\infty$ must have rational coordinates. Section~\ref{s.hj} defines viscosity solutions of~\eqref{e.main.hj}, identifies the solution through the Hopf formula, and proves the lower bound for Theorem~\ref{t.main}. Section~\ref{s.conclusion} combines the two bounds into the identification of the limit free energy under the restrictions inherited from previous arguments, and then lifts these restrictions to complete the proof of Theorem~\ref{t.main}. Section~\ref{s.balanced} discusses balanced models. Appendix~\ref{s.app_cascade} collects the identities for tilted discrete cascades used in Section~\ref{s.enriched_properties}, and Appendix~\ref{s.app_concentration} proves the concentration of the free energy for the Hamiltonians appearing in the cavity calculations.

\medskip

\noindent \textbf{Statement of AI use.} OpenAI's ChatGPT 5.5 and 5.6 and Anthropic's Claude Fable 5 and 5.1 were used to help us in the development of the proofs of some technical parts and in the polishing of the paper. 

\medskip

\noindent \textbf{Acknowledgements.} HBC acknowledges funding from the NYU Shanghai Start-Up Fund and support from the NYU--ECNU Institute of Mathematical Sciences at NYU Shanghai. JCM acknowledges the support of the MSCA grant SLOHD (101203974).

\section{Enriched model and regularity properties}
\label{s.enriched_properties}

In this section, we introduce the enriched free energy $\bar F_N(t,q)$, in which the spins are coupled to an external Gaussian field organized along a Poisson--Dirichlet probability cascade. We show that $\bar F_N$ is Lipschitz continuous, differentiable, and semi-concave in $q$. These properties are collected in Subsection~\ref{s.enriched_def}, together with the invariance of the overlap law, which is used throughout the cavity computations in the next section.

We also identify the initial condition $\psi=\lim_{N\to\infty}\bar F_N(0,\cdot)$. Following Talagrand~\cite{tal.sph}, we replace the uniform measure on the sphere by a Gaussian measure of variance $1/b$, where $b$ is a Lagrange multiplier for the spherical constraint. The one-coordinate version of this model, the Gaussian block, is introduced in Subsection~\ref{s.gaussian_block}, and we initially define $\psi$ in terms of this object. The regularity and convexity of $\psi$, as well as an explicit formula for its derivative, are proved in Subsection~\ref{s.initial_regularity}, and the convergence $\bar F_N(0,q)\to\psi(q)$ is proved in Subsection~\ref{s.initial_condition}. The multi-species version of the Gaussian block reappears in Section~\ref{s.cavity} as the limit law of the cavity coordinates.

Finally, in Subsection~\ref{s.critical_paths}, we give the precise statement of the claim \eqref{e.FN.upper}--\eqref{e.critical.p}, in Proposition~\ref{p.crit_pt_bdd_spherical}, whose proof occupies Sections~\ref{s.cavity} and~\ref{s.crit_pt_bdd}, and we explain how this proposition enters into the proof of Theorem~\ref{t.main}. 

\subsection{Cascades and the enriched free energy}
\label{s.enriched_def}

Recall the reference measure $P_N$ from~\eqref{e.reference_measure_PN}. For every $u\in\R^N$ and $s\in\sS$, we write $u^{(s)}:=(u_i)_{i\in I_{N,s}}\in\R^{I_{N,s}}$, and we denote by
\begin{equation}
\label{e.Sigma_N_def}
\Sigma_N:=\Ll\{\sigma\in\R^N:\ \text{for every } s\in\sS,\ \sigma^{(s)}\in\Sph_{|I_{N,s}|}\Rr\}
\end{equation}
the support of $P_N$.

\subsubsection*{Continuous cascades}
Let $\fR$ be a Poisson--Dirichlet probability cascade whose overlap is uniformly distributed over $[0,1]$. This is a random probability measure on the unit sphere of an abstract separable Hilbert space, whose scalar product we denote by $\wedge$; we refer to~\cite[Chapter~2]{pan} or \cite[Section~4]{chen2025free} for its construction as a limit of discrete cascades, and to~\cite[Chapter~2]{pan} or \cite[Chapter~5]{HJbook} for the properties of discrete cascades that we use below. We write $\mfk U:=\supp\fR$ for the (random) support of $\fR$, and $\alpha,\alpha',\alpha^1,\alpha^2,\ldots$ for elements of $\mfk U$. For every $n\in\N$ and every integrable function $g$ of $\alpha^1,\ldots,\alpha^n$, we write
\begin{equation}
    \la g\ra_{\fR}:=\int g(\alpha^1,\ldots,\alpha^n)\prod_{\ell=1}^n\fR(\d\alpha^\ell).
    \label{e.fR_bracket_def}
\end{equation}
The bracket $\la\cdot\ra_{\fR}$ is conditional on $\fR$, while $\E\la\cdot\ra_{\fR}$ also averages over the randomness of $\fR$. By construction, $\alpha\wedge\alpha'$ is uniformly distributed on $[0,1]$ under $\E\la\cdot\ra_\fR$, and $\mfk U$ is almost surely an ultrametric set.

Every path $q\in\mcl Q_\infty$ is extended to $[0,1]$ by setting
\begin{equation}
\label{e.continuous_endpoint_convention}
    q(1):=\lim_{u\nearrow1}q(u).
\end{equation}
By~\cite[Proposition~4.1]{chen2025free}, for every $q\in\mcl Q_\infty$ and almost every realization of $\fR$, there exists a centered Gaussian process $(w^q(\alpha))_{\alpha\in\mfk U}$, jointly measurable in $\alpha$ and in the underlying randomness, with covariance
\begin{equation}
\label{e.W_field_covariance}
    \E\Ll[w^{q}(\alpha)w^{q}(\alpha')\Rr]=q(\alpha\wedge\alpha'),
    \qquad \alpha,\alpha'\in\mfk U.
\end{equation}
In particular, the variance of $w^q(\alpha)$ is $q(1)$. Whenever we take expectations of functionals involving both $\fR$ and such Gaussian processes, we first average over the Gaussian randomness, conditionally on $\fR$, and then over $\fR$; see~\cite[Lemma~4.5]{chen2025free}.

The cascade often appears tilted by a function of these Gaussian processes, and the property of the continuous cascade that we use most often is that such a tilt does not change the law of the overlaps. Let $n\in\N$, let $q_1,\ldots,q_n\in\mcl Q_\infty$, and, conditionally on $\fR$, let $w^{q_1},\ldots,w^{q_n}$ be independent Gaussian fields with the covariances~\eqref{e.W_field_covariance}; we write $\mathbf w(\alpha):=(w^{q_i}(\alpha))_{i\le n}\in\R^n$. Given a Lipschitz function $\mathbf g:\R^n\to\R$, let $\fR^{\mathbf g}$ be the random probability measure on $\mfk U$ with density proportional to $\exp(\mathbf g(\mathbf w(\alpha)))$ with respect to $\fR$, and let $\la\cdot\ra_{\fR^{\mathbf g}}$ be defined as in~\eqref{e.fR_bracket_def}, with $\fR^{\mathbf g}$ in place of $\fR$. By~\cite[Proposition~4.8]{chen2025free}, applied to the diagonal matrix-valued path $\diag(q_1,\ldots,q_n)$, the law of the overlap array is invariant under this tilt: for every $n'\in\N$ and every bounded measurable $\mathbf h:[0,1]^{n'\times n'}\to\R$,
\begin{equation}
\label{e.invariance_cascade}
    \E\la\mathbf h\Ll((\alpha^\ell\wedge\alpha^{\ell'})_{\ell,\ell'\le n'}\Rr)\ra_{\fR^{\mathbf g}}
    =\E\la\mathbf h\Ll((\alpha^\ell\wedge\alpha^{\ell'})_{\ell,\ell'\le n'}\Rr)\ra_{\fR};
\end{equation}
in particular, $\alpha\wedge\alpha'$ remains uniformly distributed on $[0,1]$ under $\E\la\cdot\ra_{\fR^{\mathbf g}}$. In the enriched free energy introduced below, we take $n=N$ and $\mathbf g(\mathbf w(\alpha))=\log\int\exp(H^{t,q}_N(\sigma,\alpha))P_N(\d\sigma)$, with $H^{t,q}_N$ as in~\eqref{e.enriched_H}; this function is Lipschitz in the fields, since $|\sigma|^2=N$ on $\Sigma_N$, and since the disorder $H_N$ is independent of the cascade and of the fields, we apply~\eqref{e.invariance_cascade} conditionally on $H_N$.

\subsubsection*{Discrete cascades}
Let $k\ge0$ and $0=m_0<m_1<\cdots<m_k<m_{k+1}=1$. A path $\rho\in\mcl Q_\infty$ is called a finite-step path on the partition $(m_j)_{0\le j\le k+1}$ if it can be written as
\begin{equation}
\label{e.finite_step_path}
    \rho=\sum_{j=0}^k\rho_j\one_{[m_j,m_{j+1})},
    \qquad 0=\rho_{-1}\le\rho_0\le\cdots\le\rho_k,
\end{equation}
so that $\rho(1)=\rho_k$. We denote by $\mcl A:=\N^0\cup\N^1\cup\cdots\cup\N^k$ the infinitary tree of depth $k$ (with $\N^0 = \{\emptyset\}$), and for every $\alpha = (n_1, \ldots, n_k) \in \N^k$ and $\ell \in \{0, \ldots, k\}$, we write $\alpha_{|\ell} = (n_1, \ldots, n_\ell)$ for the ancestor of $\alpha$ at depth $\ell$. 
The discrete Poisson--Dirichlet probability cascade with parameters $(m_1,\ldots,m_k)$ is a random family of probability weights $(v_\alpha)_{\alpha\in\N^k}$ indexed by the leaves of $\mcl A$. Its construction is explained in detail in \cite[Section~2.3]{pan} or in~\cite[Section~5.6]{HJbook}; briefly, the probability weights are proportional to the products $w_\alpha:=\prod_{j=1}^ku_{\alpha_{|j}}$, where, for every node $\beta\in\N^{j-1}$, the weights $(u_{\beta i})_{i\ge1}$ attached to the children of $\beta$ form a Poisson--Dirichlet point process with parameter $m_j$, independently over $\beta$ and $j$. When $k=0$, the tree is reduced to its root and $v_\emptyset=1$. For two leaves $\alpha,\alpha'$, we write $\alpha\wedge\alpha'\in\{0,\ldots,k\}$ for the depth of their most recent common ancestor; under two independent samples from $(v_\alpha)_{\alpha\in\N^k}$, averaged over the weights, this depth satisfies
\begin{equation}
\label{e.discrete_overlap_law}
    \P\Ll(\alpha\wedge\alpha'=j\Rr)=m_{j+1}-m_j,\qquad j\in\{0,\ldots,k\},
\end{equation}
by~\cite[Theorem~5.28]{HJbook}. Let $(z_\beta)_{\beta\in\mcl A}$ be independent standard Gaussian random variables attached to the nodes of the tree, which we call marks. For a finite-step path $\rho$ as in~\eqref{e.finite_step_path}, the discrete field
\begin{equation}
\label{e.discrete_field}
    \msf w^\rho(\alpha):=\sum_{j=0}^k(\rho_j-\rho_{j-1})^{1/2}z_{\alpha_{|j}},
    \qquad\alpha\in\N^k,
\end{equation}
is a centered Gaussian field with covariance $\E\msf w^\rho(\alpha)\msf w^\rho(\alpha')=\rho_{\alpha\wedge\alpha'}$. 

The relation between the continuous and the discrete constructions is the following. We denote by $\pi_m:[0,1]\to\{m_0,\ldots,m_k\}$ the coarse-graining map $\pi_m(u):=m_j$ for $u\in[m_j,m_{j+1})$, with $\pi_m(1):=m_k$, so that a finite-step path on the partition is a function of $\pi_m(u)$. Since $\alpha\wedge\alpha'$ is uniformly distributed on $[0,1]$ under $\E\la\cdot\ra_\fR$, the coarse-grained overlap $\pi_m(\alpha\wedge\alpha')$ of the continuous cascade takes the value $m_j$ with probability $m_{j+1}-m_j$, which is also the law of $m_{\alpha\wedge\alpha'}$ under the discrete cascade, by~\eqref{e.discrete_overlap_law}. Both cascades satisfy the Ghirlanda--Guerra identities, which determine the law of the overlap array from that of a single overlap, and this is why the discrete cascade can replace the continuous one when the paths are finite-step paths on the partition; see the proof of~\cite[Proposition~4.6]{chen2025free}. Precisely, let $q_1,\ldots,q_n$ be finite-step paths on the partition $(m_j)_{0\le j\le k+1}$, let $\mathbf w$, $\mathbf g$ and $\fR^{\mathbf g}$ be defined from these paths as before~\eqref{e.invariance_cascade}, and let $\msf w^{q_1},\ldots,\msf w^{q_n}$ be the discrete fields~\eqref{e.discrete_field} built from independent families of marks $(z_{\beta,1})_{\beta\in\mcl A},\ldots,(z_{\beta,n})_{\beta\in\mcl A}$, with $\msf w(\alpha):=(\msf w^{q_i}(\alpha))_{i\le n}$. Then, by~\cite[Proposition~4.6]{chen2025free},
\begin{equation}
\label{e.equiv_cascade_fe}
    \E\log\sum_{\alpha\in\N^k}v_\alpha\exp\Ll(\mathbf g(\msf w(\alpha))\Rr)=\E\log\int_{\mfk U}\exp\Ll(\mathbf g(\mathbf w(\alpha))\Rr)\fR(\d\alpha),
\end{equation}
and, for every $n'\in\N$ and every bounded measurable $\mathbf h$,
\begin{equation}
\label{e.equiv_cascade_overlaps}
    \E\la\mathbf h\Ll((m_{\alpha^\ell\wedge\alpha^{\ell'}})_{\ell,\ell'\le n'}\Rr)\ra^{\msf d}_{\mathbf g}
    =\E\la\mathbf h\Ll((\pi_m(\alpha^\ell\wedge\alpha^{\ell'}))_{\ell,\ell'\le n'}\Rr)\ra_{\fR^{\mathbf g}},
\end{equation}
where $\la\cdot\ra^{\msf d}_{\mathbf g}$ denotes the average with respect to the tilted discrete cascade $v^{\mathbf g}_\alpha\propto v_\alpha\exp(\mathbf g(\msf w(\alpha)))$, with independent replicas $\alpha^1,\alpha^2,\ldots$. In the Gaussian block of Subsection~\ref{s.gaussian_block}, we use~\eqref{e.equiv_cascade_fe} with the quadratic function $\mathbf g(x)=(\sqrt2x_1+a)^2/(2b)$, which is not Lipschitz; since it is nonnegative,~\eqref{e.equiv_cascade_fe} for this $\mathbf g$ follows from~\eqref{e.equiv_cascade_fe} for the bounded Lipschitz functions $\mathbf g\wedge L$ by monotone convergence as $L\to\infty$, both sides being then possibly equal to $+\infty$. Further identities for discrete cascades, concerning the law of the marks along the leaves sampled from a tilted cascade, are recalled in Appendix~\ref{s.app_cascade}.

\subsubsection*{The enriched Hamiltonian}
Let $q=(q^s)_{s\in\sS}\in\mcl Q_\infty^{\sS}$. For every $s\in\sS$ and $i\in I_{N,s}$, let $(w_i^{q^s}(\alpha))_{\alpha\in\mfk U}$ be a centered Gaussian process with covariance~\eqref{e.W_field_covariance} for the path $q^s$; conditionally on $\fR$, these processes are independent as $s$ and $i$ vary, and independent of $H_N$. We set
\begin{equation}
\label{e.WNq_def}
W_N^q(\sigma,\alpha):=\sum_{s\in\sS}\sum_{i\in I_{N,s}}w_i^{q^s}(\alpha)\sigma_i,
\qquad \sigma\in\R^N,\quad\alpha\in\mfk U.
\end{equation}
By~\eqref{e.R_Ns_def} and~\eqref{e.W_field_covariance}, this is a centered Gaussian field with covariance
\begin{equation}
\label{e.WNq_covariance}
\E\Ll[W_N^q(\sigma,\alpha)W_N^q(\sigma',\alpha')\Rr]=Nq(\alpha\wedge\alpha')\cdot R_N(\sigma,\sigma'),
\end{equation}
where $a\cdot b:=\sum_{s\in\sS}a_sb_s$ for $a,b\in\R^{\sS}$. For $t\ge0$, we define the enriched Hamiltonian
\begin{equation}
\label{e.enriched_H}
\begin{aligned}
H_N^{t,q}(\sigma,\alpha):={}&\sqrt{2t}\,H_N(\sigma)-Nt\xi(\lambda_N)+\sqrt2\,W_N^q(\sigma,\alpha)-Nq(1)\cdot\lambda_N \\
&+\sum_{s\in\sS}h^s\sum_{i\in I_{N,s}}\sigma_i,
\end{aligned}
\end{equation}
where $q(1):=(q^s(1))_{s\in\sS}$. Since $R_N(\sigma,\sigma)=\lambda_N$ for every $\sigma\in\Sigma_N$, the deterministic terms $Nt\xi(\lambda_N)$ and $Nq(1)\cdot\lambda_N$ are the variances of $\sqrt tH_N(\sigma)$ and of $W_N^q(\sigma,\alpha)$, so that the random part of~\eqref{e.enriched_H} is a Gaussian field with its self-overlap correction, as in~\cite{chen2026ising,chen2025free}. The partition function and the free energy are
\begin{align}
Z_N(t,q)&:=\iint\exp\Ll(H_N^{t,q}(\sigma,\alpha)\Rr)P_N(\d\sigma)\fR(\d\alpha),
\label{e.ZN_def}\\
\bar F_N(t,q)&:=-\frac1N\E\log Z_N(t,q).
\label{e.F_N_spherical}
\end{align}
Note the minus sign in~\eqref{e.F_N_spherical}. If $q=0$, then~\eqref{e.F_N_spherical} agrees with~\eqref{e.def.FN.delta0}. For every $n\in\N$ and every bounded measurable function $f$ of $(\sigma^\ell,\alpha^\ell)_{1\le\ell\le n}$, we write
\begin{equation}
\label{e.gibbs_bracket_N}
    \la f\ra_{N,t,q}:=\frac{1}{Z_N(t,q)^n}\int f\Ll((\sigma^\ell,\alpha^\ell)_{1\le\ell\le n}\Rr)\prod_{\ell=1}^n e^{H_N^{t,q}(\sigma^\ell,\alpha^\ell)}P_N(\d\sigma^\ell)\fR(\d\alpha^\ell)
\end{equation}
for the associated Gibbs bracket; it is conditional on all the randomness in~\eqref{e.enriched_H}, and $\E\la\cdot\ra_{N,t,q}$ also averages over this randomness. The pairs $(\sigma^\ell,\alpha^\ell)$ are called replicas, and we often write $(\sigma,\alpha)$, $(\sigma',\alpha')$ for two of them.

\subsubsection*{Notation for paths}
For $r\in[1,\infty]$, we equip $\mcl Q_r^{\sS}=\prod_{s\in\sS}\mcl Q_r$ with the norm $|q|_{L^r}$ of $L^r([0,1);\R^{\sS})$, where $\R^{\sS}$ carries the Euclidean norm $|\cdot|$, and we write $\la\kappa,q\ra_{L^2}:=\int_0^1\kappa(u)\cdot q(u)\d u$. For $a\ge0$ and $\lambda\in\R_+^{\sS}$, we recall that $\mcl Q_{\infty,\le a}$ denotes the set of paths in $\mcl Q_\infty$ bounded by $a$, and that $\mcl Q_{\infty,\le\lambda}^{\sS}=\prod_{s\in\sS}\mcl Q_{\infty,\le\lambda_s}$. Finite-step paths are dense in $\mcl Q_r^{\sS}$ for every $r\in [1,\infty]$.

We say that a function $g:\mcl Q_2^{\sS}\to\R$ is Gateaux differentiable at $q\in\mcl Q_2^{\sS}$ if there exists $y\in L^2([0,1);\R^{\sS})$ such that, for every $\kappa\in L^2([0,1);\R^{\sS})$ for which $q+\eps\kappa\in\mcl Q_2^{\sS}$ for all sufficiently small $\eps>0$, we have
\begin{equation*}
    \lim_{\eps\searrow0}\frac{g(q+\eps\kappa)-g(q)}{\eps}=\la\kappa,y\ra_{L^2},
\end{equation*}
and if $y$ is unique with this property; we then write $\dr_qg(q):=y$. Uniqueness holds automatically here, since the admissible directions include $\mcl Q_\infty^{\sS}$, and $\mcl Q_\infty^{\sS}-\mcl Q_\infty^{\sS}$ is dense in $L^2([0,1);\R^{\sS})$.

\subsubsection*{Regularity of the enriched free energy}

\begin{proposition}[Differentiability of $\bar F_N$]
\label{p.F_N_smooth}
For every $N\in\N$, $t,t'\ge0$, and $q,q'\in\mcl Q_\infty^{\sS}$,
\begin{equation}
    \Ll|\bar F_N(t,q)-\bar F_N(t',q')\Rr|\le |q-q'|_{L^1}+|t-t'|\sup_{a\in\prod_s[-\lambda_{N,s},\lambda_{N,s}]}|\xi(a)|.
    \label{e.F_N_Lipschitz}
\end{equation}
In particular, $\bar F_N$ extends uniquely by continuity to $\R_+\times\mcl Q_1^{\sS}$, and~\eqref{e.F_N_Lipschitz} remains valid there. The restriction of $\bar F_N$ to $\R_+\times\mcl Q_2^{\sS}$ is Gateaux differentiable in $q$, with
\begin{equation}
    \dr_q\bar F_N(t,q)\in\mcl Q_{\infty,\le\lambda_N}^{\sS},
    \qquad (t,q)\in\R_+\times\mcl Q_2^{\sS},
    \label{e.bounds.der.FN}
\end{equation}
and, for every $q\in\mcl Q_\infty^{\sS}$ and $\kappa\in L^2([0,1);\R^{\sS})$,
\begin{equation}
    \la\kappa,\dr_q\bar F_N(t,q)\ra_{L^2}=\E\la\kappa(\alpha\wedge\alpha')\cdot R_N(\sigma,\sigma')\ra_{N,t,q}.
    \label{e.def.der.FN}
\end{equation}
\end{proposition}

\begin{proof}
This is~\cite[Proposition~5.1]{chen2025free}, in the diagonal multi-species form of~\cite[Proposition~4.1]{chen2024ms}, and we only explain why the argument applies to the spherical reference measure. The only property of $P_N$ that it uses is the bound $|R_{N,s}(\sigma,\sigma')|\le\lambda_{N,s}$, which follows from the identity $R_N(\sigma,\sigma)=\lambda_N$ and the Cauchy--Schwarz inequality; in particular, $|R_N(\sigma,\sigma')|\le|\lambda_N|\le1$ for the Euclidean norm, since the entries of $\lambda_N$ are nonnegative and sum to~$1$. The membership of $\dr_q\bar F_N(t,q)$ in the cone of nonnegative nondecreasing paths is obtained there by duality from a monotonicity property of $\bar F_N(t,\cdot)$, see \cite[Proposition~3.8]{mourrat2023free}, whose proof only assumes that the reference measure is supported in the ball of radius $\sqrt N$, as is $P_N$. 
\end{proof}

\subsubsection*{Semi-concavity}
For every $c\in(0,1]$, we set
\begin{equation*}
    \mcl Q_{\uparrow,c}^{\sS}:=\Big\{q\in\mcl Q_1^{\sS}:q^s(0)=0\text{ and }q^s(v)-q^s(u)\ge c(v-u)\text{ for every }s\in\sS\text{ and }0\le u\le v<1\Big\},
\end{equation*}
and
\begin{equation*}
    \mcl Q_{\infty,\uparrow}^{\sS}:=\mcl Q_\infty^{\sS}\cap\bigcup_{c\in(0,1]}\mcl Q_{\uparrow,c}^{\sS}.
\end{equation*}
For an increasing path $q$, we denote its distributional derivative by $\dot q$. The next result is a semi-concavity estimate for the free energy. It is used in Section~\ref{s.crit_pt_bdd} to compare the derivatives in $q$ of two free energies that are close to each other. 

\begin{proposition}[Semi-concavity of the free energy]
\label{p.semi-concave}
There exists a constant $C<\infty$, depending only on $\xi$, such that, for every $N\in\N$, $c\in(0,1]$, $t\ge0$, $q,q'\in\mcl Q_{\uparrow,c}^{\sS}$ with $\dot q-\dot q'\in L^2([0,1);\R^{\sS})$, and $r\in[0,1]$,
\begin{align}
    (1-r)\bar F_N(t,q)+r\bar F_N(t,q')-\bar F_N\Ll(t,(1-r)q+rq'\Rr)\le Cr(1-r)c^{-2}|\dot q-\dot q'|_{L^2}^2.
    \label{e.semi_concave_q_spherical}
\end{align}
The estimate is unchanged if one adds to the Hamiltonian a Gaussian field independent of the cascade field $W_N^q$, whose law does not depend on $q$, or a deterministic affine function of $q$.
\end{proposition}

\begin{proof}
This is~\cite[Propositions~3.7 and~3.8]{chen2025free} at fixed $t$, in the diagonal multi-species form of~\cite[Proposition~4.5]{chen2024ms}; since $t$ is fixed, the lower bound on $t$ required there for the joint semi-concavity in $(t,q)$ is not needed, and the ellipticity condition of the matrix-valued statement is void for scalar increments. The proofs use only the covariance~\eqref{e.WNq_covariance} of the cascade field, the self-overlap correction in~\eqref{e.enriched_H}, and the bound $|R_{N,s}(\sigma,\sigma')|\le\lambda_{N,s}$ verified in the proof of Proposition~\ref{p.F_N_smooth}. For the last assertion, recall that the proof of~\cite[Proposition~3.7]{chen2025free} consists of H\"older's inequality applied in the amplitudes $(q^s_j-q^s_{j-1})^{1/2}$ of a finite-step path, for which the terms of the Hamiltonian that do not depend on the amplitudes play no role, and of a bound on the derivatives of the free energy in the amplitudes obtained by Gaussian integration by parts. A Gaussian field independent of the cascade field, with a law that does not depend on $q$, is such a term, and it does not affect the derivative bound; a deterministic affine function of $q$ contributes zero to the left-hand side of~\eqref{e.semi_concave_q_spherical}.
\end{proof}

\subsection{The Gaussian block and the Crisanti--Sommers functional}
\label{s.gaussian_block}

We now introduce the function $\psi$ appearing in Theorem~\ref{t.main}. It is defined through a one-coordinate Gaussian model, the Gaussian block, in which a Lagrange multiplier $b$ replaces the spherical constraint, and through an explicit functional of Crisanti--Sommers type in the variables $(\rho,b)$, which gives $\psi(\rho)$ after minimization over $b$. We then study the minimizing multiplier $b_a(\rho)$, introduce the path that will be identified with the derivative of $\psi^\circ_a$ in Subsection~\ref{s.initial_regularity}, and characterize $b_a(\rho)$ by a second-moment identity under the block. We treat one species at a time, and write $\rho\in\mcl Q_\infty$ for the path and $a\in\R$ for the external field; the multi-species function $\psi$ is a weighted sum of scalar functions. Throughout, $\gamma(\d\tau):=(2\pi)^{-1/2}e^{-\tau^2/2}\d\tau$ denotes the standard Gaussian measure on $\R$.

\subsubsection*{The Gaussian block}
The idea behind the identification of $\lim_{N\to\infty}\bar F_N(0,q)$ is to replace the uniform measure on the sphere of radius $\sqrt n$ by the Gaussian measure $\propto e^{-b|x|^2/2}\d x$ on $\R^n$, under which the coordinates decouple, and to choose $b>0$ so that this Gaussian measure concentrates on the sphere; the parameter $b$ is a Lagrange multiplier for the spherical constraint. This is the approach of~\cite[Section~3]{tal.sph}. Since the enriched Hamiltonian at $t=0$ is linear in the spins, the Gaussian model factorizes, conditionally on the cascade, into $n$ copies of the following one-coordinate object. For $\rho\in\mcl Q_\infty$, $a\in\R$, and $b\in\R$, let
\begin{equation}
\label{e.scalar_direct_initial_functional}
    Z^\circ_{a,b}(\rho):=\int_\R\int_{\mfk U}\exp\Ll((\sqrt2w^\rho(\alpha)+a)\tau-\frac{b-1}{2}\tau^2\Rr)\fR(\d\alpha)\gamma(\d\tau),
\end{equation}
which may be infinite, and let $\la\cdot\ra_{\rho,b}^{\circ}$ denote the Gibbs bracket associated with the measure $\exp((\sqrt2w^\rho(\alpha)+a)\tau-\frac{b-1}2\tau^2)\fR(\d\alpha)\gamma(\d\tau)$ on $\R\times\mfk U$, when $Z^\circ_{a,b}(\rho)<\infty$; replicas under this bracket are denoted by $(\tau,\alpha),(\tau',\alpha')$, and the dependence on $a$ is kept implicit. We call this object the Gaussian block with one coordinate. Since $\gamma(\d\tau)\exp(-\frac{b-1}2\tau^2)=(2\pi)^{-1/2}e^{-b\tau^2/2}\d\tau$, the Gaussian integral over $\tau$ can be computed, and, when $b>0$,
\begin{equation}
\label{e.Zcirc_tau_integrated}
    Z^\circ_{a,b}(\rho)=b^{-1/2}\int_{\mfk U}\exp\Ll(\frac{(\sqrt2w^\rho(\alpha)+a)^2}{2b}\Rr)\fR(\d\alpha).
\end{equation}
Conditionally on $\alpha$, the coordinate $\tau$ is Gaussian with mean $(\sqrt2w^\rho(\alpha)+a)/b$ and variance $1/b$ under $\la\cdot\ra^\circ_{\rho,b}$. The quadratic tilt in~\eqref{e.Zcirc_tau_integrated} makes the block finite only for $b$ large enough, and the threshold is determined in terms of the quantity $\mathcal{K}$ we introduce next.

\subsubsection*{The Crisanti--Sommers functional}
For $\rho\in\mcl Q_\infty$, set
\begin{align}
    \mcl K(\rho):=\rho(1)-\int_0^1\rho(u)\d u,
    \qquad
    \mcl K_\rho(r):=\int_0^1\Ll(\rho(1)-\rho(u)\vee r\Rr)\d u,
    \qquad r\in[0,\rho(1)].
    \label{e.Krho_def}
\end{align}
The function $\mcl K_\rho$ is nonincreasing and $1$-Lipschitz on $[0,\rho(1)]$, with $\mcl K_\rho(0)=\mcl K(\rho)$ and $\mcl K_\rho(\rho(1))=0$. We say that $b\in\R$ is admissible for $\rho$ if $b>2\mcl K(\rho)$. For $a\in\R$ and admissible $b$, we define
\begin{equation}
\label{e.scalar_spherical_initial}
    \mcl D_a^\circ(\rho,b):=\frac{a^2}{2(b-2\mcl K(\rho))}
    +\int_0^{\rho(1)}\frac{\d r}{b-2\mcl K_\rho(r)}+\frac12\Ll(b-1-\log b\Rr)-\rho(1),
\end{equation}
which is finite since $b-2\mcl K_\rho(r)\ge b-2\mcl K(\rho)>0$, and we set
\begin{equation}
\label{e.scalar_initial_direct_identity}
    \psi_a^\circ(\rho):=-\inf_{b>2\mcl K(\rho)}\mcl D_a^\circ(\rho,b).
\end{equation}
With the fixed vector $h$ from~\eqref{e.external_field_h}, we then define
\begin{align}
    \psi(q):=\sum_{s\in\sS}\lambda_{\infty,s}\psi_{h^s}^\circ(q^s),
    \qquad q\in\mcl Q_\infty^{\sS}.
    \label{e.multi_species_initial}
\end{align}
We will show in Proposition~\ref{l.psi_spherical_smooth} that $\psi$ is Lipschitz continuous for the $L^1$ norm, and we extend it to $\mcl Q_1^{\sS}$ by continuity. The notation $\psi$ always refers to this fixed $h$. More generally, for every probability vector $\lambda$ with positive entries, we set
\begin{equation}
\label{e.psi_lambda_def}
    \psi_\lambda(q):=\sum_{s\in\sS}\lambda_s\psi_{h^s}^\circ(q^s),
    \qquad q\in\mcl Q_\infty^{\sS},
\end{equation}
so that $\psi=\psi_{\lambda_\infty}$. This is the initial condition of the model with asymptotic species proportions~$\lambda$, and it is used in Section~\ref{s.conclusion}, where $\lambda_\infty$ is approximated by rational probability vectors. All the properties of $\psi$ established below are proved species by species, and therefore hold for $\psi_\lambda$ as well.

The relation between the block and the functional is that, for admissible $b$,
\begin{equation}
\label{e.block_representation}
    \mcl D_a^\circ(\rho,b)=\E\log Z^\circ_{a,b}(\rho)+\frac12(b-1)-\rho(1);
\end{equation}
this is part of Lemma~\ref{l.scalar_initial_continuous_derivative} below. The two corrections in~\eqref{e.block_representation} have the following origin, which will become apparent in the proof of Proposition~\ref{l.initial_condition}. The term $-\rho(1)$ is the self-overlap correction $-Nq(1)\cdot\lambda_N$ of the enriched Hamiltonian~\eqref{e.enriched_H}, counted per coordinate. The term $\frac12(b-1)$ compensates the Gaussian weight $e^{-nb/2}$ carried by the sphere $|x|^2=n$ in the Gaussian model, and the volume factor $\kappa_n\simeq e^{n/2}$ of the polar decomposition of $\R^n$. Up to the sign convention and the term $-\rho(1)$, the functional $\mcl D^\circ_a(\rho,b)$ is the functional denoted by $W(x,b)$ in~\cite{tal.sph}, and its multi-species version is the Parisi functional of~\cite{bates2022free} with the interaction term removed. The quantity $2\mcl K(\rho)$ is denoted by $d_1$ in~\cite{chen2013aizenman,tal.sph}; we prefer the notation $\mcl K(\rho)$, which stresses that it is a functional of the path.

\begin{remark}[Finite-step paths]
\label{r.finite_step_CS}
Let $\rho$ be a finite-step path, written as in~\eqref{e.finite_step_path}. An Abel summation gives $\mcl K(\rho)=\sum_{i=1}^km_i(\rho_i-\rho_{i-1})$, and for $r\in[\rho_{i-1},\rho_i)$ with $i\ge1$ we have $\mcl K_\rho(r)=m_i(\rho_i-r)+\sum_{j>i}m_j(\rho_j-\rho_{j-1})$, while $\mcl K_\rho(r)=\mcl K(\rho)$ for $r\in[0,\rho_0)$. Writing
\begin{equation}
\label{e.c_i_def}
    c_i:=b-2\sum_{j=i}^km_j(\rho_j-\rho_{j-1}),\qquad 1\le i\le k+1,
\end{equation}
so that $c_1=b-2\mcl K(\rho)$, $c_{k+1}=b$, and $c_1\le c_2\le\cdots\le c_{k+1}$, the integral in~\eqref{e.scalar_spherical_initial} can be computed explicitly:
\begin{equation}
\label{e.CS_integral_finite_step}
    \int_0^{\rho(1)}\frac{\d r}{b-2\mcl K_\rho(r)}=\frac{\rho_0}{c_1}+\sum_{i=1}^k\frac{\log c_{i+1}-\log c_i}{2m_i}.
\end{equation}
\end{remark}

\subsubsection*{The multiplier equation}
The endpoint $\rho(1)$ enters $\mcl D^\circ_a$ through the admissibility threshold $2\mcl K(\rho)$, through the domain of integration in~\eqref{e.scalar_spherical_initial}, and through the term $-\rho(1)$. It is not controlled by the $L^1$ norm of the path, and neither is the minimizer $b_a(\rho)$. However, these occurrences cancel once~$b$ is measured from the boundary of the admissible set, in the variable $\ell:=b-2\mcl K(\rho)$. 
For $\rho\in\mcl Q_1$, $\ell>0$, and $r \ge 0$, we set $\bar\rho:=\int_0^1\rho(u)\d u$,
\begin{equation}
\label{e.half_line_denominator}
    d_{\rho,\ell}(r):=\ell+2\int_0^1\Ll(r-\rho(v)\Rr)_+\d v,
    \qquad\text{and}\qquad
    S_a(\rho,\ell):=\frac{a^2}{\ell^2}+2\int_0^\infty\frac{\d r}{d_{\rho,\ell}(r)^2},
\end{equation}
where $(\cdot)_+:=(\cdot)\vee0$ denotes the positive part.

\begin{lemma}[The multiplier equation]
\label{l.multiplier}
Let $a\in\R$.
\begin{enumerate}
    \item \label{i.multiplier_existence} For every $\rho\in\mcl Q_1$, $\ell>0$, and $r\ge0$,
    \begin{equation}
    \label{e.half_line_bounds}
        \ell+2(r-\bar\rho)_+\le d_{\rho,\ell}(r)\le\ell+2r,
        \qquad\text{hence}\qquad
        \frac1\ell\le S_a(\rho,\ell)\le\frac1\ell+\frac{a^2+2\bar\rho}{\ell^2}.
    \end{equation}
    The function $\ell\mapsto S_a(\rho,\ell)$ is continuous and strictly decreasing on $(0,\infty)$, and there is a unique $\ell_a(\rho)>0$ such that $S_a(\rho,\ell_a(\rho))=1$. It satisfies
    \begin{equation}
    \label{e.uniform_gap}
        1\le\ell_a(\rho)\le\frac{1+\sqrt{1+4a^2+8\bar\rho}}2.
    \end{equation}
    \item \label{i.multiplier_minimizer} Let $\rho\in\mcl Q_\infty$ and $b>2\mcl K(\rho)$, and set $\ell:=b-2\mcl K(\rho)$. Then $d_{\rho,\ell}(r)=b-2\mcl K_\rho(r)$ for $r\in[0,\rho(1)]$, $d_{\rho,\ell}(r)=b+2(r-\rho(1))$ for $r\ge\rho(1)$, and
    \begin{equation}
    \label{e.D_b_derivative_explicit}
        \dr_b\mcl D^\circ_a(\rho,b)=\frac12\Ll(1-S_a\Ll(\rho,b-2\mcl K(\rho)\Rr)\Rr).
    \end{equation}
    Consequently, $b\mapsto\mcl D^\circ_a(\rho,b)$ is strictly convex on $(2\mcl K(\rho),\infty)$, and the infimum in~\eqref{e.scalar_initial_direct_identity} is attained at the unique point
    \begin{equation}
    \label{e.b_a_def}
        b_a(\rho):=\ell_a(\rho)+2\mcl K(\rho)\ge1.
    \end{equation}
\end{enumerate}
\end{lemma}


\begin{proof}
\emph{Part~\eqref{i.multiplier_existence}.}
The upper bound on $d_{\rho,\ell}$ follows from $(r-\rho(v))_+\le r$, and the lower bound from Jensen's inequality applied to the convex function $x\mapsto(r-x)_+$. Since $2\int_0^\infty(\ell+2r)^{-2}\d r=\ell^{-1}$ and $2\int_0^\infty(\ell+2(r-\bar\rho)_+)^{-2}\d r=2\bar\rho\ell^{-2}+\ell^{-1}$, the bounds on $S_a$ follow. For fixed $r$, the function $\ell\mapsto d_{\rho,\ell}(r)^{-2}$ is continuous and strictly decreasing, and it is dominated by $(\ell_0+2(r-\bar\rho)_+)^{-2}$, which is integrable on $[0,\infty)$, for $\ell\ge\ell_0>0$; by dominated convergence, $\ell\mapsto S_a(\rho,\ell)$ is continuous and strictly decreasing, and by~\eqref{e.half_line_bounds}, it tends to $+\infty$ as $\ell\searrow0$ and to $0$ as $\ell\to\infty$. Hence the equation $S_a(\rho,\ell)=1$ has a unique solution $\ell_a(\rho)$. The lower bound $S_a(\rho,\ell)\ge\ell^{-1}$ gives $\ell_a(\rho)\ge1$, and the upper bound gives $\ell_a(\rho)^2-\ell_a(\rho)-(a^2+2\bar\rho)\le0$, which is the upper bound in~\eqref{e.uniform_gap}.

\smallskip
\noindent\emph{Part~\eqref{i.multiplier_minimizer}.}
For $r\in[0,\rho(1)]$, the definition~\eqref{e.Krho_def} gives $\mcl K(\rho)-\mcl K_\rho(r)=\int_0^1(\rho(u)\vee r-\rho(u))\d u=\int_0^1(r-\rho(u))_+\d u$, so that $b-2\mcl K_\rho(r)=d_{\rho,\ell}(r)$. For $r\ge\rho(1)$, we have $(r-\rho(v))_+=r-\rho(v)$ for every $v$, so that $d_{\rho,\ell}(r)=\ell+2r-2\bar\rho=b+2(r-\rho(1))$. Differentiating~\eqref{e.scalar_spherical_initial} in $b$ under the integral sign, we get
\begin{align*}
    \dr_b\mcl D^\circ_a(\rho,b) & =-\frac{a^2}{2\ell^2}-\int_0^{\rho(1)}\frac{\d r}{d_{\rho,\ell}(r)^2}+\frac12\Ll(1-\frac1b\Rr),
    \\
    \frac1b & =2\int_{\rho(1)}^\infty\frac{\d r}{(b+2(r-\rho(1)))^2}=2\int_{\rho(1)}^\infty\frac{\d r}{d_{\rho,\ell}(r)^2},
\end{align*}
which is~\eqref{e.D_b_derivative_explicit}. Since $\ell\mapsto S_a(\rho,\ell)$ is strictly decreasing, $\dr_b\mcl D^\circ_a(\rho,b)$ is strictly increasing in $b$, negative for $b<b_a(\rho)$ and positive for $b>b_a(\rho)$; this proves the strict convexity and identifies the minimizer. The bound $b_a(\rho)\ge1$ follows from~\eqref{e.uniform_gap} and $\mcl K(\rho)\ge0$.
\end{proof}

We next record the continuity properties of the multiplier, and introduce the path that will be identified with the derivative of $\psi^\circ_a$.

\begin{lemma}[Regularity of the multiplier and the derivative path]
\label{l.multiplier_regularity}
Let $a\in\R$.
\begin{enumerate}
    \item \label{i.multiplier_lipschitz} For every $\rho,\rho'\in\mcl Q_1$,
    \begin{equation}
    \label{e.ell_lipschitz}
        |\ell_a(\rho)-\ell_a(\rho')|\le2|\rho-\rho'|_{L^1},
    \end{equation}
    and, if $\rho,\rho'\in\mcl Q_\infty$, then $|b_a(\rho)-b_a(\rho')|\le4|\rho-\rho'|_{L^1}+2|\rho(1)-\rho'(1)|$.
    \item \label{i.multiplier_gradient} For $\rho\in\mcl Q_1$, define
    \begin{equation}
    \label{e.zeta_explicit}
        \zeta_{a,\rho}(u):=1-2\int_{\rho(u)}^\infty\frac{\d r}{d_{\rho,\ell_a(\rho)}(r)^2},
        \qquad u\in[0,1).
    \end{equation}
    We have $\zeta_{a,\rho}\in\mcl Q_{\infty,\le1}$, and $\zeta_{a,\rho}\ge a^2/\ell_a(\rho)^2$. If $\rho\in\mcl Q_\infty$, then
    \begin{equation}
    \label{e.zeta_bounded_formula}
        \zeta_{a,\rho}(u)=\frac{a^2}{\ell_a(\rho)^2}+2\int_0^{\rho(u)}\frac{\d r}{(b_a(\rho)-2\mcl K_\rho(r))^2}
        \qquad\text{and}\qquad
        \zeta_{a,\rho}(1)=1-\frac1{b_a(\rho)}.
    \end{equation}
    Moreover, for every $\rho,\rho'\in\mcl Q_1$ and $u\in[0,1)$,
    \begin{equation}
    \label{e.zeta_pointwise_stability}
        |\zeta_{a,\rho}(u)-\zeta_{a,\rho'}(u)|\le2|\rho(u)-\rho'(u)|+8|\rho-\rho'|_{L^1},
    \end{equation}
    and therefore, for every $r\in[1,\infty]$ and $\rho,\rho'\in\mcl Q_r$,
    \begin{equation}
    \label{e.scalar_gradient_lipschitz}
        |\zeta_{a,\rho}-\zeta_{a,\rho'}|_{L^r}\le10|\rho-\rho'|_{L^r}.
    \end{equation}
\end{enumerate}
\end{lemma}

The path $\zeta_{a,\rho}$ will be identified in Proposition~\ref{l.psi_spherical_smooth} with the derivative of $\psi^\circ_a$ at $\rho$. Note that $\ell_a(\rho)$ and $\zeta_{a,\rho}$ are defined for every integrable path, while $b_a(\rho)$ involves the endpoint; by~\eqref{e.ell_lipschitz}, the map $\rho\mapsto\ell_a(\rho)$ is Lipschitz for the $L^1$ norm, and this is not the case for $b_a$. 

\begin{proof}[Proof of Lemma~\ref{l.multiplier_regularity}]
\emph{Part~\eqref{i.multiplier_lipschitz}.}
Set $\delta:=|\rho-\rho'|_{L^1}$. Since the positive part is $1$-Lipschitz, $|d_{\rho,\ell}(r)-d_{\rho',\ell}(r)|\le2\delta$ for every $r$ and $\ell$, so that $d_{\rho',\ell+2\delta}\ge d_{\rho,\ell}$ and $S_a(\rho',\ell+2\delta)\le S_a(\rho,\ell)$. At $\ell=\ell_a(\rho)$, this gives $S_a(\rho',\ell_a(\rho)+2\delta)\le1=S_a(\rho',\ell_a(\rho'))$, hence $\ell_a(\rho')\le\ell_a(\rho)+2\delta$ by strict monotonicity. Exchanging the roles of $\rho$ and $\rho'$ proves~\eqref{e.ell_lipschitz}, and the bound on $b_a$ follows from $|\mcl K(\rho)-\mcl K(\rho')|\le|\rho(1)-\rho'(1)|+\delta$.

\smallskip
\noindent\emph{Part~\eqref{i.multiplier_gradient}.}
Write $\ell:=\ell_a(\rho)$ and $d:=d_{\rho,\ell}$. Since $\rho$ is nondecreasing and right-continuous, and since $d^{-2}$ is positive and continuous, the path $\zeta_{a,\rho}$ is nondecreasing and right-continuous, and it is at most $1$. By the equation $S_a(\rho,\ell)=1$, we have $\zeta_{a,\rho}(u)\ge1-2\int_0^\infty d^{-2}=a^2/\ell^2\ge0$, so that $\zeta_{a,\rho}\in\mcl Q_{\infty,\le1}$. If $\rho$ is bounded, then $1-2\int_{\rho(u)}^\infty d^{-2}=a^2/\ell^2+2\int_0^{\rho(u)}d^{-2}$, which is the first formula in~\eqref{e.zeta_bounded_formula} by Lemma~\ref{l.multiplier}, part~\eqref{i.multiplier_minimizer}; and letting $u\nearrow1$ and using the same part again gives $\zeta_{a,\rho}(1)=1-2\int_{\rho(1)}^\infty d^{-2}=1-1/b_a(\rho)$.

Let now $\rho'\in\mcl Q_1$, $\delta:=|\rho-\rho'|_{L^1}$, and $\widetilde d:=d_{\rho',\ell_a(\rho')}$. By~\eqref{e.uniform_gap}, $d\ge1$ and $\widetilde d\ge1$; by~\eqref{e.ell_lipschitz} and the $1$-Lipschitz continuity of the positive part, $\|d-\widetilde d\|_{L^\infty}\le4\delta$; and by the equation $S_a=1$, we have $2\int_0^\infty d^{-2}\le1$ and $2\int_0^\infty\widetilde d^{-2}\le1$. For $x,y\ge1$,
\begin{equation*}
    |x^{-2}-y^{-2}|=|x-y|\,\frac{x+y}{x^2y^2}\le|x-y|\Ll(x^{-2}+y^{-2}\Rr),
\end{equation*}
so that $\int_0^\infty|d^{-2}-\widetilde d^{-2}|\d r\le4\delta$. Writing
\begin{equation*}
    \zeta_{a,\rho}(u)-\zeta_{a,\rho'}(u)=2\int_{\rho'(u)}^\infty\Ll(\widetilde d(r)^{-2}-d(r)^{-2}\Rr)\d r-2\int_{\rho(u)}^{\rho'(u)}d(r)^{-2}\d r
\end{equation*}
and using $d^{-2}\le1$ gives~\eqref{e.zeta_pointwise_stability}. Taking the $L^r$ norm and using $\delta\le|\rho-\rho'|_{L^r}$ gives~\eqref{e.scalar_gradient_lipschitz}.
\end{proof}

\subsubsection*{Linear tilts}
The following identity is used here to show that the block is infinite outside the admissible range, and again in Section~\ref{s.cavity}.

\begin{lemma}[Linear tilts of the cascade]
\label{l.linear_cascade_identity}
Let $c\in\mcl Q_\infty$, and let $\msf W$ be a centered Gaussian field on $\mfk U$ with covariance $\E\msf W(\alpha)\msf W(\alpha')=c(\alpha\wedge\alpha')$. We have
\begin{equation}
    \E\log\int\exp\Ll(\msf W(\alpha)\Rr)\fR(\d\alpha)=\frac12\mcl K(c)=\frac12\Ll(c(1)-\int_0^1c(r)\d r\Rr).
    \label{e.cascade_gaussian_identity}
\end{equation}
\end{lemma}

\begin{proof}
For $v\in[0,1]$, set $G(v):=\E\log\int\exp(\sqrt v\,\msf W(\alpha))\fR(\d\alpha)$, which satisfies $0\le G(v)\le vc(1)/2$ by Jensen's inequality. The field $\sqrt v\msf W$ has covariance $vc(\alpha\wedge\alpha')$, which is affine in $v$. By Gaussian integration by parts in the form of~\cite[Theorem~4.6]{HJbook}, and by the invariance~\eqref{e.invariance_cascade} of the overlap law under the Lipschitz tilt $x\mapsto\sqrt vx$,
\begin{equation*}
    G'(v)=\frac12\,\E\la c(\alpha\wedge\alpha)-c(\alpha\wedge\alpha')\ra_{\fR^{\sqrt v\msf W}}=\frac12\Ll(c(1)-\int_0^1c(r)\d r\Rr),
\end{equation*}
which does not depend on $v$. Integrating over $[0,1]$ and using $G(0)=0$ gives~\eqref{e.cascade_gaussian_identity}. For a finite-step path $c$, the identity can also be read off from the recursion~\eqref{e.tilted_recursion} of Appendix~\ref{s.app_cascade}, the tilt being affine in the marks at every level.
\end{proof}

\subsubsection*{Properties of the block}
For $\rho\in\mcl Q_\infty$ and admissible $b$, we set
\begin{equation}
\label{e.X_a_def}
    \mcl X_a(\rho,b):=\mcl D^\circ_a(\rho,b)-\frac{b-1}2+\rho(1)=\frac{a^2}{2(b-2\mcl K(\rho))}+\int_0^{\rho(1)}\frac{\d r}{b-2\mcl K_\rho(r)}-\frac12\log b,
\end{equation}
the candidate value of $\E\log Z^\circ_{a,b}(\rho)$ in~\eqref{e.block_representation}.

\begin{lemma}[Gaussian block with one coordinate]
\label{l.scalar_initial_continuous_derivative}
Let $\rho\in\mcl Q_\infty$ and $a\in\R$.
\begin{enumerate}
    \item \label{i.block_admissible} If $b>2\mcl K(\rho)$, then $Z^\circ_{a,b}(\rho)$ is almost surely finite, $\log Z^\circ_{a,b}(\rho)$ is integrable, and~\eqref{e.block_representation} holds, that is, $\E\log Z^\circ_{a,b}(\rho)=\mcl X_a(\rho,b)$.
    \item \label{i.block_nonadmissible} If $b\le2\mcl K(\rho)$, then $\E\log Z^\circ_{a,b}(\rho)=+\infty$.
    \item \label{i.block_convex} The function $b\mapsto\E\log Z^\circ_{a,b}(\rho)$ is convex and differentiable on $(2\mcl K(\rho),\infty)$, with $\E\la\tau^2\ra^\circ_{\rho,b}<\infty$ and
    \begin{equation}
    \label{e.block_b_derivative}
        \dr_b\mcl D_a^\circ(\rho,b)=\frac12-\frac12\E\la\tau^2\ra^\circ_{\rho,b}.
    \end{equation}
    Consequently, the minimizer $b_a(\rho)$ of Lemma~\ref{l.multiplier} is the unique admissible $b$ such that
    \begin{equation}
    \label{e.scalar_initial_normalization}
        \E\la\tau^2\ra^\circ_{\rho,b}=1.
    \end{equation}
\end{enumerate}
\end{lemma}


\begin{proof}
\emph{Step 1: Finite-step paths.}\par
Let $\rho$ be as in~\eqref{e.finite_step_path}, and let $b>2\mcl K(\rho)$. By~\eqref{e.Zcirc_tau_integrated} and~\eqref{e.equiv_cascade_fe}, applied to the nonnegative function $g(x):=x^2/(2b)$ as explained after~\eqref{e.equiv_cascade_overlaps},
\begin{equation}
\label{e.block_discrete}
    \E\log Z^\circ_{a,b}(\rho)=\E\log\sum_{\alpha\in\N^k}v_\alpha\exp\Ll(g(\sqrt2\msf w^\rho(\alpha)+a)\Rr)-\frac12\log b,
\end{equation}
as an identity in $(-\infty,+\infty]$. We compute the right-hand side with the recursive formula for averages along Poisson--Dirichlet cascades of~\cite[Theorem~2.9]{pan} or \cite[Theorem~5.25]{HJbook}, as in~\cite[Section~3]{tal.sph}. Writing $\Delta_j:=2(\rho_j-\rho_{j-1})$ and letting $z$ be a standard Gaussian random variable, define $Y_k:=g$ and, for $j$ from $k$ down to $1$,
\begin{equation}
\label{e.scalar_recursion}
    Y_{j-1}(x):=\frac1{m_j}\log\E\exp\Ll(m_jY_j\Ll(x+\Delta_j^{1/2}z\Rr)\Rr),
    \qquad Y_{-1}:=\E Y_0\Ll(a+\Delta_0^{1/2}z\Rr).
\end{equation}
The theorem is stated in~\cite{HJbook} under the assumption that $\E\exp(m_kg(\sqrt2\msf w^\rho(\alpha)+a))<\infty$, which may fail for the quadratic function $g$ even when $b$ is admissible, and we apply it to the bounded functions $g\wedge L$ instead. As $L\to\infty$, the left-hand side of the resulting identity increases to $\E\log\sum_\alpha v_\alpha\exp(g(\sqrt2\msf w^\rho(\alpha)+a))$, by monotone convergence, since $g\wedge L\ge0$; and at each level, the function obtained from $g\wedge L$ through~\eqref{e.scalar_recursion} increases to the function obtained from $g$, by monotone convergence in each expectation, the recursion being understood with values in $(-\infty,+\infty]$. Hence $\E\log\sum_\alpha v_\alpha\exp(g(\sqrt2\msf w^\rho(\alpha)+a))=Y_{-1}$ in $(-\infty,+\infty]$, and it remains to compute the recursion. If $Y_j(x)=x^2/(2c_{j+1})+d_{j+1}$ for some $c_{j+1}>m_j\Delta_j$ and $d_{j+1}\in\R$, then a Gaussian computation gives $Y_{j-1}(x)=x^2/(2c_j)+d_j$ with
\begin{equation*}
    c_j=c_{j+1}-m_j\Delta_j,
    \qquad
    d_j=d_{j+1}+\frac1{2m_j}\Ll(\log c_{j+1}-\log c_j\Rr),
\end{equation*}
while $Y_{j-1}\equiv+\infty$ if $c_{j+1}\le m_j\Delta_j$. Starting from $c_{k+1}=b$ and $d_{k+1}=0$, the constants $c_j$ are those of~\eqref{e.c_i_def}, and since they are nondecreasing in $j$, they are all positive if and only if $c_1=b-2\mcl K(\rho)>0$. 
Finally, $Y_{-1}=(a^2+\Delta_0)/(2c_1)+d_1$, and comparing with~\eqref{e.CS_integral_finite_step} gives
\begin{equation}
\label{e.X0_explicit}
    Y_{-1}=\frac{a^2}{2(b-2\mcl K(\rho))}+\int_0^{\rho(1)}\frac{\d r}{b-2\mcl K_\rho(r)}.
\end{equation}
Together with~\eqref{e.block_discrete}, this proves $\E\log Z^\circ_{a,b}(\rho)=\mcl X_a(\rho,b)$ for finite-step paths. In particular, $\E\log Z^\circ_{a,b}(\rho)<\infty$, so that $Z^\circ_{a,b}(\rho)$ is almost surely finite; and $\log Z^\circ_{a,b}(\rho)$ is bounded from below by a deterministic constant, as we now check for general paths, so that it is integrable.

\smallskip
\noindent\emph{Step 2: Cutoff and a deterministic lower bound.}\par
Let $\rho\in\mcl Q_\infty$ and $b\in\R$. For $L>0$, let $Z^\circ_{a,b,L}(\rho)$ and $\la\cdot\ra^\circ_{\rho,b,L}$ be defined as $Z^\circ_{a,b}(\rho)$ and $\la\cdot\ra^\circ_{\rho,b}$ in~\eqref{e.scalar_direct_initial_functional}, with the integration in $\tau$ restricted to $[-L,L]$, and let $\gamma_L$ be the normalized restriction of $\gamma$ to $[-L,L]$. By Jensen's inequality under $\gamma_L\otimes\fR$, and since $\int\tau\gamma_L(\d\tau)=0$,
\begin{equation}
\label{e.block_deterministic_lower_bound}
    \log Z^\circ_{a,b,L}(\rho)\ge\log\gamma([-L,L])-\frac{|b-1|}2L^2=:-C_{b,L}.
\end{equation}
Since $Z^\circ_{a,b,L}(\rho)$ increases to $Z^\circ_{a,b}(\rho)$ as $L\to\infty$, and since $\log Z^\circ_{a,b,L}(\rho)\ge-C_{b,1}$ for $L\ge1$, the expectation $\E\log Z^\circ_{a,b}(\rho)$ is well defined in $(-\infty,+\infty]$, and $\E\log Z^\circ_{a,b,L}(\rho)\to\E\log Z^\circ_{a,b}(\rho)$ as $L\to\infty$, by monotone convergence; this holds for every $b\in\R$.

\smallskip
\noindent\emph{Step 3: Admissible $b$ for general paths.}\par
Let $\rho\in\mcl Q_\infty$ and $b>2\mcl K(\rho)$. For $\delta>0$, we discretize $\rho$ by setting
\begin{equation*}
    \rho^{(\delta)}:=\min\Ll\{\rho(1),\ \delta\Ll(\Ll\lfloor\rho/\delta\Rr\rfloor+1\Rr)\Rr\}.
\end{equation*}
The path $\rho^{(\delta)}$ is a finite-step path, since $\rho$ is bounded, nondecreasing, and right-continuous, and it satisfies $\rho\le\rho^{(\delta)}\le\rho+\delta$ and $\rho^{(\delta)}(1)=\rho(1)$. Since $\rho^{(\delta)}\ge\rho$ with the same endpoint, $\mcl K(\rho^{(\delta)})\le\mcl K(\rho)$, so that every $b'>2\mcl K(\rho)$ is admissible for $\rho^{(\delta)}$. We compare the block at $(\rho,b)$ with the blocks at $(\rho^{(\delta)},b\pm2\delta)$, which Step~1 computes, along an interpolation in which the path and the multiplier vary simultaneously, the multiplier absorbing the effect of the change of path.

For $L>0$ and $b'\in\R$, the function $g_{L,b'}(x):=\log\int_{-L}^L\exp(x\tau-\frac{b'}2\tau^2)\frac{\d\tau}{\sqrt{2\pi}}$ is smooth, with $|g_{L,b'}'|\le L$, and $Z^\circ_{a,b',L}(\rho)=\int\exp(g_{L,b'}(\sqrt2w^\rho(\alpha)+a))\fR(\d\alpha)$. Set $\kappa:=\rho^{(\delta)}-\rho$, so that $0\le\kappa\le\delta$ and $\kappa(1)=0$, and, for $s\in[0,1]$ and $\beta\in\R$, let $\rho_s:=\rho+s\kappa$ and $b_s:=b+s\beta$; we realize the fields jointly as $w^{\rho_s}=\sqrt{1-s}\,w^\rho+\sqrt s\,w^{\rho^{(\delta)}}$, with $w^\rho$ and $w^{\rho^{(\delta)}}$ independent. Gaussian integration by parts and the invariance~\eqref{e.invariance_cascade}, as in the proof of~\cite[Proposition~5.1]{chen2025free}, together with the fact that the derivative of $\log Z^\circ_{a,b',L}(\rho)$ in $b'$ is $-\frac12\la\tau^2\ra^\circ_{\rho,b',L}$, give
\begin{equation}
\label{e.cutoff_IBP}
    \frac{\d}{\d s}\E\log Z^\circ_{a,b_s,L}(\rho_s)=-\E\la\kappa(\alpha\wedge\alpha')\tau\tau'\ra^\circ_{\rho_s,b_s,L}-\frac\beta2\,\E\la\tau^2\ra^\circ_{\rho_s,b_s,L},
\end{equation}
where the diagonal term of the integration by parts vanishes because $\kappa(1)=0$, and where we used that $g_{L,b_s}'(\sqrt2w^{\rho_s}(\alpha)+a)$ is the conditional mean of $\tau$ given $\alpha$ under $\la\cdot\ra^\circ_{\rho_s,b_s,L}$. Since $0\le\kappa\le\delta$ and $|\tau\tau'|\le\frac12(\tau^2+\tau'^2)$, the first term on the right-hand side of~\eqref{e.cutoff_IBP} is at most $\delta\,\E\la\tau^2\ra^\circ_{\rho_s,b_s,L}$ in absolute value. The derivative in~\eqref{e.cutoff_IBP} is therefore nonpositive when $\beta=2\delta$, and nonnegative when $\beta=-2\delta$. Integrating over $s\in[0,1]$ in both cases, we obtain
\begin{equation*}
    \E\log Z^\circ_{a,b+2\delta,L}(\rho^{(\delta)})\le\E\log Z^\circ_{a,b,L}(\rho)\le\E\log Z^\circ_{a,b-2\delta,L}(\rho^{(\delta)}),
\end{equation*}
and letting $L\to\infty$, by the monotone convergence noted in Step~2, gives
\begin{equation}
\label{e.block_sandwich}
    \E\log Z^\circ_{a,b+2\delta}(\rho^{(\delta)})\le\E\log Z^\circ_{a,b}(\rho)\le\E\log Z^\circ_{a,b-2\delta}(\rho^{(\delta)}).
\end{equation}
For $4\delta<b-2\mcl K(\rho)$, the multipliers $b\pm2\delta$ are admissible for $\rho^{(\delta)}$, and Step~1 identifies the outer terms in~\eqref{e.block_sandwich} with $\mcl X_a(\rho^{(\delta)},b\pm2\delta)$. Since $|\mcl K_{\rho^{(\delta)}}(r)-\mcl K_\rho(r)|\le\delta$ uniformly in $r\in[0,\rho(1)]$, the denominators in~\eqref{e.X_a_def} evaluated at $(\rho^{(\delta)},b\pm2\delta)$ are bounded from below by $\frac12(b-2\mcl K(\rho))$ once $8\delta\le b-2\mcl K(\rho)$, and dominated convergence gives $\mcl X_a(\rho^{(\delta)},b\pm2\delta)\to\mcl X_a(\rho,b)$ as $\delta\to0$. Letting $\delta\to0$ in~\eqref{e.block_sandwich} therefore yields $\E\log Z^\circ_{a,b}(\rho)=\mcl X_a(\rho,b)<\infty$. In particular, $Z^\circ_{a,b}(\rho)$ is almost surely finite, and $\log Z^\circ_{a,b}(\rho)\ge-C_{b,1}$ is integrable. This completes the proof of part~\eqref{i.block_admissible}.

\smallskip
\noindent\emph{Step 4: Non-admissible $b$.}\par
Let $b\le2\mcl K(\rho)$ and $T>0$, and let $J$ be the interval $[T,2T]$ if $a\ge0$ and $[-2T,-T]$ otherwise, so that $a\tau\ge0$ on $J$. The expectation $\E\log Z^\circ_{a,b}(\rho)$ is well defined in $(-\infty,+\infty]$, by Step~2. Restricting the integral in $\tau$ to $J$ and applying Jensen's inequality with respect to the normalized Lebesgue measure on $J$, we obtain
\begin{align*}
    \E\log Z^\circ_{a,b}(\rho)
    &\ge\log\frac{T}{\sqrt{2\pi}}+\frac1T\int_J\Ll(\E\log\int\exp\Ll(\sqrt2\tau w^\rho(\alpha)\Rr)\fR(\d\alpha)+a\tau-\frac b2\tau^2\Rr)\d\tau\\
    &=\log\frac{T}{\sqrt{2\pi}}+\frac1T\int_J\Ll(\Ll(\mcl K(\rho)-\frac b2\Rr)\tau^2+a\tau\Rr)\d\tau\ge\log\frac{T}{\sqrt{2\pi}},
\end{align*}
where we used Lemma~\ref{l.linear_cascade_identity} with the path $2\tau^2\rho$, for which $\mcl K(2\tau^2\rho)=2\tau^2\mcl K(\rho)$. Letting $T\to\infty$ proves part~\eqref{i.block_nonadmissible}.

\smallskip
\noindent\emph{Step 5: Dependence on $b$.}\par
Since $Z^\circ_{a,b'}(\rho)$ is nonincreasing in $b'$, part~\eqref{i.block_admissible}, applied along a sequence of admissible values of $b'$ decreasing to $2\mcl K(\rho)$, shows that almost surely $Z^\circ_{a,b'}(\rho)<\infty$ for every $b'>2\mcl K(\rho)$. On this event, the function $b'\mapsto\log Z^\circ_{a,b'}(\rho)$ is the logarithm of the Laplace transform, in the variable $b'/2$, of the image of the positive measure $\exp((\sqrt2w^\rho(\alpha)+a)\tau+\frac12\tau^2)\fR(\d\alpha)\gamma(\d\tau)$ under $(\tau,\alpha)\mapsto\tau^2$; it is therefore convex and differentiable on the open interval $(2\mcl K(\rho),\infty)$, with derivative $-\frac12\la\tau^2\ra^\circ_{\rho,b'}$. Fix $b>2\mcl K(\rho)$. By convexity, for $\eps>0$, the difference quotients $\eps^{-1}(\log Z^\circ_{a,b+\eps}(\rho)-\log Z^\circ_{a,b}(\rho))$ decrease to $-\frac12\la\tau^2\ra^\circ_{\rho,b}$ as $\eps\searrow0$, and they are bounded above by the integrable random variable obtained for $\eps=1$. By monotone convergence, $-\frac12\E\la\tau^2\ra^\circ_{\rho,b}$ is the right derivative at $b$ of the convex function $b'\mapsto\E\log Z^\circ_{a,b'}(\rho)$, which is finite on $(2\mcl K(\rho),\infty)$ by part~\eqref{i.block_admissible}; in particular $\E\la\tau^2\ra^\circ_{\rho,b}<\infty$. The same argument with $\eps<0$ identifies the left derivative with the same quantity. Hence $b'\mapsto\E\log Z^\circ_{a,b'}(\rho)$ is differentiable, with derivative $-\frac12\E\la\tau^2\ra^\circ_{\rho,b'}$, and~\eqref{e.block_b_derivative} follows from~\eqref{e.block_representation}. Finally, by Lemma~\ref{l.multiplier}, part~\eqref{i.multiplier_minimizer}, the derivative $\dr_b\mcl D^\circ_a(\rho,b)$ vanishes exactly at $b=b_a(\rho)$, which is~\eqref{e.scalar_initial_normalization}.
\end{proof}
\subsection{Regularity of the initial condition}
\label{s.initial_regularity}

We now show that $\psi^\circ_a$ is Lipschitz continuous for the $L^1$ norm, and that $\zeta_{a,\rho}$ is its derivative. Recall from Lemma~\ref{l.multiplier_regularity} that $\zeta_{a,\rho}$ is defined for every $\rho\in\mcl Q_1$.

\begin{proposition}[Regularity of the initial condition]
\label{l.psi_spherical_smooth}
Let $a\in\R$.
\begin{enumerate}
    \item \label{i.psi_lipschitz} For every $\rho,\rho'\in\mcl Q_\infty$,
    \begin{equation}
    \label{e.scalar_regular_lipschitz_continuous}
        |\psi^\circ_a(\rho)-\psi^\circ_a(\rho')|\le|\rho-\rho'|_{L^1}.
    \end{equation}
    Consequently, $\psi^\circ_a$ extends uniquely by continuity to $\mcl Q_1$, and~\eqref{e.scalar_regular_lipschitz_continuous} holds on $\mcl Q_1$.
    \item \label{i.zeta_derivative} For every $\rho,\rho'\in\mcl Q_2$,
    \begin{equation}
    \label{e.scalar_quadratic_remainder}
        \Ll|\psi^\circ_a(\rho')-\psi_a^\circ(\rho)-\int_0^1(\rho'-\rho)\,\zeta_{a,\rho}\Rr|\le5|\rho'-\rho|_{L^2}^2.
    \end{equation}
    In particular, $\psi^\circ_a$ is Fr\'echet differentiable on $\mcl Q_2$, relative to $\mcl Q_2$, with derivative $\zeta_{a,\rho}$ at $\rho$.
\end{enumerate}
\end{proposition}

\begin{proof}
\emph{Step 1: The functional on the half-line.}\par
For $\rho\in\mcl Q_\infty$ and $\ell>0$, set $\widehat J_a(\rho,\ell):=\mcl D^\circ_a(\rho,\ell+2\mcl K(\rho))$, so that, by Lemma~\ref{l.multiplier},
\begin{equation}
\label{e.psi_as_min}
    \psi^\circ_a(\rho)=-\inf_{\ell>0}\widehat J_a(\rho,\ell)=-\widehat J_a\Ll(\rho,\ell_a(\rho)\Rr).
\end{equation}
Let $R>\rho(1)$, and write $b:=\ell+2\mcl K(\rho)$. By part~\eqref{i.multiplier_minimizer} of Lemma~\ref{l.multiplier}, $d_{\rho,\ell}(r)=b+2(r-\rho(1))$ for $r\ge\rho(1)$, so that $\int_{\rho(1)}^Rd_{\rho,\ell}(r)^{-1}\d r=\frac12\log d_{\rho,\ell}(R)-\frac12\log b$; since moreover $\frac12(b-1)-\rho(1)=\frac12(\ell-1)-\int_0^1\rho$, the definition~\eqref{e.scalar_spherical_initial} and the same part of Lemma~\ref{l.multiplier} give
\begin{equation}
\label{e.J_hat_half_line}
    \widehat J_a(\rho,\ell)=\frac{a^2}{2\ell}+\int_0^R\frac{\d r}{d_{\rho,\ell}(r)}-\frac12\log d_{\rho,\ell}(R)+\frac{\ell-1}2-\int_0^1\rho(u)\d u.
\end{equation}
The right-hand side does not depend on the choice of $R>\rho(1)$, and the endpoint $\rho(1)$ no longer appears in it explicitly: the path enters only through $d_{\rho,\ell}$ and through its integral.

\smallskip
\noindent\emph{Step 2: Derivative along a segment, at fixed $\ell$.}\par
Let $\rho,\rho'\in\mcl Q_\infty$, $\kappa:=\rho'-\rho$, and $\rho_s:=\rho+s\kappa$ for $s\in[0,1]$, which belongs to $\mcl Q_\infty$ with $\rho_s(1)=(1-s)\rho(1)+s\rho'(1)$. We fix $R>\max\{\rho(1),\rho'(1)\}$, which exceeds $\rho_s(1)$ for every $s$, and $\ell>0$, and we show that $s\mapsto\widehat J_a(\rho_s,\ell)$ is differentiable on $[0,1]$, with
\begin{equation}
\label{e.J_hat_s_derivative}
    \frac{\d}{\d s}\widehat J_a(\rho_s,\ell)=-\int_0^1\kappa(u)\Ll(1-2\int_{\rho_s(u)}^\infty\frac{\d r}{d_{\rho_s,\ell}(r)^2}\Rr)\d u,
\end{equation}
and that the right-hand side of~\eqref{e.J_hat_s_derivative} is a continuous function of $(s,\ell)\in[0,1]\times(0,\infty)$. Since the positive part is $1$-Lipschitz, we have $|d_{\rho_{s'},\ell}(r)-d_{\rho_s,\ell}(r)|\le2|s'-s|\,|\kappa|_{L^1}$ for every $r\ge0$ and $s,s'\in[0,1]$, and $d_{\rho_s,\ell}\ge\ell$; the difference quotients of $s\mapsto d_{\rho_s,\ell}(r)^{-1}$ are therefore bounded by $2|\kappa|_{L^1}\ell^{-2}$, uniformly in $r$. Fix $s\in[0,1]$. For every $r\ge0$ and every $u$ such that $\rho_s(u)\ne r$, the function $s'\mapsto(r-\rho_{s'}(u))_+$ is differentiable at $s'=s$, with derivative $-\kappa(u)\one_{\{r>\rho_s(u)\}}$. Since the sets $\{u:\rho_s(u)=r\}$, $r\ge0$, are disjoint, at most countably many of them have positive measure, and for every other $r$, dominated convergence gives
\begin{equation*}
    \frac{\d}{\d s}d_{\rho_s,\ell}(r)=-2\int_0^1\kappa(u)\one_{\{r>\rho_s(u)\}}\d u.
\end{equation*}
By dominated convergence again, now in the variable $r\in[0,R]$, and by Fubini's theorem, using that $\rho_s(u)<R$ for every $u$,
\begin{equation*}
    \frac{\d}{\d s}\int_0^R\frac{\d r}{d_{\rho_s,\ell}(r)}=\int_0^R\frac{2\int_0^1\kappa(u)\one_{\{r>\rho_s(u)\}}\d u}{d_{\rho_s,\ell}(r)^2}\d r=2\int_0^1\kappa(u)\int_{\rho_s(u)}^R\frac{\d r}{d_{\rho_s,\ell}(r)^2}\d u.
\end{equation*}
For the logarithmic term, since $R\ge\rho_s(u)$ for every $u$, we have $d_{\rho_s,\ell}(R)=\ell+2R-2\int_0^1\rho_s$, which is affine in $s$ with derivative $-2\int_0^1\kappa$; and since $d_{\rho_s,\ell}(r)=d_{\rho_s,\ell}(R)+2(r-R)$ for $r\ge R$, we have $d_{\rho_s,\ell}(R)^{-1}=2\int_R^\infty d_{\rho_s,\ell}(r)^{-2}\d r$. Hence
\begin{equation*}
    \frac{\d}{\d s}\Ll(-\frac12\log d_{\rho_s,\ell}(R)\Rr)=\frac{\int_0^1\kappa}{d_{\rho_s,\ell}(R)}=2\int_0^1\kappa(u)\int_R^\infty\frac{\d r}{d_{\rho_s,\ell}(r)^2}\d u.
\end{equation*}
Adding these two contributions to the derivative $-\int_0^1\kappa$ of the last term in~\eqref{e.J_hat_half_line} gives~\eqref{e.J_hat_s_derivative}. For the joint continuity, let $(s',\ell')\to(s,\ell)$. Then $\rho_{s'}(u)\to\rho_s(u)$ for every $u$, and $d_{\rho_{s'},\ell'}\to d_{\rho_s,\ell}$ uniformly on $[0,\infty)$, while, by~\eqref{e.half_line_bounds}, for $\ell'\ge\ell/2$ the functions $d_{\rho_{s'},\ell'}^{-2}$ are bounded by $(\ell/2+2(r-\bar\rho\vee\bar\rho')_+)^{-2}$, which is integrable on $[0,\infty)$. By dominated convergence, the inner integrals in~\eqref{e.J_hat_s_derivative} converge for every $u$, and they are bounded uniformly in $u$; since $\kappa$ is bounded, the right-hand side of~\eqref{e.J_hat_s_derivative} converges as well.

\smallskip
\noindent\emph{Step 3: The envelope argument.}\par
With the notation of Step~2, set $\ell_s:=\ell_a(\rho_s)$ and $V(s):=\widehat J_a(\rho_s,\ell_s)=\inf_{\ell>0}\widehat J_a(\rho_s,\ell)$, so that $\psi^\circ_a(\rho_s)=-V(s)$ by~\eqref{e.psi_as_min}. By~\eqref{e.ell_lipschitz}, $|\ell_{s'}-\ell_s|\le2|s'-s|\,|\kappa|_{L^1}$, so $s\mapsto\ell_s$ is continuous, and $\ell_s\ge1$ by~\eqref{e.uniform_gap}. For $s,s'\in[0,1]$, the minimality of $\ell_s$ and of $\ell_{s'}$ gives
\begin{equation}
\label{e.envelope_sandwich}
    \widehat J_a(\rho_{s'},\ell_{s'})-\widehat J_a(\rho_s,\ell_{s'})\le V(s')-V(s)\le\widehat J_a(\rho_{s'},\ell_s)-\widehat J_a(\rho_s,\ell_s).
\end{equation}
By Step~2 and the mean value theorem, the left-hand side of~\eqref{e.envelope_sandwich} is $(s'-s)$ times the right-hand side of~\eqref{e.J_hat_s_derivative} evaluated at $\ell=\ell_{s'}$ and at some point between $s$ and $s'$, and similarly for the right-hand side of~\eqref{e.envelope_sandwich}, with $\ell=\ell_s$. Dividing by $s'-s$ and letting $s'\to s$, the continuity of $s\mapsto\ell_s$ and the joint continuity established in Step~2 show that both bounds converge to the right-hand side of~\eqref{e.J_hat_s_derivative} at $\ell=\ell_s$, which is $-\int_0^1\kappa\,\zeta_{a,\rho_s}$ by~\eqref{e.zeta_explicit}. Hence $V$ is differentiable on $[0,1]$, and
\begin{equation}
\label{e.segment_derivative}
    \frac{\d}{\d s}\psi^\circ_a(\rho_s)=\int_0^1\kappa(u)\zeta_{a,\rho_s}(u)\d u.
\end{equation}
This is the envelope theorem: the minimizer does not need to be differentiated. By~\eqref{e.scalar_gradient_lipschitz} with $r=1$, the right-hand side of~\eqref{e.segment_derivative} is a continuous function of $s$, so that~\eqref{e.segment_derivative} can be integrated over $[0,1]$.

\smallskip
\noindent\emph{Step 4: Conclusion.}\par
Since $0\le\zeta_{a,\rho_s}\le1$, integrating~\eqref{e.segment_derivative} gives $|\psi^\circ_a(\rho')-\psi^\circ_a(\rho)|\le|\kappa|_{L^1}$, which is~\eqref{e.scalar_regular_lipschitz_continuous} on $\mcl Q_\infty$; and by~\eqref{e.scalar_gradient_lipschitz} with $r=2$,
\begin{multline*}
    \Ll|\psi^\circ_a(\rho')-\psi^\circ_a(\rho)-\int_0^1\kappa\,\zeta_{a,\rho}\Rr|=\Ll|\int_0^1\int_0^1\kappa(u)\Ll(\zeta_{a,\rho_s}(u)-\zeta_{a,\rho}(u)\Rr)\d u\,\d s\Rr|
    \\
    \le\int_0^1|\kappa|_{L^2}\,10s|\kappa|_{L^2}\d s=5|\kappa|_{L^2}^2,
\end{multline*}
which is~\eqref{e.scalar_quadratic_remainder} on $\mcl Q_\infty$. Since $\mcl Q_\infty$ is dense in $\mcl Q_1$ for the $L^1$ norm, \eqref{e.scalar_regular_lipschitz_continuous} shows that $\psi^\circ_a$ extends uniquely by continuity to $\mcl Q_1$, and the extension satisfies the same estimate. For $\rho,\rho'\in\mcl Q_2$, the truncations $\rho\wedge m$ and $\rho'\wedge m$ belong to $\mcl Q_\infty$ and converge to $\rho$ and $\rho'$ in $L^1$ and in $L^2$ as $m\to\infty$, so that $\psi^\circ_a(\rho\wedge m)\to\psi^\circ_a(\rho)$, and $\zeta_{a,\rho\wedge m}\to\zeta_{a,\rho}$ in $L^2$ by~\eqref{e.scalar_gradient_lipschitz}; passing to the limit in~\eqref{e.scalar_quadratic_remainder} proves it on $\mcl Q_2$. Finally, \eqref{e.scalar_quadratic_remainder} says that the remainder is $O(|\rho'-\rho|_{L^2}^2)$ as $\rho'\to\rho$ in $\mcl Q_2$, which implies the Fr\'echet differentiability relative to $\mcl Q_2$.
\end{proof}

Before turning to the species-weighted statements, we record the convexity of $\psi^\circ_a$, which is the property that allows arbitrary species-dependent external fields, and which is the input of the Hopf formula of Section~\ref{s.hj}. Its proof is a short calculation on finite-step paths: at fixed $\ell=b-2\mcl K(\rho)$, the finite-step expression of $\mcl D^\circ_a$ is concave in the increments of the path, because its logarithmic terms have nonnegative coefficients, and $\psi^\circ_a$ is a supremum of convex functions.

\begin{lemma}[Convexity of the initial condition]
\label{l.psi_convex}
For every $q_0,q_1\in\mcl Q_1^{\sS}$ and every $c\in[0,1]$, we have $\psi\Ll((1-c)q_0+cq_1\Rr)\le (1-c)\psi(q_0)+c\psi(q_1)$. The same holds for $\psi_\lambda$, for every probability vector $\lambda$ with positive entries.
\end{lemma}

\begin{proof}
Since $\psi_\lambda$ is a positive linear combination of the functions $\psi^\circ_{h^s}$, it suffices to prove that $\psi_a^\circ$ is convex on $\mcl Q_1$, for every $a\in\R$. By the Lipschitz continuity~\eqref{e.scalar_regular_lipschitz_continuous} and the density of finite-step paths, it suffices to prove the convexity inequality for two finite-step paths on a common partition $0=m_0<m_1<\cdots<m_k<m_{k+1}=1$. Along the segment between two such paths, the partition stays fixed and the values $\rho_0\le\cdots\le\rho_k$ move affinely; we parametrize them by the increments $t_i:=\rho_i-\rho_{i-1}\ge0$, $0\le i\le k$, which depend affinely on the path. It is therefore enough to prove that $\psi^\circ_a(\rho)$ is a convex function of $t=(t_0,\ldots,t_k)\in\R_+^{k+1}$.

We use the explicit form of $\mcl D^\circ_a$ given in Remark~\ref{r.finite_step_CS}, and the variable $\ell=b-2\mcl K(\rho)=c_1$ of Lemma~\ref{l.multiplier}, which ranges over $(0,\infty)$. In terms of $\ell$ and $t$, we have $c_{i+1}=\ell+2\sum_{j=1}^im_jt_j$ for $0\le i\le k$, $\rho_0=t_0$, $\rho_k=\sum_{i=0}^kt_i$, and, by~\eqref{e.CS_integral_finite_step},
\begin{equation*}
    \mcl D^\circ_a(\rho,b)=\frac{a^2}{2\ell}+\frac{t_0}{\ell}+\sum_{i=1}^k\frac{\log c_{i+1}-\log c_i}{2m_i}+\frac12\Ll(c_{k+1}-1-\log c_{k+1}\Rr)-\sum_{i=0}^kt_i=:J_{\ell}(t).
\end{equation*}
When $k=0$, that is, when $\rho$ is constant, there is no logarithmic term and $J_{\ell}$ is an affine function of $t_0$. For $k\ge1$ and fixed $\ell$, we collect the logarithmic terms: since $c_1=\ell$ does not depend on $t$,
\begin{equation*}
    J_{\ell}(t)=\sum_{j=2}^k\Ll(\frac1{2m_{j-1}}-\frac1{2m_j}\Rr)\log c_j+\Ll(\frac1{2m_k}-\frac12\Rr)\log c_{k+1}+\text{(affine function of $t$)}.
\end{equation*}
The coefficients in front of the logarithms are nonnegative since $m_1<\cdots<m_k<1$, and each $c_j$ is a positive affine function of $t$. Hence, in all cases, $J_{\ell}$ is concave in $t$, and $\psi_a^\circ(\rho)=\sup_{\ell>0}(-J_{\ell}(t))$ is a supremum of convex functions of $t$, hence convex.
\end{proof}

We now record the species-weighted statements. For a probability vector $\lambda$ with positive entries and $\pi\in\mcl Q_1^{\sS}$, we define the path
\begin{equation}
\label{e.derivative_psi_lambda}
    \dr_q\psi_\lambda(\pi):=\Ll(\lambda_s\zeta_{h^s,\pi^s}\Rr)_{s\in\sS}\in\mcl Q^{\sS}_{\infty,\le\lambda},
\end{equation}
where the inclusion follows from $\zeta_{h^s,\pi^s}\in\mcl Q_{\infty,\le1}$, and we set
\begin{equation}
\label{e.bound_derivative_spherical_psi}
    \dr_q\psi(\pi):=\dr_q\psi_{\lambda_\infty}(\pi)\in\mcl Q^{\sS}_{\infty,\le\lambda_\infty},
    \qquad\pi\in\mcl Q_1^{\sS}.
\end{equation}

\begin{corollary}[Multi-species initial condition]
\label{c.psi_multi_species}
Let $\lambda,\lambda'$ be probability vectors with positive entries. For every $\pi,\pi'\in\mcl Q_1^{\sS}$,
\begin{equation}
\label{e.psi_lambda_lipschitz}
    |\psi_\lambda(\pi)-\psi_\lambda(\pi')|\le|\pi-\pi'|_{L^1}
\end{equation}
and
\begin{equation}
\label{e.derivative_stability}
    |\dr_q\psi_{\lambda'}(\pi')-\dr_q\psi_\lambda(\pi)|_{L^1}\le\sum_{s\in\sS}|\lambda'_s-\lambda_s|+10|\pi'-\pi|_{L^1}.
\end{equation}
For every $\pi,\pi'\in\mcl Q_2^{\sS}$,
\begin{equation}
\label{e.psi_quadratic_remainder}
    \Ll|\psi_\lambda(\pi')-\psi_\lambda(\pi)-\la\pi'-\pi,\dr_q\psi_\lambda(\pi)\ra_{L^2}\Rr|\le5|\pi'-\pi|_{L^2}^2.
\end{equation}
\end{corollary}

\begin{proof}
Since $\R^{\sS}$ carries the Euclidean norm, we have $\sum_s\lambda_s|x_s|\le|x|$ and $|x|\le\sum_s|x_s|$ for $x\in\R^{\sS}$. The bound~\eqref{e.psi_lambda_lipschitz} follows from~\eqref{e.scalar_regular_lipschitz_continuous}, since $|\psi_\lambda(\pi)-\psi_\lambda(\pi')|\le\sum_s\lambda_s|\pi^s-\pi'^s|_{L^1}\le|\pi-\pi'|_{L^1}$. For~\eqref{e.derivative_stability}, we write $\lambda'_s\zeta_{h^s,\pi'^s}-\lambda_s\zeta_{h^s,\pi^s}=(\lambda'_s-\lambda_s)\zeta_{h^s,\pi'^s}+\lambda_s(\zeta_{h^s,\pi'^s}-\zeta_{h^s,\pi^s})$, use $0\le\zeta\le1$ and~\eqref{e.scalar_gradient_lipschitz}, and sum over $s$. The remainder estimate~\eqref{e.psi_quadratic_remainder} follows from~\eqref{e.scalar_quadratic_remainder} applied to each species, since $\sum_s\lambda_s|\pi'^s-\pi^s|_{L^2}^2\le|\pi'-\pi|_{L^2}^2$.
\end{proof}


\subsection{The spherical initial condition}
\label{s.initial_condition}

We now prove that $\psi$ is the limit of $\bar F_N(0,\cdot)$. For $n\in\N$ and $v\in\R^n$, we set
\begin{equation}
\label{e.Lambda_n_def}
    \Lambda_n(v):=\int_{\Sph_n}e^{v\cdot\sigma}\mu_n(\d\sigma),
\end{equation}
which depends only on $|v|$ by rotation invariance. By Jensen's inequality, we have $\Lambda_n\ge1$. Moreover, for every $v$, the function $r\mapsto\log\Lambda_n(rv)$ is convex on $\R$, vanishes at $r=0$, and is nonnegative; it is therefore nondecreasing on $[0,\infty)$. For one species with $n$ coordinates, path $\rho\in\mcl Q_\infty$, and external field $a\in\R$, the free energy~\eqref{e.F_N_spherical} at $t=0$ is
\begin{equation}
\label{e.one_species_initial_F}
    \bar F^{\circ}_{n,a}(\rho):=-\frac1n\E\log\int_{\mfk U}\Lambda_n\Ll(\sqrt2\,w^\rho_{[n]}(\alpha)+a\vecone\Rr)\fR(\d\alpha)+\rho(1),
\end{equation}
where $w^\rho_{[n]}:=(w_1^\rho,\ldots,w^\rho_n)$ is a vector of independent copies of $w^\rho$ and $\vecone:=(1,\ldots,1)$. Proposition~\ref{p.F_N_smooth} gives $|\bar F^\circ_{n,a}(\rho)-\bar F^\circ_{n,a}(\rho')|\le|\rho-\rho'|_{L^1}$; and since $|\sigma\cdot\vecone|\le n$ on $\Sph_n$, we see that $|\dr_a\bar F^\circ_{n,a}(\rho)|\le 1$.

\begin{proposition}[Initial condition]
\label{l.initial_condition}
For every $q\in\mcl Q_1^{\sS}$, we have
\begin{equation}
    \lim_{N\to\infty}\bar F_N(0,q)=\psi(q).
    \label{e.initial_condition_limit}
\end{equation}
\end{proposition}

For one species and a finite-step path, the identity~\eqref{e.one_species_limit} below is Talagrand's computation of the free energy at the initial point of Guerra's interpolation, \cite[Proposition~3.1]{tal.sph}; the multi-species statement, in the form of the convergence of the cavity functional as the number of cavity coordinates tends to infinity, is~\cite[Proposition~2.10]{bates2022free}. For the reader's convenience, we include a proof that largely follows Talagrand's Gaussian comparison argument, except that for one bound, we use the tilted-cascade identities of Appendix~\ref{s.app_cascade} to replace his successive large-deviation estimates by a law of large numbers for the Gaussian radius at the optimizing multiplier. 
\begin{proof}[Proof of Proposition~\ref{l.initial_condition}]
Since $\bar F_N(0,\cdot)$ and $\psi$ are both $1$-Lipschitz for the $L^1$ norm, by~\eqref{e.F_N_Lipschitz} and~\eqref{e.psi_lambda_lipschitz}, and since finite-step paths are dense in $\mcl Q_1^{\sS}$, it suffices to prove~\eqref{e.initial_condition_limit} when every $q^s$ is a finite-step path, which we assume from now on.

\smallskip
\noindent\emph{Step 1: Reduction to one species.}\par
Since $P_N$ is a product measure over the species, and since the field $\sqrt2W_N^q+\sum_sh^s\sum_{i\in I_{N,s}}\sigma_i$ is linear in $\sigma$, we have, for every $\alpha$,
\begin{equation*}
    \int\exp\Ll(\sqrt2W^q_N(\sigma,\alpha)+\sum_{s\in\sS}h^s\sum_{i\in I_{N,s}}\sigma_i\Rr)P_N(\d\sigma)=\prod_{s\in\sS}\Lambda_{N_s}\Ll(\sqrt2\,w^{q^s}_{I_{N,s}}(\alpha)+h^s\vecone\Rr),
\end{equation*}
where $N_s:=|I_{N,s}|$ and $w^{q^s}_{I_{N,s}}:=(w_i^{q^s})_{i\in I_{N,s}}$. The logarithm of the right-hand side is a Lipschitz function of the Gaussian fields, since $|\nabla\log\Lambda_n|\le\sqrt n$. By~\eqref{e.equiv_cascade_fe}, applied with a common partition for the paths $q^s$, the expectation of the logarithm of its integral against $\fR$ can be computed with the discrete cascade; the tilt is then a sum over the species of functions of the marks of the coordinates of that species, and by the additivity of the recursive formula over independent groups of marks, Lemma~\ref{l.tilted_cascade}, part~\eqref{i.tilted_additivity}, it equals the sum over $s$ of the corresponding quantities for a single species, that is,
\begin{equation}
\label{e.initial_species_decomposition}
    \bar F_N(0,q)=\sum_{s\in\sS}\lambda_{N,s}\,\bar F^\circ_{N_s,h^s}(q^s).
\end{equation}
Since $\lambda_{N,s}\to\lambda_{\infty,s}$, it suffices to prove that, for every finite-step $\rho\in\mcl Q_\infty$ and $a\in\R$, we have
\begin{equation}
\label{e.one_species_limit}
    \lim_{n\to\infty}\bar F^\circ_{n,a}(\rho)=\psi^\circ_a(\rho).
\end{equation}

\smallskip
\noindent\emph{Step 2: The Gaussian model.}\par
Let $\rho$ be a finite-step path, written as in~\eqref{e.finite_step_path}, and let $b>2\mcl K(\rho)$. We use the discrete cascade with $n$ marks per node, and we let $\msf v_i(\alpha):=\sqrt2\msf w^\rho_i(\alpha)+a$, $1\le i\le n$, where the $\msf w^\rho_i$ are the discrete fields~\eqref{e.discrete_field} built from the marks of coordinate $i$. We write $\msf v(\alpha):=(\msf v_i(\alpha))_{i\le n}$, so that, by~\eqref{e.equiv_cascade_fe}, $\bar F^\circ_{n,a}(\rho)=-\frac1n\E\log A_n+\rho(1)$ with $A_n:=\sum_\alpha v_\alpha\Lambda_n(\msf v(\alpha))$. We introduce the Gaussian model
\begin{equation}
\label{e.gaussian_model_n}
    Z_b:=\sum_{\alpha\in\N^k}v_\alpha\int_{\R^n}\exp\Ll(\msf v(\alpha)\cdot x-\frac b2|x|^2\Rr)\frac{\d x}{(2\pi)^{n/2}},
\end{equation}
with Gibbs bracket $\la\cdot\ra_b$ on $\R^n\times\N^k$. Conditionally on $\alpha$, the coordinates $x_i$ are independent Gaussian random variables with mean $\msf v_i(\alpha)/b$ and variance $1/b$ under $\la\cdot\ra_b$, and integrating over $x$ gives $Z_b=\sum_\alpha v_\alpha\prod_{i\le n}b^{-1/2}e^{\msf v_i(\alpha)^2/(2b)}$. In the notation of Appendix~\ref{s.app_cascade}, this is a discrete cascade tilted by $X(\Omega_\alpha):=\sum_{i\le n}g(\msf v_i(\alpha))$ with $g(x)=x^2/(2b)$, a sum of identical functions of the marks of the $n$ coordinates. We collect the three facts about this tilted cascade that the proof uses.

First, by the additivity of the recursion, Lemma~\ref{l.tilted_cascade}, part~\eqref{i.tilted_additivity}, and by Step~1 of the proof of Lemma~\ref{l.scalar_initial_continuous_derivative}, in which the one-coordinate recursion was computed, we have
\begin{equation}
\label{e.gaussian_model_free_energy}
    \frac1n\E\log Z_b=\E\log Z^\circ_{a,b}(\rho)=\mcl X_a(\rho,b)=\mcl D^\circ_a(\rho,b)-\frac{b-1}2+\rho(1).
\end{equation}
Second, by Lemma~\ref{l.tilted_cascade}, part~\eqref{i.tilted_marks}, the marks $(\Omega_{\alpha,i})_{i\le n}$ of the $n$ coordinates along the sampled leaf are independent and identically distributed under $\E\la\cdot\ra_b$, with a law that does not depend on $n$; since, conditionally on $\alpha$ and on the marks, the coordinate $x_i$ is Gaussian with a law that only depends on $\msf v_i(\alpha)$, the pairs $(\msf v_i(\alpha),x_i)_{i\le n}$ are independent and identically distributed under $\E\la\cdot\ra_b$ as well, with a law $\nu$ that does not depend on $n$. Under $\nu$, all moments are finite, since the tilted path law in \eqref{e.tilted_path_law} is Gaussian, the tilts $X_j$ of the recursion being quadratic. Moreover, $b\mapsto\log Z_b$ is almost surely convex, with derivative $-\frac12\la|x|^2\ra_b$, so that the monotone convergence argument of Step~5 of the proof of Lemma~\ref{l.scalar_initial_continuous_derivative} shows that $b\mapsto\frac1n\E\log Z_b$ is differentiable with derivative $-\frac1{2n}\E\la|x|^2\ra_b=-\frac12\E\la x_1^2\ra_b$, by exchangeability. Comparing with~\eqref{e.gaussian_model_free_energy} and~\eqref{e.block_b_derivative}, we get $\E\la x_1^2\ra_b=\E\la\tau^2\ra^\circ_{\rho,b}$, which equals $1$ when $b=b_a(\rho)$, by~\eqref{e.scalar_initial_normalization}. For this value of $b$, the independence gives the law of large numbers
\begin{equation}
\label{e.radial_concentration_gaussian_model}
    \E\la\Ll(\frac1n|x|^2-1\Rr)^2\ra_{b_a(\rho)}=\frac1n\Var_\nu\Ll(x_1^2\Rr)=:\frac{C_0}n,
\end{equation}
where $C_0<\infty$ does not depend on $n$.

Third, by the variance bound of Lemma~\ref{l.tilted_cascade}, part~\eqref{i.tilted_variance}, applied to the same tilt, we have $\Var(\log Z_b)\le Cn$; together with $|\E\log Z_b|=n|\mcl X_a(\rho,b)|$, this gives
\begin{equation}
\label{e.log_Z_second_moment}
    \E\Ll[(\log Z_b)^2\Rr]\le Cn^2,
\end{equation}
for a constant $C$ depending only on $\rho,a,b$.

We now write $x\in\R^n$ in polar coordinates, $x=r\sigma$ with $r:=|x|/\sqrt n\in[0,\infty)$ and $\sigma:=\sqrt n\,x/|x|\in\Sph_n$. The Lebesgue measure decomposes as $(2\pi)^{-n/2}\d x=\kappa_nr^{n-1}\d r\,\mu_n(\d\sigma)$, with
\begin{equation}
\label{e.kappa_n_def}
    \kappa_n:=\frac{2(n/2)^{n/2}}{\Gamma(n/2)},
    \qquad\text{so that}\qquad\lim_{n\to\infty}\frac1n\log\kappa_n=\frac12
\end{equation}
by Stirling's formula. Since $\msf v(\alpha)\cdot x=r\,\msf v(\alpha)\cdot\sigma$ and $|x|^2=nr^2$, we obtain
\begin{equation}
\label{e.Z_b_polar}
    Z_b=\int_0^\infty B(r)\,e^{-nbr^2/2}\kappa_nr^{n-1}\d r,
    \qquad
    B(r):=\sum_{\alpha}v_\alpha\Lambda_n\Ll(r\msf v(\alpha)\Rr).
\end{equation}
The function $B$ is nondecreasing, by the monotonicity of $r\mapsto\Lambda_n(rv)$ noted above, and $B(1)=A_n$. Under $\la\cdot\ra_b$, the radial variable $r$ has density $B(r)e^{-nbr^2/2}\kappa_nr^{n-1}/Z_b$. The identity~\eqref{e.Z_b_polar} expresses the Gaussian model as a mixture of spherical models with radius $r\sqrt n$; the multiplier $b_a(\rho)$ is the one for which this mixture concentrates near $r=1$.

\smallskip
\noindent\emph{Step 3: Lower bound on $\bar F^\circ_{n,a}(\rho)$.}\par
Let $\delta>0$. Restricting the integral in~\eqref{e.Z_b_polar} to $[1,1+\delta]$, and using $B(r)\ge B(1)=A_n$, $r^{n-1}\ge1$, and $b>0$ there, we obtain
\begin{equation*}
    Z_b\ge A_n\,\delta\kappa_n\exp\Ll(-\frac{nb(1+\delta)^2}2\Rr).
\end{equation*}
Taking logarithms and expectations, dividing by $n$, and using~\eqref{e.gaussian_model_free_energy} and~\eqref{e.kappa_n_def}, we get
\begin{equation*}
    \limsup_{n\to\infty}\Ll(\frac1n\E\log A_n-\rho(1)\Rr)\le\mcl X_a(\rho,b)-\rho(1)+\frac{b(1+\delta)^2}2-\frac12=\mcl D^\circ_a(\rho,b)+\frac b2\Ll((1+\delta)^2-1\Rr).
\end{equation*}
Letting $\delta\to0$ and choosing $b=b_a(\rho)$, for which $\mcl D^\circ_a(\rho,b)=-\psi^\circ_a(\rho)$, we obtain
\begin{equation}
\label{e.initial_lower_bound}
    \liminf_{n\to\infty}\bar F^\circ_{n,a}(\rho)\ge\psi^\circ_a(\rho).
\end{equation}

\smallskip
\noindent\emph{Step 4: Upper bound on $\bar F^\circ_{n,a}(\rho)$.}\par
We now fix $b:=b_a(\rho)$ and $\delta\in(0,\frac12)$. Since $|r-1|\le|r^2-1|$ for $r\ge0$, Chebyshev's inequality gives $\la\one_{\{|r-1|>\delta\}}\ra_b\le\delta^{-2}\la(r^2-1)^2\ra_b$. Let $G$ be the event $\{\la(r^2-1)^2\ra_b\le\delta^2/2\}$, whose complement has probability at most $2C_0/(n\delta^2)$ by~\eqref{e.radial_concentration_gaussian_model} and Markov's inequality. On $G$, the radial variable lies in $[1-\delta,1+\delta]$ with $\la\cdot\ra_b$-probability at least $\frac12$, so that, by~\eqref{e.Z_b_polar} and the monotonicity of $B$,
\begin{equation*}
    \frac{Z_b}2\le\int_{1-\delta}^{1+\delta}B(r)e^{-nbr^2/2}\kappa_nr^{n-1}\d r\le2\delta\,B(1+\delta)\,\kappa_n(1+\delta)^{n}e^{-nb(1-\delta)^2/2}.
\end{equation*}
Hence, on $G$,
\begin{equation}
\label{e.window_bound}
    \log B(1+\delta)\ge\log Z_b+D_n,
    \qquad D_n:=\frac{nb(1-\delta)^2}2-\log\kappa_n-n\log(1+\delta)-\log(4\delta),
\end{equation}
where $D_n$ is deterministic with $|D_n|\le C_\delta n$. Outside $G$, we only use $\log B(1+\delta)\ge0$, which holds since $\Lambda_n\ge1$. Multiplying~\eqref{e.window_bound} by $\one_G$ and taking expectations, we get
\begin{equation*}  
\E\log B(1+\delta)\ge\E\log Z_b-\E[|\log Z_b|\one_{G^c}]-|D_n|\P(G^c)+D_n,
\end{equation*}
and by the Cauchy--Schwarz inequality and~\eqref{e.log_Z_second_moment},
\begin{equation*}
    \E\Ll[|\log Z_b|\one_{G^c}\Rr]\le\Ll(\E\Ll[(\log Z_b)^2\Rr]\Rr)^{1/2}\P(G^c)^{1/2}\le C_\delta\sqrt n.
\end{equation*}
Therefore,
\begin{equation}
\label{e.window_bound_expectation}
    \E\log B(1+\delta)\ge\E\log Z_b+\frac{nb(1-\delta)^2}2-\log\kappa_n-n\log(1+\delta)-C_\delta\Ll(\sqrt n+1\Rr),
\end{equation}
where $C_\delta$ depends on $\delta$ and on $\rho,a$, but not on $n$. Finally, we relate $B(1+\delta)$ to the free energy. Since $(1+\delta)\msf v(\alpha)=\sqrt2(1+\delta)\msf w^\rho(\alpha)+(1+\delta)a$ and $(1+\delta)\msf w^\rho$ is the discrete field~\eqref{e.discrete_field} associated with the path $\rho_\delta:=(1+\delta)^2\rho$, we have $\E\log B(1+\delta)=n\rho_\delta(1)-n\bar F^\circ_{n,a_\delta}(\rho_\delta)$ with $a_\delta:=(1+\delta)a$, by~\eqref{e.equiv_cascade_fe}, and the Lipschitz bounds recorded below~\eqref{e.one_species_initial_F} give
\begin{equation*}
    \Ll|\bar F^\circ_{n,a_\delta}(\rho_\delta)-\bar F^\circ_{n,a}(\rho)\Rr|\le(2\delta+\delta^2)|\rho|_{L^1}+\delta|a|.
\end{equation*}
Dividing~\eqref{e.window_bound_expectation} by $n$, inserting~\eqref{e.gaussian_model_free_energy} and~\eqref{e.kappa_n_def}, and letting $n\to\infty$, we arrive at
\begin{equation*}
    \liminf_{n\to\infty}\Ll(-\bar F^\circ_{n,a}(\rho)\Rr)\ge\mcl D_a^\circ(\rho,b)-\frac b2\Ll(1-(1-\delta)^2\Rr)-\log(1+\delta)-\Ll((1+\delta)^2-1\Rr)\rho(1)-(2\delta+\delta^2)|\rho|_{L^1}-\delta|a|.
\end{equation*}
The right-hand side is $\mcl D^\circ_a(\rho,b)+O(\delta)$ as $\delta\to0$, and $\mcl D_a^\circ(\rho,b_a(\rho))=-\psi^\circ_a(\rho)$. Hence, we obtain that $\limsup_n\bar F^\circ_{n,a}(\rho)\le\psi^\circ_a(\rho)$, which together with~\eqref{e.initial_lower_bound} proves~\eqref{e.one_species_limit}, and completes the proof.
\end{proof}
\subsection{Critical paths and the inner infimum}
\label{s.critical_paths}

Let $p\in\mcl Q_\infty^{\sS}$. Since the partial derivatives of~$\xi$ are nonnegative and nondecreasing in each coordinate on $\R_+^{\sS}$, by the nonnegativity of the coefficients in~\eqref{e.xi_power_series_def}, the path $\nabla\xi(p):=(\dr_s\xi(p(\cdot)))_{s\in\sS}$ belongs to $\mcl Q_\infty^{\sS}$, with endpoint $\nabla\xi(p(1))$. Consequently, for every $t\ge0$ and $r\in[1,\infty]$,
\begin{equation}
\label{e.shifted_path_admissible}
    q\in\mcl Q_r^{\sS}
    \qquad\Longrightarrow\qquad
    q+t\nabla\xi(p)\in\mcl Q_r^{\sS}.
\end{equation}
For every $a\in\R^{\sS}$, we set
\begin{equation}
\label{e.theta_def}
\theta(a):=a\cdot\nabla\xi(a)-\xi(a),
\end{equation}
and we define the Parisi functional
\begin{align}
    \sP_{t,q}(p):=\psi\Ll(q+t\nabla\xi(p)\Rr)-t\int_0^1\theta(p(r))\d r,
    \qquad t\ge0,\quad q\in\mcl Q_1^{\sS},\quad p\in\mcl Q_{\infty,\le\lambda_\infty}^{\sS}.
    \label{e.spherical_Parisi_functional}
\end{align}
Comparing with~\eqref{e.mcJ}, we see that $\sP_{t,q}(p)=\mcl J_{t,q}(q+t\nabla\xi(p),p)$.

The derivative of $\mcl J_{t,q}(q',p)$ in $p$ is $q-q'+t\nabla\xi(p)$, and its derivative in $q'$ is $\dr_q\psi(q')-p$. Formally, the pair $(q',p)$ is thus a critical point of $\mcl J_{t,q}$ if and only if $q'=q+t\nabla\xi(p)$ and $p=\dr_q\psi(q+t\nabla\xi(p))$. We call the latter identity the \emph{critical relation}, and a path $p$ satisfying it a \emph{critical path}. We now state the precise upper bound on the free energy proved in the next two sections.

\begin{proposition}[Upper bound at a critical path]
\label{p.crit_pt_bdd_spherical}
Assume that $\lambda_\infty\in\Q^{\sS}$, let $M\in\N$ be such that $M_s:=M\lambda_{\infty,s}\in\N$ for every $s\in\sS$, and consider the sequence of sizes
\begin{gather}
    N_n:=nM,
    \qquad
    |I_{N_n,s}|=nM_s,
    \qquad n\in\N,
    \label{e.compatible_sequence}
\end{gather}
along which $\lambda_{N_n}=\lambda_\infty$. For every $t>0$ and $q\in\mcl Q_{\infty,\uparrow}^{\sS}$, there exists $p\in\mcl Q_{\infty,\le\lambda_\infty}^{\sS}$ such that
\begin{gather}
    p=\dr_q\psi\Ll(q+t\nabla\xi(p)\Rr),
    \label{e.spherical_critical_relation}\\
    \limsup_{n\to\infty}\bar F_{N_n}(t,q)\le\sP_{t,q}(p).
    \label{e.spherical_critical_bounds}
\end{gather}
\end{proposition}
The integer $M$ is the size of the cavity block of Section~\ref{s.cavity}, which adds $M_s$ coordinates to species~$s$, so that the passage from $N_n$ to $N_{n+1}$ in~\eqref{e.compatible_sequence} adds one block. The class $\mcl Q_{\infty,\uparrow}^{\sS}$ of regular paths, which are bounded, vanish at the origin, and have slope bounded from below, was introduced in Subsection~\ref{s.enriched_def}; the regularity is used through the semi-concavity estimate of Proposition~\ref{p.semi-concave} when matching derivatives. The restrictions to rational proportions and to regular paths are removed in Section~\ref{s.conclusion}. The next lemma explains why Proposition~\ref{p.crit_pt_bdd_spherical} is enough to obtain the desired upper bound on the free energy.

\begin{lemma}[The inner infimum at a critical path]
\label{l.inner_infimum}
Let $t\ge0$, $q\in\mcl Q_\infty^{\sS}$, and $p\in\mcl Q_{\infty,\le\lambda_\infty}^{\sS}$, and set $\pi:=q+t\nabla\xi(p)\in\mcl Q_\infty^{\sS}$. If $p=\dr_q\psi(\pi)$, then
\begin{equation}
\label{e.inner_infimum}
    \inf_{q'\in\mcl Q_\infty^{\sS}}\mcl J_{t,q}(q',p)=\mcl J_{t,q}(\pi,p)=\sP_{t,q}(p).
\end{equation}
\end{lemma}

\begin{proof}
Let $q'\in\mcl Q_\infty^{\sS}$, and set $\kappa:=q'-\pi$, so that $\pi+\eps\kappa=(1-\eps)\pi+\eps q'\in\mcl Q_\infty^{\sS}$ for $\eps\in[0,1]$. By the convexity of $\psi$ (Lemma~\ref{l.psi_convex}),
\begin{equation}
\label{e.supporting_hyperplane}
    \psi(q')-\psi(\pi)\ge\lim_{\eps\searrow0}\frac{\psi(\pi+\eps\kappa)-\psi(\pi)}{\eps}=\la q'-\pi,\dr_q\psi(\pi)\ra_{L^2}.
\end{equation}
With $p=\dr_q\psi(\pi)$, it reads $\psi(q')-\la q',p\ra_{L^2}\ge\psi(\pi)-\la\pi,p\ra_{L^2}$, and adding $\la q,p\ra_{L^2}+t\int_0^1\xi(p(r))\d r$ to both sides yields $\mcl J_{t,q}(q',p)\ge\mcl J_{t,q}(\pi,p)$. Since $\pi\in\mcl Q_\infty^{\sS}$ is itself an admissible competitor, the infimum in~\eqref{e.inner_infimum} equals $\mcl J_{t,q}(\pi,p)$, and the second equality in~\eqref{e.inner_infimum} is the identity noted after~\eqref{e.spherical_Parisi_functional}.
\end{proof}

Let $f(t,q)$ denote the right-hand side of~\eqref{e.main.1}, as in Theorem~\ref{t.main}. Theorem~\ref{t.vis_sol} identifies $f$ with the Lipschitz viscosity solution of~\eqref{e.main.hj}, and Proposition~\ref{p.HJ_bdd_spherical} shows that $f(t,q)\le\liminf_{N\to\infty}\bar F_N(t,q)$, for arbitrary species proportions, along the full sequence of sizes, and for every $q\in\mcl Q_2^{\sS}$. In the setting of Proposition~\ref{p.crit_pt_bdd_spherical}, with $p$ the critical path it provides, we therefore have
\begin{equation}
\label{e.identification_chain}
    f(t,q)\le\liminf_{n\to\infty}\bar F_{N_n}(t,q)\le\limsup_{n\to\infty}\bar F_{N_n}(t,q)\le\sP_{t,q}(p)=\inf_{q'\in\mcl Q_\infty^{\sS}}\mcl J_{t,q}(q',p)\le f(t,q).
\end{equation}
The third inequality is~\eqref{e.spherical_critical_bounds}; the equality is Lemma~\ref{l.inner_infimum}, which applies since $p$ satisfies~\eqref{e.spherical_critical_relation} and $q$ is bounded; and the last inequality holds because $p\in\mcl Q_{\infty,\le\lambda_\infty}^{\sS}$ is a competitor in the supremum defining $f$. Once the Hamilton--Jacobi lower bound and the upper bound at a critical path are proved, in Sections~\ref{s.hj} and~\ref{s.crit_pt_bdd} respectively, the chain~\eqref{e.identification_chain} identifies the limit free energy for rational proportions and regular paths; this is Corollary~\ref{c.rational_identification}, from which Theorem~\ref{t.main} follows by density and continuity arguments, as will be explained in Section~\ref{s.conclusion}.

\section{Cavity computation and synchronization}
\label{s.cavity}

The purpose of this section is to compare the systems with $N$ and $N+M$ spins, with $M$ kept fixed, and to provide the cavity input for Proposition~\ref{p.crit_pt_bdd_spherical}. That proposition requires two conclusions for the same path $p$: the upper bound~\eqref{e.spherical_critical_bounds} and the critical relation~\eqref{e.spherical_critical_relation}. In~\cite{bates2022free,chen2013aizenman}, the corresponding one-sided bound at $q=0$, combined with Guerra's bound under the convexity hypothesis, suffices to identify the free energy. Here we also need the critical relation to complete the identification~\eqref{e.identification_chain}.

The general strategy follows~\cite[Section~6]{chen2025free}, see also~\cite[Section~5]{chen2026ising}. A small perturbation, parametrized by $x$, enforces the Ghirlanda--Guerra identities, so that the bulk overlaps synchronize: the overlap $R_N(\sigma,\sigma')$ of two replicas becomes asymptotically a function $p(\alpha\wedge\alpha')$ of the overlap of their cascade variables, for a single path $p$. To establish the critical relation, Section~\ref{s.crit_pt_bdd} also compares the derivatives in $q$ of the $N$- and $(N+M)$-spin free energies. A small perturbation of the path $q$, parametrized by $z$, ensures that these derivatives are close to one another. The first of these two derivatives is an average of the bulk overlaps, and its limit defines $p$. By the exchangeability of the coordinates within each species, the second one is an average of the overlaps of the cavity coordinates, and in order to identify its limit as the right-hand side of~\eqref{e.spherical_critical_relation}, we will determine the limit law of these overlaps in terms of the limit behavior of the bulk.

This limit law is that described by the Gaussian block of Subsection~\ref{s.gaussian_block}, which is also the object through which $\psi$ is defined. We recall its main features. In the version with one coordinate, a real variable~$\tau$ and a cascade variable $\alpha$ are sampled with density proportional to
\begin{equation*}
    \exp\Ll(\Ll(\sqrt2w^\rho(\alpha)+a\Rr)\tau-\frac b2\tau^2\Rr)
\end{equation*}
with respect to $\d\tau\,\fR(\d\alpha)$, where $\fR$ is the Poisson--Dirichlet cascade of Subsection~\ref{s.enriched_def}, $w^\rho$ is the centered Gaussian field with covariance $\rho(\alpha\wedge\alpha')$, and $a$ is the external field. The coefficient~$b$ is a Lagrange multiplier for the spherical constraint. The cascade field favors large values of $\tau$, and the free energy of the block is finite if and only if $b>2\mcl K(\rho)$, with $\mcl K$ as in~\eqref{e.Krho_def}; we then say that $b$ is admissible for $\rho$. The free energy of the block is a convex function of $b$, and up to an affine term and the weighting of the species, the function $-\psi$ is defined by minimizing it over admissible $b$, see~\eqref{e.scalar_initial_direct_identity} and~\eqref{e.block_representation}. The minimizer is the unique admissible $b$ for which $\E\la\tau^2\ra=1$, as can be expected from a coordinate of a point on the sphere, see~\eqref{e.scalar_initial_normalization}. In the present section, the block has $M$ coordinates, $M_s$ of them in each species $s$, with a path $\pi=(\pi^s)_{s\in\sS}$ and a multiplier $b=(b^s)_{s\in\sS}$; we call it the canonical Gaussian block, and we write $b_\pi$ for the minimizing multiplier, see Subsection~\ref{s.canonical_block}.

In order to perform the cavity calculation, and since $P_{N+M}$ is not a product measure, we write a configuration in $\Sigma_{N+M}$ as $(S_{r_N(\tau)}\sigma,\tau)$, where $\sigma\in\Sigma_N$, $\tau\in\R^M$, and where $S_{r_N(\tau)}$ rescales the coordinates of $\sigma$ in species~$s$ by the factor $r_N^s(\tau)^{1/2}$, with $r^s_N(\tau)=1+(M_s-|\tau^{(s)}|^2)/N_s$; see~\eqref{e.radial_factor}. We refer to this parametrization of $\Sigma_{N+M}$ as the chart. The rescaling factor deviates from $1$ by $O(N^{-1})$, but it multiplies a radial derivative of the Hamiltonian whose Gibbs average can be of order $N$, so its contribution to the cavity increment is of order one and must be retained. Explicitly, in the cavity computation of~\cite{chen2025free}, the Hamiltonian of the $(N+M)$-spin system is, up to negligible terms, the sum of the Hamiltonian of the $N$-spin system and of a term linear in $\tau$. Here, an additional term of the form
\begin{equation}
\label{e.partial_r_new}
    \frac12\sum_{s\in\sS}\Ll(M_s-|\tau^{(s)}|^2\Rr)\mcl E_{N,s}(\sigma,\alpha)
\end{equation}
appears, where $\mcl E_{N,s}$ is the derivative of the Hamiltonian of the $N$-spin system with respect to the radius of species $s$, suitably normalized to be of order $1$. In~\cite{bates2022free, chen2013aizenman}, the cavity coordinates are confined to a thin annulus around the product of spheres of radii $\sqrt{M_s}$, on which this term is small and can be discarded at a cost that vanishes in the limit.
The confinement operation is sufficient for the purpose of proving \eqref{e.spherical_critical_bounds}, but here we also aim to establish \eqref{e.spherical_critical_relation} and thus need to be more precise.

To control the term in~\eqref{e.partial_r_new}, we introduce a small perturbation in the directions of $\mcl E_{N,s}$, with a small coefficient parametrized by $y=(y^s)_{s\in\sS}$. The logarithm of the partition function is convex in $y$, and a classical argument shows that $\mcl E_{N,s}$ concentrates around its Gibbs average $e^s_N$ for most values of $y$. Replacing $\mcl E_{N,s}$ by $e^s_N$ turns~\eqref{e.partial_r_new} into the additive constant $\frac12\sum_sM_se_N^s$ and the quadratic term $-\frac12\sum_se_N^s|\tau^{(s)}|^2$, so that in the limit, the cavity coordinates are those of a canonical Gaussian block with path $\pi:=q + t \nabla \xi(p)$, whose multiplier $b$ is an explicit function of $e:=\lim e_N$. However, $e_N$ is a quantity on which we have no direct information, so at this stage~$b$ is an unknown parameter of the limit. It is identified through the spherical constraint: by the permutation symmetry between the coordinates of each species in the $(N+M)$-spin system, we have
\begin{equation*}
    \E\la|\tau^{(s)}|^2\ra=M_s.
\end{equation*}
For the limiting block, this identity is the normalization recalled above, and we can therefore conclude that $b=b_\pi$ (Lemma~\ref{l.canonical_normalization_identity}). This allows us to connect the Gaussian block with $\psi$, and thereby to complete the argument.

The unbounded nature of the cavity variables creates several difficulties, such as in justifying the replacement of $\mcl E_{N,s}$ by $e^s_N$, in using that the density of the cavity coordinates is close to Gaussian, or in discussing Gaussian blocks for values of $b$ that we do not know in advance to be admissible. We therefore insert a cutoff $\chi_L(\tau)$, under which all objects are defined for every $b$, and first let $N\to\infty$ at fixed $L$. Since the cutoff can only decrease the partition function of the $(N+M)$-spin system, this gives a lower bound on the cavity increment $\E\log Z_{N+M}-\E\log Z_N$ by the free energy of the block with cutoff, for every $L$. The increment is bounded (Lemma~\ref{l.fixed_block_increment_bound}), while the free energy of the block with cutoff diverges as $L\to\infty$ when $b$ is not admissible (Lemma~\ref{l.multispecies_nonadmissible_escape}); hence $b$ is admissible. We then let $L\to\infty$ (Lemma~\ref{l.remove_cutoffs}), and we show that the moments of the cavity coordinates are bounded uniformly in $N$, using the rotation invariance of the reference measure (Lemma~\ref{l.true_cavity_radius_ui}). This extends the convergence of the Gibbs averages from bounded observables to observables of polynomial growth, such as $|\tau^{(s)}|^2$ in the identity above.

These steps are combined in Theorem~\ref{t.spherical_cavity_lim}, stated in Subsection~\ref{s.GG_cavity_prop}. This theorem applies along sequences of perturbation parameters $(x,y)$ for which two error terms vanish: the Ghirlanda--Guerra error $\Delta_N$, which measures the failure of the Ghirlanda--Guerra identities in the bulk, and the radial error $\operatorname{Rad}_{N,L}$, which measures the fluctuations of $\mcl E_{N,s}$ at cutoff $L$, for every fixed $L$. The selection of such parameters is carried out in Section~\ref{s.crit_pt_bdd}. Along a subsequence, the theorem provides a path $p$ with $0\le p\le\lambda$ describing the bulk overlaps; the convergence of the overlaps of the cavity coordinates under the $(N+M)$-spin Gibbs measure, jointly with the cascade overlaps and including polynomial moments, to those of the canonical Gaussian block with path $\pi=q+t\nabla\xi(p)$ and multiplier $b=b_\pi$; and the asymptotic lower bound $-M\sP_{t,q}(p)$ on the cavity increment, where $\sP_{t,q}$ is the Parisi functional~\eqref{e.spherical_Parisi_functional}. In view of the sign convention in~\eqref{e.F_N_spherical}, a lower bound on the increment is what is needed for the upper bound~\eqref{e.spherical_critical_bounds}.

The section is organized as follows. Subsections~\ref{s.cavity.setting} and~\ref{s.cavity.perturbations} set up the notation and introduce the perturbed Hamiltonians. The proof of Theorem~\ref{t.spherical_cavity_lim} then goes through four stages.
\begin{enumerate}
    \item \emph{Geometry and scale.} Subsection~\ref{s.spherical_chart} describes the chart and the chart density; Subsection~\ref{s.gaussian_comparison_notation} collects two Gaussian comparison lemmas; and Subsection~\ref{s.ASS_scheme} shows that the cavity increment is uniformly bounded.
    \item \emph{Comparison at fixed cutoff.} Subsections~\ref{s.bulk_and_chart_expansion} to~\ref{s.radial_and_gaussian} replace the $(N+M)$-spin system in the chart by an explicit Gaussian model for the cavity coordinates, at a fixed cutoff $L$, up to the radial error $\operatorname{Rad}_{N,L}$ and a vanishing deterministic error. The main comparison result is Proposition~\ref{p.chart_to_gaussian}.
    \item \emph{The canonical Gaussian block.} Subsection~\ref{s.canonical_block} introduces the block and relates it to $\psi$ and to $\sP_{t,q}$. It contains the characterization of $b_\pi$ by the normalization, the formula for the derivative of $\psi$, and the divergence of the block with cutoff when $b$ is not admissible.
    \item \emph{Passage to the limit and removal of the cutoff.} Subsection~\ref{s.finite_overlap_cutoff} contains the finite-overlap approximation, used to pass to the limit $N\to\infty$ at fixed cutoff, the uniform integrability estimate for the cavity coordinates, and the removal of the cutoff. Subsection~\ref{s.GG_cavity_prop} states and proves the theorem.
\end{enumerate}

\subsection{Setting and notation}
\label{s.cavity.setting}

\subsubsection*{Standing assumptions}
Throughout Sections~\ref{s.cavity} and~\ref{s.crit_pt_bdd}, we assume that $\lambda_\infty\in\Q^{\sS}$, and we write $\lambda:=\lambda_\infty$. We fix positive integers $(M_s)_{s\in\sS}$ and a partition $[M]=\bigsqcup_{s\in\sS}\msf M_s$ such that
\begin{align}
    M:=\sum_{s\in\sS}M_s,
    \qquad
    \frac{M_s}{M}=\lambda_{s}\quad\text{for every }s\in\sS,
    \qquad
    |\msf M_s|=M_s.
    \label{e.lambda_compatible_block}
\end{align}
The set $\msf M_s$ is the set of labels of the cavity coordinates of species $s$. We restrict our attention to system sizes $N$ that are multiples of $M$, and we assume that the species decomposition is compatible with the cavity block, in the sense that
\begin{align}
    N_s:=|I_{N,s}|=\lambda_sN,
    \qquad
    I_{N+M,s}=I_{N,s}\sqcup\{N+j:j\in\msf M_s\},
    \qquad\text{for every }s\in\sS.
    \label{e.N_compatible_block}
\end{align}
In particular, the last $M$ coordinates of $[N+M]$ are identified with $[M]$, and the cavity coordinates of species $s$ are those with labels in $N+\msf M_s$. Under~\eqref{e.N_compatible_block}, the proportions in~\eqref{e.def.lambda_Nd} do not depend on $N$:
\begin{equation}
    \lambda_N=\lambda_{N+M}=\lambda.
    \label{e.lambda_N_equals_lambda}
\end{equation}


\subsubsection*{Bulk and cavity coordinates}
For $\rho\in\R^{N+M}$, we write
\begin{align}
    \sigma:=(\rho_1,\ldots,\rho_N)\in\R^N
    \qquad\text{and}\qquad
    \tau:=(\rho_{N+1},\ldots,\rho_{N+M})\in\R^M,
    \label{e.tau=}
\end{align}
and we call $\sigma$ the bulk coordinates and $\tau$ the cavity coordinates of $\rho$. For $\tau\in\R^M$ and $s\in\sS$, we write $\tau^{(s)}:=(\tau_j)_{j\in\msf M_s}$. For $r=(r^s)_{s\in\sS}\in(0,\infty)^{\sS}$ and $u\in\R^N$, we define the specieswise dilation $S_ru\in\R^N$ by
\begin{align}
    (S_ru)_i:=\sqrt{r^s}\,u_i\qquad\text{for every $s\in\sS$ and $i\in I_{N,s}$}.
    \label{e.S_r=}
\end{align}
Thus $S_r$ multiplies the squared norm of the species-$s$ block by $r^s$, and
\begin{equation}
    R_{N,s}(S_r\sigma,S_{r'}\sigma')=\sqrt{r^sr'^s}\,R_{N,s}(\sigma,\sigma'),
    \qquad \sigma,\sigma'\in\R^N.
    \label{e.overlap_dilation}
\end{equation}
We set $\delta_0:=1/2$ and
\begin{equation}
    U_0:=[1-\delta_0,1+\delta_0]^{\sS}.
    \label{e.U_0_def}
\end{equation}
The field $H_N$ is defined on all of $\R^N$, see~\eqref{e.def_H_N}, and the formula~\eqref{e.WNq_def} defines $W_N^q(\sigma,\alpha)$ for every $\sigma\in\R^N$, with the covariance~\eqref{e.WNq_covariance}; in particular, both fields are defined at the dilated configurations $S_r\sigma$. By~\eqref{e.overlap_dilation}, the covariance of $H_N(S_r\sigma)$ and $H_N(S_{r'}\sigma')$ is $N\xi\Ll((\sqrt{r^sr'^s}\,R_{N,s}(\sigma,\sigma'))_{s\in\sS}\Rr)$, a smooth function of $(r,r')\in U_0^2$, and similarly for $W_N^q$. The fields are therefore differentiable in $r$ in the $L^2$ sense, and their radial derivatives are jointly Gaussian with the fields themselves. All the radial derivatives appearing below are understood in this sense.

\subsubsection*{Fixed parameters}
We fix $t>0$ for the whole section. The dependence on $t$ is kept in the notation $H_N^{t,q}$ from~\eqref{e.enriched_H}, but is suppressed from the notation of the objects that are specific to the cavity computation. We also fix the sequence
\begin{equation}
    \eta_N:=N^{-\frac1{16}},
    \qquad N\in\N,
    \label{e.eta_N_def}
\end{equation}
which sets the size of the perturbations of the Hamiltonian introduced in the next subsection. The exponent $1/16$ is chosen to match~\cite[Proposition~6.8]{chen2025free}, whose proof we invoke in Lemma~\ref{l.overlap_identities_radial_spherical} below, and any exponent in $(0,1/16]$ would do.

\subsubsection*{Parameters and uniformity}
The perturbations depend on parameters
\begin{equation*}
    x=(x_{\msf h})_{\msf h\in\N^4}\in[0,3]^{\N^4}
    \qquad\text{and}\qquad
    y=(y^s)_{s\in\sS}\in[0,3]^{\sS},
\end{equation*}
which will be chosen in Section~\ref{s.crit_pt_bdd}. We also fix a subset $Q\subset\mcl Q_\infty^{\sS}$ that is bounded in $L^\infty$, and a bounded subset $B\subset\R^{\sS}$. Unless stated otherwise, when we say that an estimate holds uniformly over $x$, $y$, $q$, or $b$, we mean that it holds uniformly over
\begin{equation}
    x\in[0,3]^{\N^4},
    \qquad
    y\in[0,3]^{\sS},
    \qquad
    q\in Q,
    \qquad
    b\in B,
    \label{e.xyqbR_uniform}
\end{equation}
and every supremum indexed by $x$, $y$, $q$, or $b$ is taken over these sets. Constants denoted by $C$ may depend on $\xi$, $h$, $t$, $M$, $Q$, and $B$, and may change from one occurrence to the next; additional dependencies are indicated by subscripts, as in $C_L$. In Section~\ref{s.crit_pt_bdd}, the parameters $x$ and $y$ are restricted to the smaller sets $[1,2]^{\N^4}$ and $[1,2]^{\sS}$; the estimates of the present section apply there verbatim.

Recall that $\mfk U=\supp\fR$ denotes the support of the cascade, and that $\la\cdot\ra_\fR$ in~\eqref{e.fR_bracket_def} denotes the average with respect to independent copies of $\alpha$ distributed according to $\fR$, conditionally on $\fR$.

\subsection{Perturbed Hamiltonians}
\label{s.cavity.perturbations}

We add two perturbations to the Hamiltonian $H_N^{t,q}$ in~\eqref{e.enriched_H}. The first one is the standard perturbation that yields the Ghirlanda--Guerra identities; the second one is specific to the spherical setting, and controls the radial derivative of the Hamiltonian.

\subsubsection{Perturbation enforcing the Ghirlanda--Guerra identities}
\label{s.GG_perturbation}

We fix enumerations $(\iota_m)_{m\ge1}$ and $(a_m)_{m\ge1}$ such that
\begin{equation*}
    \{\iota_m:m\ge1\}=[0,1]\cap\Q,
    \qquad
    \{a_m:m\ge1\}=((0,\infty)\cap\Q)^{\sS}.
\end{equation*}
For $\msf h=(\msf h_1,\msf h_2,\msf h_3,\msf h_4)\in\N^4$, we define the kernel
\begin{equation}
    \msf C_{\msf h}(v,u):=\Ll(a_{\msf h_1}\cdot v^{\odot \msf h_2}+\iota_{\msf h_3}u\Rr)^{\msf h_4},
    \qquad v\in\R^{\sS},\quad u\in[0,1],
    \label{e.C_h}
\end{equation}
where $\odot$ denotes componentwise multiplication, so that $v^{\odot\msf h_2}=(v_s^{\msf h_2})_{s\in\sS}$. This is a polynomial in $(v,u)$ with nonnegative coefficients. For each $\msf h\in\N^4$, let $(H_N^{\msf h}(\sigma,\alpha))_{\sigma\in\R^N,\alpha\in\mfk U}$ be a centered Gaussian process with covariance
\begin{align}
    \E H_N^{\msf h}(\sigma,\alpha)H_N^{\msf h}(\sigma',\alpha')
    =N\msf C_{\msf h}\Ll(R_N(\sigma,\sigma'),\alpha\wedge\alpha'\Rr),
    \label{e.overlap_perturbation_covariance}
\end{align}
these processes being independent over $\msf h$ and independent of the other sources of randomness.

We briefly explain why these processes exist, in a form that yields versions which are polynomial in $\sigma$, and which is also used in Subsection~\ref{s.bulk_and_chart_expansion} to couple the processes at sizes $N$ and $N+M$. Expanding the power in~\eqref{e.C_h} by the binomial formula,
\begin{equation}
    \msf C_{\msf h}(v,u)=\sum_{j=0}^{\msf h_4}\binom{\msf h_4}{j}\,\iota_{\msf h_3}^j\,u^j\,\Ll(a_{\msf h_1}\cdot v^{\odot\msf h_2}\Rr)^{\msf h_4-j}.
    \label{e.C_h_binomial}
\end{equation}
For $\sigma\in\R^N$, let $\varphi_{N,s}(\sigma):=N^{-1/2}\sigma^{(s)}\in\R^{I_{N,s}}$, so that $\la\varphi_{N,s}(\sigma),\varphi_{N,s}(\sigma')\ra=R_{N,s}(\sigma,\sigma')$ by~\eqref{e.R_Ns_def}, and, for $0\le j\le\msf h_4$, consider the finite-dimensional space $E_{N,j}:=\big(\bigoplus_{s\in\sS}(\R^{I_{N,s}})^{\otimes\msf h_2}\big)^{\otimes(\msf h_4-j)}$ and the polynomial map
\begin{equation}
    \Phi_N^{\msf h,j}(\sigma):=\Big(\bigoplus_{s\in\sS}\sqrt{a^s_{\msf h_1}}\,\varphi_{N,s}(\sigma)^{\otimes\msf h_2}\Big)^{\otimes(\msf h_4-j)}\in E_{N,j}.
    \label{e.feature_vector_def}
\end{equation}
The inner product of two direct sums is the sum of the inner products, and the inner product of two tensor powers is the power of the inner product, so
\begin{equation}
    \big\langle\Phi_N^{\msf h,j}(\sigma),\Phi_N^{\msf h,j}(\sigma')\big\rangle=\Big(\sum_{s\in\sS}a^s_{\msf h_1}R_{N,s}(\sigma,\sigma')^{\msf h_2}\Big)^{\msf h_4-j}.
    \label{e.feature_inner_product}
\end{equation}
The path $u\mapsto\iota_{\msf h_3}^ju^j$ belongs to $\mcl Q_\infty$; it is constant when $j=0$. By~\cite[Proposition~4.1]{chen2025free}, we can therefore attach to every element $e$ of an orthonormal basis $\mcl E_{N,j}$ of $E_{N,j}$ a cascade field $w^{\msf h,j}_e$ on $\mfk U$ with covariance $\E w^{\msf h,j}_e(\alpha)w^{\msf h,j}_e(\alpha')=\iota_{\msf h_3}^j(\alpha\wedge\alpha')^j$, as in~\eqref{e.W_field_covariance}, these fields being independent over $(\msf h,j,e)$ conditionally on $\fR$, jointly measurable in $\alpha$ and in the underlying randomness, and independent of the other sources of randomness. We set
\begin{equation}
    H_N^{\msf h}(\sigma,\alpha):=\sqrt N\sum_{j=0}^{\msf h_4}\binom{\msf h_4}{j}^{1/2}\sum_{e\in\mcl E_{N,j}}w^{\msf h,j}_e(\alpha)\,\big\langle\Phi_N^{\msf h,j}(\sigma),e\big\rangle.
    \label{e.feature_field_construction}
\end{equation}
Since $\sum_{e\in\mcl E_{N,j}}\la\Phi(\sigma),e\ra\la\Phi(\sigma'),e\ra=\la\Phi(\sigma),\Phi(\sigma')\ra$, the covariance of this field is given by~\eqref{e.C_h_binomial} and~\eqref{e.feature_inner_product}, that is, by~\eqref{e.overlap_perturbation_covariance}. By construction, $H_N^{\msf h}(\sigma,\alpha)$ is a polynomial in $\sigma$ whose coefficients are finitely many Gaussian random variables, measurable in $\alpha$. In particular, it is defined for every $\sigma\in\R^N$, including the dilated configurations $S_r\sigma$ with $r\in U_0$, and its radial derivatives are obtained by differentiating the polynomial; they are linear combinations of the same fields~$w^{\msf h,j}_e$. 

The kernel $\msf C_{\msf h}$ and its derivatives are bounded on the relevant range of overlaps. Setting
\begin{equation}
    \Lambda_{\msf h}:=1+\sum_{a=0}^{4}\sup_{v\in[-2,2]^\sS,\,u\in[0,1]}\|\dr_v^a\msf C_{\msf h}(v,u)\|,
    \label{e.Lambda_h_def}
\end{equation}
we choose constants $c_{\msf h}>0$ such that
\begin{equation}
    \sum_{\msf h\in\N^4}c_{\msf h}^2\Lambda_{\msf h}\le1.
    \label{e.ch_perturbation_lipschitz_summability}
\end{equation}
Derivatives of order up to four are included in~\eqref{e.Lambda_h_def} because the radial perturbation introduced below involves derivatives of the covariance kernels, and its own covariance therefore involves up to four derivatives. For $x\in[0,3]^{\N^4}$, we define
\begin{align}\label{e.perturbed_enriched_H}
    H_N^{x}(\sigma,\alpha)
    :=\sum_{\msf h\in\N^4}x_{\msf h}c_{\msf h}H_N^{\msf h}(\sigma,\alpha)\qquad \text{and}\qquad
    H_N^{x,q}(\sigma,\alpha):=H_N^{t,q}(\sigma,\alpha)+\eta_NH_N^x(\sigma,\alpha).
\end{align}
By~\eqref{e.overlap_perturbation_covariance} and~\eqref{e.ch_perturbation_lipschitz_summability}, the series defining $H_N^x(\sigma,\alpha)$ converges in $L^2$, with variance at most $9N$, for every $\sigma\in\R^N$ with $R_N(\sigma,\sigma)\in[0,2]^{\sS}$; this covers $\Sigma_N$ and the dilated configurations $S_r\sigma$ with $r\in U_0$, which are the only configurations at which $H_N^x$ is evaluated. The coefficient $\eta_N$ makes the contribution of this perturbation to the free energy per spin of order $\eta_N^2$, hence negligible, while keeping it large enough to enforce the Ghirlanda--Guerra identities; see Lemma~\ref{l.overlap_identities_radial_spherical}.

\subsubsection{Radial evaluation and the radial perturbation}
\label{s.radial_perturbation}

We now define the evaluation of the Hamiltonian along dilated configurations. For $r\in U_0$, set
\begin{align}
    H_N^{x,q}(r;\sigma,\alpha):={}&\sqrt{2t}H_N(S_r\sigma)-Nt\xi\Ll(R_N(S_r\sigma,S_r\sigma)\Rr)+\sqrt2W_N^q(S_r\sigma,\alpha)-\sum_{s\in\sS}q^s(1)|(S_r\sigma)^{(s)}|^2\notag\\
    &+\sum_{s\in\sS}h^s\sum_{i\in I_{N,s}}(S_r\sigma)_i+\eta_NH_N^x(S_r\sigma,\alpha).
    \label{e.full_radially_evaluated_H}
\end{align}
This is the expression of $H_N^{x,q}$ in~\eqref{e.perturbed_enriched_H}, with $\sigma$ replaced by $S_r\sigma$ everywhere, including in the deterministic self-overlap corrections. In particular, $H_N^{x,q}(\vecone;\sigma,\alpha)=H_N^{x,q}(\sigma,\alpha)$, where $\vecone=(1,\ldots,1)$.

As explained at the beginning of the section, writing the $(N+M)$-spin system in the chart leads us to evaluate the Hamiltonian of the bulk at a dilation $S_r\sigma$ with $r-\vecone$ of order $N^{-1}$. The radial derivatives of the Hamiltonian have means and covariances of order $N$, see~\eqref{e.second_radial_bound} below, so that this small displacement produces a variation of order one, which is measured by the normalized radial derivatives
\begin{align}
    \mcl E_{N,s}^{x,q}(\sigma,\alpha)
    :=\frac2{N_s}\,\frac{\partial}{\partial r^s}H_N^{x,q}(r;\sigma,\alpha)\Big|_{r=\vecone},
    \qquad s\in\sS.
    \label{e.radial_energy_def}
\end{align}
The factor $2/N_s$ is chosen so that, for a function $f$ of $\sigma$, we have
\begin{equation*}
    \frac{2}{N_s}\partial_{r^s}f(S_r\sigma)\Big|_{r=\vecone}=\frac1{N_s}\sum_{i\in I_{N,s}}\sigma_i\partial_{\sigma_i}f(\sigma);
\end{equation*}
that is, $\mcl E_{N,s}^{x,q}$ is the average over the species-$s$ coordinates of $\sigma_i\partial_{\sigma_i}H_N^{x,q}$. We will need these quantities to be asymptotically deterministic under the Gibbs measure. For this purpose, we perturb the Hamiltonian in their direction: for $y\in[0,3]^{\sS}$, we set
\begin{align}
    V_N^{x,y,q}(\sigma,\alpha)
    :=\eta_N\sum_{s\in\sS}y^sN_s\,\mcl E_{N,s}^{x,q}(\sigma,\alpha)
    \qquad\text{and}\qquad
    H_N^{x,y,q}(\sigma,\alpha):=H_N^{x,q}(\sigma,\alpha)+V_N^{x,y,q}(\sigma,\alpha).
    \label{e.full_radial_perturbation}
\end{align}
The perturbed partition function, free energy, and Gibbs bracket are
\begin{gather}
Z_N^{x,y,q}:=\iint\exp\Ll(H_N^{x,y,q}(\sigma,\alpha)\Rr)P_N(\d\sigma)\fR(\d\alpha),
\qquad \bar F_N^{x,y}(t,q):=-\frac1N\E\log Z_N^{x,y,q},\notag\\
\la f\ra_N^{x,y,q}:=\frac{1}{(Z_N^{x,y,q})^n}\int f\Ll((\sigma^\ell,\alpha^\ell)_{\ell\le n}\Rr)\prod_{\ell=1}^n e^{H_N^{x,y,q}(\sigma^\ell,\alpha^\ell)}P_N(\d\sigma^\ell)\fR(\d\alpha^\ell),
\label{e.ZNx_def}
\end{gather}
where $f$ is a bounded measurable function of $n$ replicas. When $x=0$ and $y=0$, we recover $\bar F_N^{0,0}(t,q)=\bar F_N(t,q)$. The mechanism through which $V_N^{x,y,q}$ produces the concentration of $\mcl E_{N,s}^{x,q}$ is explained in Subsection~\ref{s.radial_and_gaussian}: the logarithm of the partition function is convex in $y$, with derivative proportional to $\la\mcl E^{x,q}_{N,s}\ra$, and a convex function whose fluctuations are small has, for most values of its argument, a derivative whose fluctuations are small as well.

\begin{remark}[Linear form of the perturbed Hamiltonians]
\label{r.linear_form_hamiltonians}
The Hamiltonian $H^{x,y,q}_N$ depends on the cascade variable $\alpha$ only through the cascade fields $w^{q^s}_i(\alpha)$ and the fields $w^{\msf h,j}_e(\alpha)$ of the construction~\eqref{e.feature_field_construction}, and it is linear in these fields, with coefficients that are polynomials in~$\sigma$; this remains true of its radial derivatives, hence of $V^{x,y,q}_N$. It is therefore of the form~\eqref{e.linear_cascade_hamiltonian} considered in Appendix~\ref{s.app_concentration}, with the finite family of the cascade fields and the countable family of the perturbation fields. The tail condition~\eqref{e.linear_cascade_tail} holds by~\eqref{e.ch_perturbation_lipschitz_summability}, since the contribution to the variance of the indices $\msf h$ outside a finite set $F$ is at most $CN\eta_N^2\sum_{\msf h\notin F}c_{\msf h}^2\Lambda_{\msf h}$, for some constant $C<\infty$. The same applies to the bulk Hamiltonian $\widetilde H^{x,y,q}_N$ of Subsection~\ref{s.bulk_and_chart_expansion} and to the Hamiltonian of the $(N+M)$-spin system. The invariance~\eqref{e.invariance_cascade} is stated for finitely many fields; we apply it to the Hamiltonian in which the sum over $\msf h$ is restricted to a finite set $F$, conditionally on the fields that do not depend on the cascade, and we then let $F$ increase to $\N^4$, using Lemma~\ref{l.gaussian_comparison_covariance_norm} below: by the Cauchy--Schwarz inequality, the covariances of the restricted and unrestricted Hamiltonians differ by at most the bound on the variance just given. Consequently, for every $n\in\N$, the overlap array $(\alpha^\ell\wedge\alpha^{\ell'})_{\ell,\ell'\le n}$ of $n$ replicas has the same law under every Gibbs measure built from these Hamiltonians as under $\E\la\cdot\ra_\fR$; in particular, the overlap $\alpha\wedge\alpha'$ of two replicas is uniformly distributed on $[0,1]$. Moreover, the corresponding free energies concentrate around their expectations at the rate given by Lemma~\ref{l.cascade_concentration}. These two facts are used repeatedly below without further comment.
\end{remark}

\subsubsection{Covariance of radially perturbed fields}

All the Gaussian fields we consider have a covariance of the form $f(R_N(\sigma,\sigma'),\alpha\wedge\alpha')$, and the following lemma computes the covariance of their radial perturbation. For $y\in\R^{\sS}$, we introduce the first-order differential operator
\begin{equation}
    \msf D_y:=\sum_{s\in\sS}y^sv_s\dr_{v_s},
    \label{e.D_y_def}
\end{equation}
acting on functions of $v\in\R^{\sS}$, and we write $\msf D_y^2:=\msf D_y\circ\msf D_y$.

\begin{lemma}[Covariance of radial perturbations]
\label{l.comput_cov_radial}
Let $f=f(v,u)$ be a function which is smooth in $v$ on a neighborhood of $\{(\sqrt{r^sr'^s}v_s)_{s\in\sS}:r,r'\in U_0,\ |v_s|\le\lambda_s\}$, and let $\msf W$ be a centered Gaussian field on $\{S_r\sigma:r\in U_0,\ \sigma\in\Sigma_N\}\times\mfk U$ with covariance $\E\msf W(\sigma,\alpha)\msf W(\sigma',\alpha')=f(R_N(\sigma,\sigma'),\alpha\wedge\alpha')$. For $\eta>0$ and $y\in\R^{\sS}$, set
\begin{align*}
    \msf W^y(\sigma,\alpha):= \msf W(\sigma,\alpha) + 2\eta\sum_{s\in\sS}y^s\partial_{r^s}\msf W(S_r\sigma,\alpha)\big|_{r=\vecone},
    \qquad \sigma\in\Sigma_N.
\end{align*}
Then $\E \msf W^y(\sigma,\alpha)\msf W^y(\sigma',\alpha') = f_{\eta,y}(R_N(\sigma,\sigma'),\alpha\wedge\alpha')$, where
\begin{align}\label{e.f_eta_N,y=}
    f_{\eta,y}(v,u):=f(v,u)+2\eta\,\msf D_yf(v,u)+\eta^2\msf D_y^2f(v,u).
\end{align}
\end{lemma}

\begin{proof}
Set $v:=R_N(\sigma,\sigma')$ and $u:=\alpha\wedge\alpha'$. By~\eqref{e.overlap_dilation}, the covariance of $\msf W(S_r\sigma,\alpha)$ and $\msf W(S_{r'}\sigma',\alpha')$ is $f((\sqrt{r^sr'^s}v_s)_{s\in\sS},u)$. Since radial derivatives are $L^2$ derivatives, they commute with expectations, and
\begin{align*}
    \E\msf W^y(\sigma,\alpha)\msf W^y(\sigma',\alpha')={}&\Ll(1+2\eta\sum_{s\in\sS}y^s\partial_{r^s}\Rr)\Ll(1+2\eta\sum_{s\in\sS}y^s\partial_{r'^s}\Rr)f\Ll((\sqrt{r^sr'^s}\,v_s)_{s\in\sS},u\Rr)\Big|_{r=r'=\vecone}.
\end{align*}
By the chain rule, $2\partial_{r^s}f((\sqrt{r^sr'^s}v_s)_s,u)|_{r=\vecone}=v_s\dr_{v_s}f((\sqrt{r'^s}v_s)_s,u)$, and similarly for $r'$. Expanding the product of the two operators gives~\eqref{e.f_eta_N,y=}.
\end{proof}

In particular, the covariance of $H_N^{x,y,q}$ is obtained from that of $H_N^{x,q}$ by applying the transformation in~\eqref{e.f_eta_N,y=} to each covariance kernel, with $\eta=\eta_N$. Explicitly, since the fields $H_N$, $W_N^q$, and $(H_N^{\msf h})_{\msf h}$ are independent, we obtain, for $v:=R_N(\sigma,\sigma')$ and $u:=\alpha\wedge\alpha'$,
\begin{align}
    \operatorname{Cov}\Ll(H_N^{x,y,q}(\sigma,\alpha),H_N^{x,y,q}(\sigma',\alpha')\Rr)={}&2tN\xi_{\eta_N,y}(v)+2N\sum_{s\in\sS}\Ll(1+\eta_Ny^s\Rr)^2q^s(u)v_s\notag\\
    &+N\eta_N^2\sum_{\msf h\in\N^4}x_{\msf h}^2c_{\msf h}^2(\msf C_{\msf h})_{\eta_N,y}(v,u),
    \label{e.covariance_perturbed_H}
\end{align}
where we used that $\msf D_y(q^s(u)v_s)=y^sq^s(u)v_s$ for the middle term. The mean of $H_N^{x,y,q}(\sigma,\alpha)$ is
\begin{align}
    \E H_N^{x,y,q}(\sigma,\alpha)={}&-Nt\Ll(\xi(\lambda)+2\eta_N\msf D_y\xi(\lambda)\Rr)-N\sum_{s\in\sS}q^s(1)\lambda_{s}\Ll(1+2\eta_Ny^s\Rr)\notag\\
    &+\sum_{s\in\sS}h^s\Ll(1+\eta_Ny^s\Rr)\sum_{i\in I_{N,s}}\sigma_i,
    \label{e.mean_perturbed_H}
\end{align}
since $R_N(S_r\sigma,S_r\sigma)=(r^s\lambda_s)_{s\in\sS}$, $|(S_r\sigma)^{(s)}|^2=r^sN_s$, and $\partial_{r^s}\sqrt{r^s}|_{r^s=1}=1/2$. By the absolute convergence in~\eqref{e.xi_power_series_def} and by~\eqref{e.ch_perturbation_lipschitz_summability}, the functions $\xi_{\eta_N,y}$ and $(\msf C_{\msf h})_{\eta_N,y}$, as well as their first two derivatives in $v$, are bounded uniformly over $y\in[0,3]^{\sS}$ and $N$ on the range of overlaps of interest.

\subsection{The spherical chart}
\label{s.spherical_chart}

We now describe the chart on $\Sigma_{N+M}$ that we use to separate the bulk and cavity coordinates. For $m,n\in\N$, set
\begin{align}
    \begin{split}
        \mcl B_{m,n}:=\{\tau\in\R^m: |\tau|^2<n+m\},
    \qquad
    \phi_{m,n}(\sigma,\tau):=\Ll(\sqrt{\frac{n+m-|\tau|^2}{n}}\,\sigma,\tau\Rr),
    \\
    p_{m,n}(\tau):=
    \begin{cases}
        \dfrac{\Gamma\Ll(\frac{n+m}{2}\Rr)}{\Gamma\Ll(\frac n2\Rr)(\pi(n+m))^{m/2}}
        \Ll(1-\dfrac{|\tau|^2}{n+m}\Rr)^{\frac n2-1},&\tau\in\mcl B_{m,n},\\
        0,&\tau\in\R^m\setminus\mcl B_{m,n},
    \end{cases}
    \end{split}
    \label{e.block_density}
\end{align}
where $\Gamma$ is Euler's Gamma function. The map $\phi_{m,n}$ sends $\Sph_n\times\mcl B_{m,n}$ onto $\Sph_{n+m}$ minus a null set. For $m\in\N$, we write
\begin{equation}\label{e.bgamma_m=}
\bgamma_m(\d u):=(2\pi)^{-m/2}\exp\Ll(-|u|^2/2\Rr)\d u
\end{equation}
for the standard Gaussian measure on $\R^m$. The following lemma is a block version of the disintegration used in~\cite{bates2022free,chen2013aizenman}, which peels off one coordinate at a time.

\begin{lemma}[Disintegration of the surface measure]
\label{l.block_disintegration}
For every measurable function $f:\Sph_{n+m}\to[0,\infty)$, we have
\begin{align}
    \int_{\Sph_{n+m}}f(\rho)\,\mu_{n+m}(\d\rho)
    =\iint_{\Sph_n\times \mcl B_{m,n}}f\Ll(\phi_{m,n}(\sigma,\tau)\Rr)
    p_{m,n}(\tau)\,\d\tau\,\mu_n(\d\sigma).
    \label{e.block_disintegration}
\end{align}
Moreover, for each fixed $m$, the density $p_{m,n}$ converges to the density of $\bgamma_m$ locally uniformly on $\R^m$ as $n\to\infty$.
\end{lemma}

\begin{proof}
Set $R:=\sqrt{n+m}$, and parametrize $\Sph_{n+m}$ minus a null set by
\begin{equation*}
    \mcl B_{m,n}\times\{\omega\in\R^n:|\omega|=1\}\ni(\tau,\omega)\longmapsto\Ll(\sqrt{R^2-|\tau|^2}\,\omega,\tau\Rr).
\end{equation*}
The partial derivative of this map with respect to $\tau_k$ is $(-\tau_k(R^2-|\tau|^2)^{-1/2}\omega,e_k)$, and its differential in a unit direction $\zeta$ tangent to the sphere at $\omega$ is $((R^2-|\tau|^2)^{1/2}\zeta,0)$. Each vector of the second kind is orthogonal to each vector of the first kind, since $\zeta\perp\omega$. The Gram determinant of the $m$ vectors of the first kind is $\det(I_m+\tau\tau^\intercal/(R^2-|\tau|^2))=R^2/(R^2-|\tau|^2)$, and that of the $n-1$ vectors of the second kind, with $\zeta$ ranging over an orthonormal basis of the tangent space, is $(R^2-|\tau|^2)^{n-1}$. Denoting by $\omega_{n-1}$ the surface measure on the unit sphere of $\R^n$, and taking the square root of the product of these two determinants, we find that the surface measure on $\Sph_{n+m}$ is
\begin{equation*}
    R\Ll(R^2-|\tau|^2\Rr)^{\frac{n-2}{2}}\d\tau\,\omega_{n-1}(\d\omega)
\end{equation*}
in these coordinates. Normalizing this measure, we see that under $\mu_{n+m}$, the marginal law of $\tau$ has a density proportional to $(1-|\tau|^2/(n+m))^{n/2-1}$ on $\mcl B_{m,n}$ and that, conditionally on $\tau$, the vector $\omega$ is uniform on the unit sphere of $\R^n$; equivalently, $\sigma:=\sqrt n\,\omega$ is distributed according to $\mu_n$, and the first $n$ coordinates of $\rho$ are $\sqrt{(n+m-|\tau|^2)/n}\,\sigma$. This is the content of~\eqref{e.block_disintegration}, up to the identification of the normalizing constant, which follows from the classical formula
\begin{equation*}
    \int_{\{|v|<1\}}\Ll(1-|v|^2\Rr)^{a-1}\d v=\frac{\pi^{m/2}\Gamma(a)}{\Gamma(a+\frac m2)},
    \qquad a>0,
\end{equation*}
applied with $a=n/2$ after the change of variables $\tau=Rv$. For the last assertion, Stirling's formula gives $\Gamma(\frac{n+m}2)/\Gamma(\frac n2)=(\frac n2)^{m/2}(1+o(1))$ as $n\to\infty$, so that the constant in~\eqref{e.block_density} converges to $(2\pi)^{-m/2}$, while $(1-|\tau|^2/(n+m))^{n/2-1}$ converges to $e^{-|\tau|^2/2}$ uniformly on compact sets.
\end{proof}

We apply this disintegration to each species. We set
\begin{align}
    \mcl B_N:=\prod_{s\in\sS}\mcl B_{M_s,N_s},
    \qquad
    p_N^{\mathrm{ch}}(\tau):=\prod_{s\in\sS}p_{M_s,N_s}(\tau^{(s)}),\qquad
    r_N^s(\tau):=1+\frac{M_s-|\tau^{(s)}|^2}{N_s}.
    \label{e.radial_factor}
\end{align}
Comparing with~\eqref{e.block_density}, we see that $\sqrt{r^s_N(\tau)}$ is the factor by which $\phi_{M_s,N_s}$ rescales the species-$s$ block of $\sigma$. For $(\sigma,\tau)\in\Sigma_N\times\mcl B_N$, the point $(S_{r_N(\tau)}\sigma,\tau)$ belongs to $\Sigma_{N+M}$, with the identification in~\eqref{e.N_compatible_block}. We call this parametrization of $\Sigma_{N+M}$, minus a null set, by $\Sigma_N\times\mcl B_N$ the \emph{chart}, and we call $p_N^{\mathrm{ch}}$ the \emph{chart density}; the superscript ``ch'' stands for ``chart''. Since $P_{N+M}$ is a product of surface measures over the species, Lemma~\ref{l.block_disintegration} gives, for every nonnegative measurable $f$ on $\Sigma_{N+M}$,
\begin{align}
    \int_{\Sigma_{N+M}}f(\rho)\,P_{N+M}(\d\rho)
    =\iint_{\Sigma_N\times\mcl B_N}f\Ll(S_{r_N(\tau)}\sigma,\tau\Rr)
    p_N^{\mathrm{ch}}(\tau)\,\d\tau\,P_N(\d\sigma).
    \label{e.product_disintegration}
\end{align}
We record two consequences. First, the second moment of the cavity coordinates under the reference measure is
\begin{equation}
    \int_{\mcl B_N}|\tau^{(s)}|^2\,p_N^{\mathrm{ch}}(\tau)\,\d\tau=M_s,
    \qquad s\in\sS,
    \label{e.chart_second_moment}
\end{equation}
since each coordinate of a point of $\Sph_{N_s+M_s}$ has second moment $1$ under the uniform measure. Second, for every $L<\infty$, if $|\tau|\le L+1$ then
\begin{equation}
    \Ll|r_N^s(\tau)-1\Rr|\le\frac{M_s+(L+1)^2}{N_s},
    \qquad s\in\sS,
    \label{e.r_N_close_to_one}
\end{equation}
so that $r_N(\tau)\in U_0$ for all sufficiently large $N$, where $U_0$ is defined in~\eqref{e.U_0_def}. In the remainder of the section, we tacitly take $N$ large enough, depending on $L$, for this to hold.

\subsection{Two Gaussian comparison lemmas}
\label{s.gaussian_comparison_notation}

The cavity computation consists of a chain of replacements of one Gaussian field by another, each of which is justified by the closeness of their means and covariances. We introduce a notation for such comparisons, and record the two interpolation lemmas that we use.

\subsubsection*{Asymptotic notation}
Let $(\delta_N)$ be a sequence of positive numbers, not necessarily tending to zero as $N$ tends to infinity. For two sequences $(s_N)$ and $(s'_N)$ of real numbers, we write
\begin{equation*}
    s_N \stackrel{\delta_N}{\lapprox} s'_N \qquad\text{whenever}\qquad \Ll|s_N-s'_N\Rr|=O(\delta_N)\quad\text{as }N\to\infty.
\end{equation*}
For two sequences $(\msf f_N)$ and $(\msf f'_N)$ of functions defined on sets $\mathbb X_N$, we write
\begin{equation*}
    \msf f_N(\msf x)\stackrel{\delta_N}{\lapprox}\msf f'_N(\msf x)\quad\text{uniformly over $\msf x\in\mathbb X_N$}
    \qquad\text{whenever}\qquad
    \sup_{\msf x\in\mathbb X_N}\Ll|\msf f_N(\msf x)-\msf f'_N(\msf x)\Rr|=O(\delta_N).
\end{equation*}
Let $(\delta'_N)$ be another sequence of positive numbers. For two sequences $(\msf W_N)$ and $(\msf W'_N)$ of Gaussian fields indexed by $\mathbb X_N$, we write
\begin{equation}
    \msf W_N(\msf x)\stackrel{\delta_N,\delta'_N}{\lapprox}\msf W'_N(\msf x)\quad\text{uniformly over $\msf x\in\mathbb X_N$}
    \label{e.Gaussian_approx}
\end{equation}
whenever
\begin{equation*}
    \E\msf W_N(\msf x)\stackrel{\delta_N}{\lapprox}\E\msf W'_N(\msf x)
    \qquad\text{and}\qquad
    \operatorname{Cov}\Ll(\msf W_N(\msf x),\msf W_N(\msf x')\Rr)\stackrel{\delta'_N}{\lapprox}\operatorname{Cov}\Ll(\msf W'_N(\msf x),\msf W'_N(\msf x')\Rr),
\end{equation*}
uniformly over $\msf x\in\mathbb X_N$ and over $(\msf x,\msf x')\in\mathbb X_N^2$ respectively. When $\delta_N$ and $\delta'_N$ tend to zero as $N$ tends to infinity and we do not wish to keep track of the rates, we simply write $\lapprox$. The implicit constants in these relations are allowed to depend on the fixed parameters of the problem, including $L$ when a cutoff is present, but not on $N$; when the relation is claimed to hold uniformly over $x$, $y$, $q$, or $b$, the constants do not depend on these parameters either. Note that if $s_N\stackrel{\delta_N}{\lapprox}s'_N$ and $s'_N\stackrel{\delta'_N}{\lapprox}s''_N$, then $s_N\stackrel{\delta_N+\delta'_N}{\lapprox}s''_N$, and similarly for the other relations. Two remarks on the meaning of~\eqref{e.Gaussian_approx} are in order. First, the relation only compares the laws of the two fields, through their means and covariances; it does not assert a pathwise approximation, and in the estimates below a field may be replaced by any other field with the same law. Second, when the index set $\mathbb X_N$ contains a factor~$\mfk U$, the fields are Gaussian conditionally on the cascade $\fR$, and the relation~\eqref{e.Gaussian_approx} is understood conditionally on $\fR$, with implicit constants that do not depend on its realization.

The first comparison lemma is the standard Gaussian interpolation, in a form that keeps track of the errors.

\begin{lemma}[Gaussian comparison in covariance norm]
\label{l.gaussian_comparison_covariance_norm}
Let $\msf P$ be a probability measure on a measurable space $\mathbb X$, and let $\msf W^0$ and $\msf W^1$ be Gaussian fields indexed by $\mathbb X$, with bounded means $m^0$, $m^1$ and bounded covariances $C^0$, $C^1$. Set $Z^i:=\int_{\mathbb X}e^{\msf W^i(\msf x)}\msf P(\d\msf x)$, and let $\la\cdot\ra_i$ be the associated Gibbs bracket, for $i\in\{0,1\}$. Then
\begin{equation*}
    \Ll|\E\log Z^1-\E\log Z^0\Rr|\le\sup_{\mathbb X}\Ll|m^1-m^0\Rr|+\sup_{\mathbb X\times\mathbb X}\Ll|C^1-C^0\Rr|,
\end{equation*}
and, for every $n\in\N$, there is a constant $C_n<\infty$, depending only on $n$, such that, for every bounded measurable function $f$ of $n$ replicas,
\begin{equation*}
    \Ll|\E\la f\ra_1-\E\la f\ra_0\Rr|\le C_n\|f\|_\infty\Ll(\sup_{\mathbb X}\Ll|m^1-m^0\Rr|+\sup_{\mathbb X\times\mathbb X}\Ll|C^1-C^0\Rr|\Rr).
\end{equation*}
In particular, with the notation~\eqref{e.Gaussian_approx}, if $\msf W^0_N\stackrel{\delta_N,\delta'_N}{\lapprox}\msf W^1_N$ uniformly over $\mathbb X_N$, then
\begin{align*}
    \E\log Z^0_N\stackrel{\delta_N+\delta'_N}{\lapprox}\E\log Z^1_N
    \qquad\text{and}\qquad
    \E\la f\ra_{0,N}\stackrel{\delta_N+\delta'_N}{\lapprox} \E\la f\ra_{1,N}.
\end{align*}
\end{lemma}

\begin{proof}
Let $\msf G^i:=\msf W^i-m^i$, and realize $\msf G^0$ and $\msf G^1$ as independent fields; this does not change the law of $Z^0$, of $Z^1$, nor of the individual brackets. For $u\in[0,1]$, set $\msf W^u:=(1-u)m^0+um^1+\sqrt{1-u}\,\msf G^0+\sqrt u\,\msf G^1$, and let $Z^u$ and $\la\cdot\ra_u$ be the associated partition function and bracket. Gaussian integration by parts gives
\begin{equation*}
    \frac{\d}{\d u}\E\log Z^u=\E\la m^1(\msf x)-m^0(\msf x)\ra_u+\frac12\E\la (C^1-C^0)(\msf x,\msf x)\ra_u-\frac12\E\la(C^1-C^0)(\msf x,\msf x')\ra_u,
\end{equation*}
where $\msf x,\msf x'$ denote two replicas, and the first estimate follows by integration over $u\in[0,1]$. Similarly, $\frac{\d}{\d u}\E\la f\ra_u$ is a finite sum of Gibbs averages of products of $f$ with $(m^1-m^0)(\msf x^\ell)$ or $(C^1-C^0)(\msf x^\ell,\msf x^{\ell'})$, for $\ell,\ell'\le n+2$, with coefficients depending only on $n$; this gives the second estimate. The last assertion is a restatement of the first two.
\end{proof}

The second lemma allows the difference between the two fields to grow quadratically in the cavity coordinates. It is tailored to the comparison of the $(N+M)$-spin system in the chart with the $N$-spin system; see Lemmas~\ref{l.fixed_block_increment_bound} and~\ref{l.one_coordinate_comparison}.

\begin{lemma}[Gaussian comparison with quadratic confinement]
\label{l.quadratically_confined_Gaussian_comparison}
For every $C_0<\infty$ and $M\in\N$, the following holds. Let $\mu^0$ be a probability measure on a measurable space $\mathbb X^0$, let $B\subset\R^M$ be a bounded measurable set, let $\nu$ be a probability measure on $B$ with $\int|\tau|^2\nu(\d\tau)=M$, and set $\mathbb X:=\mathbb X^0\times B$ and $\mu:=\mu^0\otimes\nu$. Let $\msf W^0$ and $\msf W^1$ be Gaussian fields indexed by $\msf z=(\msf z_0,\tau)\in\mathbb X$, with means $m^0$, $m^1$ and covariances $C^0$, $C^1$, such that $\msf W^0(\msf z_0,\tau)$ does not depend on $\tau$, $m^0$ and $C^0$ are bounded, and
\begin{equation}
    |m^1(\msf z)-m^0(\msf z)|+|C^1(\msf z,\msf z')-C^0(\msf z,\msf z')|\le C_0\Ll(1+|\tau|^2+|\tau'|^2\Rr),
    \qquad \msf z,\msf z'\in\mathbb X.
    \label{e.quadratic_kernel_comparison}
\end{equation}
Let $Z^i:=\int_{\mathbb X}e^{\msf W^i(\msf z)}\mu(\d\msf z)$ for $i\in\{0,1\}$, and let $\la\cdot\ra_1$ be the Gibbs bracket associated with $Z^1$. If
\begin{equation}
    \E\la|\tau|^2\ra_1=M,
    \label{e.quadratic_comparison_second_moments}
\end{equation}
then $|\E\log Z^1-\E\log Z^0|\le 3C_0(1+M)$.
\end{lemma}

\begin{proof}
Since $B$, $m^0$, and $C^0$ are bounded, the fields have bounded means and variances on $\mathbb X$, so that $Z^0$ and $Z^1$ are almost surely finite with integrable logarithms. As in the previous proof, we take the centered fields $\msf G^i:=\msf W^i-m^i$ to be independent. For $u\in[0,1]$ and $c\ge0$, we consider the interpolation with a quadratic confinement,
\begin{equation*}
    \Phi(u,c):=\E\log\int_{\mathbb X}\exp\Ll((1-u)m^0(\msf z)+um^1(\msf z)+\sqrt{1-u}\,\msf G^0(\msf z)+\sqrt u\,\msf G^1(\msf z)-c|\tau|^2\Rr)\mu(\d\msf z),
\end{equation*}
and we denote by $\la\cdot\ra_{u,c}$ the associated Gibbs bracket. Gaussian integration by parts, as in the previous proof, together with~\eqref{e.quadratic_kernel_comparison}, gives
\begin{equation}
    \Ll|\partial_u\Phi(u,c)\Rr|\le 3C_0\Ll(1+\E\la|\tau|^2\ra_{u,c}\Rr)\qquad\text{and}\qquad \partial_c\Phi(u,c)=-\E\la|\tau|^2\ra_{u,c}.
    \label{e.confined_interpolation_derivatives}
\end{equation}
Set $c_0:=3C_0$. Along the path $u\mapsto(u,c_0(1-u))$, we obtain from~\eqref{e.confined_interpolation_derivatives} that
\begin{equation*}
    \frac{\d}{\d u}\Phi\Ll(u,c_0(1-u)\Rr)=\partial_u\Phi+c_0\E\la|\tau|^2\ra_{u,c_0(1-u)}\ge-3C_0.
\end{equation*}
Moreover, since $\msf W^0$ does not depend on $\tau$ and $\mu$ is a product measure, Jensen's inequality and the assumption on $\nu$ give
\begin{equation*}
    \Phi(0,c_0)-\Phi(0,0)=\log\int_Be^{-c_0|\tau|^2}\nu(\d\tau)\ge-c_0M.
\end{equation*}
Combining the last two displays yields $\E\log Z^1=\Phi(1,0)\ge\Phi(0,c_0)-3C_0\ge\E\log Z^0-c_0M-3C_0$. For the reverse inequality, we use the path $u\mapsto(u,c_0u)$, along which~\eqref{e.confined_interpolation_derivatives} gives $\frac{\d}{\d u}\Phi(u,c_0u)\le3C_0$, so that $\Phi(1,c_0)\le\Phi(0,0)+3C_0$. By Jensen's inequality and~\eqref{e.quadratic_comparison_second_moments},
\begin{equation*}
    \Phi(1,0)-\Phi(1,c_0)=\E\Ll[-\log\la e^{-c_0|\tau|^2}\ra_{1}\Rr]\le c_0\E\la|\tau|^2\ra_{1}=c_0M.
\end{equation*}
Hence $\E\log Z^1=\Phi(1,0)\le\E\log Z^0+3C_0+c_0M$, which completes the proof.
\end{proof}

\subsection{The cavity increment is bounded}
\label{s.ASS_scheme}

We recall the partition function and the free energy from~\eqref{e.ZNx_def}, and we define the cavity increment
\begin{equation}
    A_N(x,y,q):=\E\log Z_{N+M}^{x,y,q}-\E\log Z_N^{x,y,q}=-(N+M)\bar F_{N+M}^{x,y}(t,q)+N\bar F_N^{x,y}(t,q).
    \label{e.ASS_increment_def}
\end{equation}
The main result of this section, Theorem~\ref{t.spherical_cavity_lim}, gives a lower bound on the $\liminf$ of $A_N$ along suitable sequences. As a preliminary step, we show that $A_N$ is bounded. The proof is a first instance of the use of the chart of Subsection~\ref{s.spherical_chart}, and it also introduces the estimates on overlaps that we use repeatedly below.

\begin{lemma}[Boundedness of the cavity increment]
\label{l.fixed_block_increment_bound}
There exists a constant $C<\infty$ such that, for every sufficiently large $N$,
\begin{equation*}
    \sup_{x,y,q}|A_N(x,y,q)|\le C.
\end{equation*}
\end{lemma}

\begin{proof}
The Hamiltonian $H_{N+M}^{x,y,q}$ of the $(N+M)$-spin system is built with the coefficient $\eta_{N+M}$, whereas $H_N^{x,y,q}$ is built with $\eta_N$. We first compare $H_N^{x,y,q}$ with a version of the $(N+M)$-spin Hamiltonian built with $\eta_N$, and then return to $\eta_{N+M}$.

\smallskip
\noindent\emph{Step 1: Overlaps in the chart.}\par
Let $\rho=(S_{r_N(\tau)}\sigma,\tau)$ and $\rho'=(S_{r_N(\tau')}\sigma',\tau')$ with $(\sigma,\tau),(\sigma',\tau')\in\Sigma_N\times\mcl B_N$. By~\eqref{e.R_Ns_def}, \eqref{e.overlap_dilation}, and~\eqref{e.radial_factor},
\begin{equation}
    R_{N+M,s}(\rho,\rho')=\frac{N}{N+M}\sqrt{r_N^s(\tau)r_N^s(\tau')}\,R_{N,s}(\sigma,\sigma')+\frac{1}{N+M}\sum_{j\in\msf M_s}\tau_j\tau_j'.
    \label{e.full_and_bulk_overlap_identity}
\end{equation}
Since $|\sqrt a-1|\le|a-1|$ for $a\ge0$ and $0<r_N^s(\tau)\le2$, we have $|\sqrt{r_N^s(\tau)r_N^s(\tau')}-1|\le2(|r_N^s(\tau)-1|+|r_N^s(\tau')-1|)$, and $N|r^s_N(\tau)-1|\le(M_s+|\tau^{(s)}|^2)/\lambda_s$ by~\eqref{e.radial_factor} and~\eqref{e.N_compatible_block}. Using also $|R_{N,s}(\sigma,\sigma')|\le\lambda_s$, we deduce from~\eqref{e.full_and_bulk_overlap_identity} that
\begin{equation}
    (N+M)\Ll|R_{N+M}(\rho,\rho')-R_N(\sigma,\sigma')\Rr|\le C\Ll(1+|\tau|^2+|\tau'|^2\Rr),
    \label{e.full_and_bulk_overlap_bound}
\end{equation}
for a constant $C$ depending only on $M$.

\smallskip
\noindent\emph{Step 2: The $(N+M)$-spin system with coefficient $\eta_N$.}\par
For $K\in\{N,N+M\}$, let $\widehat H_{K;N}^{x,y,q}$ be defined as $H_K^{x,y,q}$ in~\eqref{e.perturbed_enriched_H} and~\eqref{e.full_radial_perturbation}, except that $\eta_K$ is replaced by $\eta_N$ throughout, and let $\widehat Z_{K;N}^{x,y,q}$ be the associated partition function. Thus $\widehat H_{N;N}^{x,y,q}=H_N^{x,y,q}$. The mean and covariance of $\widehat H_{K;N}^{x,y,q}$ are given by~\eqref{e.mean_perturbed_H} and~\eqref{e.covariance_perturbed_H}, with $N$ replaced by $K$ in the prefactors and in the overlaps, and with $\eta_N$ unchanged; we use here that $\lambda_K=\lambda$ by~\eqref{e.lambda_N_equals_lambda}.

Let $\rho,\rho'$ be as in Step~1. We claim that
\begin{equation}
    \Big|\E\widehat H_{N+M;N}^{x,y,q}(\rho,\alpha)-\E H_N^{x,y,q}(\sigma,\alpha)\Big|\le C\Ll(1+|\tau|^2\Rr)
    \label{e.frozen_chart_mean_bound}
\end{equation}
and
\begin{multline}
    \Big|\operatorname{Cov}\Ll(\widehat H_{N+M;N}^{x,y,q}(\rho,\alpha),\widehat H_{N+M;N}^{x,y,q}(\rho',\alpha')\Rr)-\operatorname{Cov}\Ll(H_N^{x,y,q}(\sigma,\alpha),H_N^{x,y,q}(\sigma',\alpha')\Rr)\Big|\\
    \le C\Ll(1+|\tau|^2+|\tau'|^2\Rr),
    \label{e.frozen_chart_covariance_bound}
\end{multline}
uniformly over $x,y,q$, $\alpha,\alpha'\in\mfk U$, and the configurations. For~\eqref{e.frozen_chart_mean_bound}, the first two terms in~\eqref{e.mean_perturbed_H} are proportional to $K$ with bounded coefficients, so the difference of their values at $K=N+M$ and $K=N$ is bounded by $CM$. For the external-field term, we note that $(S_{r_N(\tau)}\sigma)_i=\sqrt{r_N^s(\tau)}\sigma_i$ for $i\in I_{N,s}$, and that
\begin{equation*}
    \Big|\Big(\sqrt{r_N^s(\tau)}-1\Big)\sum_{i\in I_{N,s}}\sigma_i\Big|\le N_s\Ll|r_N^s(\tau)-1\Rr|\le M_s+|\tau^{(s)}|^2,
    \qquad
    \Big|\sum_{j\in\msf M_s}\tau_j\Big|\le\sqrt{M_s}\,|\tau^{(s)}|,
\end{equation*}
by the Cauchy--Schwarz inequality and~\eqref{e.radial_factor}. This proves~\eqref{e.frozen_chart_mean_bound}. For~\eqref{e.frozen_chart_covariance_bound}, each term in~\eqref{e.covariance_perturbed_H} is of the form $Kg(R_K,u)$, where $g$ is one of the kernels $2t\xi_{\eta_N,y}$, $2\sum_s(1+\eta_Ny^s)^2q^s(u)v_s$, and $\eta_N^2\sum_{\msf h}x^2_{\msf h}c_{\msf h}^2(\msf C_{\msf h})_{\eta_N,y}$, and $R_K$ stands for $R_{N+M}(\rho,\rho')$ or $R_N(\sigma,\sigma')$. These kernels and their gradients in $v$ are bounded uniformly over $x,y,q$ on the range of overlaps, by the remark following~\eqref{e.mean_perturbed_H}. Writing
\begin{equation*}
    (N+M)g(R_{N+M},u)-Ng(R_N,u)=Mg(R_{N+M},u)+N\Ll(g(R_{N+M},u)-g(R_N,u)\Rr)
\end{equation*}
and using~\eqref{e.full_and_bulk_overlap_bound} for the second term proves~\eqref{e.frozen_chart_covariance_bound}.

We now apply Lemma~\ref{l.quadratically_confined_Gaussian_comparison} with $\mathbb X^0:=\Sigma_N\times\mfk U$, $\mu^0:=P_N\otimes\fR$, $B:=\mcl B_N$, $\nu:=p_N^{\mathrm{ch}}(\tau)\d\tau$, and the fields
\begin{equation*}
    \msf W^0(\sigma,\tau,\alpha):=H_N^{x,y,q}(\sigma,\alpha),
    \qquad
    \msf W^1(\sigma,\tau,\alpha):=\widehat H_{N+M;N}^{x,y,q}\Ll((S_{r_N(\tau)}\sigma,\tau),\alpha\Rr).
\end{equation*}
The set $\mcl B_N$ is bounded, $\nu$ is a probability measure on $\mcl B_N$ with second moment $M$ by~\eqref{e.chart_second_moment}, and~\eqref{e.quadratic_kernel_comparison} holds by~\eqref{e.frozen_chart_mean_bound} and~\eqref{e.frozen_chart_covariance_bound}. By~\eqref{e.product_disintegration}, we have $Z^0=Z_N^{x,y,q}$ and $Z^1=\widehat Z_{N+M;N}^{x,y,q}$, and the bracket $\la\cdot\ra_1$ is the Gibbs bracket of $\widehat H_{N+M;N}^{x,y,q}$ expressed in the chart. It remains to verify~\eqref{e.quadratic_comparison_second_moments}. The mean and covariance of $\widehat H^{x,y,q}_{N+M;N}$ depend on the configurations only through their species overlaps and the sums $\sum_{i\in I_{N+M,s}}\rho_i$, which are invariant under permutations of the coordinates within each species. The law of the field is therefore invariant under such permutations, and so is $\E\la\rho_i^2\ra_1$ as a function of $i\in I_{N+M,s}$. Since $\sum_{i\in I_{N+M,s}}\rho_i^2=N_s+M_s$ on $\Sigma_{N+M}$, we get $\E\la\rho_i^2\ra_1=1$ for every $i$, hence $\E\la|\tau^{(s)}|^2\ra_1=M_s$ and $\E\la|\tau|^2\ra_1=M$. Lemma~\ref{l.quadratically_confined_Gaussian_comparison} thus yields
\begin{equation}
    \sup_{x,y,q}\Ll|\E\log\widehat Z_{N+M;N}^{x,y,q}-\E\log Z_N^{x,y,q}\Rr|\le C.
    \label{e.frozen_fixed_block_increment}
\end{equation}

\smallskip
\noindent\emph{Step 3: Back to the coefficient $\eta_{N+M}$.}\par
Set $K:=N+M$. The fields $H_K^{x,y,q}$ and $\widehat H_{K;N}^{x,y,q}$ are built from the same Gaussian fields $H_K$, $W_K^q$, and $(H_K^{\msf h})_{\msf h}$; they differ only in that $\eta_K$ is replaced by $\eta_N$. By~\eqref{e.mean_perturbed_H}, their means differ by
\begin{equation*}
    (\eta_K-\eta_N)\Big(-2Kt\msf D_y\xi(\lambda)-2K\sum_{s\in\sS}q^s(1)\lambda_sy^s+\sum_{s\in\sS}h^sy^s\sum_{i\in I_{K,s}}\rho_i\Big),
\end{equation*}
which is bounded by $CK|\eta_K-\eta_N|$ in absolute value. By~\eqref{e.f_eta_N,y=}, for every kernel $f$,
\begin{equation*}
    f_{\eta_K,y}-f_{\eta_N,y}=(\eta_K-\eta_N)\Ll(2\msf D_yf+(\eta_K+\eta_N)\msf D_y^2f\Rr),
\end{equation*}
and $|\eta_K^2-\eta_N^2|\le2|\eta_K-\eta_N|$. Applying this to the three terms in~\eqref{e.covariance_perturbed_H}, and using~\eqref{e.ch_perturbation_lipschitz_summability} to sum over $\msf h$, we find that the covariances of $H_K^{x,y,q}$ and $\widehat H_{K;N}^{x,y,q}$ differ by at most $CK|\eta_K-\eta_N|$, uniformly over $x,y,q$ and over $\Sigma_K\times\mfk U$. Lemma~\ref{l.gaussian_comparison_covariance_norm} therefore gives
\begin{equation}
    \sup_{x,y,q}\Ll|\E\log Z_{N+M}^{x,y,q}-\E\log\widehat Z_{N+M;N}^{x,y,q}\Rr|\le CN|\eta_{N+M}-\eta_N|,
    \label{e.actual_and_frozen_increment}
\end{equation}
and the right-hand side tends to zero, since $N|\eta_{N+M}-\eta_N|\le MN^{-1/16}$ by~\eqref{e.eta_N_def}. Combining~\eqref{e.ASS_increment_def}, \eqref{e.frozen_fixed_block_increment}, and~\eqref{e.actual_and_frozen_increment} proves the lemma.
\end{proof}

\subsection{The bulk system and the expansion of the Hamiltonian in the chart}
\label{s.bulk_and_chart_expansion}

We now begin the cavity computation proper. Throughout the rest of the section, we work in the chart of Subsection~\ref{s.spherical_chart}: a configuration $\rho\in\Sigma_{N+M}$ is written as $\rho=(S_{r_N(\tau)}\sigma,\tau)$ with $(\sigma,\tau)\in\Sigma_N\times\mcl B_N$, and we often write $\widetilde\sigma:=S_{r_N(\tau)}\sigma$ for brevity.

\subsubsection{Cutoff on the cavity coordinates}
\label{s.cavity_partitions}

For each $L\ge1$, we fix a continuous function $\chi_L:\R^M\to[0,1]$ such that
\begin{align}
    \one_{\{|\tau|\le L\}}\le \chi_L(\tau)\le \one_{\{|\tau|\le L+1\}},
    \qquad
    \chi_L\le \chi_{L'}\quad\text{if }L\le L',
    \label{e.chi_L_def}
\end{align}
and we set $\chi_\infty:=1$. The partition function and Gibbs bracket of the $(N+M)$-spin system with cutoff $L$ are
\begin{align}
    \begin{split}
        Z^{x,y,q}_{N+M,L}
        :={}&\iiint_{\Sigma_N\times \mcl B_N\times \mfk U}
        \chi_L(\tau)
    \exp\Ll(H_{N+M}^{x,y,q}\Ll((S_{r_N(\tau)}\sigma,\tau),\alpha\Rr)\Rr)
    p_N^{\mathrm{ch}}(\tau)\,\d\tau\,P_N(\d\sigma)\,\fR(\d\alpha),
        \\
        \la f\ra_{N+M,L}^{x,y,q}
         :={}&\frac{1}{(Z^{x,y,q}_{N+M,L})^n}
        \int f\Ll((\rho^\ell,\alpha^\ell)_{\ell\le n}\Rr)\prod_{\ell=1}^n\chi_L(\tau^\ell)
    e^{H_{N+M}^{x,y,q}(\rho^\ell,\alpha^\ell)}
    p_N^{\mathrm{ch}}(\tau^\ell)\,\d\tau^\ell\, P_N(\d\sigma^\ell)\,\fR(\d\alpha^\ell),
    \end{split}
    \label{e.truncated_true_partition}
\end{align}
where $\rho^\ell=(S_{r_N(\tau^\ell)}\sigma^\ell,\tau^\ell)$ and $f$ is a bounded measurable function of $n$ replicas. By~\eqref{e.product_disintegration}, the bracket $\la\cdot\ra^{x,y,q}_{N+M,\infty}$ is the bracket $\la\cdot\ra^{x,y,q}_{N+M}$ in~\eqref{e.ZNx_def}, and $\la\cdot\ra^{x,y,q}_{N+M,L}$ is the same Gibbs measure conditioned, in each replica, by the factor $\chi_L(\tau)$. Since $0\le\chi_L\le1$, we also have
\begin{equation}
    Z^{x,y,q}_{N+M} = Z^{x,y,q}_{N+M,\infty}\ge Z^{x,y,q}_{N+M,L}.
    \label{e.full_partition_dominates_cutoff}
\end{equation}
This inequality is the source of the one-sided nature of our cavity estimate.

\subsubsection{The bulk system}
\label{s.bulk_system}

Let $\rho=(u,\tau)$ and $\rho'=(u',\tau')$ with $u,u'\in\R^N$ and $\tau,\tau'\in\R^M$. By the definition~\eqref{e.R_Ns_def} of the overlaps and the identification~\eqref{e.N_compatible_block},
\begin{align}
    R_{N+M}(\rho,\rho')&=\msf a_N(u,u')+\msf b_N(\tau,\tau'),
    \qquad\text{where}\notag\\
    \msf a_N(u,u')&:=\frac{N}{N+M}R_N(u,u')
    \qquad\text{and}\qquad
    \msf b_N(\tau,\tau'):=\frac1{N+M}\Big(\sum_{j\in\msf M_s}\tau_j\tau'_j\Big)_{s\in\sS}.
    \label{e.overlap_split}
\end{align}
The covariance of $H_{N+M}$ is therefore $(N+M)\xi(\msf a_N+\msf b_N)$. When $\tau$ and $\tau'$ are bounded, $\msf b_N$ is of order $N^{-1}$, and the leading term $(N+M)\xi(\msf a_N)$ only involves the bulk coordinates. This leads us to define the bulk covariance function
\begin{equation}
    \widetilde\xi_N(v):=\frac{N+M}{N}\,\xi\Ll(\frac{N}{N+M}v\Rr),
    \qquad v\in\R^{\sS},
    \label{e.tilde_xi_def}
\end{equation}
and to let $\widetilde H_N$ be a centered Gaussian field on $\R^N$ with covariance
\begin{align}\label{e.tH=}
    \E \widetilde H_N(u)\widetilde H_N(u')
    =N\widetilde\xi_N\Ll(R_N(u,u')\Rr)=(N+M)\xi\Ll(\msf a_N(u,u')\Rr),
    \qquad u,u'\in\R^N.
\end{align}
The function $\widetilde\xi_N$ is obtained from $\xi$ by multiplying the coefficient of each monomial of total degree $d$ in~\eqref{e.xi_power_series_def} by $(N/(N+M))^{d-1}\in(0,1]$. Hence $\widetilde\xi_N$ is again a covariance function of the form~\eqref{e.xi_power_series_def}, so that the field $\widetilde H_N$ exists on $\R^N$, as was the case for $H_N$, and $\widetilde\xi_N$ together with its derivatives of order at most four converge to those of $\xi$, uniformly on $\prod_{s\in\sS}[-(1+\delta_0)\lambda_s,(1+\delta_0)\lambda_s]$, as $N\to\infty$. In other words, the bulk system is the $N$-spin system with $\xi$ replaced by $\widetilde\xi_N$, and every estimate that we have obtained so far for the $N$-spin system applies to the bulk system, uniformly in $N$.

The fields $H_N$ and $\widetilde H_N$ differ by a field whose covariance is of order one. Explicitly, set
\begin{equation}
    \theta_N(v):=\frac NM\Ll(\xi(v)-\widetilde\xi_N(v)\Rr)=\frac1M\Ll[N\xi(v)-(N+M)\xi\Ll(\frac{N}{N+M}v\Rr)\Rr].
    \label{e.theta_N_def}
\end{equation}
The coefficient of a monomial of degree $d$ in $\theta_N$ is the corresponding coefficient in $\xi$ multiplied by $\frac NM[1-(\frac N{N+M})^{d-1}]\ge0$, which converges to $d-1$ as $N\to\infty$. Since $\theta(v)=v\cdot\nabla\xi(v)-\xi(v)$ is obtained from $\xi$ by multiplying the coefficient of each monomial of degree $d$ by $d-1$, the absolute convergence in~\eqref{e.xi_power_series_def} gives
\begin{equation}
    \lim_{N\to\infty}\sup_{v\in\prod_s[-\lambda_{s},\lambda_{s}]}\big|\theta_N(v)-\theta(v)\big|=0,
    \label{e.cavity_reaction_kernel_limit}
\end{equation}
and $\theta_N$ is a valid covariance function, again of the form~\eqref{e.xi_power_series_def}. We let $\widetilde Y_N$ be a centered Gaussian field on $\R^N$, independent of $\widetilde H_N$, with covariance $\E\widetilde Y_N(u)\widetilde Y_N(u')=\theta_N(R_N(u,u'))$, and we henceforth realize $H_N$ as
\begin{equation}
    H_N(u)=\widetilde H_N(u)+\sqrt M\,\widetilde Y_N(u),
    \qquad u\in\R^N,
    \label{e.coupling_H_tildeH}
\end{equation}
which is legitimate since the right-hand side has the covariance~\eqref{e.def_H_N}. This provides a coupling between the $N$-spin system and the bulk system, which includes the radial derivatives of the two fields. The field $\widetilde Y_N$ will produce the ``reaction term'' $\mcl I_{M,t}(p)$ in~\eqref{e.reaction_value} below.

The Hamiltonian of the bulk system is defined as in~\eqref{e.full_radially_evaluated_H}--\eqref{e.full_radial_perturbation}, with $H_N$ replaced by $\widetilde H_N$. That is, for $r\in U_0$, we set
\begin{align}
    \widetilde H_N^{x,q}(r;\sigma,\alpha)
    :=&\sqrt{2t}\,\widetilde H_N(S_r\sigma)
    -Nt\widetilde\xi_N\Ll(R_N(S_r\sigma,S_r\sigma)\Rr)
    +\sqrt2W_N^q(S_r\sigma,\alpha)
    -\sum_{s\in\sS}q^s(1)|(S_r\sigma)^{(s)}|^2\notag\\
    &+\sum_{s\in\sS}h^s\sum_{i\in I_{N,s}}(S_r\sigma)_i
    +\eta_NH_N^x(S_r\sigma,\alpha),
    \qquad
    \widetilde H_N^{x,q}(\sigma,\alpha):=\widetilde H_N^{x,q}(\vecone;\sigma,\alpha),
    \label{e.radially_evaluated_H}
\end{align}
and
\begin{align} \label{e.radial_generic_perturbation}
    \begin{split}
        \widetilde{\mcl E}_{N,s}^{x,q}(\sigma,\alpha)
    :=\frac{2}{N_s}\dr_{r^s}\widetilde H_N^{x,q}(r;\sigma,\alpha)\Big|_{r=\vecone},\qquad
    \widetilde V_N^{x,y,q}(\sigma,\alpha)
    :=\eta_N\sum_{s\in\sS}y^sN_s\widetilde{\mcl E}_{N,s}^{x,q}(\sigma,\alpha),
    \\
    \widetilde H_N^{x,y,q}(\sigma,\alpha)
    :=\widetilde H_N^{x,q}(\sigma,\alpha)+\widetilde V_N^{x,y,q}(\sigma,\alpha).
    \end{split}
\end{align}
The partition function, free energy, and Gibbs bracket of the bulk system are
\begin{align}\label{e.ZcircMN_def}
    \begin{split}
        Z_N^{\mathord{\sim},x,y,q}
    &:=\iint e^{\widetilde H_N^{x,y,q}(\sigma,\alpha)}P_N(\d\sigma)\fR(\d\alpha),
    \qquad
    \widetilde F_N^{x,y}(t,q)
    :=-\frac1N\E\log Z_N^{\mathord{\sim},x,y,q},
    \\
    \la f\ra_N^{\mathord{\sim},x,y,q}
    &:= \frac{1}{(Z_N^{\mathord{\sim},x,y,q})^n}\int f\Ll((\sigma^\ell,\alpha^\ell)_{\ell\le n}\Rr)\prod_{\ell=1}^ne^{\widetilde H_N^{x,y,q}(\sigma^\ell,\alpha^\ell)}P_N(\d\sigma^\ell)\fR(\d\alpha^\ell).
    \end{split}
\end{align}
The mean and covariance of $\widetilde H_N^{x,y,q}$ are given by~\eqref{e.mean_perturbed_H} and~\eqref{e.covariance_perturbed_H} with $\xi$ replaced by $\widetilde\xi_N$. In particular, by Lemma~\ref{l.comput_cov_radial} and the uniform bounds on the kernels and their derivatives, we have, for every radial multi-index $\nu$ with $1\le|\nu|\le2$,
\begin{equation}
    \dr_r^\nu\widetilde H_N^{x,q}(r;\sigma,\alpha)\stackrel{N,N}{\lapprox}0
    \qquad\text{uniformly over $x,q$, $r\in U_0$, $\sigma\in\Sigma_N$, and $\alpha\in\mfk U$},
    \label{e.second_radial_bound}
\end{equation}
that is, these derivatives have means and covariances of order $N$. In view of the normalizations in~\eqref{e.radial_generic_perturbation}, this implies
\begin{align}
    \widetilde{\mcl E}_{N,s}^{x,q}(\sigma,\alpha)\stackrel{1,N^{-1}}{\lapprox} 0\qquad\text{and}\qquad \widetilde V_N^{x,y,q}(\sigma,\alpha)\stackrel{N\eta_N,N\eta_N^2}{\lapprox}0,
    \label{e.radial_energy_covariance_bound}
\end{align}
uniformly over $x,y,q$, $\sigma\in\Sigma_N$, and $\alpha\in\mfk U$. The same estimates hold for $\mcl E_{N,s}^{x,q}$ and $V_N^{x,y,q}$ in~\eqref{e.radial_energy_def}--\eqref{e.full_radial_perturbation}:
\begin{align}\label{e.full_radial_energy_covariance_bound}
    \mcl E_{N,s}^{x,q}(\sigma,\alpha)\stackrel{1,N^{-1}}{\lapprox} 0\qquad\text{and}\qquad V_N^{x,y,q}(\sigma,\alpha)\stackrel{N\eta_N,N\eta_N^2}{\lapprox}0.
\end{align}
Applying Lemma~\ref{l.gaussian_comparison_covariance_norm} to the fields $H_N^{x,y,q}$ and $H_N^{x,0,q}=H_N^{x,q}$, we deduce from~\eqref{e.full_radial_energy_covariance_bound} that the radial perturbation has a negligible effect on the free energy:
\begin{align}\label{e.full_radial_free_energy_cost}
    \sup_{x,y,q}\Ll|\bar F_N^{x,y}(t,q)-\bar F_N^{x,0}(t,q)\Rr|\le C\eta_N .
\end{align}
Similarly, since the variance of $\eta_NH^x_N$ is at most $9N\eta_N^2$, Lemma~\ref{l.gaussian_comparison_covariance_norm} gives
\begin{equation}
    \sup_{x,q}\Ll|\bar F_N^{x,0}(t,q)-\bar F_N(t,q)\Rr|\le C\eta_N^2 .
    \label{e.GG_free_energy_cost}
\end{equation}

\subsubsection{Decomposition of the \texorpdfstring{$(N+M)$}{(N+M)}-spin Hamiltonian}
\label{s.decomposition_NM}

We now decompose each of the Gaussian fields entering $H^{x,q}_{N+M}$ into a bulk part, a part linear in the cavity coordinates, and a remainder. We start with the interaction field. In view of Taylor's formula applied to the covariance $(N+M)\xi(\msf a_N+\msf b_N)$ in~\eqref{e.overlap_split}, we let $(\widetilde{\msf Z}_{N,j})_{j\in[M]}$ and $\widetilde D_N$ be centered Gaussian fields, indexed by $u\in\R^N$ and by $(u,\tau)\in\R^N\times\R^M$ respectively, independent of one another and of $\widetilde H_N$, with covariances
\begin{gather}
\label{e.D_MN_explicit}
    \E \widetilde{\msf Z}_{N,j}(u)\widetilde{\msf Z}_{N,j'}(u')
    =\one_{\{j=j'\}}\partial_s\xi\Ll(\msf a_N(u,u')\Rr),\qquad \text{for $j\in\msf M_s$,}
    \\
    \E \widetilde D_N(u,\tau)\widetilde D_N(u',\tau')
    =(N+M)\Ll(
    \xi\Ll(
        \msf a_N+\msf b_N
    \Rr)
    -\xi\Ll(\msf a_N\Rr)
    -\msf b_N\cdot\nabla\xi\Ll(\msf a_N\Rr)\Rr), \label{e.D_covariance_residual}
\end{gather}
where $\msf a_N=\msf a_N(u,u')$ and $\msf b_N=\msf b_N(\tau,\tau')$. The right-hand side of~\eqref{e.D_covariance_residual} is the sum of the monomials of $(N+M)\xi(\msf a_N+\msf b_N)$ that contain at least two factors from $\msf b_N$; since $\msf a_{N,s}$ and $\msf b_{N,s}$ are nonnegative definite kernels and the coefficients of $\xi$ are nonnegative, it is a nonnegative definite kernel by the Schur product theorem, and the field $\widetilde D_N$ exists. Comparing covariances, we see that the field
\begin{equation}
    \widetilde H_N(u)+\sum_{s\in\sS}\sum_{j\in\msf M_s}\tau_j\widetilde{\msf Z}_{N,j}(u)+\widetilde D_N(u,\tau)
    \label{e.interaction_decomposition}
\end{equation}
has the same law as $H_{N+M}((u,\tau))$. We henceforth realize $H_{N+M}$ through~\eqref{e.interaction_decomposition}; together with~\eqref{e.coupling_H_tildeH}, this couples the three fields $H_{N+M}$, $\widetilde H_N$, and $H_N$.

The cascade field and the external field are linear in $\rho$, and thus decompose exactly. By~\eqref{e.WNq_def} and~\eqref{e.N_compatible_block},
\begin{equation}
    W_{N+M}^q((u,\tau),\alpha)=W_N^q(u,\alpha)+\sum_{s\in\sS}\sum_{j\in\msf M_s}\tau_jw_j^{q^s}(\alpha),
    \label{e.cascade_decomposition}
\end{equation}
where $(w_j^{q^s})_{j\in\msf M_s}$ are the cascade fields attached to the cavity coordinates, which are independent of $W_N^q$; and $\sum_sh^s\sum_{i\in I_{N+M,s}}\rho_i=\sum_sh^s\sum_{i\in I_{N,s}}u_i+\sum_sh^s\sum_{j\in\msf M_s}\tau_j$.

Finally, we couple the perturbation fields $H^{\msf h}_{N+M}$ and $H^{\msf h}_N$ through the construction~\eqref{e.feature_field_construction}. Regarding $\R^{I_{N,s}}$ as the subspace of $\R^{I_{N+M,s}}$ of vectors vanishing on the cavity coordinates $\msf M_s$, we have $E_{N,j}\subset E_{N+M,j}$ isometrically for every $j$; we choose the orthonormal basis $\mcl E_{N+M,j}$ so that it contains $\mcl E_{N,j}$, and we build both fields from the same cascade fields $(w^{\msf h,j}_e)_{e\in\mcl E_{N+M,j}}$, for every $\msf h$ and $j$. For $\rho=(u,\tau)$ with $u\in\R^N$ and $\tau\in\R^M$, and for $u'\in\R^N$, only the bulk coordinates contribute to the inner product
\begin{equation*}
    \la\varphi_{N+M,s}(\rho),\varphi_{N,s}(u')\ra=\frac1{\sqrt{N(N+M)}}\sum_{i\in I_{N,s}}u_iu_i'=\sqrt{\tfrac{N}{N+M}}\,R_{N,s}(u,u'),
\end{equation*}
so that $\la\Phi^{\msf h,j}_{N+M}(\rho),\Phi^{\msf h,j}_N(u')\ra=(a_{\msf h_1}\cdot(\sqrt{N/(N+M)}\,R_N(u,u'))^{\odot\msf h_2})^{\msf h_4-j}$, while the fields $w^{\msf h,j}_e$ contribute the factors $\iota_{\msf h_3}^j(\alpha\wedge\alpha')^j$ as before. The binomial formula~\eqref{e.C_h_binomial} therefore gives
\begin{align}
    \E\Big[H_{N+M}^{\msf h}((u,\tau),\alpha)\,
    H_N^{\msf h}(u',\alpha')\Big]
    =\bigl(N(N+M)\bigr)^{1/2}
    \msf C_{\msf h}\Ll(\sqrt{\tfrac{N}{N+M}}R_N(u,u'),\alpha\wedge\alpha'\Rr),
    \label{e.overlap_perturbation_cross_covariance}
\end{align}
for every $u,u'\in\R^N$, $\tau\in\R^M$, and $\alpha,\alpha'\in\mfk U$. Note that $(N(N+M))^{1/2}=N+O(1)$ and that $\sqrt{N/(N+M)}=1+O(N^{-1})$. This cross-covariance therefore differs by $O(\Lambda_{\msf h})$ from the covariance $N\msf C_{\msf h}(R_N(u,u'),\alpha\wedge\alpha')$ of $H_N^{\msf h}$, and the same is true of the covariance $(N+M)\msf C_{\msf h}(R_{N+M}(\rho,\rho'),\alpha\wedge\alpha')$ of $H^{\msf h}_{N+M}$ when the cavity coordinates are bounded, by~\eqref{e.full_and_bulk_overlap_bound}. Since the perturbation carries the prefactor $\eta_N$, these differences are negligible.

\subsubsection{The cavity field and the expansion in the chart}

The linear parts in~\eqref{e.interaction_decomposition} and~\eqref{e.cascade_decomposition} combine, together with the external field, into the field that drives the cavity coordinates. Its bulk part $\widetilde{\msf Z}_{N,j}$ still depends on $N$ through $\msf a_N$, and we replace it by an $N$-independent version. Let $(\msf Z_{N,j})_{j\in[M]}$ be centered Gaussian fields on $\Sigma_N$, independent over $j$ and independent of all the other fields, such that
\begin{equation}
    \E \msf Z_{N,j}(\sigma)\msf Z_{N,j}(\sigma')=\dr_s\xi\Ll(R_N(\sigma,\sigma')\Rr),
    \qquad j\in\msf M_s,
    \label{e.Z_Nj_covariance}
\end{equation}
and set
\begin{align}
    \mathsf X_{N,j}^{q}(\sigma,\alpha):=h^s+\sqrt{2t}\,\msf Z_{N,j}(\sigma)+\sqrt2w_j^{q^s}(\alpha),
    \qquad s\in\sS,\quad j\in\msf M_s.
    \label{e.cavity_field_def}
\end{align}
The fields $(\mathsf X_{N,j}^{q})_{j\in[M]}$ are independent over $j$ and independent of the fields entering $\widetilde H_N^{x,y,q}$. We also introduce the coefficients
\begin{align}
    \vartheta^s:=q^s(1)+t\dr_s\xi(\lambda),
    \qquad s\in\sS,
    \label{e.prelimit_vartheta_def}
\end{align}
which collect the quadratic terms in $\tau$ produced by the self-overlap corrections in the Hamiltonian: the term $-Nq(1)\cdot\lambda_N$ in~\eqref{e.enriched_H} at size $N+M$ exceeds its value at size $N$ by $-\sum_sq^s(1)|\tau^{(s)}|^2$ once the bulk is evaluated at the radius $r_N(\tau)$, and similarly for $-Nt\xi(\lambda_N)$, up to negligible terms. Note that $\vartheta$ does not depend on $N$, thanks to~\eqref{e.lambda_N_equals_lambda}.

\begin{lemma}[Expansion of the Hamiltonian in the chart]
\label{l.chart_hamiltonian_expansion}
Fix $L\ge1$. Under the couplings described above, we have
\begin{align}
    &H_{N+M}^{x,q}\Ll((S_{r_N(\tau)}\sigma,\tau),\alpha\Rr)+\widetilde V_N^{x,y,q}(\sigma,\alpha)\notag\\
    &\qquad\lapprox\widetilde H_N^{x,q}(r_N(\tau);\sigma,\alpha)+\widetilde V_N^{x,y,q}(\sigma,\alpha)+\sum_{s\in\sS}\sum_{j\in\msf M_s}\tau_j\mathsf X^q_{N,j}(\sigma,\alpha)-\sum_{s\in\sS}\vartheta^s|\tau^{(s)}|^2,
    \label{e.chart_H_decomposition}
\end{align}
uniformly over $x,y,q$ and $(\sigma,\tau,\alpha)\in\Sigma_N\times\{|\tau|\le L+1\}\times\mfk U$.
\end{lemma}

\begin{proof}
Write $\widetilde\sigma:=S_{r_N(\tau)}\sigma$ and $\rho:=(\widetilde\sigma,\tau)$, and recall from~\eqref{e.r_N_close_to_one} that $|r_N(\tau)-\vecone|\le CN^{-1}$ on the range of $\tau$ considered. We go through the terms of $H_{N+M}^{x,q}(\rho,\alpha)$ in~\eqref{e.perturbed_enriched_H} and~\eqref{e.enriched_H}, at size $N+M$. The term $\widetilde V_N^{x,y,q}$ is the same on both sides of~\eqref{e.chart_H_decomposition}; since it is not independent of the other fields, we must also keep track of its cross-covariances with the terms that we modify.

\smallskip
\noindent\emph{Interaction field.}\quad By~\eqref{e.interaction_decomposition},
\begin{equation*}
    \sqrt{2t}H_{N+M}(\rho)=\sqrt{2t}\widetilde H_N(\widetilde\sigma)+\sqrt{2t}\sum_{s\in\sS}\sum_{j\in\msf M_s}\tau_j\widetilde{\msf Z}_{N,j}(\widetilde\sigma)+\sqrt{2t}\widetilde D_N(\widetilde\sigma,\tau).
\end{equation*}
The first term is the interaction part of $\widetilde H_N^{x,q}(r_N(\tau);\sigma,\alpha)$ in~\eqref{e.radially_evaluated_H}. The last two terms are independent of all the other fields. By~\eqref{e.D_covariance_residual} and~\eqref{e.overlap_split}, the variance of $\widetilde D_N(\widetilde\sigma,\tau)$ is bounded by $(N+M)\sup|D^2\xi|\,|\msf b_N(\tau,\tau)|^2\le C_LN^{-1}$, so that $\widetilde D_N(\widetilde\sigma,\tau)\stackrel{0,N^{-1}}{\lapprox}0$. By~\eqref{e.D_MN_explicit}, \eqref{e.overlap_dilation}, and the Lipschitz continuity of $\nabla\xi$, the covariance of $\widetilde{\msf Z}_{N,j}(\widetilde\sigma)$ and $\widetilde{\msf Z}_{N,j}(\widetilde\sigma')$ is
\begin{equation*}
    \partial_s\xi\Ll(\Big(\tfrac{N}{N+M}\sqrt{r^u_N(\tau)r^u_N(\tau')}\,R_{N,u}(\sigma,\sigma')\Big)_{u\in\sS}\Rr)=\partial_s\xi\Ll(R_N(\sigma,\sigma')\Rr)+O(N^{-1}),
\end{equation*}
uniformly over the range considered, and the same holds after multiplication by $\tau_j\tau'_j$. Comparing with~\eqref{e.Z_Nj_covariance}, we obtain
\begin{equation*}
    \sqrt{2t}\sum_{s\in\sS}\sum_{j\in\msf M_s}\tau_j\widetilde{\msf Z}_{N,j}(\widetilde\sigma)+\sqrt{2t}\widetilde D_N(\widetilde\sigma,\tau)\stackrel{0,N^{-1}}{\lapprox}\sqrt{2t}\sum_{s\in\sS}\sum_{j\in\msf M_s}\tau_j\msf Z_{N,j}(\sigma),
\end{equation*}
and both sides are independent of all the other fields.

\smallskip
\noindent\emph{Self-overlap corrections.}\quad Since $\rho\in\Sigma_{N+M}$, the interaction correction in $H^{x,q}_{N+M}$ is $-(N+M)t\xi(\lambda)$. On the other hand, by~\eqref{e.overlap_dilation} and~\eqref{e.radial_factor}, $R_{N,s}(\widetilde\sigma,\widetilde\sigma)=r_N^s(\tau)\lambda_s$, so that $\msf a_N(\widetilde\sigma,\widetilde\sigma)=\lambda-(|\tau^{(s)}|^2/(N+M))_{s\in\sS}$. The interaction correction in $\widetilde H_N^{x,q}(r_N(\tau);\sigma,\alpha)$ is therefore
\begin{align*}
    -Nt\widetilde\xi_N\Ll(R_N(\widetilde\sigma,\widetilde\sigma)\Rr)
    &=-(N+M)t\xi\Ll(\lambda-\big(\tfrac{|\tau^{(s)}|^2}{N+M}\big)_{s\in\sS}\Rr)\\
    &=-(N+M)t\xi(\lambda)+t\sum_{s\in\sS}\partial_s\xi(\lambda)|\tau^{(s)}|^2+O_L(N^{-1}),
\end{align*}
by Taylor's formula. Similarly, the cascade correction at size $N+M$ is $-\sum_sq^s(1)(N_s+M_s)$, while that in $\widetilde H_N^{x,q}(r_N(\tau);\sigma,\alpha)$ is $-\sum_sq^s(1)|\widetilde\sigma^{(s)}|^2=-\sum_sq^s(1)(N_s+M_s-|\tau^{(s)}|^2)$. Taking the difference of the two sides and recalling~\eqref{e.prelimit_vartheta_def}, the deterministic corrections account for the term $-\sum_s\vartheta^s|\tau^{(s)}|^2$ in~\eqref{e.chart_H_decomposition}, up to an error $O_L(N^{-1})$.

\smallskip
\noindent\emph{Cascade and external fields.}\quad By~\eqref{e.cascade_decomposition}, $\sqrt2W_{N+M}^q(\rho,\alpha)=\sqrt2W_N^q(\widetilde\sigma,\alpha)+\sqrt2\sum_s\sum_{j\in\msf M_s}\tau_jw_j^{q^s}(\alpha)$, where the first term is the cascade part of $\widetilde H_N^{x,q}(r_N(\tau);\sigma,\alpha)$ and the second is the cascade part of $\sum_j\tau_j\msf X^q_{N,j}$. The external field decomposes in the same way, producing the terms $h^s\tau_j$ in $\msf X^q_{N,j}$. These decompositions are exact.

\smallskip
\noindent\emph{Perturbation field.}\quad It remains to compare $\eta_{N+M}H^x_{N+M}(\rho,\alpha)$ with $\eta_NH^x_N(\widetilde\sigma,\alpha)$, which is the perturbation part of $\widetilde H_N^{x,q}(r_N(\tau);\sigma,\alpha)$. Both are independent of all the fields considered so far, but not of $\widetilde V_N^{x,y,q}$, which contains the term $2\eta_N^2\sum_sy^s\partial_{r^s}H_N^x(S_r\sigma,\alpha)|_{r=\vecone}$. By~\eqref{e.overlap_perturbation_covariance}, \eqref{e.overlap_perturbation_cross_covariance}, and Lemma~\ref{l.comput_cov_radial}, all the relevant covariances are obtained by evaluating and differentiating the kernels $\msf C_{\msf h}$ at overlaps that differ from $R_N(\sigma,\sigma')$ by $O_L(N^{-1})$, by~\eqref{e.full_and_bulk_overlap_bound} and~\eqref{e.overlap_dilation}, with prefactors $(N+M)\eta_{N+M}^2$, $(N(N+M))^{1/2}\eta_{N+M}\eta_N$, or $N\eta_N^2$, which differ from one another by $O(N|\eta_{N+M}-\eta_N|+\eta_N^2)$. Using~\eqref{e.ch_perturbation_lipschitz_summability} to sum over $\msf h$, we find that replacing $\eta_{N+M}H^x_{N+M}(\rho,\alpha)$ by $\eta_NH^x_N(\widetilde\sigma,\alpha)$ changes the variance by $O(N|\eta_{N+M}-\eta_N|+\eta_N^2)$, and the cross-covariance with $\widetilde V_N^{x,y,q}$ by $O(\eta_N(N|\eta_{N+M}-\eta_N|+\eta_N^2))$. Both quantities vanish as $N\to\infty$, by~\eqref{e.eta_N_def}. The means are zero.

Collecting the four steps, the means of the two sides of~\eqref{e.chart_H_decomposition} differ by $O_L(N^{-1})$ and their covariances by $o(1)$, uniformly over the stated range. This is~\eqref{e.chart_H_decomposition}.
\end{proof}

\subsection{Radial linearization and replacement of the radial perturbation}
\label{s.radial_replacement}

The right-hand side of~\eqref{e.chart_H_decomposition} still evaluates the bulk Hamiltonian at the random radius $r_N(\tau)$. Since $r_N(\tau)-\vecone$ is of order $N^{-1}$, while the radial derivatives of the Hamiltonian have means and covariances of order $N$ by~\eqref{e.second_radial_bound}, the first-order term of the Taylor expansion in the radius is of order one and the second-order remainder is negligible; the first-order term is expressed through the normalized radial derivatives $\widetilde{\mcl E}^{x,q}_{N,s}$. We then show that the radial perturbation of the $(N+M)$-spin system can be replaced by that of the bulk system.

\begin{lemma}[First-order radial expansion]
\label{l.radial_first_order_expansion}
Fix $L\ge1$. We have
\begin{align}
    \widetilde H_N^{x,q}(r_N(\tau);\sigma,\alpha)+\widetilde V_N^{x,y,q}(\sigma,\alpha)\stackrel{N^{-1},N^{-1}}{\lapprox}\widetilde H_N^{x,y,q}(\sigma,\alpha)+\frac12\sum_{s\in\sS}\Ll(M_s-|\tau^{(s)}|^2\Rr)\widetilde{\mcl E}_{N,s}^{x,q}(\sigma,\alpha)
    \label{e.radial_first_order}
\end{align}
uniformly over $x,y,q$ and $(\sigma,\tau,\alpha)\in\Sigma_N\times\{|\tau|\le L+1\}\times\mfk U$. The relation remains valid, with the same rates, if a Gaussian field independent of $\widetilde H_N$, $W^q_N$, and $H^x_N$ is added to both sides.
\end{lemma}

\begin{proof}
Set $\delta_s(\tau):=r_N^s(\tau)-1=(M_s-|\tau^{(s)}|^2)/N_s$, so that $|\delta_s(\tau)|\le C_LN^{-1}$ by~\eqref{e.r_N_close_to_one}. Taylor's formula in the radial variable gives
\begin{align*}
    \widetilde H_N^{x,q}(r_N(\tau);\sigma,\alpha)={}&\widetilde H_N^{x,q}(\sigma,\alpha)+\sum_{s\in\sS}\delta_s(\tau)\,\partial_{r^s}\widetilde H^{x,q}_N(r;\sigma,\alpha)\big|_{r=\vecone}+\boldsymbol{\eps}_N(\sigma,\tau,\alpha),\\
    \boldsymbol{\eps}_N(\sigma,\tau,\alpha):={}&\sum_{s,u\in\sS}\delta_s(\tau)\delta_u(\tau)\int_0^1(1-\msf v)\,\dr_{r^s}\dr_{r^u}\widetilde H_N^{x,q}(\vecone+\msf v\delta(\tau);\sigma,\alpha)\,\d\msf v.
\end{align*}
By the definition~\eqref{e.radial_generic_perturbation} of $\widetilde{\mcl E}^{x,q}_{N,s}$, the first-order term is $\frac12\sum_s(M_s-|\tau^{(s)}|^2)\widetilde{\mcl E}^{x,q}_{N,s}(\sigma,\alpha)$. By~\eqref{e.second_radial_bound} and the bound on $\delta(\tau)$, the remainder satisfies $\boldsymbol{\eps}_N\stackrel{N^{-1},N^{-3}}{\lapprox}0$. The field on the right-hand side of~\eqref{e.radial_first_order} has variance at most $CN$, by~\eqref{e.radial_energy_covariance_bound} and the covariance formula of $\widetilde H^{x,q}_N$; the Cauchy--Schwarz inequality therefore bounds its covariance with $\boldsymbol{\eps}_N$ by $C_LN^{-1}$. This proves~\eqref{e.radial_first_order}. A field independent of $\widetilde H_N$, $W^q_N$, and $H^x_N$ is independent of $\boldsymbol{\eps}_N$, which gives the last assertion.
\end{proof}

We next compare the radial perturbations $V_{N+M}^{x,y,q}$ and $V_N^{x,y,q}$ in~\eqref{e.full_radial_perturbation} with the bulk radial perturbation $\widetilde V_N^{x,y,q}$ in~\eqref{e.radial_generic_perturbation}. Recall that these three fields are, respectively, $2\eta_{N+M}$, $2\eta_N$, and $2\eta_N$ times a $y$-weighted sum of radial derivatives of $H^{x,q}_{N+M}$, $H^{x,q}_N$, and $\widetilde H^{x,q}_N$. The radial derivative in $V^{x,y,q}_{N+M}$ acts on all the coordinates of species $s$, including the cavity ones. The fields $H_{N+M}$, $H_N$, and $\widetilde H_N$, together with their radial derivatives, are coupled through~\eqref{e.interaction_decomposition} and~\eqref{e.coupling_H_tildeH}; the cascade and perturbation fields are coupled through~\eqref{e.cascade_decomposition} and~\eqref{e.overlap_perturbation_cross_covariance}.

\begin{lemma}[Replacement of the radial perturbation]
\label{l.symmetric_block_radial_covariance}
Fix $L\ge1$, and write $\rho:=(S_{r_N(\tau)}\sigma,\tau)$. Under the couplings above, we have
\begin{align}
    H_{N+M}^{x,y,q}(\rho,\alpha)=H_{N+M}^{x,q}(\rho,\alpha)+V_{N+M}^{x,y,q}(\rho,\alpha)
    &\lapprox H_{N+M}^{x,q}(\rho,\alpha)+\widetilde V_N^{x,y,q}(\sigma,\alpha)
    \label{e.replace_V_NM}
\end{align}
uniformly over $x,y,q$ and $(\sigma,\tau,\alpha)\in\Sigma_N\times\{|\tau|\le L+1\}\times\mfk U$, and
\begin{equation}
    H_N^{x,y,q}(\sigma,\alpha)=H_N^{x,q}(\sigma,\alpha)+V_N^{x,y,q}(\sigma,\alpha)\lapprox H_N^{x,q}(\sigma,\alpha)+\widetilde V_N^{x,y,q}(\sigma,\alpha)
    \label{e.replace_V_N}
\end{equation}
uniformly over $x,y,q$ and $(\sigma,\alpha)\in\Sigma_N\times\mfk U$.
\end{lemma}

\begin{proof}
We first replace $\eta_{N+M}$ by $\eta_N$ in $V^{x,y,q}_{N+M}$. By~\eqref{e.full_radial_energy_covariance_bound} at size $N+M$, the field $V^{x,y,q}_{N+M}/\eta_{N+M}$ has mean and covariance of order $N$, and its covariance with $H^{x,q}_{N+M}$ is of order $N$ as well. Multiplying it by $\eta_N$ instead of $\eta_{N+M}$ therefore changes the mean, the variance, and the covariance with $H^{x,q}_{N+M}$ of the left-hand side of~\eqref{e.replace_V_NM} by at most $CN|\eta_{N+M}-\eta_N|$, which vanishes by~\eqref{e.eta_N_def}.

Both sides of~\eqref{e.replace_V_NM} and~\eqref{e.replace_V_N} are now of the form $\msf H+2\eta_N\sum_sy^s\msf D_s$, where $\msf H$ is a Hamiltonian and $\msf D_s$ is a radial derivative. It therefore suffices to show that, in each of the two comparisons, the means of the two radial derivatives, their covariances with the Hamiltonian $\msf H$ and with the corresponding Hamiltonian evaluated at another point, and their mutual covariances, differ by $O_L(1)$; multiplication by $2\eta_N$ (once for the cross terms, twice for the covariances of the radial derivatives) then gives errors of order $\eta_N$ and $\eta_N^2$.

Consider~\eqref{e.replace_V_NM}. The means of the radial derivatives are obtained by differentiating the deterministic terms in~\eqref{e.enriched_H}, at sizes $N+M$ and $N$, as in~\eqref{e.mean_perturbed_H}. For species $s$, the derivative $2\partial_{r^s}$ of the interaction correction is $-2(N+M)t\lambda_s\partial_s\xi(\lambda)$ for the $(N+M)$-spin system and $-2Nt\lambda_s\partial_s\widetilde\xi_N(\lambda)=-2Nt\lambda_s\partial_s\xi(\frac{N}{N+M}\lambda)$ for the bulk system; these differ by $O(1)$. The derivatives of the cascade corrections are $-2q^s(1)(N_s+M_s)$ and $-2q^s(1)N_s$; and the derivatives of the external-field terms are $h^s\sum_{i\in I_{N+M,s}}\rho_i$ and $h^s\sum_{i\in I_{N,s}}\sigma_i$, whose difference is $h^s((\sqrt{r^s_N(\tau)}-1)\sum_{i\in I_{N,s}}\sigma_i+\sum_{j\in\msf M_s}\tau_j)=O_L(1)$, as in the proof of Lemma~\ref{l.fixed_block_increment_bound}. The covariances are obtained by differentiating covariance kernels. For instance, by~\eqref{e.def_H_N} and~\eqref{e.overlap_dilation}, and with $\rho'=(S_{r_N(\tau')}\sigma',\tau')$,
\begin{equation*}
    \operatorname{Cov}\Ll(2\partial_{r^s}H_{N+M}(S_r\rho)\big|_{r=\vecone},H_{N+M}(\rho')\Rr)=(N+M)R_{N+M,s}(\rho,\rho')\,\partial_s\xi\Ll(R_{N+M}(\rho,\rho')\Rr),
\end{equation*}
while, by~\eqref{e.interaction_decomposition}, \eqref{e.tH=}, and the independence of $\widetilde{\msf Z}_{N,j}$ and $\widetilde D_N$ from $\widetilde H_N$,
\begin{equation*}
    \operatorname{Cov}\Ll(2\partial_{r^s}\widetilde H_{N}(S_r\sigma)\big|_{r=\vecone},H_{N+M}(\rho')\Rr)=N\sqrt{r^s_N(\tau')}R_{N,s}(\sigma,\sigma')\,\partial_s\xi\Ll(\tfrac{N}{N+M}R_N(\sigma,S_{r_N(\tau')}\sigma')\Rr).
\end{equation*}
By~\eqref{e.full_and_bulk_overlap_bound} and~\eqref{e.r_N_close_to_one}, the overlaps and radii appearing in these two expressions differ by $O_L(N^{-1})$, and the prefactors differ by $M$; since $\xi$ is smooth, the two expressions differ by $O_L(1)$. The same argument applies to the covariance of the radial derivatives with $H_{N+M}(\rho)$ itself, to the cascade field through~\eqref{e.cascade_decomposition} and~\eqref{e.WNq_covariance}, to the perturbation field through~\eqref{e.overlap_perturbation_cross_covariance} and~\eqref{e.ch_perturbation_lipschitz_summability}, in which case the kernels $\msf C_{\msf h}$ are differentiated at most twice and the bound is $O_L(\Lambda_{\msf h})$ before summation, and to the covariances between two radial derivatives, which involve one more derivative of the kernels. This proves~\eqref{e.replace_V_NM}.

For~\eqref{e.replace_V_N}, the cascade, external-field, and perturbation parts of $V_N^{x,y,q}$ and $\widetilde V_N^{x,y,q}$ coincide, and the only difference lies in the interaction part, where $H_N$ is replaced by $\widetilde H_N$ and $\xi$ by $\widetilde\xi_N$. Under the coupling~\eqref{e.coupling_H_tildeH}, we have $\operatorname{Cov}(\partial_{r^s}\widetilde H_N(S_r\sigma),H_N(\sigma'))=\operatorname{Cov}(\partial_{r^s}\widetilde H_N(S_r\sigma),\widetilde H_N(\sigma'))$, and the kernels $\xi$ and $\widetilde\xi_N$, together with their derivatives of order at most three, differ by $O(N^{-1})$ uniformly on the overlap domain $[-1,1]^\sS$. The differences of means and covariances are therefore $O(1)$ again, and the conclusion follows as above.
\end{proof}

We now introduce the partition functions in which the radial perturbation has been replaced by the bulk one. Set
\begin{align}
    Z^{x,y,q,\mathrm{ch}}_{N+M,L,\widetilde V}
    :={}&\iiint_{\Sigma_N\times \mcl B_N\times \mfk U} \chi_L(\tau)
    \exp\Ll(H_{N+M}^{x,q}\Ll((S_{r_N(\tau)}\sigma,\tau),\alpha\Rr)+\widetilde V_N^{x,y,q}(\sigma,\alpha)\Rr)\notag\\
    &\qquad\qquad\times p_N^{\mathrm{ch}}(\tau)\,\d\tau\,P_N(\d\sigma)\,\fR(\d\alpha),
\label{e.radially_inserted_full_partition_L}
\end{align}
and let $\la \cdot\ra^{x,y,q,\mathrm{ch}}_{N+M,L,\widetilde V}$ be the associated Gibbs bracket, defined as in~\eqref{e.truncated_true_partition}. Similarly, set
\begin{equation}
    Z^{x,y,q}_{N,\widetilde V}:=\iint \exp\Ll(H_N^{x,q}(\sigma,\alpha)+\widetilde V_N^{x,y,q}(\sigma,\alpha)\Rr)P_N(\d\sigma)\,\fR(\d\alpha).
    \label{e.radially_inserted_N_partition}
\end{equation}
Recall the cavity increment $A_N(x,y,q)$ from~\eqref{e.ASS_increment_def}, and define its counterpart with cutoff and bulk radial perturbation,
\begin{equation}
\label{e.radial_cutoff_increment_def}
    A_{N,L}^{\mathrm{rad}}(x,y,q):=\E\log Z^{x,y,q,\mathrm{ch}}_{N+M,L,\widetilde V}-\E\log Z^{x,y,q}_{N,\widetilde V}.
\end{equation}

\begin{proposition}[Replacement of the radial perturbation at fixed cutoff]
\label{p.radial_insertion_harmless}
Let $L\ge1$. Uniformly over $x,y,q$, we have
\begin{gather}
    \E\log Z^{x,y,q}_{N+M,L}-\E\log Z_N^{x,y,q}\lapprox A^{\mathrm{rad}}_{N,L}(x,y,q),\label{e.radial_insertion_cutoff_increment_negligible}
    \\
    \E\la f\ra_{N+M,L}^{x,y,q}
    \lapprox\E\la f\ra^{x,y,q,\mathrm{ch}}_{N+M,L,\widetilde V},
    \label{e.radial_insertion_observable_negligible}
\end{gather}
for every bounded measurable function $f$ of finitely many replicas of $(\tau,\alpha)$. Consequently, the increment with cutoff is asymptotically a lower bound for the true increment:
\begin{equation}
    \lim_{N\to\infty}\sup_{x,y,q}
    \Ll(A_{N,L}^{\mathrm{rad}}(x,y,q)-A_N(x,y,q)\Rr)_+=0.
    \label{e.one_sided_increment}
\end{equation}
\end{proposition}

\begin{proof}
We apply Lemma~\ref{l.gaussian_comparison_covariance_norm} on the space $\Sigma_N\times\{|\tau|\le L+1\}\times\mfk U$, with the finite measure $\chi_L(\tau)p_N^{\mathrm{ch}}(\tau)\d\tau\,P_N(\d\sigma)\fR(\d\alpha)$ normalized to a probability measure (the normalization cancels in the differences of logarithms and does not affect the brackets), to the two fields in~\eqref{e.replace_V_NM}. The partition functions are $Z^{x,y,q}_{N+M,L}$ and $Z^{x,y,q,\mathrm{ch}}_{N+M,L,\widetilde V}$, by~\eqref{e.truncated_true_partition} and~\eqref{e.radially_inserted_full_partition_L}, so that $\E\log Z^{x,y,q}_{N+M,L}\lapprox\E\log Z^{x,y,q,\mathrm{ch}}_{N+M,L,\widetilde V}$ and~\eqref{e.radial_insertion_observable_negligible} holds. Applying the same lemma on $\Sigma_N\times\mfk U$ to the two fields in~\eqref{e.replace_V_N} gives $\E\log Z^{x,y,q}_N\lapprox\E\log Z^{x,y,q}_{N,\widetilde V}$. Combining the two estimates on partition functions proves~\eqref{e.radial_insertion_cutoff_increment_negligible}. Finally, by~\eqref{e.full_partition_dominates_cutoff}, we have $A_N(x,y,q)\ge\E\log Z^{x,y,q}_{N+M,L}-\E\log Z^{x,y,q}_N$, so that $(A^{\mathrm{rad}}_{N,L}-A_N)_+$ is bounded by the absolute value of the difference between the two sides of~\eqref{e.radial_insertion_cutoff_increment_negligible}, which tends to zero uniformly over $x,y,q$. This proves~\eqref{e.one_sided_increment}. 
\end{proof}
We close this subsection with a comparison between the free energy of the bulk system and that of the $(N+M)$-spin system, which is used in Section~\ref{s.crit_pt_bdd} to match their derivatives in $q$.

\begin{lemma}[Bulk and true free energies]
\label{l.bulk_true_comparison}
There exists a constant $C<\infty$ such that, for every sufficiently large $N$,
\begin{equation}
    \varepsilon_N(Q):=\sup_{x,y,q}\Ll|\widetilde F_N^{x,y}(t,q)-\bar F_{N+M}^{x,y}(t,q)\Rr|\le\frac{C}{N}.
    \label{e.bulk_true_comparison}
\end{equation}
\end{lemma}

\begin{proof}
We first compare the bulk system with the $N$-spin system. Under the coupling~\eqref{e.coupling_H_tildeH}, the fields $H_N^{x,q}+\widetilde V_N^{x,y,q}$ and $\widetilde H_N^{x,q}+\widetilde V^{x,y,q}_N$ differ by $\sqrt{2tM}\,\widetilde Y_N(\sigma)$, whose covariance $2tM\theta_N(R_N(\sigma,\sigma'))$ is bounded by~\eqref{e.cavity_reaction_kernel_limit}, and by the deterministic term $-Nt(\xi(\lambda)-\widetilde\xi_N(\lambda))=-tM\theta_N(\lambda)$, which is bounded as well. Lemma~\ref{l.gaussian_comparison_covariance_norm} therefore gives
\begin{equation*}
    \sup_{x,y,q}\Ll|\E\log Z_N^{\mathord{\sim},x,y,q}-\E\log Z_{N,\widetilde V}^{x,y,q}\Rr|\le C.
\end{equation*}
By~\eqref{e.replace_V_N} and Lemma~\ref{l.gaussian_comparison_covariance_norm}, we also have $\sup_{x,y,q}|\E\log Z_{N,\widetilde V}^{x,y,q}-\E\log Z_N^{x,y,q}|\le C$ for all sufficiently large $N$. Dividing by $N$ and recalling~\eqref{e.ZNx_def} and~\eqref{e.ZcircMN_def}, we obtain $\sup_{x,y,q}|\widetilde F_N^{x,y}(t,q)-\bar F_N^{x,y}(t,q)|\le C/N$. Next, by~\eqref{e.ASS_increment_def},
\begin{equation*}
    \widetilde F_N^{x,y}(t,q)-\frac{N+M}{N}\bar F_{N+M}^{x,y}(t,q)=\widetilde F_N^{x,y}(t,q)-\bar F_N^{x,y}(t,q)+\frac1N A_N(x,y,q),
\end{equation*}
so that Lemma~\ref{l.fixed_block_increment_bound} gives $\sup_{x,y,q}|\widetilde F_N^{x,y}(t,q)-\frac{N+M}{N}\bar F_{N+M}^{x,y}(t,q)|\le C/N$. It remains to observe that $\bar F^{x,y}_{N+M}(t,q)$ is bounded uniformly over $x,y,q$ and $N$, so that replacing $\frac{N+M}N\bar F^{x,y}_{N+M}$ by $\bar F^{x,y}_{N+M}$ costs at most $CM/N$. Indeed, Jensen's inequality and the self-overlap correction in~\eqref{e.enriched_H} give
\begin{equation*}
    -t\xi(\lambda)\le\frac1N\E\log Z_N(t,0)\le\frac1N\log\int\exp\Ll(\sum_{s\in\sS}h^s\sum_{i\in I_{N,s}}\sigma_i\Rr)P_N(\d\sigma)\le\max_{s\in\sS}|h^s|,
\end{equation*}
the Lipschitz estimate~\eqref{e.F_N_Lipschitz} extends this bound to $\bar F_N(t,q)$ for $q$ in the bounded set $Q$, and~\eqref{e.full_radial_free_energy_cost}--\eqref{e.GG_free_energy_cost} extend it to $\bar F_N^{x,y}(t,q)$.
\end{proof}

\subsection{Radial self-averaging and Gaussian reduction}
\label{s.radial_and_gaussian}

Combining Lemmas~\ref{l.chart_hamiltonian_expansion} and~\ref{l.radial_first_order_expansion}, the Hamiltonian of the $(N+M)$-spin system in the chart is, up to negligible terms,
\begin{equation}
    \widetilde H_N^{x,y,q}(\sigma,\alpha)+\sum_{s\in\sS}\sum_{j\in\msf M_s}\tau_j\msf X^q_{N,j}(\sigma,\alpha)-\sum_{s\in\sS}\vartheta^s|\tau^{(s)}|^2+\frac12\sum_{s\in\sS}\Ll(M_s-|\tau^{(s)}|^2\Rr)\widetilde{\mcl E}_{N,s}^{x,q}(\sigma,\alpha).
    \label{e.expanded_hamiltonian_informal}
\end{equation}
The first three terms are of the same form as in the cavity computation of~\cite{chen2025free}: a bulk Hamiltonian, a field linear in $\tau$ and independent of the bulk, and a deterministic quadratic term. The last term is specific to the spherical setting. In this subsection, we show that it can be replaced by its Gibbs average, at a cost controlled by the self-averaging of $\widetilde{\mcl E}^{x,q}_{N,s}$, and that this self-averaging holds for most values of the parameter $y$.

\subsubsection{The Gaussian model and the interpolation}

Recall the bulk bracket $\la\cdot\ra_N^{\mathord{\sim},x,y,q}$ from~\eqref{e.ZcircMN_def}, and set
\begin{equation}\label{e.e^N_s=}
    e_N^s:=\E\la\widetilde{\mcl E}_{N,s}^{x,q}\ra_N^{\mathord{\sim},x,y,q},\qquad s\in\sS.
\end{equation}
This quantity depends on $x,y,q$, which we omit from the notation; it is bounded uniformly in $N,x,y,q$ by Lemma~\ref{l.fixed_tilt_uniform_bounds} below. Replacing $\widetilde{\mcl E}^{x,q}_{N,s}$ by $e^s_N$ in~\eqref{e.expanded_hamiltonian_informal} turns the last term into $\frac12\sum_sM_se^s_N-\frac12\sum_se^s_N|\tau^{(s)}|^2$, which is a constant plus a quadratic term. Since the cavity coordinates will be integrated against the Gaussian measure $\bgamma_M$, whose density already contains the factor $e^{-|\tau|^2/2}$, it is convenient to set
\begin{align}
    b_N^s:=1+e^s_N+2\vartheta^s,
    \qquad s\in\sS,
    \label{e.b_prelimit_def}
\end{align}
so that, once $\widetilde{\mcl E}^{x,q}_{N,s}$ is replaced by $e^s_N$, the total quadratic term in~\eqref{e.expanded_hamiltonian_informal} is $-\frac12\sum_s(b_N^s-1)|\tau^{(s)}|^2$. For a general vector $b=(b^s)_{s\in\sS}\in\R^{\sS}$, we define the Gaussian model
\begin{align}
    Z^{x,y,q,\bgamma}_{N,L,\mathsf X,b}:={}&\iiint \chi_L(\tau)
    \exp\Ll(\widetilde H_N^{x,y,q}(\sigma,\alpha)
    +\sum_{s\in\sS}\sum_{j\in\msf M_s}\tau_j\mathsf X_{N,j}^{q}(\sigma,\alpha)-\frac12\sum_{s\in\sS}(b^s-1)|\tau^{(s)}|^2\Rr)\notag\\
    &\qquad\bgamma_M(\d\tau)\,P_N(\d\sigma)\,\fR(\d\alpha),
    \label{e.gaussian_bracket_L_b}
\end{align}
and we let $\la \cdot\ra^{x,y,q,\bgamma}_{N,L,\mathsf X,b}$ be the associated Gibbs bracket on $\Sigma_N\times\R^M\times\mfk U$. The superscript $\bgamma$ indicates that the reference measure for the cavity coordinates is the Gaussian measure rather than the chart density $p^{\mathrm{ch}}_N$.

The passage from~\eqref{e.expanded_hamiltonian_informal} to the exponent in~\eqref{e.gaussian_bracket_L_b} with $b=b_N$ is achieved by interpolation. For $\mt\in[0,1]$, set
\begin{align}
    \mcl H_{N,\mt}(\sigma,\tau,\alpha):={}&\widetilde H_N^{x,y,q}(\sigma,\alpha)+\sum_{s\in\sS}\sum_{j\in\msf M_s}\tau_j\mathsf X_{N,j}^{q}(\sigma,\alpha)-\frac12\sum_{s\in\sS}\Ll(e_N^s+2\vartheta^s\Rr)|\tau^{(s)}|^2\notag\\
    &+\frac{\mt}{2}\sum_{s\in\sS}\Ll(M_s-|\tau^{(s)}|^2\Rr)\Ll(\widetilde{\mcl E}_{N,s}^{x,q}(\sigma,\alpha)-e_N^s\Rr),
    \label{e.interpolating_hamiltonian}
\end{align}
and let $\la\cdot\ra_{N,\mt,L}$ be the Gibbs bracket associated with the measure
\begin{align}
    \chi_L(\tau)\exp\Ll(\mcl H_{N,\mt}(\sigma,\tau,\alpha)\Rr)\bgamma_M(\d\tau)\,P_N(\d\sigma)\,\fR(\d\alpha)
    \label{e.radial_interpolation_measure}
\end{align}
on $\Sigma_N\times\R^M\times\mfk U$. A direct computation shows that the two endpoints of the interpolation are
\begin{align}
    \mcl H_{N,1}(\sigma,\tau,\alpha)+\frac12\sum_{s\in\sS}M_se_N^s
    &=\widetilde H_N^{x,y,q}(\sigma,\alpha)+\sum_{s\in\sS}\sum_{j\in\msf M_s}\tau_j\msf X^q_{N,j}(\sigma,\alpha)-\sum_{s\in\sS}\vartheta^s|\tau^{(s)}|^2\notag\\
    &\qquad+\frac12\sum_{s\in\sS}\Ll(M_s-|\tau^{(s)}|^2\Rr)\widetilde{\mcl E}_{N,s}^{x,q}(\sigma,\alpha),\notag\\
    \mcl H_{N,0}(\sigma,\tau,\alpha)&=\widetilde H_N^{x,y,q}(\sigma,\alpha)+\sum_{s\in\sS}\sum_{j\in\msf M_s}\tau_j\mathsf X_{N,j}^{q}(\sigma,\alpha)-\frac12\sum_{s\in\sS}\Ll(b_N^s-1\Rr)|\tau^{(s)}|^2,
    \label{e.radial_interpolation_endpoints}
\end{align}
so that $\mcl H_{N,1}$ is~\eqref{e.expanded_hamiltonian_informal} up to an additive constant, while $\mcl H_{N,0}$ is the exponent in~\eqref{e.gaussian_bracket_L_b} with $b=b_N$. The cost of the interpolation is measured by
\begin{align}
    \operatorname{Rad}_{N,L}(x,y,q):=\int_0^1\sum_{s\in\sS}\E\la\Ll|\widetilde{\mcl E}_{N,s}^{x,q}-e^s_N\Rr|\ra_{N,\mt,L}\,\d\mt.
    \label{e.radial_error_def}
\end{align}
The main result of this subsection, Proposition~\ref{p.chart_to_gaussian}, asserts that the chart partition function and Gibbs averages are, up to errors bounded by a constant times $\operatorname{Rad}_{N,L}$ and by vanishing terms, those of the Gaussian model with $b=b_N$. Theorem~\ref{t.spherical_cavity_lim} is then stated under the assumption that $\operatorname{Rad}_{N,L}$ vanishes along the sequence under consideration, and Lemma~\ref{l.radial_self_average} shows that this can be arranged by an appropriate choice of $y$.

\subsubsection{Radial self-averaging}

The mechanism through which the perturbation $\widetilde V^{x,y,q}_N$ forces $\widetilde{\mcl E}^{x,q}_{N,s}$ to concentrate is the following elementary fact about random convex functions, which is a variant of~\cite[Lemma~3.2]{pan}.

\begin{lemma}[Derivatives of random convex functions]
\label{l.integrated_convex_derivative_random}
Let $0<\delta<1/2$. Let $\mcl F$ be a random differentiable convex function on $[0,3]$, and set $f(u):=\E\mcl F(u)$. Assume that $f$ is finite and differentiable on $[0,3]$ and that, for some $K<\infty$,
\begin{equation}
    \sup_{u\in[0,3]}|f'(u)|\le K.
    \label{e.random_convex_derivative_l1_hypothesis}
\end{equation}
Then
\begin{equation}
    \int_1^2\E\Ll|\mcl F'(u)-f'(u)\Rr|\,\d u\le\frac3\delta\int_0^3\E\Ll|\mcl F(u)-f(u)\Rr|\,\d u+2K\delta.
    \label{e.random_integrated_convex_derivative_bound}
\end{equation}
\end{lemma}

\begin{proof}
Set $A(u):=\mcl F(u)-f(u)$. For every differentiable convex function $g:[0,3]\to\R$ and every $u\in[1,2]$, we have
\begin{equation*}
    \frac{g(u)-g(u-\delta)}{\delta}\le g'(u)\le\frac{g(u+\delta)-g(u)}{\delta}.
\end{equation*}
Using the upper difference quotient for $\mcl F$ and the lower one for $f$, and then the other way around, we obtain
\begin{align*}
    \mcl F'(u)-f'(u)
    &\le\frac{f(u+\delta)-2f(u)+f(u-\delta)+A(u+\delta)-A(u)}{\delta},\\
    \mcl F'(u)-f'(u)
    &\ge\frac{-f(u+\delta)+2f(u)-f(u-\delta)+A(u)-A(u-\delta)}{\delta}.
\end{align*}
Since $f$ is convex, $f(u+\delta)-2f(u)+f(u-\delta)\ge0$, and therefore
\begin{align*}
    \Ll|\mcl F'(u)-f'(u)\Rr|
    \le\frac{f(u+\delta)-2f(u)+f(u-\delta)}{\delta}
    +\frac{|A(u-\delta)|+|A(u)|+|A(u+\delta)|}{\delta}.
\end{align*}
We integrate over $u\in[1,2]$ and take expectations. The contribution of the second term is at most $\frac3\delta\int_0^3\E|A(u)|\d u$. For the first term, we write
\begin{align*}
    \int_1^2\Ll(f(u+\delta)-2f(u)+f(u-\delta)\Rr)\d u=\int_0^\delta\left(\int_{2-\delta+v}^{2+v}f'(w)\,\d w-\int_{1-\delta+v}^{1+v}f'(w)\,\d w\right)\d v
    \le2K\delta^2,
\end{align*}
by the fundamental theorem of calculus and~\eqref{e.random_convex_derivative_l1_hypothesis}. This proves~\eqref{e.random_integrated_convex_derivative_bound}.
\end{proof}

We will apply this lemma to the free energy of the bulk system, as a function of $y^s$; this yields the concentration of $\widetilde{\mcl E}^{x,q}_{N,s}$ under the bulk measure for most values of $y$. We then transfer this concentration to the interpolating measures~\eqref{e.radial_interpolation_measure} by a change of density. The point of this two-step procedure is that the centering $e_N$, which depends on $y$, never needs to be differentiated.

For $L\ge1$, let $\la\cdot\ra^{\otimes}_{N,L}$ denote the Gibbs bracket associated with the measure
\begin{equation}
    \chi_L(\tau)\exp\Ll(\widetilde H_N^{x,y,q}(\sigma,\alpha)\Rr)\bgamma_M(\d\tau)\,P_N(\d\sigma)\,\fR(\d\alpha)
    \label{e.product_bulk_cavity_measure}
\end{equation}
on $\Sigma_N\times\R^M\times\mfk U$. This is the product of the bulk Gibbs measure $\la\cdot\ra_N^{\mathord{\sim},x,y,q}$ and of the probability measure proportional to $\chi_L\bgamma_M$ on the cavity coordinates; as for the other brackets of this subsection, the dependence on $x,y,q$ is kept implicit. Functions of $(\sigma,\alpha)$ alone have the same averages under $\la\cdot\ra^{\otimes}_{N,L}$ and under $\la\cdot\ra_N^{\mathord{\sim},x,y,q}$. For $\mt\in[0,1]$, set
\begin{equation}
\begin{aligned}
    W_\mt(\sigma,\tau,\alpha):={}&\mcl H_{N,\mt}(\sigma,\tau,\alpha)-\widetilde H_N^{x,y,q}(\sigma,\alpha)\\
    ={}&\sum_{s\in\sS}\sum_{j\in\msf M_s}\tau_j\mathsf X_{N,j}^{q}(\sigma,\alpha)-\frac12\sum_{s\in\sS}\Ll(e_N^s+2\vartheta^s\Rr)|\tau^{(s)}|^2\\
    &+\frac{\mt}{2}\sum_{s\in\sS}\Ll(M_s-|\tau^{(s)}|^2\Rr)\Ll(\widetilde{\mcl E}_{N,s}^{x,q}(\sigma,\alpha)-e_N^s\Rr).
\end{aligned}
\label{e.cavity_tilt_W}
\end{equation}
Comparing~\eqref{e.radial_interpolation_measure} with~\eqref{e.product_bulk_cavity_measure}, we see that the interpolating measure $\la\cdot\ra_{N,\mt,L}$ has density
\begin{equation}
    R_\mt:=\frac{e^{W_\mt}}{\la e^{W_\mt}\ra^{\otimes}_{N,L}}
    \label{e.cavity_tilt_density}
\end{equation}
with respect to $\la\cdot\ra^{\otimes}_{N,L}$: for every nonnegative measurable function $g$ on $\Sigma_N\times\R^M\times\mfk U$,
\begin{equation}
    \la g\ra_{N,\mt,L}=\la gR_\mt\ra^{\otimes}_{N,L}.
    \label{e.cavity_tilt_change_of_density}
\end{equation}
On the support of $\chi_L$, the field $W_\mt$ is a linear combination of the fields $(\widetilde{\mcl E}^{x,q}_{N,u})_{u\in\sS}$ and $(\msf X^q_{N,j})_{j\in[M]}$ with bounded coefficients, plus a bounded deterministic term. This is what makes the second moment of $R_\mt$ bounded at fixed $L$, as we now show.

\begin{lemma}[Moment bounds under the cavity tilts]
\label{l.fixed_tilt_uniform_bounds}
Fix $L\ge1$. There exists $C_L<\infty$ such that, for every $N$, $x$, $y\in[0,3]^{\sS}$, $q$, $\mt\in[0,1]$, and $s\in\sS$,
\begin{equation}
    \E\la|\widetilde{\mcl E}_{N,s}^{x,q}|^4\ra_N^{\mathord{\sim},x,y,q}
    +\E\la|\widetilde{\mcl E}_{N,s}^{x,q}|^4\ra_{N,\mt,L}
    +\E\la R_\mt^2\ra^{\otimes}_{N,L}\le C_L.
    \label{e.cavity_tilt_moment_bounds}
\end{equation}
In particular, $\sup_N\sup_{x,y,q}\max_{s\in\sS}|e_N^s|<\infty$ and $\sup_N\sup_{x,y,q}\operatorname{Rad}_{N,L}(x,y,q)\le C_L$.
\end{lemma}

\begin{proof}
We first prove a general estimate. Let $\mu$ be a finite measure on a space $\mathbb X$, let $\msf G$ be a centered Gaussian field on $\mathbb X$ and $m$ a bounded function, and let $\la\cdot\ra$ be the Gibbs bracket associated with the measure $e^{\msf G(\msf z)+m(\msf z)}\mu(\d\msf z)$. Let $X$ be a Gaussian field on $\mathbb X$, jointly Gaussian with $\msf G$, and let $A<\infty$ be such that
\begin{equation*}
    \sup_{\msf z\in\mathbb X}\Ll(|\E X(\msf z)|+\operatorname{Var}X(\msf z)\Rr)\le A
    \qquad\text{and}\qquad
    \sup_{\msf z,\msf z'\in\mathbb X}\Ll|\operatorname{Cov}\Ll(X(\msf z),\msf G(\msf z')\Rr)\Rr|\le A.
\end{equation*}
We claim that, for every $\lambda\in\R$,
\begin{equation}
    \E\la e^{\lambda X}\ra\le\exp\Ll(3A|\lambda|+A\lambda^2/2\Rr).
    \label{e.tilted_exponential_moment}
\end{equation}
To see this, fix $\msf z\in\mathbb X$, write $X(\msf z)=\E X(\msf z)+X^\circ(\msf z)$, and set $C_{\msf z}(\msf z'):=\operatorname{Cov}(X(\msf z),\msf G(\msf z'))$. The Cameron--Martin formula gives, for every nonnegative measurable functional $\Phi$ of the field $\msf G$,
\begin{equation*}
    \E\Ll[e^{\lambda X^\circ(\msf z)}\Phi(\msf G)\Rr]
    =e^{\lambda^2\operatorname{Var}(X(\msf z))/2}\,\E\Ll[\Phi(\msf G+\lambda C_{\msf z})\Rr].
\end{equation*}
We apply this identity to the normalized Gibbs density at $\msf z$, namely $\Phi(\msf G):=e^{\msf G(\msf z)+m(\msf z)}/\int e^{\msf G+m}\d\mu$. Since $|C_{\msf z}|\le A$, replacing $\msf G$ by $\msf G+\lambda C_{\msf z}$ multiplies the numerator and the denominator of $\Phi$ by factors in $[e^{-A|\lambda|},e^{A|\lambda|}]$, so that $\Phi(\msf G+\lambda C_{\msf z})\le e^{2A|\lambda|}\Phi(\msf G)$. Hence
\begin{equation*}
    \E\la e^{\lambda X}\ra=\int_{\mathbb X}\E\Ll[e^{\lambda X(\msf z)}\Phi(\msf G)\Rr]\mu(\d\msf z)\le e^{A|\lambda|+A\lambda^2/2+2A|\lambda|}\int_{\mathbb X}\E\Ll[\Phi(\msf G)\Rr]\mu(\d\msf z),
\end{equation*}
and the last integral equals $1$. This proves~\eqref{e.tilted_exponential_moment}. Since $|u|^4\le24(e^u+e^{-u})$ and $|u|^8\le8!\,(e^u+e^{-u})$, we deduce that
\begin{equation}
    \E\la|X|^4\ra\le48\exp(3A+A/2)
    \qquad\text{and}\qquad
    \E\la|X|^8\ra\le2\cdot8!\exp(3A+A/2).
    \label{e.tilted_polynomial_moments}
\end{equation}

We apply this estimate first to the bulk bracket $\la\cdot\ra_N^{\mathord{\sim},x,y,q}$, with $\mathbb X:=\Sigma_N\times\mfk U$, $\mu:=P_N\otimes\fR$, and $\msf G$ the centered part of $\widetilde H_N^{x,y,q}$, and with $X:=\widetilde{\mcl E}^{x,q}_{N,s}$. The mean and the variance of $X$ are bounded by~\eqref{e.radial_energy_covariance_bound}. For the cross-covariances, we note that the covariance of $\widetilde{\mcl E}^{x,q}_{N,s}(\sigma,\alpha)$ with $\widetilde H_N^{x,q}(\sigma',\alpha')$ is $2/N_s$ times a first radial derivative of a covariance kernel of order $N$, hence is bounded by~\eqref{e.second_radial_bound}, and that its covariance with $\widetilde V_N^{x,y,q}(\sigma',\alpha')=\eta_N\sum_uy^uN_u\widetilde{\mcl E}^{x,q}_{N,u}(\sigma',\alpha')$ is $O(\eta_N)$ by~\eqref{e.radial_energy_covariance_bound}. The hypotheses therefore hold with a constant $A=C$ independent of $L$, and~\eqref{e.tilted_polynomial_moments} gives the bound on the first term in~\eqref{e.cavity_tilt_moment_bounds}, together with
\begin{equation}
    \E\la|\widetilde{\mcl E}_{N,s}^{x,q}|^8\ra_N^{\mathord{\sim},x,y,q}\le C.
    \label{e.bulk_eighth_moment}
\end{equation}
Since $e^s_N=\E\la\widetilde{\mcl E}^{x,q}_{N,s}\ra_N^{\mathord{\sim},x,y,q}$ by~\eqref{e.e^N_s=}, Jensen's inequality gives $|e^s_N|\le C$, which is the first assertion after~\eqref{e.cavity_tilt_moment_bounds}.

We next apply the estimate to the product bracket $\la\cdot\ra^{\otimes}_{N,L}$, with
\begin{equation*}
    \mathbb X:=\Sigma_N\times\{|\tau|\le L+1\}\times\mfk U,
    \qquad
    \mu:=\chi_L(\tau)\,\bgamma_M(\d\tau)\,P_N(\d\sigma)\,\fR(\d\alpha),
\end{equation*}
the same field $\msf G$, and $X:=W_\mt$. By~\eqref{e.cavity_tilt_W}, on the set $\mathbb X$, the field $W_\mt$ is a linear combination of $(\widetilde{\mcl E}^{x,q}_{N,u})_{u\in\sS}$ and $(\msf X^q_{N,j})_{j\in[M]}$ with coefficients bounded by $C_L$, plus a deterministic term bounded by $C_L$; here we use $|\tau|\le L+1$, the bound $|e^s_N|\le C$ just proved, and the boundedness of the coefficients $\vartheta^s$ in~\eqref{e.prelimit_vartheta_def} for $q$ in the bounded set $Q$ of~\eqref{e.xyqbR_uniform}. By~\eqref{e.cavity_field_def}, the fields $\msf X^q_{N,j}$ have mean $h^s$ and variance bounded by $C$, and they are independent of $\widetilde H_N^{x,y,q}$ and of $(\widetilde{\mcl E}^{x,q}_{N,u})_u$. Together with the bounds of the previous paragraph, this shows that the hypotheses hold with $A=C_L$, so that~\eqref{e.tilted_exponential_moment} gives
\begin{equation}
    \E\la e^{\lambda W_\mt}\ra^{\otimes}_{N,L}\le C_L\qquad\text{for every }\lambda\in[-4,4].
    \label{e.W_exponential_moments}
\end{equation}
We now bound the second moment of the density~\eqref{e.cavity_tilt_density}. By the Cauchy--Schwarz inequality, $1=(\la e^{W_\mt/2}e^{-W_\mt/2}\ra^{\otimes}_{N,L})^2\le\la e^{W_\mt}\ra^{\otimes}_{N,L}\la e^{-W_\mt}\ra^{\otimes}_{N,L}$, and by Jensen's inequality, $(\la e^{-W_\mt}\ra^{\otimes}_{N,L})^2\le\la e^{-2W_\mt}\ra^{\otimes}_{N,L}$. Hence
\begin{align*}
    \E\la R_\mt^2\ra^{\otimes}_{N,L}
    =\E\Bigg[\frac{\la e^{2W_\mt}\ra^{\otimes}_{N,L}}{\big(\la e^{W_\mt}\ra^{\otimes}_{N,L}\big)^2}\Bigg]
    &\le\E\Ll[\la e^{2W_\mt}\ra^{\otimes}_{N,L}\,\la e^{-2W_\mt}\ra^{\otimes}_{N,L}\Rr]\\
    &\le\Ll(\E\la e^{4W_\mt}\ra^{\otimes}_{N,L}\Rr)^{1/2}\Ll(\E\la e^{-4W_\mt}\ra^{\otimes}_{N,L}\Rr)^{1/2}\le C_L,
\end{align*}
where we used the Cauchy--Schwarz inequality for $\E$, Jensen's inequality for the brackets, and the bound~\eqref{e.W_exponential_moments}. This is the bound on the third term in~\eqref{e.cavity_tilt_moment_bounds}.

For the second term, since $\widetilde{\mcl E}^{x,q}_{N,s}$ does not depend on $\tau$, the change of density~\eqref{e.cavity_tilt_change_of_density}, the Cauchy--Schwarz inequality, and~\eqref{e.bulk_eighth_moment} give
\begin{equation*}
    \E\la|\widetilde{\mcl E}_{N,s}^{x,q}|^4\ra_{N,\mt,L}
    =\E\la|\widetilde{\mcl E}_{N,s}^{x,q}|^4R_\mt\ra^{\otimes}_{N,L}
    \le\Ll(\E\la|\widetilde{\mcl E}_{N,s}^{x,q}|^8\ra_N^{\mathord{\sim},x,y,q}\Rr)^{1/2}\Ll(\E\la R_\mt^2\ra^{\otimes}_{N,L}\Rr)^{1/2}\le C_L.
\end{equation*}
Finally, by Jensen's inequality and $|e_N^s|\le C$, we have $\E\la|\widetilde{\mcl E}^{x,q}_{N,s}-e^s_N|\ra_{N,\mt,L}\le C_L$, and the bound on $\operatorname{Rad}_{N,L}$ follows from its definition~\eqref{e.radial_error_def}.
\end{proof}

\begin{lemma}[Radial self-averaging]
\label{l.radial_self_average}
For every $L\ge1$, we have
\begin{equation*}
    \lim_{N\to\infty}\sup_{x,q}\int_{[1,2]^{\sS}}\operatorname{Rad}_{N,L}(x,y,q)\,\d y=0.
\end{equation*}
\end{lemma}

\begin{proof}
Fix $L\ge1$, $N$, $x$, and $q$, and write $E_s:=\widetilde{\mcl E}_{N,s}^{x,q}$ and $\la\cdot\ra_0:=\la\cdot\ra_N^{\mathord{\sim},x,y,q}$ for brevity. The constants $C$ below may depend on $L$ but not on $N,x,y,q,\mt$. Recall from~\eqref{e.e^N_s=} that $e_N^s=\E\la E_s\ra_0$ depends on $y$. The proof has two steps: we first show that $E_s$ concentrates around $e^s_N$ under the bulk measure for most values of $y$, and we then transfer this concentration to the interpolating measures through the density $R_\mt$.

\smallskip
\noindent\emph{Step 1: Concentration under the bulk measure.}\par
Define
\begin{equation*}
    \mcl F(y):=\frac1N\log Z_N^{\mathord{\sim},x,y,q},
    \qquad
    f(y):=\E\mcl F(y),
\end{equation*}
where $Z_N^{\mathord{\sim},x,y,q}$ is the bulk partition function in~\eqref{e.ZcircMN_def}. By~\eqref{e.radial_generic_perturbation}, the dependence of $\widetilde H_N^{x,y,q}$ on $y$ is through the affine term $\eta_N\sum_sy^sN_sE_s$. Hence, using $N_s=\lambda_sN$,
\begin{gather*}
    \dr_{y^s}\mcl F(y)=\eta_N\lambda_s\la E_s\ra_0,\qquad 
    \dr_{y^s}f(y)=\eta_N\lambda_s\E\la E_s\ra_0=\eta_N\lambda_se^s_N,\\
    \dr_{y^s}^2f(y)=\eta_N^2\lambda_s^2N\,\E\la\Ll(E_s-\la E_s\ra_0\Rr)^2\ra_0.
\end{gather*}
In particular, $\mcl F$ and $f$ are convex in each coordinate of $y$. Since $|e^s_N|\le C$ by Lemma~\ref{l.fixed_tilt_uniform_bounds}, we have $\sup_{y\in[0,3]^{\sS}}|\dr_{y^s}f(y)|\le C\eta_N$. Moreover, the Hamiltonian $\widetilde H_N^{x,y,q}$ is of the form~\eqref{e.linear_cascade_hamiltonian}, see Remark~\ref{r.linear_form_hamiltonians}, and its variance is at most $CN$ uniformly over $x,y,q$, by~\eqref{e.covariance_perturbed_H} and~\eqref{e.radial_energy_covariance_bound}. Lemma~\ref{l.cascade_concentration}, which controls the fluctuations coming from the Gaussian fields and from the cascade, gives
\begin{equation*}
    \sup_{y\in[0,3]^{\sS}}\E\Ll|\mcl F(y)-f(y)\Rr|\le CN^{-1/2}.
\end{equation*}
Set $\delta_N:=(N^{-1/2}/\eta_N)^{1/2}$, which is smaller than $1/2$ for all sufficiently large $N$ and tends to zero, by~\eqref{e.eta_N_def}. Fix $s\in\sS$ and the coordinates $y^u$, $u\ne s$. Applying Lemma~\ref{l.integrated_convex_derivative_random} to $\mcl F$ and $f$ as functions of $u=y^s$, with $K=C\eta_N$ and $\delta=\delta_N$, gives
\begin{align*}
    \eta_N\lambda_s\int_1^2\E\Ll|\la E_s\ra_0-e^s_N\Rr|\,\d y^s
    &\le\frac{3}{\delta_N}\int_0^3\E|\mcl F-f|\,\d y^s+2C\eta_N\delta_N\\
    &\le C\Ll(\frac{N^{-1/2}}{\delta_N}+\eta_N\delta_N\Rr)=2C\eta_N\Ll(\frac{N^{-1/2}}{\eta_N}\Rr)^{1/2}.
\end{align*}
Dividing by $\eta_N\lambda_s$, we obtain
\begin{equation*}
    \int_1^2\E\Ll|\la E_s\ra_0-e^s_N\Rr|\,\d y^s\le C\Ll(\frac{N^{-1/2}}{\eta_N}\Rr)^{1/2}.
\end{equation*}
Next, integrating the formula for $\dr^2_{y^s}f$ over $y^s\in[1,2]$ and using the bound on $\dr_{y^s}f$ gives
\begin{equation*}
    \int_1^2\E\la\Ll(E_s-\la E_s\ra_0\Rr)^2\ra_0\,\d y^s\le\frac{2C\eta_N}{\eta_N^2\lambda_s^2N}=\frac{C}{N\eta_N}.
\end{equation*}
Combining the last two displays with the triangle inequality and the Cauchy--Schwarz inequality, in the form
\begin{equation*}
    \E\la|E_s-e^s_N|\ra_0\le\Ll(\E\la(E_s-\la E_s\ra_0)^2\ra_0\Rr)^{1/2}+\E\Ll|\la E_s\ra_0-e^s_N\Rr|,
\end{equation*}
and using Jensen's inequality for the concave function $u\mapsto u^{1/2}$ after an integration over $y^s\in[1,2]$, we get
\begin{equation*}
    \int_1^2\E\la\Ll|E_s-e^s_N\Rr|\ra_0\,\d y^s\le\eps_N:=C\Ll((N\eta_N)^{-1/2}+(N^{-1/2}/\eta_N)^{1/2}\Rr),
\end{equation*}
and $\eps_N\to0$ by~\eqref{e.eta_N_def}. Integrating over the remaining coordinates of $y\in[1,2]^{\sS}$ and summing over $s$, we conclude that
\begin{equation}
    \sup_{x,q}\int_{[1,2]^{\sS}}a_N(y)\,\d y\le|\sS|\eps_N,
    \qquad\text{where}\qquad
    a_N(y):=\sum_{s\in\sS}\E\la\Ll|E_s-e^s_N\Rr|\ra_0.
    \label{e.bulk_radial_self_averaging}
\end{equation}

\smallskip
\noindent\emph{Step 2: Transfer to the interpolating measures.}\par
Fix $\mt\in[0,1]$. Since $E_s$ does not depend on $\tau$, its averages under $\la\cdot\ra^{\otimes}_{N,L}$ coincide with those under $\la\cdot\ra_0$. By the change of density~\eqref{e.cavity_tilt_change_of_density}, the Cauchy--Schwarz inequality, and the bound on $\E\la R_\mt^2\ra^{\otimes}_{N,L}$ in Lemma~\ref{l.fixed_tilt_uniform_bounds},
\begin{align*}
    \E\la|E_s-e^s_N|\ra_{N,\mt,L}
    =\E\la|E_s-e^s_N|R_\mt\ra^{\otimes}_{N,L}
    &\le\Ll(\E\la|E_s-e^s_N|^2\ra_0\Rr)^{1/2}\Ll(\E\la R_\mt^2\ra^{\otimes}_{N,L}\Rr)^{1/2}\\
    &\le C\Ll(\E\la|E_s-e^s_N|^2\ra_0\Rr)^{1/2}.
\end{align*}
By H\"older's inequality, $\E\la X^2\ra_0\le(\E\la|X|\ra_0)^{2/3}(\E\la|X|^4\ra_0)^{1/3}$ for every random variable $X$, and for $X:=E_s-e^s_N$ the fourth moment is bounded by Lemma~\ref{l.fixed_tilt_uniform_bounds} and $|e^s_N|\le C$. Hence
\begin{equation*}
    \E\la|E_s-e^s_N|\ra_{N,\mt,L}\le C\Ll(\E\la|E_s-e^s_N|\ra_0\Rr)^{1/3}\le C\,a_N(y)^{1/3}.
\end{equation*}
Summing over $s$ and integrating over $\mt\in[0,1]$, we obtain from the definition~\eqref{e.radial_error_def} that
\begin{equation*}
    \operatorname{Rad}_{N,L}(x,y,q)\le C\,a_N(y)^{1/3}.
\end{equation*}
Integrating over $y\in[1,2]^{\sS}$, using Jensen's inequality for the concave function $u\mapsto u^{1/3}$, and then~\eqref{e.bulk_radial_self_averaging}, we conclude that
\begin{equation*}
    \sup_{x,q}\int_{[1,2]^{\sS}}\operatorname{Rad}_{N,L}(x,y,q)\,\d y\le C\Ll(|\sS|\eps_N\Rr)^{1/3}\xrightarrow[N\to\infty]{}0.
\end{equation*}
This completes the proof.
\end{proof}

\subsubsection{Gaussian reduction of the cavity block}

We can now state the main result of this subsection: at fixed cutoff, the $(N+M)$-spin system in the chart is described by the Gaussian model~\eqref{e.gaussian_bracket_L_b} with $b=b_N$, up to the radial error $\operatorname{Rad}_{N,L}$.

\begin{proposition}[Fixed-cutoff chart-to-Gaussian comparison]
\label{p.chart_to_gaussian}
Fix $L\ge1$ and a bounded continuous function $f$ of finitely many replicas of $(\tau,\alpha)$. There exist constants $C_L,C_{f,L}<\infty$ and sequences $(\varepsilon_{N,L})_N$ and $(\varepsilon_{N,f,L})_N$ of nonnegative numbers converging to zero, all independent of $x,y,q$, such that
\begin{gather}
    \Big|\E\log Z^{x,y,q,\mathrm{ch}}_{N+M,L,\widetilde V}
    -\E\log Z^{x,y,q,\bgamma}_{N,L,\mathsf X,b_N}-\frac12\sum_{s\in\sS}M_se^s_N\Big|
    \le C_L\operatorname{Rad}_{N,L}(x,y,q)+\varepsilon_{N,L},
    \label{e.partition_chart_to_gaussian}
    \\
    \Ll|\E\la f\ra^{x,y,q,\mathrm{ch}}_{N+M,L,\widetilde V}-\E\la f\ra^{x,y,q,\bgamma}_{N,L,\mathsf X,b_N}\Rr|
    \le C_{f,L}\operatorname{Rad}_{N,L}(x,y,q)+\varepsilon_{N,f,L}.
    \label{e.normalized_chart_to_gaussian}
\end{gather}
\end{proposition}

\begin{proof}
We take $N$ large enough that $\{|\tau|\le L+1\}\subset\mcl B_N$, and we suppose that $f$ depends on $n$ replicas.

\smallskip
\noindent\emph{Step 1: Replacement of $p_N^{\mathrm{ch}}(\tau)\d\tau$ by $\bgamma_M(\d\tau)$.}\par
For a finite measure $\nu$ on $\R^M$, set
\begin{equation*}
    Z_N[\nu]:=\iiint\chi_L(\tau)e^{H_{N+M}^{x,q}((S_{r_N(\tau)}\sigma,\tau),\alpha)+\widetilde V_N^{x,y,q}(\sigma,\alpha)}\,\nu(\d\tau)\,P_N(\d\sigma)\,\fR(\d\alpha),
\end{equation*}
and let $\la\cdot\ra_{N,\nu}$ be the associated Gibbs bracket. With $\nu=p^{\mathrm{ch}}_N(\tau)\d\tau$, we recover $Z^{x,y,q,\mathrm{ch}}_{N+M,L,\widetilde V}$ and $\la\cdot\ra^{x,y,q,\mathrm{ch}}_{N+M,L,\widetilde V}$ from~\eqref{e.radially_inserted_full_partition_L}. Set
\begin{equation*}
    \delta_{N,L}:=\sup_{|\tau|\le L+1}\Ll|\log\frac{p_N^{\mathrm{ch}}(\tau)}{(2\pi)^{-M/2}e^{-|\tau|^2/2}}\Rr|,
\end{equation*}
which tends to zero as $N\to\infty$ by Lemma~\ref{l.block_disintegration} and~\eqref{e.radial_factor}. On the support of $\chi_L$, the two densities are within a factor $e^{\pm\delta_{N,L}}$ of each other, so that $Z_N[p^{\mathrm{ch}}_N\d\tau]$ and $Z_N[\bgamma_M]$ are within a factor $e^{\pm\delta_{N,L}}$ of each other, and the Radon--Nikodym derivative of the one-replica Gibbs measure $\la\cdot\ra_{N,p^{\mathrm{ch}}_N\d\tau}$ with respect to $\la\cdot\ra_{N,\bgamma_M}$ takes values in $[e^{-2\delta_{N,L}},e^{2\delta_{N,L}}]$. Taking the $n$-fold product, we obtain
\begin{equation}
\begin{aligned}
    \Ll|\E\log Z_N[p^{\mathrm{ch}}_N\d\tau]-\E\log Z_N[\bgamma_M]\Rr|&\le\delta_{N,L}\qquad\text{and}\\
    \Ll|\E\la f\ra_{N,p^{\mathrm{ch}}_N\d\tau}-\E\la f\ra_{N,\bgamma_M}\Rr|&\le\|f\|_\infty\Ll(e^{2n\delta_{N,L}}-1\Rr).
\end{aligned}
\label{e.chart_density_replacement}
\end{equation}

\smallskip
\noindent\emph{Step 2: Expansion of the Hamiltonian.}\par
By Lemma~\ref{l.chart_hamiltonian_expansion}, then by Lemma~\ref{l.radial_first_order_expansion} applied with the cavity field $\sum_j\tau_j\msf X^q_{N,j}$ added to both sides, which is allowed since this field is independent of $\widetilde H_N$, $W^q_N$, and $H^x_N$, and finally by the first identity in~\eqref{e.radial_interpolation_endpoints}, we have
\begin{equation*}
    H_{N+M}^{x,q}\Ll((S_{r_N(\tau)}\sigma,\tau),\alpha\Rr)+\widetilde V_N^{x,y,q}(\sigma,\alpha)\lapprox\mcl H_{N,1}(\sigma,\tau,\alpha)+\frac12\sum_{s\in\sS}M_se^s_N,
\end{equation*}
uniformly over $x,y,q$ and $(\sigma,\tau,\alpha)\in\Sigma_N\times\{|\tau|\le L+1\}\times\mfk U$. For $\mt\in[0,1]$, let
\begin{equation*}
    \mcl Z_N(\mt):=\iiint\chi_L(\tau)e^{\mcl H_{N,\mt}(\sigma,\tau,\alpha)}\bgamma_M(\d\tau)\,P_N(\d\sigma)\,\fR(\d\alpha),
\end{equation*}
whose Gibbs bracket is $\la\cdot\ra_{N,\mt,L}$. Since the constant $\frac12\sum_sM_se^s_N$ is deterministic and bounded, Lemma~\ref{l.gaussian_comparison_covariance_norm} gives sequences $\varepsilon'_{N,L},\varepsilon'_{N,f,L}\to0$, independent of $x,y,q$, such that
\begin{equation}
    \Big|\E\log Z_N[\bgamma_M]-\E\log\mcl Z_N(1)-\frac12\sum_{s\in\sS}M_se^s_N\Big|\le\varepsilon'_{N,L}
    \qquad\text{and}\qquad
    \Ll|\E\la f\ra_{N,\bgamma_M}-\E\la f\ra_{N,1,L}\Rr|\le\varepsilon'_{N,f,L}.
    \label{e.chart_field_comparison}
\end{equation}

\smallskip
\noindent\emph{Step 3: Interpolation in $\mt$.}\par
Differentiating in $\mt$ and using~\eqref{e.interpolating_hamiltonian}, we find
\begin{equation*}
    \frac{\d}{\d \mt}\E\log\mcl Z_N(\mt)=\frac12\sum_{s\in\sS}\E\la\Ll(M_s-|\tau^{(s)}|^2\Rr)\Ll(\widetilde{\mcl E}_{N,s}^{x,q}-e_N^s\Rr)\ra_{N,\mt,L}.
\end{equation*}
Since $|M_s-|\tau^{(s)}|^2|\le M_s+(L+1)^2$ on the support of $\chi_L$, we obtain from~\eqref{e.radial_error_def} that
\begin{equation}
    \Ll|\E\log\mcl Z_N(1)-\E\log\mcl Z_N(0)\Rr|\le C_L\operatorname{Rad}_{N,L}(x,y,q).
    \label{e.radial_interpolation_partition_bound}
\end{equation}
Similarly, writing $z^\ell=(\sigma^\ell,\tau^\ell,\alpha^\ell)$ for the replicas,
\begin{align*}
    \frac{\d}{\d \mt}\E\la f\ra_{N,\mt,L}=\frac12\sum_{s\in\sS}\E\Bigg\langle f\Bigg(&\sum_{\ell=1}^n\Ll(M_s-|\tau^{\ell,(s)}|^2\Rr)\Ll(\widetilde{\mcl E}_{N,s}^{x,q}(\sigma^\ell,\alpha^\ell)-e_N^s\Rr)\\
    &-n\Ll(M_s-|\tau^{n+1,(s)}|^2\Rr)\Ll(\widetilde{\mcl E}_{N,s}^{x,q}(\sigma^{n+1},\alpha^{n+1})-e_N^s\Rr)\Bigg)\Bigg\rangle_{N,\mt,L},
\end{align*}
and since the replicas are exchangeable,
\begin{equation}
    \Ll|\E\la f\ra_{N,1,L}-\E\la f\ra_{N,0,L}\Rr|\le C_{f,L}\operatorname{Rad}_{N,L}(x,y,q).
    \label{e.radial_interpolation_replica_bound}
\end{equation}

\smallskip
\noindent\emph{Step 4: Conclusion.}\par
By the second identity in~\eqref{e.radial_interpolation_endpoints}, we have $\mcl Z_N(0)=Z^{x,y,q,\bgamma}_{N,L,\mathsf X,b_N}$ and $\la\cdot\ra_{N,0,L}=\la\cdot\ra^{x,y,q,\bgamma}_{N,L,\mathsf X,b_N}$. Combining~\eqref{e.chart_density_replacement}, \eqref{e.chart_field_comparison}, \eqref{e.radial_interpolation_partition_bound}, and~\eqref{e.radial_interpolation_replica_bound} proves~\eqref{e.partition_chart_to_gaussian} and~\eqref{e.normalized_chart_to_gaussian}, with $\varepsilon_{N,L}:=\delta_{N,L}+\varepsilon'_{N,L}$ and $\varepsilon_{N,f,L}:=\|f\|_\infty(e^{2n\delta_{N,L}}-1)+\varepsilon'_{N,f,L}$.
\end{proof}

\subsection{The canonical Gaussian block and the cavity functional}
\label{s.canonical_block}

The cavity coordinates of the Gaussian model~\eqref{e.gaussian_bracket_L_b} interact with the bulk only through the fields $\msf X^q_{N,j}$ in~\eqref{e.cavity_field_def}, whose covariance is
\begin{equation*}
    \operatorname{Cov}\Ll(\msf X^q_{N,j}(\sigma,\alpha),\msf X^q_{N,j}(\sigma',\alpha')\Rr)=2t\partial_s\xi\Ll(R_N(\sigma,\sigma')\Rr)+2q^s(\alpha\wedge\alpha'),
    \qquad j\in\msf M_s.
\end{equation*}
Theorem~\ref{t.spherical_cavity_lim} will show that, along suitable sequences, the overlaps of the bulk system synchronize, in the sense that $R_N(\sigma,\sigma')$ is asymptotically a function $p(\alpha\wedge\alpha')$ of the cascade overlap. The covariance above then becomes $2\pi^s(\alpha\wedge\alpha')$ with $\pi:=q+t\nabla\xi(p)$, and the field $\msf X^q_{N,j}$ is replaced by $h^s+\sqrt2w^{\pi^s}_j(\alpha)$, where $w^{\pi^s}_j$ is a cascade field with covariance $\pi^s(\alpha\wedge\alpha')$. This subsection introduces the resulting limiting model, which involves only the cavity coordinates and the cascade, and relates it to the spherical initial condition $\psi$ and to the functional $\sP_{t,q}$ in~\eqref{e.spherical_Parisi_functional}.

\subsubsection{Definitions}

Recall $\mcl K$ from~\eqref{e.Krho_def}. Let $\pi=(\pi^s)_{s\in\sS}\in\mcl Q_\infty^{\sS}$ and $b=(b^s)_{s\in\sS}\in\R^{\sS}$. We say that $b$ is admissible for $\pi$ if
\begin{align}
    b^s>2\mcl K(\pi^s)\qquad\text{for every }s\in\sS.
    \label{e.b_admissible_display}
\end{align}
For $s\in\sS$ and $j\in\msf M_s$, let $w_j^{\pi^s}$ be independent centered Gaussian fields on $\mfk U$, with covariance $\E w_j^{\pi^s}(\alpha)w_j^{\pi^s}(\alpha')=\pi^s(\alpha\wedge\alpha')$ as in~\eqref{e.W_field_covariance}. Set
\begin{align}
    U_{\pi,b}(\tau,\alpha):=\sum_{s\in\sS}\sum_{j\in\msf M_s}\Ll((\sqrt2w_j^{\pi^s}(\alpha)+h^s)\tau_j-\frac{b^s-1}{2}\tau_j^2\Rr),
    \qquad \tau\in\R^M,\quad\alpha\in\mfk U,
    \label{e.U_pi_b_def}
\end{align}
and, for $L\in[1,\infty]$,
\begin{align}
    Z^{\pi,\fR}_{M,L,b}
    &:=
    \iint
    \chi_L(\tau)e^{U_{\pi,b}(\tau,\alpha)}
    \bgamma_M(\d\tau)\,\fR(\d\alpha),
    \notag\\
    \la f\ra^{\pi,\fR}_{M,L,b}
    &:=\frac{1}{(Z^{\pi,\fR}_{M,L,b})^n}
    \int f\Ll((\tau^\ell,\alpha^\ell)_{\ell\le n}\Rr)\prod_{\ell=1}^n\chi_L(\tau^\ell)e^{U_{\pi,b}(\tau^\ell,\alpha^\ell)}
    \bgamma_M(\d\tau^\ell)\,\fR(\d\alpha^\ell),
    \label{e.canonical_cutoff_partition_def}
\end{align}
where $f$ is a function of $n$ replicas. We write $Z^{\pi,\fR}_{M,b}:=Z^{\pi,\fR}_{M,\infty,b}$ and $\la\cdot\ra^{\pi,\fR}_{M,b}:=\la\cdot\ra^{\pi,\fR}_{M,\infty,b}$, and we call $\la\cdot\ra^{\pi,\fR}_{M,b}$ the canonical Gaussian block. When $L=\infty$, the integral defining $Z^{\pi,\fR}_{M,b}$ is finite with integrable logarithm precisely when $b$ is admissible for $\pi$; this is the content of Lemmas~\ref{l.multispecies_nonadmissible_escape} and~\ref{l.remove_cutoffs} below. We define the functional
\begin{equation}
    \mcl D_{M,L}(\pi,b):=\E\log Z^{\pi,\fR}_{M,L,b}
    +\frac12\sum_{s\in\sS}M_s(b^s-1)-\sum_{s\in\sS}M_s\pi^s(1),
    \qquad
    \mcl D_M(\pi,b):=\mcl D_{M,\infty}(\pi,b).
\label{e.direct_canonical}
\end{equation}
The additive constant in~\eqref{e.direct_canonical} is chosen so that, in the proof of Theorem~\ref{t.spherical_cavity_lim}, $\mcl D_{M,L}(\pi,b)$ is the limit of the chart partition function in~\eqref{e.partition_chart_to_gaussian}, including the term $\frac12\sum_sM_se^s_N$. The cavity overlaps are
\begin{equation}
    \widetilde R_{M,s}(u,u'):=\frac1M\sum_{j\in\msf M_s}u_ju_j',
    \qquad
    \widetilde R_M:=(\widetilde R_{M,s})_{s\in\sS},
    \qquad u,u'\in\R^M.
    \label{e.cavity_overlap_def}
\end{equation}

The functional $\mcl D_M$ is a multi-species version of the functional $\mcl D_a^\circ$ in Lemma~\ref{l.scalar_initial_continuous_derivative}. Indeed, the fields $(w^{\pi^s}_j)_{s,j}$ are independent, and $U_{\pi,b}$ is a sum over $(s,j)$ of terms each involving only $\tau_j$ and $w_j^{\pi^s}$. Since the recursive formula for cascade averages is additive over independent groups of fields, see~\cite[Corollary~5.26]{HJbook} and Step~2 of the proof of Proposition~\ref{l.initial_condition}, the representation~\eqref{e.block_representation} gives
\begin{equation}
    \mcl D_{M}(\pi,b)=\sum_{s\in\sS}M_s\,\mcl D^\circ_{h^s}(\pi^s,b^s),
    \label{e.canonical_additivity}
\end{equation}
where the left-hand side is finite if and only if $b$ is admissible for $\pi$, by Lemma~\ref{l.scalar_initial_continuous_derivative}. More precisely, the additivity of the recursion gives~\eqref{e.canonical_additivity} for finite-step paths, and the general case follows by the argument of Step~3 of the proof of Lemma~\ref{l.scalar_initial_continuous_derivative}, which applies verbatim to the block with $M$ coordinates: one discretizes each path $\pi^s$ as there, and compares with the multipliers $(b^s\pm2\delta)_{s\in\sS}$; along the interpolation, the term produced by the change of paths is at most $\delta\sum_{s\in\sS}\E\la|\tau^{(s)}|^2\ra$ in absolute value, since $|\tau_j\tau_j'|\le\frac12(\tau_j^2+\tau_j'^2)$, and it is therefore dominated by the term produced by the change of multipliers. Consequently, by~\eqref{e.scalar_initial_direct_identity} and~\eqref{e.multi_species_initial},
\begin{equation}
    \inf_{b}\mcl D_M(\pi,b)=-\sum_{s\in\sS}M_s\psi^\circ_{h^s}(\pi^s)=-M\psi(\pi),
    \label{e.canonical_infimum}
\end{equation}
where the infimum is over admissible $b$, and it is attained at the unique point
\begin{equation}
    b_\pi:=\Ll(b_{h^s}(\pi^s)\Rr)_{s\in\sS},
    \label{e.b_pi_def}
\end{equation}
with $b_{h^s}(\pi^s)$ the minimizer of Lemma~\ref{l.multiplier}, characterized by the normalization~\eqref{e.scalar_initial_normalization} of Lemma~\ref{l.scalar_initial_continuous_derivative}. 

Finally, we introduce the counterpart of the reaction field $\widetilde Y_N$ in~\eqref{e.coupling_H_tildeH}. For $p\in\mcl Q_{\infty,\le\lambda}^{\sS}$, let $\mathsf y^p$ be a centered Gaussian field on $\mfk U$ with covariance $\E \mathsf y^p(\alpha)\mathsf y^p(\alpha')=\theta(p(\alpha\wedge\alpha'))$; this is a valid covariance since $\theta$ has nonnegative coefficients and $p$ is nondecreasing. Lemma~\ref{l.linear_cascade_identity}, applied with $c=2tM\theta(p)$, gives
\begin{align}
    \mcl I_{M,t}(p):=\E\log\int\exp\Ll(\sqrt{2tM}\,\mathsf y^p(\alpha)-tM\theta(p(1))\Rr)\fR(\d\alpha)=-tM\int_0^1\theta(p(r))\,\d r.
    \label{e.reaction_value}
\end{align}
We then define the cavity functional
\begin{align}
    \mcl P_{M,t,q}(p,b)
    :=\mcl D_M\Ll(q+t\nabla\xi(p),b\Rr)-\mcl I_{M,t}(p),
    \qquad p\in\mcl Q_{\infty,\le\lambda}^{\sS},\quad b\text{ admissible for }q+t\nabla\xi(p).
    \label{e.canonical_ASS_functional}
\end{align}
Theorem~\ref{t.spherical_cavity_lim} bounds the cavity increment from below with this functional, and the next lemma identifies the functional with $-M\sP_{t,q}(p)$ once $b$ is determined by the spherical constraint.

\subsubsection{Diagonal compensation}
\label{s.diagonal_compensation}

In the proof of Theorem~\ref{t.spherical_cavity_lim}, the limiting overlap array of the bulk system has a deterministic diagonal, $R_N(\sigma,\sigma)=\lambda$, while its off-diagonal entries are $p(\alpha\wedge\alpha')$; the endpoint $p(1)$ need not equal $\lambda$. Correspondingly, the cavity fields $\msf X^q_{N,j}$ in~\eqref{e.cavity_field_def} have the deterministic variance $2\vartheta^s=2(q^s(1)+t\dr_s\xi(\lambda))$, whereas the cascade field of the canonical block with path $\pi=q+t\nabla\xi(p)$ has variance $2\pi^s(1)=2(q^s(1)+t\dr_s\xi(p(1)))$, which may be smaller. The block is nevertheless insensitive to such a change of the diagonal variance, provided the quadratic coefficient is shifted accordingly. To formulate this, we parametrize the coefficient by the radial part
\begin{equation}
    b(\pi,e):=\Ll(1+e^s+2\pi^s(1)\Rr)_{s\in\sS},
    \qquad \pi\in\mcl Q_\infty^{\sS},\quad e\in\R^{\sS},
    \label{e.b_pi_e_def}
\end{equation}
so that, relative to the standard Gaussian measure $\bgamma_M$, the exponent~\eqref{e.U_pi_b_def} of the block with coefficient $b(\pi,e)$ reads $\sum_{s}\sum_{j\in\msf M_s}(\sqrt2w_j^{\pi^s}(\alpha)+h^s)\tau_j-\sum_s(\pi^s(1)+\frac{e^s}2)|\tau^{(s)}|^2$. In this parametrization, the additive constant in~\eqref{e.direct_canonical} is $\frac12\sum_sM_se^s$, and does not depend on the path.

\begin{lemma}[Diagonal compensation at fixed cutoff]
\label{l.diagonal_compensation}
Let $L\ge1$, $e\in\R^{\sS}$, and $\pi,\pi'\in\mcl Q_\infty^{\sS}$. Then
\begin{equation}
    \Ll|\mcl D_{M,L}\Ll(\pi,b(\pi,e)\Rr)-\mcl D_{M,L}\Ll(\pi',b(\pi',e)\Rr)\Rr|\le M(L+1)^2\sum_{s\in\sS}|\pi^s-\pi'^s|_{L^1},
    \label{e.diagonal_compensation_partition}
\end{equation}
and, for every $n\in\N$ and every bounded measurable function $f$ of $n$ replicas of $(\tau,\alpha)$, there is a constant $C_n<\infty$, depending only on $n$, such that
\begin{equation}
    \Ll|\E\la f\ra^{\pi,\fR}_{M,L,b(\pi,e)}-\E\la f\ra^{\pi',\fR}_{M,L,b(\pi',e)}\Rr|\le C_n\|f\|_\infty M(L+1)^2\sum_{s\in\sS}|\pi^s-\pi'^s|_{L^1}.
    \label{e.diagonal_compensation_observable}
\end{equation}
\end{lemma}

\begin{proof}
For $\mt\in[0,1]$, set $u_\mt:=(1-\mt)\pi+\mt\pi'\in\mcl Q_\infty^{\sS}$ and $b_\mt:=b(u_\mt,e)$, so that $(u_0,b_0)=(\pi,b(\pi,e))$ and $(u_1,b_1)=(\pi',b(\pi',e))$, and realize the cascade fields jointly as $w^{u_\mt^s}_j:=\sqrt{1-\mt}\,w^{\pi^s}_j+\sqrt{\mt}\,w^{\pi'^s}_j$, with independent fields $w^{\pi^s}_j$ and $w^{\pi'^s}_j$. The covariance of the field $\sqrt2w^{u_\mt^s}_j(\alpha)\tau_j$ is $2u_\mt^s(\alpha\wedge\alpha')\tau_j\tau'_j$, whose derivative in $\mt$ is $2(\pi'^s-\pi^s)(\alpha\wedge\alpha')\tau_j\tau'_j$. Gaussian integration by parts, as in the proof of Lemma~\ref{l.gaussian_comparison_covariance_norm}, together with the derivative $\frac{\d}{\d\mt}b_\mt^s=2(\pi'^s(1)-\pi^s(1))$ of the quadratic coefficient, gives
\begin{align*}
    \frac{\d}{\d\mt}\E\log Z_{M,L,b_\mt}^{u_\mt,\fR}={}&\sum_{s\in\sS}\sum_{j\in\msf M_s}\Ll(\pi'^s(1)-\pi^s(1)\Rr)\E\la\tau_j^2\ra^{u_\mt,\fR}_{M,L,b_\mt}-\frac12\sum_{s\in\sS}\sum_{j\in\msf M_s}\frac{\d b^s_\mt}{\d\mt}\,\E\la\tau_j^2\ra^{u_\mt,\fR}_{M,L,b_\mt}\\
    &-\sum_{s\in\sS}\sum_{j\in\msf M_s}\E\la\tau_j\tau_j'\Ll(\pi'^s-\pi^s\Rr)(\alpha\wedge\alpha')\ra_{M,L,b_\mt}^{u_\mt,\fR}.
\end{align*}
The first term is the diagonal term of the integration by parts, the second one comes from the derivative of the quadratic coefficient, and the two cancel exactly, since $\frac{\d}{\d\mt}b^s_\mt=2(\pi'^s(1)-\pi^s(1))$. This cancellation is the point of the parametrization~\eqref{e.b_pi_e_def}. On the support of $\chi_L$, we have $|\tau_j\tau'_j|\le(L+1)^2$. Moreover, integrating out the cavity coordinates, the bracket $\la\cdot\ra^{u_\mt,\fR}_{M,L,b_\mt}$ tilts the cascade by the factor $\exp(\mathbf g(w^{u_\mt}(\alpha)))$, where $w^{u_\mt}=(w^{u_\mt^s}_j)_{j\in[M]}$ and
\begin{equation*}
    \mathbf g(w):=\log\int\chi_L(\tau)\exp\Big(\sum_{s\in\sS}\sum_{j\in\msf M_s}\Ll((\sqrt2w_j+h^{s})\tau_j-\tfrac12(b^s_\mt-1)\tau_j^2\Rr)\Big)\bgamma_M(\d\tau)
\end{equation*}
is a Lipschitz function on $\R^M$, since its gradient is $\sqrt2$ times a Gibbs average of $\tau$ restricted to $\{|\tau|\le L+1\}$. By the invariance~\eqref{e.invariance_cascade}, the overlap $\alpha\wedge\alpha'$ is therefore uniformly distributed on $[0,1]$ under this bracket, so that the remaining sum is bounded in absolute value by $(L+1)^2\sum_sM_s|\pi'^s-\pi^s|_{L^1}\le M(L+1)^2\sum_s|\pi^s-\pi'^s|_{L^1}$. Integrating over $\mt\in[0,1]$, and recalling that the additive constant in~\eqref{e.direct_canonical} is $\frac12\sum_sM_se^s$ at both endpoints, proves~\eqref{e.diagonal_compensation_partition}. The same interpolation applied to $\E\la f\ra^{u_\mt,\fR}_{M,L,b_\mt}$ produces, as in the proof of Lemma~\ref{l.gaussian_comparison_covariance_norm}, a finite sum of Gibbs averages of $f$ multiplied by the terms $(\pi'^s-\pi^s)(\alpha^\ell\wedge\alpha^{\ell'})\tau^\ell_j\tau^{\ell'}_j$ for $\ell,\ell'\le n+2$, again with an exact cancellation of the diagonal terms; the overlaps $\alpha^\ell\wedge\alpha^{\ell'}$ with $\ell\ne\ell'$ are uniformly distributed by the same invariance, and~\eqref{e.diagonal_compensation_observable} follows.
\end{proof}

The lemma is a fixed-cutoff statement: nothing is claimed about the admissibility of $b(\pi,e)$, which is a separate matter, treated in Subsection~\ref{s.finite_overlap_cutoff} and in the proof of Theorem~\ref{t.spherical_cavity_lim}. Its role is to allow the passage to the limit along paths that converge in $L^1$ but whose endpoint values differ.

\subsubsection{Stationarity in \texorpdfstring{$b$}{b} and the link with the initial condition}

\begin{lemma}[Stationarity in $b$]
\label{l.canonical_normalization_identity}
Let $\pi\in\mcl Q_\infty^{\sS}$ and let $b$ be admissible for $\pi$. The function $b\mapsto\mcl D_M(\pi,b)$ is convex on the set of admissible $b$, with
\begin{gather}
    \dr_{b^s}\mcl D_M(\pi,b)
    =\frac{M_s}{2}-\frac12\E\la |\tau^{(s)}|^2\ra^{\pi,\fR}_{M,b},
    \qquad s\in\sS,
    \label{e.direct_canonical_derivative}
\end{gather}
and $\mcl D_M(\pi,b)\ge -M\psi(\pi)$. Moreover, if
\begin{equation}
    \E\la|\tau^{(s)}|^2\ra^{\pi,\fR}_{M,b}=M_s\qquad\text{for every }s\in\sS,
    \label{e.block_normalization_condition}
\end{equation}
then $b=b_\pi$ and $\mcl D_M(\pi,b)=-M\psi(\pi)$; if in addition $\pi=q+t\nabla\xi(p)$ for some $p\in\mcl Q_{\infty,\le\lambda}^{\sS}$, then
\begin{align}
    \mcl P_{M,t,q}(p,b)=-M\sP_{t,q}(p).
    \label{e.canonical_equals_Parisi}
\end{align}
\end{lemma}

\begin{proof}
By~\eqref{e.canonical_additivity} and part~\eqref{i.multiplier_minimizer} of Lemma~\ref{l.multiplier}, the function $b\mapsto\mcl D_M(\pi,b)$ is a sum over $s\in\sS$ of strictly convex functions of $b^s$, hence it is convex. For each $s\in\sS$, the function $b^s\mapsto\log Z^{\pi,\fR}_{M,b}$ is almost surely convex on the admissible set, with derivative $-\frac12\la|\tau^{(s)}|^2\ra^{\pi,\fR}_{M,b}$, and the monotone convergence argument of Step~5 of the proof of Lemma~\ref{l.scalar_initial_continuous_derivative}, which applies verbatim to the block with $M$ coordinates, shows that $b^s\mapsto\E\log Z^{\pi,\fR}_{M,b}$ is differentiable with derivative $-\frac12\E\la|\tau^{(s)}|^2\ra^{\pi,\fR}_{M,b}$; together with~\eqref{e.direct_canonical}, this gives~\eqref{e.direct_canonical_derivative}. The lower bound is~\eqref{e.canonical_infimum}. Under~\eqref{e.block_normalization_condition}, the derivatives in~\eqref{e.direct_canonical_derivative} vanish, so that $b$ is a critical point of the convex function $b\mapsto\mcl D_M(\pi,b)$, hence a minimizer. By~\eqref{e.canonical_additivity} and the uniqueness of the minimizer in Lemma~\ref{l.scalar_initial_continuous_derivative}, we get $b=b_\pi$ and $\mcl D_M(\pi,b)=-M\psi(\pi)$. Finally, if $\pi=q+t\nabla\xi(p)$, then~\eqref{e.canonical_ASS_functional}, \eqref{e.reaction_value}, and~\eqref{e.spherical_Parisi_functional} give
\begin{equation*}
    \mcl P_{M,t,q}(p,b)=-M\psi\Ll(q+t\nabla\xi(p)\Rr)+tM\int_0^1\theta(p(r))\,\d r=-M\sP_{t,q}(p). \qedhere
\end{equation*}
\end{proof}

The condition~\eqref{e.block_normalization_condition} is the species-wise form of the identity satisfied by the multiplier of a minimizer of the multi-species Parisi functional in~\cite[Theorem~2.12, (2.13a)]{bates2022crisanti}, see also (4.3) there; in Theorem~\ref{t.spherical_cavity_lim}, it is derived from the spherical constraint of the finite system.

The derivative of the initial condition $\psi$ can also be expressed through the canonical block. This is used in Section~\ref{s.crit_pt_bdd} to identify the critical relation satisfied by the limiting path $p$.

\begin{lemma}[Derivative of the initial condition]
\label{l.derivative_initial}
Let $\pi\in\mcl Q_\infty^{\sS}$, and let $b_\pi$ be as in~\eqref{e.b_pi_def}. For every $\kappa\in\mcl Q_\infty^{\sS}$, we have
\begin{align}
    \la\kappa,\dr_q\psi(\pi)\ra_{L^2}
    =\sum_{s\in\sS}\E\la\kappa^s(\alpha\wedge\alpha')
    \widetilde R_{M,s}(\tau,\tau')\ra^{\pi,\fR}_{M,b_\pi}.
    \label{e.derivative_initial_formula_block}
\end{align}
\end{lemma}

\begin{proof}
Write $b:=b_\pi$, $\bar\kappa:=\max_{s\in\sS}\kappa^s(1)$, and $\pi_\eps:=\pi+\eps\kappa\in\mcl Q_\infty^{\sS}$ for $\eps\ge0$, and fix $\eps_0\in(0,1]$ such that $4\eps_0\mcl K(\kappa^s)\le1$ for every $s\in\sS$. Since $\mcl K$ is linear, and since $b^s-2\mcl K(\pi^s)=\ell_{h^s}(\pi^s)\ge1$ by~\eqref{e.b_a_def} and~\eqref{e.uniform_gap}, we have $b^s-2\mcl K(\pi^s_\eps)=\ell_{h^s}(\pi^s)-2\eps\mcl K(\kappa^s)\ge\frac12$ for every $\eps\in[0,\eps_0]$ and $s\in\sS$; in particular, $b$ is admissible for $\pi_\eps$. We compute the derivative of $\eps\mapsto\mcl D_M(\pi_\eps,b)$ at $\eps=0$ in two ways: from the explicit formula, and by Gaussian integration by parts. Since the covariance of the block without cutoff is unbounded, the integration by parts is performed with a cutoff, which is then removed.

\smallskip
\noindent\emph{Step 1: The explicit formula.}\par
Recall the function $\widehat J_a$ from Step~1 of the proof of Proposition~\ref{l.psi_spherical_smooth}. By~\eqref{e.canonical_additivity} and~\eqref{e.b_a_def},
\begin{equation*}
    \mcl D_M(\pi_\eps,b)=\sum_{s\in\sS}M_s\,\widehat J_{h^s}\Ll(\pi^s_\eps,\ell_{h^s}(\pi^s)-2\eps\mcl K(\kappa^s)\Rr).
\end{equation*}
By Step~2 of that proof, applied to the segment from $\pi^s$ to $\pi^s+\kappa^s$, the function $(\eps,\ell)\mapsto\widehat J_{h^s}(\pi^s_\eps,\ell)$ is differentiable in $\eps$, and its derivative, given by~\eqref{e.J_hat_s_derivative}, is jointly continuous on $[0,1]\times(0,\infty)$. It is also differentiable in $\ell$, with $\dr_\ell\widehat J_{h^s}(\pi^s_\eps,\ell)=\frac12(1-S_{h^s}(\pi^s_\eps,\ell))$ by~\eqref{e.D_b_derivative_explicit}, and this derivative is jointly continuous as well, by the dominated convergence argument given at the end of that step. We can therefore apply the chain rule. Since $S_{h^s}(\pi^s,\ell_{h^s}(\pi^s))=1$, the derivative in $\ell$ does not contribute at $\eps=0$, and by~\eqref{e.J_hat_s_derivative}, \eqref{e.zeta_explicit}, \eqref{e.derivative_psi_lambda}, and $M_s=\lambda_sM$, we obtain
\begin{equation}
    \frac{\d}{\d\eps}\mcl D_M(\pi_\eps,b)\Big|_{\eps=0}=-\sum_{s\in\sS}M_s\int_0^1\kappa^s(u)\zeta_{h^s,\pi^s}(u)\,\d u=-M\la\kappa,\dr_q\psi(\pi)\ra_{L^2}.
    \label{e.derivative_initial_explicit}
\end{equation}
Moreover, by~\eqref{e.direct_canonical_derivative}, \eqref{e.canonical_additivity}, and~\eqref{e.D_b_derivative_explicit},
\begin{equation}
    \E\la|\tau^{(s)}|^2\ra^{\pi_\eps,\fR}_{M,b}=M_sS_{h^s}\Ll(\pi^s_\eps,b^s-2\mcl K(\pi^s_\eps)\Rr),
    \qquad s\in\sS,
    \label{e.derivative_initial_second_moment}
\end{equation}
which is a continuous function of $\eps\in[0,\eps_0]$, by the joint continuity just mentioned, and is therefore bounded on this interval.

\smallskip
\noindent\emph{Step 2: Integration by parts with a cutoff.}\par
Conditionally on $\fR$, let $(w_j^{\kappa^s})_{s\in\sS,j\in\msf M_s}$ be independent Gaussian fields with covariances $\kappa^s(\alpha\wedge\alpha')$, independent of the fields $(w_j^{\pi^s})_{s\in\sS,j\in\msf M_s}$, and realize the fields of the block with path $\pi_\eps$ as $w_j^{\pi^s_\eps}:=w_j^{\pi^s}+\sqrt\eps\,w_j^{\kappa^s}$. The exponent in~\eqref{e.U_pi_b_def} then reads
\begin{equation*}
    U_\eps:=U_{\pi_\eps,b}=U_{\pi,b}+\sqrt\eps\,\msf v,
    \qquad
    \msf v(\tau,\alpha):=\sqrt2\sum_{s\in\sS}\sum_{j\in\msf M_s}w_j^{\kappa^s}(\alpha)\tau_j.
\end{equation*}
For $L\in[1,\infty]$ and $\eps\in[0,\eps_0]$, let $Z_{L,\eps}$ and $\la\cdot\ra_{L,\eps}$ be defined as in~\eqref{e.canonical_cutoff_partition_def}, with the exponent $U_\eps$ and with $\chi_L$ replaced by the indicator function of $\{|\tau|\le L\}$, so that $Z_{\infty,\eps}=Z^{\pi_\eps,\fR}_{M,b}$ and $\la\cdot\ra_{\infty,\eps}=\la\cdot\ra^{\pi_\eps,\fR}_{M,b}$. We set
\begin{equation*}
    X_{L,\eps}:=\la|\tau|^2\ra_{L,\eps},
    \qquad
    D_{L,\eps}:=\sum_{s\in\sS}\kappa^s(1)\la|\tau^{(s)}|^2\ra_{L,\eps},
    \qquad
    Y_{L,\eps}:=\sum_{s\in\sS}\la\kappa^s(\alpha\wedge\alpha')\sum_{j\in\msf M_s}\tau_j\tau_j'\ra_{L,\eps},
\end{equation*}
and $\Phi_L(\eps):=\E D_{L,\eps}-\E Y_{L,\eps}$. Since $0\le\kappa^s\le\bar\kappa$, the Cauchy--Schwarz inequality, the independence of the replicas under the bracket, and Jensen's inequality give
\begin{equation}
    0\le D_{L,\eps}\le\bar\kappa X_{L,\eps}
    \qquad\text{and}\qquad
    |Y_{L,\eps}|\le\bar\kappa\la|\tau|\,|\tau'|\ra_{L,\eps}\le\bar\kappa X_{L,\eps}.
    \label{e.derivative_initial_domination}
\end{equation}
By~\eqref{e.derivative_initial_second_moment}, the random variable $X_{\infty,\eps}$ is integrable, so that $\Phi_\infty(\eps)$ is well defined. Let $L<\infty$. The covariance $2\pi^s_\eps(\alpha\wedge\alpha')\tau_j\tau_j'$ of the field $\sqrt2w_j^{\pi^s_\eps}(\alpha)\tau_j$ is bounded on $\{|\tau|\le L\}$, and its derivative in $\eps$ is $2\kappa^s(\alpha\wedge\alpha')\tau_j\tau_j'$. The function $\eps\mapsto\E\log Z_{L,\eps}$ is continuous on $[0,\eps_0]$, and Gaussian integration by parts, as in the proofs of Lemmas~\ref{l.gaussian_comparison_covariance_norm} and~\ref{l.diagonal_compensation}, shows that it is differentiable on $(0,\eps_0]$ with derivative $\Phi_L(\eps)$, the term $\E D_{L,\eps}$ coming from the diagonal of the covariance. Since $|\Phi_L|\le2\bar\kappa L^2$, we obtain
\begin{equation}
    \E\log Z_{L,\eps}-\E\log Z_{L,0}=\int_0^\eps\Phi_L(\eps')\,\d\eps',
    \qquad\eps\in[0,\eps_0].
    \label{e.derivative_initial_cutoff}
\end{equation}

\smallskip
\noindent\emph{Step 3: Removal of the cutoff.}\par
Fix $\eps\in[0,\eps_0]$. Almost surely, $Z_{\infty,\eps}$ and $X_{\infty,\eps}$ are finite, by Lemma~\ref{l.remove_cutoffs} below and~\eqref{e.derivative_initial_second_moment}. For $L<\infty$, the average $X_{\infty,\eps}$ is a convex combination of $X_{L,\eps}\le L^2$ and of an average of $|\tau|^2$ over $\{|\tau|>L\}$, which is at least $L^2$; hence $X_{L,\eps}\le X_{\infty,\eps}$. By dominated convergence with respect to the measure $e^{U_\eps}\bgamma_M\otimes\fR$ and its tensor square, we have $D_{L,\eps}\to D_{\infty,\eps}$ and $Y_{L,\eps}\to Y_{\infty,\eps}$ almost surely as $L\to\infty$, and by~\eqref{e.derivative_initial_domination}, these random variables are bounded by $\bar\kappa X_{\infty,\eps}$, which is integrable. Hence $\Phi_L(\eps)\to\Phi_\infty(\eps)$ as $L\to\infty$, with $|\Phi_L(\eps)|\le2\bar\kappa\,\E X_{\infty,\eps}$, and the right-hand side is bounded uniformly over $\eps\in[0,\eps_0]$ by Step~1. Moreover, $\E\log Z_{L,\eps}\to\E\log Z_{\infty,\eps}$ by monotone convergence, as in the proof of Lemma~\ref{l.remove_cutoffs}, whose lower bound on the partition function uses only the unit ball. Letting $L\to\infty$ in~\eqref{e.derivative_initial_cutoff}, we obtain, by dominated convergence,
\begin{equation}
    \E\log Z^{\pi_\eps,\fR}_{M,b}-\E\log Z^{\pi,\fR}_{M,b}=\int_0^\eps\Phi_\infty(\eps')\,\d\eps',
    \qquad\eps\in[0,\eps_0].
    \label{e.derivative_initial_integrated}
\end{equation}

\smallskip
\noindent\emph{Step 4: Continuity at $\eps=0$.}\par
We show that $\Phi_\infty(\eps)\to\Phi_\infty(0)$ as $\eps\to0$. The term $\E D_{\infty,\eps}$ is a continuous function of $\eps$, by~\eqref{e.derivative_initial_second_moment}. For the other term, we use that $e^{\sqrt\eps\msf v}\le1+e^{\sqrt{\eps_0}\msf v}$, so that $e^{U_\eps}\le e^{U_0}+e^{U_{\eps_0}}$ for every $\eps\in[0,\eps_0]$. Almost surely, the function $(1+|\tau|^2)(e^{U_0}+e^{U_{\eps_0}})$ is integrable with respect to $\bgamma_M\otimes\fR$, since $Z_{\infty,\eps}$ and $X_{\infty,\eps}$ are finite for $\eps\in\{0,\eps_0\}$. By dominated convergence, applied to the numerators and to the denominators of the brackets, we have $X_{\infty,\eps}\to X_{\infty,0}$ and $Y_{\infty,\eps}\to Y_{\infty,0}$ almost surely as $\eps\to0$. Moreover, $\E X_{\infty,\eps}\to\E X_{\infty,0}$ by~\eqref{e.derivative_initial_second_moment}. Since $\bar\kappa X_{\infty,\eps}\pm Y_{\infty,\eps}\ge0$ by~\eqref{e.derivative_initial_domination}, Fatou's lemma gives
\begin{equation*}
    \bar\kappa\,\E X_{\infty,0}\pm\E Y_{\infty,0}\le\liminf_{\eps\to0}\Ll(\bar\kappa\,\E X_{\infty,\eps}\pm\E Y_{\infty,\eps}\Rr)=\bar\kappa\,\E X_{\infty,0}+\liminf_{\eps\to0}\Ll(\pm\E Y_{\infty,\eps}\Rr),
\end{equation*}
that is, $\E Y_{\infty,\eps}\to\E Y_{\infty,0}$.

\smallskip
\noindent\emph{Step 5: Conclusion.}\par
By~\eqref{e.derivative_initial_integrated} and Step~4, the function $\eps\mapsto\E\log Z^{\pi_\eps,\fR}_{M,b}$ is differentiable at $\eps=0$, with derivative $\Phi_\infty(0)$. By~\eqref{e.direct_canonical}, we thus have
\begin{align*}
    \frac{\d}{\d\eps}\mcl D_M(\pi_\eps,b_\pi)\Big|_{\eps=0}
    ={}&\sum_{s\in\sS}\kappa^s(1)\Ll(\E\la|\tau^{(s)}|^2\ra^{\pi,\fR}_{M,b_\pi}-M_s\Rr)-\sum_{s\in\sS}\E\la\kappa^s(\alpha\wedge\alpha')\sum_{j\in\msf M_s}\tau_j\tau_j'\ra^{\pi,\fR}_{M,b_\pi},
\end{align*}
where the terms $\kappa^s(1)$ come from the diagonal of the covariance and from the last term in~\eqref{e.direct_canonical}. Since $b_\pi$ minimizes $b\mapsto\mcl D_M(\pi,b)$, the first sum vanishes by~\eqref{e.direct_canonical_derivative}. Comparing with~\eqref{e.derivative_initial_explicit}, dividing by $-M$, and using~\eqref{e.cavity_overlap_def} gives~\eqref{e.derivative_initial_formula_block}.
\end{proof}

\subsubsection{Admissibility of \texorpdfstring{$b$}{b}}

The next two lemmas show that the truncated functional $\mcl D_{M,L}(\pi,b)$ diverges as $L\to\infty$ when $b$ is not admissible. This will be used in the proof of Theorem~\ref{t.spherical_cavity_lim} to show that the limiting coefficient $b$ is admissible, using only the boundedness of the cavity increment.

\begin{lemma}[Quadratic gain from the cascade]
\label{l.linear_rpc_quadratic_gain}
Let $\pi\in\mcl Q_\infty^{\sS}$, and for $\tau\in\R^M$, set $W_M^\pi(\tau,\alpha):=\sum_{s\in\sS}\sum_{j\in\msf M_s}\tau_jw_j^{\pi^s}(\alpha)$. We have
\begin{equation*}
    \E\log\int\exp\Ll(\sqrt2W_M^\pi(\tau,\alpha)\Rr)\fR(\d\alpha)
    =\sum_{s\in\sS}\mcl K(\pi^s)|\tau^{(s)}|^2.
\end{equation*}
\end{lemma}

\begin{proof}
For fixed $\tau$, the field $\sqrt2W_M^\pi(\tau,\cdot)$ is a centered Gaussian field on $\mfk U$ with covariance $c(\alpha\wedge\alpha')$, where $c(r):=2\sum_{s\in\sS}|\tau^{(s)}|^2\pi^s(r)$ belongs to $\mcl Q_\infty$. The result follows from~\eqref{e.cascade_gaussian_identity} and the definition of $\mcl K$ in~\eqref{e.Krho_def}.
\end{proof}

\begin{lemma}[Divergence for non-admissible $b$]
\label{l.multispecies_nonadmissible_escape}
Let $\pi\in\mcl Q_\infty^{\sS}$ and $b\in\R^{\sS}$. If $b$ is not admissible for $\pi$, then
\begin{equation*}
    \lim_{L\to\infty}\E\log Z^{\pi,\fR}_{M,L,b}=+\infty,\qquad\text{and equivalently}\qquad\lim_{L\to\infty}\mcl D_{M,L}(\pi,b)=+\infty.
\end{equation*}
\end{lemma}

\begin{proof}
Set $\Delta_s:=2\mcl K(\pi^s)-b^s$. Since $b$ is not admissible, there is $s_*\in\sS$ with $\Delta_{s_*}\ge0$. For $\tau\in\R^M$, let
\begin{equation*}
    v(\tau):=\log\int\exp\Ll(\sqrt2W_M^\pi(\tau,\alpha)\Rr)\fR(\d\alpha)
    \qquad\text{and}\qquad
    \Phi(\tau):=\frac12\sum_{s\in\sS}\Delta_s|\tau^{(s)}|^2+\sum_{s\in\sS}h^s\sum_{j\in\msf M_s}\tau_j,
\end{equation*}
where $W^\pi_M$ is as in Lemma~\ref{l.linear_rpc_quadratic_gain}. The random function $v$ is convex in $\tau$, and by Lemma~\ref{l.linear_rpc_quadratic_gain}, $\E v(\tau)=\sum_{s}\mcl K(\pi^s)|\tau^{(s)}|^2$. By~\eqref{e.U_pi_b_def} and $\bgamma_M(\d\tau)=(2\pi)^{-M/2}e^{-|\tau|^2/2}\d\tau$, we can therefore write
\begin{equation*}
    Z^{\pi,\fR}_{M,L,b}=(2\pi)^{-M/2}\int\chi_L(\tau)\,e^{\Phi(\tau)}\,e^{v(\tau)-\E v(\tau)}\,\d\tau.
\end{equation*}
Let $\nu_L$ be the probability measure on $\R^M$ with density proportional to $\chi_Le^{\Phi}$, which is well defined since $\chi_L$ has compact support. Jensen's inequality with respect to $\nu_L$ gives
\begin{equation*}
    \log Z^{\pi,\fR}_{M,L,b}\ge-\frac M2\log(2\pi)+\log\int\chi_L(\tau)e^{\Phi(\tau)}\,\d\tau+\int\Ll(v-\E v\Rr)\d\nu_L.
\end{equation*}
The last term has zero expectation by Fubini's theorem; the required integrability follows from $\E|v(\tau)|\le C(1+|\tau|^2)$, which is a consequence of the lower bound $v(\tau)\ge\sqrt2\int W^\pi_M(\tau,\alpha)\,\fR(\d\alpha)$ (Jensen's inequality) and of the upper bound $\E e^{v(\tau)}=\exp(\sum_s\pi^s(1)|\tau^{(s)}|^2)$. Hence
\begin{equation}
\label{e.escape_deterministic_lower_bound}
    \E\log Z^{\pi,\fR}_{M,L,b}\ge-\frac M2\log(2\pi)+\log\int\chi_L(\tau)\exp\Bigg(\frac12\sum_{s\in\sS}\Delta_s|\tau^{(s)}|^2+\sum_{s\in\sS}h^s\sum_{j\in\msf M_s}\tau_j\Bigg)\d\tau.
\end{equation}
It remains to check that this deterministic integral tends to infinity with $L$. Let $A_L$ be the set of $\tau\in\R^M$ such that $|\tau^{(s_*)}|\le L/2$, $h^{s_*}\sum_{j\in\msf M_{s_*}}\tau_j\ge0$, and $|\tau^{(u)}|\le1$ for every $u\ne s_*$. For $L\ge2|\sS|$, every $\tau\in A_L$ satisfies $|\tau|^2\le L^2/4+|\sS|\le L^2$, so $\chi_L=1$ on $A_L$ by~\eqref{e.chi_L_def}. On $A_L$, the term $\frac12\Delta_{s_*}|\tau^{(s_*)}|^2$ and the external-field term of the species $s_*$ are nonnegative, while the contributions of the other species are bounded from below by a constant $-C$ independent of $L$, since their coordinates lie in a fixed ball. By symmetry, the half-ball $\{\tau^{(s_*)}:|\tau^{(s_*)}|\le L/2,\ h^{s_*}\sum_{j\in\msf M_{s_*}}\tau_j\ge0\}$ has volume at least $cL^{M_{s_*}}$ for some $c>0$ (if $h^{s_*}=0$, it is the whole ball). Restricting the integral in~\eqref{e.escape_deterministic_lower_bound} to $A_L$, we obtain
\begin{equation*}
    \E\log Z^{\pi,\fR}_{M,L,b}\ge-\frac M2\log(2\pi)-C+\log\Ll(cL^{M_{s_*}}\prod_{u\ne s_*}|\mcl B_1(\R^{M_u})|\Rr)\xrightarrow[L\to\infty]{}+\infty,
\end{equation*}
where $|\mcl B_1(\R^{M_u})|$ is the volume of the unit ball. The statement for $\mcl D_{M,L}(\pi,b)$ follows from~\eqref{e.direct_canonical}, since the remaining terms do not depend on $L$.
\end{proof}

\subsection{Finite-overlap approximation and cutoff removal}
\label{s.finite_overlap_cutoff}

This subsection collects the tools needed to pass to the limit in the Gaussian model~\eqref{e.gaussian_bracket_L_b}, and then to remove the cutoff $L$.

\subsubsection{Approximation by finite overlap arrays}

For replicas $(\sigma^\ell,\alpha^\ell)_{\ell\ge1}$, $n\in\N$, and $p,q\in\mcl Q_\infty^{\sS}$, set
\begin{gather}
    R_N^{\ell,\ell'}:=R_N(\sigma^\ell,\sigma^{\ell'}),
    \qquad
    R_\alpha^{\ell,\ell'}:=\alpha^\ell\wedge\alpha^{\ell'},\qquad
    R_N^{\le n}:=(R_N^{\ell,\ell'})_{\ell,\ell'\le n},\qquad
    R_\alpha^{\le n}:=(R_\alpha^{\ell,\ell'})_{\ell,\ell'\le n}, \notag
    \\
    q(R_\alpha^{\le n}):=\Ll(q(R_\alpha^{\ell,\ell'})\Rr)_{\ell,\ell'\le n},\qquad
    p(R_\alpha^{\le n}):=\Ll(p(R_\alpha^{\ell,\ell'})\Rr)_{\ell,\ell'\le n}. \label{e.overlap_array_def}
\end{gather}
For $a,b\in\R$ and $\eps>0$, we write
\begin{equation}\label{e.asymp_eps}
    a\asymp_\eps b \qquad\text{whenever}\qquad |a-b|\le\eps.
\end{equation}
This relation is not transitive; a chain of $k$ such approximations yields an error of at most $k\eps$.

In the Gaussian model~\eqref{e.gaussian_bracket_L_b}, the cavity coordinates and the cascade fields attached to them can be integrated out conditionally on the bulk. The result is a functional of the Gaussian fields $\msf X^q_{N,j}$ restricted to the bulk replicas, whose joint law depends only on the overlap array $(R_N^{\le n},R_\alpha^{\le n},q(R_\alpha^{\le n}))$. The following lemma, which is the overlap-approximation argument of~\cite[Lemmas~6.6 and~6.7]{chen2025free}, makes this precise. Its main point is that the same approximating functions serve for the Gaussian model at finite $N$ and for the canonical block, so that convergence in law of the overlap array transfers to convergence of the corresponding quantities. Recall the bulk bracket $\la\cdot\ra_N^{\mathord{\sim},x,y,q}$ from~\eqref{e.ZcircMN_def} and the cascade bracket $\la\cdot\ra_{\fR}$ from~\eqref{e.fR_bracket_def}.

\begin{lemma}[Finite-overlap approximation]
\label{l.finite_overlap_approx_spherical}
For every $L\ge1$ and $\eps>0$, we have the following.
\begin{enumerate}
    \item \label{i.finite_approx_increment} There exist $n\in\N$ and a bounded continuous function $f_\eps$ such that, uniformly over $x,y,q,b$,
    \begin{gather*}
        \E\log\frac{Z^{x,y,q,\bgamma}_{N,L,\mathsf X,b}}{Z_N^{\mathord{\sim},x,y,q}}
        \asymp_\eps\E\la f_\eps\Ll(R_N^{\le n},q(R_\alpha^{\le n}),b\Rr)\ra_N^{\mathord{\sim},x,y,q},
    \end{gather*}
    and, uniformly over $q\in Q$, $b\in B$, and $p\in\mcl Q_{\infty,\leq\lambda}^{\sS}$,
    \begin{gather*}
        \E\log Z^{q+t\nabla\xi(p),\fR}_{M,L,b}
        \asymp_\eps\E\la f_\eps\Ll(p(R_\alpha^{\le n}),q(R_\alpha^{\le n}),b\Rr)\ra_{\fR}.
    \end{gather*}

    \item \label{i.finite_approx_cavity_array} Let $m\ge1$ and let $G:(\R^{\sS}\times[0,1])^{m\times m}\to\R$ be bounded and continuous. There exist $n\in\N$ and a bounded continuous function $f_\eps^G$ such that, uniformly over $x,y,q,b$,
    \begin{equation*}
        \E\la G\Ll(\Ll(\widetilde R_M(\tau^\ell,\tau^{\ell'}),R_\alpha^{\ell,\ell'}\Rr)_{\ell,\ell'\le m}\Rr)\ra^{x,y,q,\bgamma}_{N,L,\mathsf X,b}
        \asymp_\eps\E\la f_\eps^G\Ll(R_N^{\le n},R_\alpha^{\le n},q(R_\alpha^{\le n}),b\Rr)\ra_N^{\mathord{\sim},x,y,q},
    \end{equation*}
    and, uniformly over $q\in Q$, $b\in B$, and $p\in\mcl Q_{\infty,\leq\lambda}^{\sS}$,
    \begin{equation*}
        \E\la G\Ll(\Ll(\widetilde R_M(\tau^\ell,\tau^{\ell'}),R_\alpha^{\ell,\ell'}\Rr)_{\ell,\ell'\le m}\Rr)\ra^{q+t\nabla\xi(p),\fR}_{M,L,b}
        \asymp_\eps\E\la f_\eps^G\Ll(p(R_\alpha^{\le n}),R_\alpha^{\le n},q(R_\alpha^{\le n}),b\Rr)\ra_{\fR}.
    \end{equation*}

    \item \label{i.finite_approx_reaction} For every $N\in\N$, let $\mathsf y_N$ be a centered Gaussian field on $\Sigma_N$, independent of the other Gaussian fields, with covariance $\E\mathsf y_N(\sigma)\mathsf y_N(\sigma')=\theta(R_N(\sigma,\sigma'))$. There exist $n\in\N$ and a bounded continuous function $f_\eps^{\mcl I}$ such that, uniformly over $x,y,q$,
    \begin{equation*}
        \E\log\la\exp\Ll(\sqrt{2tM}\,\mathsf y_N(\sigma)-tM\theta(\lambda)\Rr)\ra_N^{\mathord{\sim},x,y,q}
        \asymp_\eps\E\la f_\eps^{\mcl I}(R_N^{\le n})\ra_N^{\mathord{\sim},x,y,q},
    \end{equation*}
    and, uniformly over $p\in\mcl Q_{\infty,\leq\lambda}^{\sS}$,
    \begin{equation*}
        \mcl I_{M,t}(p)\asymp_\eps\E\la f_\eps^{\mcl I}\Ll(p(R_\alpha^{\le n})\Rr)\ra_{\fR}.
    \end{equation*}
\end{enumerate}
\end{lemma}

\begin{proof}
The quantities on the left-hand sides are of the form considered in~\cite[Lemmas~6.6 and~6.7]{chen2025free}, and the statements follow from~\cite[Proposition~4.4]{chen2025free} as explained there. We only indicate how the arguments of the approximating functions arise. In the Gaussian model, the covariance of the cavity fields between two bulk replicas is $2t\partial_s\xi(R_N^{\ell,\ell'})+2q^s(R_\alpha^{\ell,\ell'})$, while in the canonical block with path $q+t\nabla\xi(p)$ it is $2t\partial_s\xi(p(R_\alpha^{\ell,\ell'}))+2q^s(R_\alpha^{\ell,\ell'})$: these are the same continuous function of $(R^{\ell,\ell'}_N,q(R_\alpha^{\ell,\ell'}))$, respectively of $(p(R_\alpha^{\ell,\ell'}),q(R_\alpha^{\ell,\ell'}))$. Since $\nabla\xi$ is continuous, the dependence on the overlaps can be absorbed into the approximating functions, but the path $q$ need not be continuous, so $q(R^{\le n}_\alpha)$ is kept as an argument. In part~\eqref{i.finite_approx_cavity_array}, the function $G$ depends explicitly on $R^{\ell,\ell'}_\alpha$, which therefore remains a separate argument; in general, it cannot be recovered from $q(R_\alpha^{\le n})$ when $q$ is constant on an interval. Part~\eqref{i.finite_approx_reaction} is the same argument applied to the scalar covariance $\theta(R_N^{\ell,\ell'})$, respectively $\theta(p(R_\alpha^{\ell,\ell'}))$. Since $L$ is fixed, the cavity coordinates range over a compact set, and the boundedness of $Q$, $B$, and of the overlaps make the approximations uniform in the stated parameters.
\end{proof}

\subsubsection{Uniform integrability of the cavity coordinates}

The extension of the cavity limit from bounded observables to observables of polynomial growth rests on the following moment bounds, which are a consequence of the rotation invariance of the reference measure within each species.

\begin{lemma}[Uniform moments of the cavity coordinates]
\label{l.true_cavity_radius_ui}
For every integer $k\ge1$, there exists a constant $C_k<\infty$, depending only on $k$, such that, for every $s\in\sS$,
\begin{gather*}
    \sup_N\sup_{x,y,q}\E\la|\tau^{(s)}|^{2k}\ra_{N+M}^{x,y,q}\le C_kM_s^k,
    \\
    \text{and consequently}\qquad
    \lim_{a\to\infty}\sup_N\sup_{x,y,q}\E\la\frac{|\tau^{(s)}|^2}{M_s}\one_{\{|\tau^{(s)}|^2>aM_s\}}\ra_{N+M}^{x,y,q}=0.
\end{gather*}
\end{lemma}

\begin{proof}
Let $(e_i)_{i\le N+M}$ be the standard basis of $\R^{N+M}$. For each $s'\in\sS$, let $\mathsf O_{N,s'}$ be the compact group of orthogonal transformations of $\R^{I_{N+M,s'}}$ which fix the vector $\sum_{i\in I_{N+M,s'}}e_i$, and let $\mathsf O_N:=\prod_{s'\in\sS}\mathsf O_{N,s'}$ act on $\R^{N+M}$ blockwise. The reference measure $P_{N+M}$ is invariant under $\mathsf O_N$. Every $U\in\mathsf O_N$ preserves the species overlaps and the sums $\sum_{i\in I_{N+M,s'}}\rho_i$, and commutes with the dilations~$S_r$. By~\eqref{e.mean_perturbed_H} and~\eqref{e.covariance_perturbed_H}, the law of the field $(H_{N+M}^{x,y,q}(U\rho,\alpha))_{\rho,\alpha}$ is therefore the same as that of $(H_{N+M}^{x,y,q}(\rho,\alpha))_{\rho,\alpha}$. It follows that, for every bounded Borel function $F$ on $\Sigma_{N+M}$,
\begin{equation}
    \E\la F(\rho)\ra_{N+M}^{x,y,q}=\E\la\int_{\mathsf O_N}F(U\rho)\,\d U\ra_{N+M}^{x,y,q},
    \label{e.residual_rotation_average}
\end{equation}
where $\d U$ denotes the Haar probability measure. We are thus reduced to a deterministic estimate on the average over $\mathsf O_{N,s}$.

Set $n_s:=|I_{N+M,s}|=N_s+M_s$ and $\vecone_s:=\sum_{i\in I_{N+M,s}}e_i$. Let $u\in\R^{I_{N+M,s}}$ with $|u|^2=n_s$, and write
\begin{equation*}
    u=m\vecone_s+v,
    \qquad
    m:=\frac1{n_s}\sum_{i\in I_{N+M,s}}u_i,
    \qquad
    v\perp\vecone_s,
\end{equation*}
so that $n_sm^2+|v|^2=n_s$; in particular, $|m|\le1$ and $|v|^2=n_s(1-m^2)$. Every $U\in\mathsf O_{N,s}$ fixes $\vecone_s$, so that $(Uu)_i=m+(Uv)_i$ for every $i\in I_{N+M,s}$. If $n_s=1$, then $\mathsf O_{N,s}$ is trivial, $v=0$, and $|(Uu)_i|=1$. Assume now that $n_s\ge2$, and fix $i\in I_{N+M,s}$. Under the Haar measure on $\mathsf O_{N,s}$, the vector $Uv$ is uniform on the sphere of radius $|v|$ in the hyperplane $\vecone_s^\perp$, whose dimension is $d:=n_s-1$, and $(Uv)_i$ is the scalar product of $Uv$ with the orthogonal projection of $e_i$ onto $\vecone_s^\perp$, which has squared norm $1-n_s^{-1}=d/n_s$. If $Z$ is uniform on the unit sphere of $\R^d$, then $\E Z_1^{2k}=(2k-1)!!\prod_{j=0}^{k-1}(d+2j)^{-1}$: writing $Z=g/|g|$ with $g$ a standard Gaussian vector in $\R^d$, and using the independence of $|g|$ and $g/|g|$, we get $\E g_1^{2k}=\E|g|^{2k}\,\E Z_1^{2k}$, with $\E g_1^{2k}=(2k-1)!!$ and $\E|g|^{2k}=\prod_{j=0}^{k-1}(d+2j)$. Hence
\begin{equation*}
    \int_{\mathsf O_{N,s}}(Uv)_i^{2k}\,\d U
    =(2k-1)!!\,\Ll(\frac{d}{n_s}\Rr)^k|v|^{2k}\prod_{j=0}^{k-1}\frac1{d+2j}
    =(2k-1)!!\,(1-m^2)^k\prod_{j=0}^{k-1}\frac{d}{d+2j}
    \le(2k-1)!!,
\end{equation*}
and, since $|m+w|^{2k}\le2^{2k-1}(m^{2k}+w^{2k})$ and $|m|\le1$,
\begin{equation}
    \sup_N\ \sup_{u\in\R^{I_{N+M,s}}:\:|u|^2=n_s}\ \max_{i\in I_{N+M,s}}\int_{\mathsf O_{N,s}}|(Uu)_i|^{2k}\,\d U
    \le2^{2k-1}\Ll(1+(2k-1)!!\Rr)=:C_k,
    \label{e.deterministic_orbit_moment_bound}
\end{equation}
which also covers the case $n_s=1$. Fix $j\in\msf M_s$, so that $N+j\in I_{N+M,s}$ by~\eqref{e.N_compatible_block}, and let $F(\rho):=\rho_{N+j}^{2k}=\tau_j^{2k}$, which is bounded on $\Sigma_{N+M}$. For $\rho\in\Sigma_{N+M}$, we have $|\rho^{(s)}|^2=n_s$, and $F(U\rho)$ depends only on the component of $U$ in $\mathsf O_{N,s}$, so that~\eqref{e.deterministic_orbit_moment_bound}, applied with $u:=\rho^{(s)}$ and $i:=N+j$, gives $\int_{\mathsf O_N}F(U\rho)\,\d U\le C_k$. By~\eqref{e.residual_rotation_average}, we obtain $\E\la\tau_j^{2k}\ra_{N+M}^{x,y,q}\le C_k$ for every $j\in\msf M_s$. By Jensen's inequality, $|\tau^{(s)}|^{2k}=(\sum_{j\in\msf M_s}\tau_j^2)^k\le M_s^{k-1}\sum_{j\in\msf M_s}\tau_j^{2k}$, and summing the previous bound over $j\in\msf M_s$ gives the moment bound. The second assertion follows from the case $k=2$: since $u\one_{\{u>a\}}\le u^2/a$ for $u\ge0$, the quantity under the limit is at most $C_2/a$.
\end{proof}

\subsubsection{Removal of the cutoff}

For the $(N+M)$-spin system, the exchangeability of the coordinates within each species gives an exact formula for the second moment of the cavity coordinates: for every $i\in I_{N+M,s}$, $\E\la\rho_i^2\ra_{N+M}^{x,y,q}$ does not depend on $i$, and $\sum_{i\in I_{N+M,s}}\rho_i^2=N_s+M_s$ on $\Sigma_{N+M}$, so that
\begin{equation}\label{e.true_block_second_moment}
    \E\la|\tau^{(s)}|^2\ra_{N+M}^{x,y,q}= M_s,
    \qquad s\in\sS,
\end{equation}
as already observed in the proof of Lemma~\ref{l.fixed_block_increment_bound}. This identity is the spherical constraint in the form that we use.

\begin{lemma}[Removal of the cutoff]
\label{l.remove_cutoffs}
For every $n\in\N$, there exists $C_n<\infty$ such that the following holds.
\begin{enumerate}
    \item \label{i.remove_cutoff_true} For every bounded measurable function $f$ of $n$ replicas of $(\tau,\alpha)$, and every $L\ge1$,
    \begin{equation*}
        \sup_N\sup_{x,y,q}\Ll|\E\la f\ra_{N+M,L}^{x,y,q}-\E\la f\ra_{N+M}^{x,y,q}\Rr|\le C_n\|f\|_\infty\frac{M}{L^2}.
    \end{equation*}
    \item \label{i.remove_cutoff_canonical} Let $\pi\in\mcl Q_\infty^{\sS}$ and let $b$ be admissible for $\pi$. The quantity $Z^{\pi,\fR}_{M,b}$ is almost surely finite with integrable logarithm, and for every bounded measurable function $f$ of $n$ replicas of $(\tau,\alpha)$,
    \begin{gather*}
        \lim_{L\to\infty}\Ll|\E\la f\ra^{\pi,\fR}_{M,L,b}-\E\la f\ra^{\pi,\fR}_{M,b}\Rr|=0
        \qquad\text{and}\qquad
        \lim_{L\to\infty}\mcl D_{M,L}(\pi,b)=\mcl D_M(\pi,b).
    \end{gather*}
\end{enumerate}
\end{lemma}

\begin{proof}
We first record an estimate valid for both brackets. Let $\la\cdot\ra$ be a Gibbs bracket without cutoff, let $\mu$ be the corresponding one-replica Gibbs measure, and let $\la\cdot\ra_L$ be the bracket obtained by inserting the factor $\chi_L(\tau^\ell)$ for each replica $\ell\le n$. Setting $c_L:=\la\chi_L(\tau)\ra\in(0,1]$ and $u_L:=1-c_L$, the bracket $\la\cdot\ra_L$ is the $n$-fold tensor power of the probability measure $\mu_L$ with density $\chi_L/c_L$ with respect to $\mu$. Since $0\le\chi_L\le1$, the measure $\mu-c_L\mu_L$ is nonnegative with total mass $u_L$, so that $\mu=c_L\mu_L+u_L\nu_L$ for some probability measure $\nu_L$ (when $u_L=0$, we have $\mu=\mu_L$). Expanding the $n$-fold tensor power of this decomposition, we find $\mu^{\otimes n}=c_L^n\mu_L^{\otimes n}+(1-c_L^n)\nu'$ for some probability measure $\nu'$. Hence, for every bounded measurable $f$ of $n$ replicas,
\begin{equation*}
    \Ll|\la f\ra-\la f\ra_L\Rr|=(1-c_L^n)\Ll|\nu'(f)-\la f\ra_L\Rr|\le2\|f\|_\infty(1-c_L^n)\le2n\|f\|_\infty u_L,
\end{equation*}
where we used that $1-c^n\le n(1-c)$ for $c\in[0,1]$. Taking expectations, we obtain, with $C_n:=2n$,
\begin{equation}\label{e.|E<G>_L-E<G>|<}
    \Ll|\E\la f\ra_L-\E\la f\ra\Rr|\le C_n\|f\|_\infty\,\E u_L.
\end{equation}

For part~\eqref{i.remove_cutoff_true}, we apply this to $\la\cdot\ra^{x,y,q}_{N+M}$. By~\eqref{e.chi_L_def}, Chebyshev's inequality, and~\eqref{e.true_block_second_moment},
\begin{align*}
    \E u_L
    \le\E\la\one_{\{|\tau|>L\}}\ra_{N+M}^{x,y,q}\le\frac1{L^2}\sum_{s\in\sS}\E\la|\tau^{(s)}|^2\ra_{N+M}^{x,y,q}=\frac{M}{L^2}.
\end{align*}

For part~\eqref{i.remove_cutoff_canonical}, we first verify the integrability statements. By~\eqref{e.canonical_additivity} and Lemma~\ref{l.scalar_initial_continuous_derivative}, the quantity $\E\log Z^{\pi,\fR}_{M,b}$ is finite when $b$ is admissible; in particular $Z^{\pi,\fR}_{M,b}<\infty$ almost surely. We also need a lower bound on $\log Z^{\pi,\fR}_{M,1,b}$. Let $\nu$ be the probability measure obtained by normalizing the restriction of $\bgamma_M$ to the unit ball $\mcl B_1$ of $\R^M$. Since $\chi_1=1$ on $\mcl B_1$, Jensen's inequality gives
\begin{align*}
    \log Z^{\pi,\fR}_{M,1,b}
    &\ge\log\bgamma_M(\mcl B_1)+\iint_{\mcl B_1\times\mfk U}U_{\pi,b}(\tau,\alpha)\,\nu(\d\tau)\,\fR(\d\alpha)\\
    &=\log\bgamma_M(\mcl B_1)-\frac12\sum_{s\in\sS}\sum_{j\in\msf M_s}(b^s-1)\int_{\mcl B_1}\tau_j^2\,\nu(\d\tau),
\end{align*}
where the terms linear in $\tau$ vanish because $\nu$ is symmetric. Since $Z^{\pi,\fR}_{M,1,b}\le Z^{\pi,\fR}_{M,L,b}\le Z^{\pi,\fR}_{M,b}$, both $\log Z^{\pi,\fR}_{M,1,b}$ and $\log Z^{\pi,\fR}_{M,b}$ are integrable, and by the monotonicity in~\eqref{e.chi_L_def}, $\log Z^{\pi,\fR}_{M,L,b}$ increases to $\log Z^{\pi,\fR}_{M,b}$ as $L\to\infty$. Monotone convergence gives $\lim_{L\to\infty}\E\log Z^{\pi,\fR}_{M,L,b}=\E\log Z^{\pi,\fR}_{M,b}$, hence the convergence of $\mcl D_{M,L}(\pi,b)$ by~\eqref{e.direct_canonical}. Finally, since $Z^{\pi,\fR}_{M,b}$ is almost surely finite, dominated convergence gives $u_L=1-\la\chi_L(\tau)\ra^{\pi,\fR}_{M,b}\to0$ almost surely, hence $\E u_L\to0$, and~\eqref{e.|E<G>_L-E<G>|<} concludes the proof.
\end{proof}

\subsection{Ghirlanda--Guerra identities and the cavity limit}
\label{s.GG_cavity_prop}

We finally combine the results of this section. We first quantify the extent to which the bulk system satisfies the Ghirlanda--Guerra identities. Recall the overlaps in~\eqref{e.overlap_array_def} and the kernels $\msf C_{\msf h}$ in~\eqref{e.C_h}, and set
\begin{align*}
    R^{\ell,\ell'}_{N,\msf h} :=\msf C_{\msf h}\Ll(R^{\ell,\ell'}_N,R^{\ell,\ell'}_\alpha\Rr)= \Ll(a_{\msf h_1}\cdot \Ll(R^{\ell,\ell'}_{N}\Rr)^{\odot \msf h_2}+\iota_{\msf h_3} R^{\ell,\ell'}_\alpha\Rr)^{\msf h_4}.
\end{align*}
For $N\in\N$, an integer $n\geq 2$, $\msf h\in\N^4$, and a bounded measurable function $\bff:\Ll(\R^\sS \times \R\Rr)^{n\times n}\to \R$, we define, with $\la\cdot\ra=\la\cdot\ra_N^{\mathord{\sim},x,y,q}$,
\begin{align}
    \Delta^{x,y,q}_N(\bff,n,\msf h) := \Big|\E \la \bff\Ll(R^{\leq n}_N,R^{\leq n}_\alpha\Rr)R^{1,n+1}_{N,\msf h} \ra &- \frac{1}{n}\E\la \bff\Ll(R^{\leq n}_N,R^{\leq n}_\alpha\Rr)\ra \E \la  R^{1,2}_{N,\msf h}\ra\notag\\
    &- \frac{1}{n}\sum_{\ell=2}^n \E \la \bff\Ll(R^{\leq n}_N,R^{\leq n}_\alpha\Rr) R^{1,\ell}_{N,\msf h}\ra\Big|. \label{e.Delta_r_def}
\end{align}
Here and throughout this subsection, $\E$ integrates the Gaussian randomness in the Hamiltonian and the randomness of $\fR$, but not the parameters $x,y$. We enumerate as $((\bff_j,n_j,\msf h_j))_{j\in\N}$ all the triples $(\bff,n,\msf h)$ in which $n\ge2$, $\msf h\in\N^4$, and $\bff$ is a monomial with coefficient $1$ in the entries of the array. Since overlaps are bounded, we may modify each $\bff_j$ outside a bounded set so that it becomes bounded, and then rescale it so that
\begin{align}\label{e.Delta<1}
    \Delta^{x,y,q}_N(\bff_j,n_j,\msf h_j)\leq 1,\qquad\text{for every } j\in\N,\ N\in\N,\ x\in[0,3]^{\N^4},\ y\in[0,3]^\sS,\ q\in\mcl Q_\infty^\sS.
\end{align}
We then set
\begin{align}\label{e.Delta_MN_def}
    \Delta_N(x,y,q) := \sum_{j=1}^\infty 2^{-j} \Delta^{x,y,q}_N(\bff_j,n_j,\msf h_j).
\end{align}
Thus $\Delta_N(x,y,q)$ vanishes asymptotically if and only if the overlap array of the bulk system asymptotically satisfies the Ghirlanda--Guerra identities. Lemma~\ref{l.overlap_identities_radial_spherical} below shows that this holds on average over $(x,y)$.

We can now state the main result of this section. Recall $\operatorname{Rad}_{N,L}$ from~\eqref{e.radial_error_def}, $A_N$ from~\eqref{e.ASS_increment_def}, $\mcl P_{M,t,q}$ from~\eqref{e.canonical_ASS_functional}, $\sP_{t,q}$ from~\eqref{e.spherical_Parisi_functional}, the canonical block $\la\cdot\ra^{\pi,\fR}_{M,b}$ from~\eqref{e.canonical_cutoff_partition_def}, the cavity overlaps $\widetilde R_M$ from~\eqref{e.cavity_overlap_def}, and $b_\pi$ from~\eqref{e.b_pi_def}. The four conclusions of the theorem are stated in their logical order: the synchronization of the bulk overlaps; the convergence of the cavity coordinates to the canonical block, with an admissible coefficient; the spherical normalization, which identifies the coefficient; and the lower bound on the increment, whose value is identified by the normalization.

\begin{theorem}[Spherical cavity limit]
\label{t.spherical_cavity_lim}
Fix $t>0$ and $q\in\mcl Q_\infty^{\sS}$. Let $(N_k,x_k,y_k,q_k)_{k\ge1}$ be a sequence such that $N_k\in M\N$ tends to infinity, $x_k\in[0,3]^{\N^4}$, $y_k\in[0,3]^{\sS}$, and $q_k\in\mcl Q_\infty^{\sS}$ satisfies
\begin{align}
    \sup_k|q_k|_{L^\infty}<\infty,\qquad \lim_{k\to\infty}q_k(r)=q(r)\quad\text{for a.e.\ }r\in[0,1],\qquad\lim_{k\to\infty}q_k(1)=q(1).
    \label{e.cavity_theorem_path_convergence}
\end{align}
Assume that
\begin{equation}
    \lim_{k\to\infty}\Delta_{N_k}(x_k,y_k,q_k)=0\qquad\text{and}\qquad \lim_{k\to\infty}\operatorname{Rad}_{N_k,L}(x_k,y_k,q_k)=0\quad\text{for every }L\in\N.
    \label{e.cavity_assumptions_delta}
\end{equation}
Then there exist a subsequence, which we still denote by $(N_k,x_k,y_k,q_k)_{k\ge1}$, a path $p\in\mcl Q_{\infty,\le\lambda}^{\sS}$, and a vector $b\in\R^{\sS}$ admissible for $\pi:=q+t\nabla\xi(p)$, such that the following holds.
\begin{enumerate}
    \item\label{i.cav_overlaps}
    \emph{Synchronization of the bulk overlaps.} For every $n\in\N$, the array $(R^{\leq n}_{N_k}, R^{\leq n}_\alpha)$ under $\E\la\cdot\ra_{N_k}^{\mathord{\sim},x_k,y_k,q_k}$ converges in law, as $k\to\infty$, to
    \begin{align}
        \Ll(\Ll(p(R^{\ell,\ell'}_\alpha),R^{\ell,\ell'}_\alpha\Rr)\one_{\{\ell\neq\ell'\}}+\Ll(\lambda,1\Rr)\one_{\{\ell=\ell'\}}\Rr)_{1\leq \ell,\ell'\leq n}
        \qquad\text{under $\E\la\cdot\ra_\fR$.}
        \label{e.overlap_limit_statement}
    \end{align}

    \item\label{i.cav_marginal}
    \emph{Convergence of the cavity coordinates.} For every $n\in\N$ and every continuous function $g:(\R^{\sS}\times[0,1])^{n\times n}\to\R$ such that $|g((a^{\ell,\ell'},r^{\ell,\ell'})_{\ell,\ell'\le n})|\le C(1+\sum_{\ell,\ell'\le n}|a^{\ell,\ell'}|^d)$ for some $C<\infty$ and $d\in\N$, we have, with $\tau$ as in~\eqref{e.tau=},
    \begin{align}
        \lim_{k\to\infty}
        \E\la g\Ll(\Ll(\widetilde R_M(\tau^\ell,\tau^{\ell'}),\alpha^\ell\wedge\alpha^{\ell'}\Rr)_{\ell,\ell'\le n}\Rr)\ra_{N_k+M}^{x_k,y_k,q_k}
        =\E\la g\Ll(\Ll(\widetilde R_M(\tau^\ell,\tau^{\ell'}),\alpha^\ell\wedge\alpha^{\ell'}\Rr)_{\ell,\ell'\le n}\Rr)\ra^{\pi,\fR}_{M,b}.
        \label{e.cavity_marginal_limit}
    \end{align}
    The arrays include their diagonal entries, so that~\eqref{e.cavity_marginal_limit} covers in particular the one-replica second moments $|\tau^{(s)}|^2=M\widetilde R_{M,s}(\tau,\tau)$.

    \item\label{i.cav_normalization}
    \emph{Spherical normalization.} The canonical block satisfies
    \begin{equation}
        \E\la|\tau^{(s)}|^2\ra^{\pi,\fR}_{M,b}=M_s\qquad\text{for every }s\in\sS,
        \label{e.cavity_canonical_normalization}
    \end{equation}
    and consequently $b=b_\pi$.

    \item\label{i.cav_increment}
    \emph{The cavity increment.} We have
    \begin{equation}
        \liminf_{k\to\infty}A_{N_k}(x_k,y_k,q_k)\ge\mcl P_{M,t,q}(p,b)=-M\sP_{t,q}(p).
        \label{e.cavity_increment_lower_bound}
    \end{equation}
\end{enumerate}
\end{theorem}

The bounded set $Q$ in~\eqref{e.xyqbR_uniform} is understood to contain $q$ and all the $q_k$, and the bounded set $B$ to contain all the vectors $b_{N_k}$ in~\eqref{e.b_prelimit_def}, which are bounded by Lemma~\ref{l.fixed_tilt_uniform_bounds}, as well as the vectors $\bar b$ and $b$ constructed in the proof. The proof follows the order of limits described at the beginning of the section. Part~1 uses the Ghirlanda--Guerra identities and the synchronization theorem of~\cite{pan.multi} to identify the limiting overlap array, and defines the coefficient $b$ by diagonal compensation. Part~2 computes the limits, at fixed cutoff, of the cavity increment and of the Gibbs averages of bounded functions of the cavity coordinates, using Proposition~\ref{p.chart_to_gaussian} and Lemmas~\ref{l.finite_overlap_approx_spherical} and~\ref{l.diagonal_compensation}. Part~3 shows that $b$ is admissible, from the boundedness of the increment and Lemma~\ref{l.multispecies_nonadmissible_escape}. Part~4 removes the cutoff and proves assertions~\eqref{i.cav_marginal} and~\eqref{i.cav_normalization}, using Lemmas~\ref{l.remove_cutoffs} and~\ref{l.true_cavity_radius_ui}. Part~5 completes the proof of assertion~\eqref{i.cav_increment}.

\begin{proof}[Proof of Theorem~\ref{t.spherical_cavity_lim}]
Throughout the proof, we write $\vartheta_k^s:=q_k^s(1)+t\partial_s\xi(\lambda)$ for the coefficient in~\eqref{e.prelimit_vartheta_def} associated with $q_k$, and $\vartheta^s:=q^s(1)+t\partial_s\xi(\lambda)$; by~\eqref{e.cavity_theorem_path_convergence}, $\vartheta_k\to\vartheta$.

\smallskip
\noindent\emph{Part 1: Convergence of the overlap array and choice of $b$.}\par
The arrays $(R_{N_k}^{\ell,\ell'},R_\alpha^{\ell,\ell'})_{\ell,\ell'\ge1}$ are uniformly bounded. After passing to a subsequence, they converge in finite-dimensional distributions under $\E\la\cdot\ra_{N_k}^{\mathord{\sim},x_k,y_k,q_k}$ to an array $(R_\infty^{\ell,\ell'},R_{\alpha,\infty}^{\ell,\ell'})_{\ell,\ell'\ge1}$; we denote by $\mathbf{E}$ the expectation under the law of the limiting array. The diagonal is deterministic, since $R_{N_k}^{\ell,\ell}=\lambda$ by~\eqref{e.lambda_N_equals_lambda} and $R_\alpha^{\ell,\ell}=1$.

We next identify the off-diagonal entries. Since every term in~\eqref{e.Delta_MN_def} is nonnegative, the first limit in~\eqref{e.cavity_assumptions_delta} gives $\lim_{k\to\infty}\Delta_{N_k}^{x_k,y_k,q_k}(\bff_j,n_j,\msf h_j)=0$ for every $j\in\N$. The linear span of the constants and of the monomials $\bff_j$ is dense among the continuous functions of finitely many entries of the bounded overlap array, and, by polarization, the linear span of the constants and of the functions $(v,u)\mapsto\msf C_{\msf h}(v,u)$, $\msf h\in\N^4$, is dense among the continuous functions on the overlap domain. Passing to the limit in~\eqref{e.Delta_r_def}, we obtain
\begin{align*}
    \mathbf{E}\Ll[ \bff(\cdots)\varphi\Ll(R_\infty^{1,n+1},R_{\alpha,\infty}^{1,n+1}\Rr)\Rr]
    =\frac1n\mathbf{E} \Ll[\bff(\cdots)\Rr]\mathbf{E}\Ll[\varphi\Ll(R_\infty^{1,2},R_{\alpha,\infty}^{1,2}\Rr)\Rr]+\frac1n\sum_{\ell=2}^n\mathbf{E} \Ll[\bff(\cdots)\varphi\Ll(R_\infty^{1,\ell},R_{\alpha,\infty}^{1,\ell}\Rr)\Rr]
\end{align*}
for every $n\ge2$, every bounded continuous function $\bff$ of $((R^{\ell,\ell'}_\infty, R^{\ell,\ell'}_{\alpha,\infty}))_{\ell,\ell'\leq n}$, abbreviated as $\bff(\cdots)$, and every bounded continuous function $\varphi$. In other words, the limiting array satisfies the Ghirlanda--Guerra identities. To identify it, set
\begin{equation*}
    T^{\ell,\ell'}:=\sum_{s\in\sS}R_{\infty,s}^{\ell,\ell'}+R_{\alpha,\infty}^{\ell,\ell'},
\end{equation*}
where $R_{\infty,s}^{\ell,\ell'}$ is the $s$-component of $R_\infty^{\ell,\ell'}$. The entries of $R_\infty$ are nonnegative, by Talagrand's positivity principle applied to each species array, see~\cite[Theorem~2.16]{pan} and~\cite[Section~4]{pan.multi}, so that $T$ takes values in $[0,2]$. The synchronization theorem~\cite[Theorem~4]{pan.multi} gives a Lipschitz, coordinatewise nondecreasing function $\Psi=(\Psi_{\sS},\Psi_\alpha):[0,2]\to\R_+^{\sS}\times[0,1]$ such that
\begin{equation*}
    \Ll(R_\infty^{\ell,\ell'},R_{\alpha,\infty}^{\ell,\ell'}\Rr)=\Psi(T^{\ell,\ell'})\qquad\text{almost surely, for every }\ell,\ell'.
\end{equation*}
The scalar array $(T^{\ell,\ell'})_{\ell,\ell'\ge1}$ has diagonal $2$, because $\sum_s\lambda_s=1$. We identify the limiting array by inverting $\Psi_\alpha$. For $r\in[0,1)$, set
\begin{equation*}
    \bkappa(r):=\inf\Ll\{v\in[0,2]:\Psi_\alpha(v)>r\Rr\}.
\end{equation*}
Fix $\ell\ne\ell'$, and let $U:=R_{\alpha,\infty}^{\ell,\ell'}=\Psi_\alpha(T^{\ell,\ell'})$. By the invariance~\eqref{e.invariance_cascade}, the overlap $\alpha^\ell\wedge\alpha^{\ell'}$ is uniformly distributed on $[0,1]$ under $\E\la\cdot\ra^{\mathord{\sim},x_k,y_k,q_k}_{N_k}$ for every $k$, hence so is $U$. In particular, for every $r\in[0,1)$, the event $\{\Psi_\alpha(T^{\ell,\ell'})>r\}$ has probability $1-r>0$, so that the set in the definition of $\bkappa(r)$ is nonempty, and $\bkappa$ is a nondecreasing function from $[0,1)$ to $[0,2]$. It is right-continuous, since the sets $\{v\in[0,2]:\Psi_\alpha(v)>r'\}$ increase to $\{v\in[0,2]:\Psi_\alpha(v)>r\}$ as $r'\downarrow r$. Since $\Psi_\alpha$ is nondecreasing, each level set $\Psi_\alpha^{-1}(\{r\})$ is an interval, and the level sets containing more than one point have pairwise disjoint nonempty interiors; hence the set $E$ of those $r\in[0,1]$ for which $\Psi_\alpha^{-1}(\{r\})$ contains more than one point is countable. Let $r\in[0,1)\setminus E$ and $v\in[0,2]$ be such that $\Psi_\alpha(v)=r$. For $v'<v$, we have $\Psi_\alpha(v')\le r$, while for $v'>v$, we have $\Psi_\alpha(v')\ge r$ and $\Psi_\alpha(v')\ne r$, since the level set of $r$ is reduced to $\{v\}$. The set in the definition of $\bkappa(r)$ is therefore $(v,2]$, which is nonempty, so that $v<2$ and $\bkappa(r)=v$. Since $U$ is uniformly distributed and $E\cup\{1\}$ is countable, we have $U\notin E\cup\{1\}$ almost surely, and applying the previous observation with $v=T^{\ell,\ell'}$ and $r=U$ gives
\begin{equation*}
    T^{\ell,\ell'}=\bkappa\Ll(R_{\alpha,\infty}^{\ell,\ell'}\Rr)\qquad\text{almost surely, for every }\ell\ne\ell'.
\end{equation*}
Set $p:=\Psi_{\sS}\circ\bkappa$ on $[0,1)$. Since $\Psi_{\sS}$ is continuous and coordinatewise nondecreasing, with nonnegative values, and since $\bkappa$ is nondecreasing and right-continuous, the path $p$ is nondecreasing, right-continuous, bounded, and nonnegative, so that $p\in\mcl Q_\infty^{\sS}$, with the endpoint convention~\eqref{e.continuous_endpoint_convention}. Applying $\Psi_{\sS}$ to the previous display, we obtain
\begin{equation}
    R_\infty^{\ell,\ell'}=p(R_{\alpha,\infty}^{\ell,\ell'}),
    \qquad \ell\ne\ell'.
    \label{e.cavity_limit_synchronization}
\end{equation}
For every $s\in\sS$, the Cauchy--Schwarz inequality gives $R_{N_k,s}^{\ell,\ell'}\le\lambda_{s}$, so that $p^s\le\lambda_s$ almost everywhere, and then everywhere on $[0,1]$ by right-continuity and the endpoint convention. Hence $p\in\mcl Q_{\infty,\le\lambda}^{\sS}$. Finally, the invariance~\eqref{e.invariance_cascade} identifies the joint law of the cascade overlaps, and not only the law of one overlap: for every $n\in\N$ and every $k$, the array $R_\alpha^{\le n}$ has the same law under $\E\la\cdot\ra^{\mathord{\sim},x_k,y_k,q_k}_{N_k}$ as under $\E\la\cdot\ra_\fR$, hence so does the limiting array $(R_{\alpha,\infty}^{\ell,\ell'})_{\ell,\ell'\le n}$. By~\eqref{e.cavity_limit_synchronization} and the deterministic diagonal, the limiting array $(R_\infty^{\ell,\ell'},R_{\alpha,\infty}^{\ell,\ell'})_{\ell,\ell'\le n}$ is almost surely the image of $(R_{\alpha,\infty}^{\ell,\ell'})_{\ell,\ell'\le n}$ under the map which applies $p$ entrywise off the diagonal and puts $(\lambda,1)$ on the diagonal, so its law is that of the array in~\eqref{e.overlap_limit_statement}. This proves part~\eqref{i.cav_overlaps}.

We now prepare the passage to the limit in the Gaussian model, which requires two additional ingredients. The first is a modification of $p$ near $r=1$, which allows us to represent the diagonal value $\lambda$ of the limiting array by a cascade path. For $m\ge2$, define
\begin{equation*}
    p_m(r):=
    \begin{cases}
        p(r),&0\le r<1-m^{-1},\\
        \lambda,&1-m^{-1}\le r\le1.
    \end{cases}
\end{equation*}
Since $p\le\lambda$, we have $p_m\in\mcl Q_{\infty,\leq\lambda}^{\sS}$, $p_m(1)=\lambda$, and $|p_m-p|_{L^1}\le m^{-1}|\lambda-p|_{L^\infty}\to0$. Since the off-diagonal overlaps $\alpha^\ell\wedge\alpha^{\ell'}$ are uniformly distributed on $[0,1]$ under $\E\la\cdot\ra_\fR$ while the diagonal ones equal $1$, we have, for every $n\in\N$,
\begin{equation}
    \Ll(p_m(R_\alpha^{\ell,\ell'}),R_\alpha^{\ell,\ell'}\Rr)_{\ell,\ell'\le n}
    \xrightarrow[m\to\infty]{\mathrm{law}}
    \Ll(\Ll(p(R_\alpha^{\ell,\ell'}),R_\alpha^{\ell,\ell'}\Rr)\one_{\{\ell\ne\ell'\}}+\Ll(\lambda,1\Rr)\one_{\{\ell=\ell'\}}\Rr)_{\ell,\ell'\le n}
    \qquad\text{under }\E\la\cdot\ra_\fR.
    \label{e.cavity_terminal_array_approximation}
\end{equation}
The second ingredient is the inclusion of the paths $q_k$ in the convergence of part~\eqref{i.cav_overlaps}. For $\ell\ne\ell'$, the overlap $R_\alpha^{\ell,\ell'}$ is uniformly distributed on $[0,1]$ both under $\E\la\cdot\ra_\fR$ and, by the invariance~\eqref{e.invariance_cascade}, under $\E\la\cdot\ra_{N_k}^{\mathord{\sim},x_k,y_k,q_k}$. The uniform bound on the $q_k$ and the almost everywhere convergence in~\eqref{e.cavity_theorem_path_convergence} therefore give $\lim_k\E\la|q_k(R_\alpha^{\ell,\ell'})-q(R_\alpha^{\ell,\ell'})|\ra=0$ under both measures, while for $\ell=\ell'$ we use $q_k(1)\to q(1)$. Since $q$ has at most countably many discontinuities and the off-diagonal overlaps are uniform, the continuous mapping theorem applied to the convergence of part~\eqref{i.cav_overlaps} shows that, with
\begin{equation*}
    R_p^{\ell,\ell'}:=p(R_\alpha^{\ell,\ell'})\one_{\{\ell\ne\ell'\}}+\lambda\one_{\{\ell=\ell'\}},
\end{equation*}
we have, for every $n\in\N$,
\begin{equation}
    \Ll(R_{N_k}^{\le n},R_\alpha^{\le n},q_k(R_\alpha^{\le n})\Rr)\ \text{under }\E\la\cdot\ra_{N_k}^{\mathord{\sim},x_k,y_k,q_k}
    \xrightarrow[k\to\infty]{\mathrm{law}}
    \Ll(R_p^{\le n},R_\alpha^{\le n},q(R_\alpha^{\le n})\Rr)\ \text{under }\E\la\cdot\ra_\fR.
    \label{e.cavity_joint_array_convergence}
\end{equation}

\smallskip
\noindent\emph{Choice of $b$.}\par
By Lemma~\ref{l.fixed_tilt_uniform_bounds}, the vectors $e_{N_k}$ in~\eqref{e.e^N_s=}, computed with the parameters $(x_k,y_k,q_k)$, are bounded. Passing to a further subsequence, we may assume that $e_{N_k}\to e\in\R^{\sS}$. By~\eqref{e.b_prelimit_def}, the vectors $b_{N_k}=\vecone+e_{N_k}+2\vartheta_k$ then converge:
\begin{equation}
    \lim_{k\to\infty}b_{N_k}=\bar b:=\vecone+e+2\vartheta.
    \label{e.cavity_prelimit_b_convergence}
\end{equation}
The vector $\bar b$ is the natural limit of the quadratic coefficients in the Gaussian model, whose cavity fields have the deterministic variance $2\vartheta^s$. In the canonical block with path $\pi=q+t\nabla\xi(p)$, however, the cascade fields have diagonal variance $\pi^s(1)=q^s(1)+t\partial_s\xi(p(1))$, which may be smaller than $\vartheta^s$ since $p(1)\le\lambda$. We compensate for this difference in the quadratic coefficient, as explained in Subsection~\ref{s.diagonal_compensation}, and set
\begin{equation}
    \pi:=q+t\nabla\xi(p),
    \qquad
    \pi_m:=q+t\nabla\xi(p_m),
    \qquad
    b:=b(\pi,e)=\Ll(1+e^s+2\pi^s(1)\Rr)_{s\in\sS},
    \label{e.cavity_b_def}
\end{equation}
with the notation~\eqref{e.b_pi_e_def}. By the nonnegativity of the coefficients in~\eqref{e.xi_power_series_def}, $\pi,\pi_m\in\mcl Q_\infty^{\sS}$, and by the Lipschitz continuity of $\nabla\xi$ on the overlap domain, $|\pi_m-\pi|_{L^1}\to0$ as $m\to\infty$. Moreover, $\pi_m(1)=q(1)+t\nabla\xi(\lambda)=\vartheta$, so that $\bar b=b(\pi_m,e)$ for every $m$. Lemma~\ref{l.diagonal_compensation} therefore gives, for every $L\ge1$ and every bounded continuous function $f$ of finitely many replicas of $(\tau,\alpha)$,
\begin{gather}
    \lim_{m\to\infty}\mcl D_{M,L}(\pi_m,\bar b)=\mcl D_{M,L}(\pi,b),
    \qquad
    \lim_{m\to\infty}\E\la f\ra_{M,L,\bar b}^{\pi_m,\fR}=\E\la f\ra_{M,L,b}^{\pi,\fR}.
    \label{e.cavity_compensated_path_continuity}
\end{gather}
This is the mechanism through which the diagonal value $\lambda$ of the limiting overlap array, which differs from $p(1)$, is transferred to the quadratic coefficient of the block. We also record that, by~\eqref{e.cavity_b_def},
\begin{equation}
    \frac12\sum_{s\in\sS}M_se^s=\frac12\sum_{s\in\sS}M_s(b^s-1)-\sum_{s\in\sS}M_s\pi^s(1),
    \label{e.cavity_diagonal_compensation}
\end{equation}
which is the additive constant in the definition~\eqref{e.direct_canonical} of $\mcl D_{M,L}(\pi,b)$.

\smallskip
\noindent\emph{Part 2: The limits at fixed cutoff.}\par
Fix $L\in\N$. We first consider the partition functions. By~\eqref{e.partition_chart_to_gaussian} and the second limit in~\eqref{e.cavity_assumptions_delta}, and after subtracting $\E\log Z^{\mathord{\sim},x_k,y_k,q_k}_{N_k}$ from both partition functions,
\begin{align}
    \lim_{k\to\infty}\Bigg(&\E\log\frac{Z^{x_k,y_k,q_k,\mathrm{ch}}_{N_k+M,L,\widetilde V}}{Z_{N_k}^{\mathord{\sim},x_k,y_k,q_k}}
    -\E\log\frac{Z^{x_k,y_k,q_k,\bgamma}_{N_k,L,\mathsf X,b_{N_k}}}{Z_{N_k}^{\mathord{\sim},x_k,y_k,q_k}}
    -\frac12\sum_{s\in\sS}M_se_{N_k}^s
    \Bigg)=0.
    \label{e.cavity_chart_partition_reduction}
\end{align}
We compute the limit of the second term. Fix $\eps>0$, and let $n$ and $f_\eps$ be given by part~\eqref{i.finite_approx_increment} of Lemma~\ref{l.finite_overlap_approx_spherical}. For all sufficiently large $k$ and $m$, we have
\begin{align*}
    \E\log\frac{Z^{x_k,y_k,q_k,\bgamma}_{N_k,L,\mathsf X,b_{N_k}}}{Z_{N_k}^{\mathord{\sim},x_k,y_k,q_k}}
    &\stackrel{\text{L.\ref{l.finite_overlap_approx_spherical}\eqref{i.finite_approx_increment}}}{\asymp_\eps}\E\la f_\eps\Ll(R_{N_k}^{\le n},q_k(R_\alpha^{\le n}),b_{N_k}\Rr)\ra_{N_k}^{\mathord{\sim},x_k,y_k,q_k}\\
    &\stackrel{\eqref{e.cavity_joint_array_convergence},\eqref{e.cavity_prelimit_b_convergence}}{\asymp_\eps}\E\la f_\eps\Ll(R_p^{\le n},q(R_\alpha^{\le n}),\bar b\Rr)\ra_{\fR}\\
    &\stackrel{\eqref{e.cavity_terminal_array_approximation}}{\asymp_\eps}\E\la f_\eps\Ll(p_m(R_\alpha^{\le n}),q(R_\alpha^{\le n}),\bar b\Rr)\ra_{\fR}\stackrel{\text{L.\ref{l.finite_overlap_approx_spherical}\eqref{i.finite_approx_increment}}}{\asymp_\eps}\E\log Z^{\pi_m,\fR}_{M,L,\bar b},
\end{align*}
where the second and third approximations use the continuity and boundedness of $f_\eps$ together with the convergences in law. Letting $k\to\infty$, then $m\to\infty$ using the first limit in~\eqref{e.cavity_compensated_path_continuity} and the fact that the additive constant in~\eqref{e.direct_canonical} is the same for $(\pi_m,\bar b)$ and $(\pi,b)$, and finally $\eps\downarrow0$, we obtain
\begin{equation*}
    \lim_{k\to\infty}
    \E\log\frac{Z^{x_k,y_k,q_k,\bgamma}_{N_k,L,\mathsf X,b_{N_k}}}{Z_{N_k}^{\mathord{\sim},x_k,y_k,q_k}}
    =\E\log Z^{\pi,\fR}_{M,L,b}.
\end{equation*}
Moreover, $\lim_ke_{N_k}=e$, so that, by~\eqref{e.cavity_diagonal_compensation} and the definition~\eqref{e.direct_canonical}, we conclude from~\eqref{e.cavity_chart_partition_reduction} that
\begin{equation}
    \lim_{k\to\infty}
    \E\log\frac{Z^{x_k,y_k,q_k,\mathrm{ch}}_{N_k+M,L,\widetilde V}}{Z_{N_k}^{\mathord{\sim},x_k,y_k,q_k}}
    =\mcl D_{M,L}(\pi,b).
    \label{e.cavity_direct_term_limit}
\end{equation}

We next identify the reaction term. Under the coupling~\eqref{e.coupling_H_tildeH}, and by~\eqref{e.theta_N_def}, the Hamiltonians in~\eqref{e.radially_inserted_N_partition} and~\eqref{e.ZcircMN_def} satisfy $H_N^{x,q}+\widetilde V_N^{x,y,q}=\widetilde H^{x,y,q}_N+\sqrt{2tM}\,\widetilde Y_N-tM\theta_N(\lambda)$, so that
\begin{equation*}
    \E\log\frac{Z^{x,y,q}_{N,\widetilde V}}{Z_N^{\mathord{\sim},x,y,q}}
    =\E\log\la\exp\Ll(\sqrt{2tM}\,\widetilde Y_N(\sigma)-tM\theta_N(\lambda)\Rr)\ra_N^{\mathord{\sim},x,y,q}.
\end{equation*}
Let $\mathsf y_N$ be the field in part~\eqref{i.finite_approx_reaction} of Lemma~\ref{l.finite_overlap_approx_spherical}. Conditionally on the fields entering $\widetilde H^{x,y,q}_N$, the bracket $\la\cdot\ra^{\mathord{\sim},x,y,q}_N$ is a probability measure, and the fields $\widetilde Y_N$ and $\mathsf y_N$ are independent of it, with covariances $\theta_N(R_N(\sigma,\sigma'))$ and $\theta(R_N(\sigma,\sigma'))$. By~\eqref{e.cavity_reaction_kernel_limit} and Lemma~\ref{l.gaussian_comparison_covariance_norm},
\begin{align*}
    \lim_{N\to\infty}\sup_{x,y,q}\Bigg|{}&\E\log\la\exp\Ll(\sqrt{2tM}\,\widetilde Y_N(\sigma)-tM\theta_N(\lambda)\Rr)\ra_N^{\mathord{\sim},x,y,q}\\
    &-\E\log\la\exp\Ll(\sqrt{2tM}\,\mathsf y_N(\sigma)-tM\theta(\lambda)\Rr)\ra_N^{\mathord{\sim},x,y,q}\Bigg|=0.
\end{align*}
It therefore suffices to compute the limit of
\begin{equation*}
    B_k:=\E\log\la\exp\Ll(\sqrt{2tM}\,\mathsf y_{N_k}(\sigma)-tM\theta(\lambda)\Rr)\ra_{N_k}^{\mathord{\sim},x_k,y_k,q_k}.
\end{equation*}
Fix $\eps>0$, and let $n$ and $f_\eps^{\mcl I}$ be given by part~\eqref{i.finite_approx_reaction} of Lemma~\ref{l.finite_overlap_approx_spherical}. For all sufficiently large $k$ and $m$,
\begin{align*}
    B_k&\stackrel{\text{L.\ref{l.finite_overlap_approx_spherical}\eqref{i.finite_approx_reaction}}}{\asymp_\eps}\E\la f_\eps^{\mcl I}(R_{N_k}^{\le n})\ra_{N_k}^{\mathord{\sim},x_k,y_k,q_k}\stackrel{\eqref{e.cavity_joint_array_convergence}}{\asymp_\eps}\E\la f_\eps^{\mcl I}(R_p^{\le n})\ra_{\fR}\\
    &\stackrel{\eqref{e.cavity_terminal_array_approximation}}{\asymp_\eps}\E\la f_\eps^{\mcl I}\Ll(p_m(R_\alpha^{\le n})\Rr)\ra_{\fR}\stackrel{\text{L.\ref{l.finite_overlap_approx_spherical}\eqref{i.finite_approx_reaction}}}{\asymp_\eps}\mcl I_{M,t}(p_m).
\end{align*}
By~\eqref{e.reaction_value}, the continuity of $\theta$, and the $L^1$ convergence of $p_m$ to $p$, we have $\lim_{m\to\infty}\mcl I_{M,t}(p_m)=\mcl I_{M,t}(p)$. Letting $k\to\infty$, then $m\to\infty$, and finally $\eps\downarrow0$, we obtain
\begin{equation}
    \lim_{k\to\infty}
    \E\log\frac{Z^{x_k,y_k,q_k}_{N_k,\widetilde V}}{Z_{N_k}^{\mathord{\sim},x_k,y_k,q_k}}
    =\mcl I_{M,t}(p).
    \label{e.cavity_reaction_term_limit}
\end{equation}

We finally consider the Gibbs averages, still at fixed cutoff $L$. Let $n\in\N$, let $g:(\R^{\sS}\times[0,1])^{n\times n}\to\R$ be bounded and continuous, and write $G:=g((\widetilde R_M(\tau^\ell,\tau^{\ell'}),\alpha^\ell\wedge\alpha^{\ell'})_{\ell,\ell'\le n})$, a bounded continuous function of $n$ replicas. By~\eqref{e.radial_insertion_observable_negligible}, \eqref{e.normalized_chart_to_gaussian}, and the second limit in~\eqref{e.cavity_assumptions_delta},
\begin{equation}
    \lim_{k\to\infty}\Ll(\E\la G\ra_{N_k+M,L}^{x_k,y_k,q_k}
    -\E\la G\ra_{N_k,L,\mathsf X,b_{N_k}}^{x_k,y_k,q_k,\bgamma}\Rr)=0.
    \label{e.cavity_observable_gaussian_reduction}
\end{equation}
Fix $\eps>0$, and let $n'$ and $f_\eps^G$ be given by part~\eqref{i.finite_approx_cavity_array} of Lemma~\ref{l.finite_overlap_approx_spherical}, applied with $m=n$ replicas of the cavity coordinates and $n'$ replicas of the bulk. For all sufficiently large $k$ and $m$,
\begin{align*}
    \E\la G\ra_{N_k,L,\mathsf X,b_{N_k}}^{x_k,y_k,q_k,\bgamma}
    &\stackrel{\text{L.\ref{l.finite_overlap_approx_spherical}\eqref{i.finite_approx_cavity_array}}}{\asymp_\eps}\E\la f_\eps^G\Ll(R_{N_k}^{\le n'},R_\alpha^{\le n'},q_k(R_\alpha^{\le n'}),b_{N_k}\Rr)\ra_{N_k}^{\mathord{\sim},x_k,y_k,q_k}
    \\
    &\stackrel{\eqref{e.cavity_joint_array_convergence},\eqref{e.cavity_prelimit_b_convergence}}{\asymp_\eps}\E\la f_\eps^G\Ll(R_p^{\le n'},R_\alpha^{\le n'},q(R_\alpha^{\le n'}),\bar b\Rr)\ra_{\fR}
    \\
    &\stackrel{\eqref{e.cavity_terminal_array_approximation}}{\asymp_\eps}\E\la f_\eps^G\Ll(p_m(R_\alpha^{\le n'}),R_\alpha^{\le n'},q(R_\alpha^{\le n'}),\bar b\Rr)\ra_{\fR}
    \\
    &\stackrel{\text{L.\ref{l.finite_overlap_approx_spherical}\eqref{i.finite_approx_cavity_array}}}{\asymp_\eps}\E\la G\ra_{M,L,\bar b}^{\pi_m,\fR}.
\end{align*}
Letting $k\to\infty$, then $m\to\infty$ using the second limit in~\eqref{e.cavity_compensated_path_continuity}, and finally $\eps\downarrow0$, and combining with~\eqref{e.cavity_observable_gaussian_reduction}, we get
\begin{equation}
    \lim_{k\to\infty}\E\la G\ra_{N_k+M,L}^{x_k,y_k,q_k}
    =\E\la G\ra_{M,L,b}^{\pi,\fR}.
    \label{e.cavity_observable_fixed_cutoff_limit}
\end{equation}

\smallskip
\noindent\emph{Part 3: Admissibility of $b$.}\par
By~\eqref{e.radial_cutoff_increment_def}, the increment $A_{N,L}^{\mathrm{rad}}(x,y,q)$ is the difference of the two quantities whose limits are identified in~\eqref{e.cavity_direct_term_limit} and~\eqref{e.cavity_reaction_term_limit}. The one-sided comparison~\eqref{e.one_sided_increment} of Proposition~\ref{p.radial_insertion_harmless} therefore yields
\begin{equation}
    \liminf_{k\to\infty}A_{N_k}(x_k,y_k,q_k)
    \ge\mcl D_{M,L}(\pi,b)-\mcl I_{M,t}(p)
    \qquad\text{for every }L\in\N.
    \label{e.cavity_fixed_cutoff_increment_bound}
\end{equation}
By Lemma~\ref{l.fixed_block_increment_bound}, the left-hand side is finite, so that $\sup_{L\ge1}\mcl D_{M,L}(\pi,b)<+\infty$. Lemma~\ref{l.multispecies_nonadmissible_escape} then shows that $b$ is admissible for $\pi$. This is what makes the untruncated block $\la\cdot\ra^{\pi,\fR}_{M,b}$ well defined, by Lemma~\ref{l.remove_cutoffs}.

\smallskip
\noindent\emph{Part 4: Removal of the cutoff.}\par
Since $b$ is admissible for $\pi$, Lemma~\ref{l.remove_cutoffs} allows us to let $L\to\infty$ on both sides of~\eqref{e.cavity_observable_fixed_cutoff_limit}, uniformly in $k$ on the left. This proves~\eqref{e.cavity_marginal_limit} for bounded continuous $g$.

We now remove the boundedness assumption on $g$. By hypothesis, there are $C<\infty$ and $d\in\N$ such that $|g((a^{\ell,\ell'},r^{\ell,\ell'})_{\ell,\ell'\le n})|\le C(1+\sum_{\ell,\ell'\le n}|a^{\ell,\ell'}|^d)$. By the Cauchy--Schwarz inequality, $|\widetilde R_{M,s}(\tau^\ell,\tau^{\ell'})|\le M^{-1}|\tau^{\ell,(s)}|\,|\tau^{\ell',(s)}|$, and the Gibbs average over the replicas of a product of functions of $\tau^\ell$ and $\tau^{\ell'}$ is bounded by the average of the square of either function, by the Cauchy--Schwarz inequality and the exchangeability of the replicas. Together with Jensen's inequality, the moment bounds of Lemma~\ref{l.true_cavity_radius_ui}, applied with an integer $k\ge d(1+\delta)$, therefore imply that, for every $\delta>0$,
\begin{equation}
    \sup_k\E\la|G|^{1+\delta}\ra_{N_k+M}^{x_k,y_k,q_k}<+\infty.
    \label{e.cavity_polynomial_uniform_integrability}
\end{equation}
For $C'>0$, the function $|g|^{1+\delta}\wedge C'$ is bounded and continuous, so the bounded case and~\eqref{e.cavity_polynomial_uniform_integrability} give
\begin{equation*}
    \E\la |G|^{1+\delta}\wedge C'\ra_{M,b}^{\pi,\fR}
    =\lim_{k\to\infty}\E\la |G|^{1+\delta}\wedge C'\ra_{N_k+M}^{x_k,y_k,q_k}
    \le\sup_k\E\la|G|^{1+\delta}\ra_{N_k+M}^{x_k,y_k,q_k},
\end{equation*}
and monotone convergence as $C'\to\infty$ shows that $\E\la |G|^{1+\delta}\ra_{M,b}^{\pi,\fR}$ is finite. Let $T_{C'}(u):=(-C')\vee(u\wedge C')$. The bounded case gives~\eqref{e.cavity_marginal_limit} with $g$ replaced by $T_{C'}\circ g$, and since $|u-T_{C'}(u)|\le|u|\one_{\{|u|>C'\}}\le C'^{-\delta}|u|^{1+\delta}$, we have
\begin{gather*}
    \sup_k\E\la|G-T_{C'}(G)|\ra_{N_k+M}^{x_k,y_k,q_k}
    \le C'^{-\delta}\sup_k\E\la|G|^{1+\delta}\ra_{N_k+M}^{x_k,y_k,q_k},\\
    \E\la|G-T_{C'}(G)|\ra_{M,b}^{\pi,\fR}
    \le C'^{-\delta}\E\la|G|^{1+\delta}\ra_{M,b}^{\pi,\fR}.
\end{gather*}
Both right-hand sides tend to zero as $C'\to\infty$, which proves~\eqref{e.cavity_marginal_limit} for every $g$ in the statement, and completes the proof of assertion~\eqref{i.cav_marginal}.

We next prove the normalization~\eqref{e.cavity_canonical_normalization}. Fix $s\in\sS$, and let $g((a^{11},r^{11})):=Ma_s^{11}$, which is a continuous function of the array of one replica with growth of degree $d=1$, and satisfies $g((\widetilde R_M(\tau,\tau),\alpha\wedge\alpha))=M\widetilde R_{M,s}(\tau,\tau)=|\tau^{(s)}|^2$. Applying~\eqref{e.cavity_marginal_limit} with $n=1$ and this function $g$, and using~\eqref{e.true_block_second_moment}, we get
\begin{equation*}
    \E\la|\tau^{(s)}|^2\ra_{M,b}^{\pi,\fR}
    =\lim_{k\to\infty}\E\la|\tau^{(s)}|^2\ra_{N_k+M}^{x_k,y_k,q_k}
    =M_s,
\end{equation*}
which is~\eqref{e.cavity_canonical_normalization}. The uniform moment bounds of Lemma~\ref{l.true_cavity_radius_ui} enter here through~\eqref{e.cavity_polynomial_uniform_integrability}: the identity~\eqref{e.true_block_second_moment} alone would not allow the second moment to pass to the limit. Lemma~\ref{l.canonical_normalization_identity} now gives $b=b_\pi$, which completes the proof of assertion~\eqref{i.cav_normalization}.

\smallskip
\noindent\emph{Part 5: The cavity increment.}\par
Since $b$ is admissible for $\pi$, Lemma~\ref{l.remove_cutoffs} gives $\lim_{L\to\infty}\mcl D_{M,L}(\pi,b)=\mcl D_M(\pi,b)$. Letting $L\to\infty$ in~\eqref{e.cavity_fixed_cutoff_increment_bound} and recalling~\eqref{e.canonical_ASS_functional}, we obtain
\begin{equation}
    \liminf_{k\to\infty}A_{N_k}(x_k,y_k,q_k)
    \ge\mcl D_M(\pi,b)-\mcl I_{M,t}(p)=\mcl P_{M,t,q}(p,b).
    \label{e.cavity_increment_canonical_bound}
\end{equation}
Finally, since $b=b_\pi$ by Part~4, Lemma~\ref{l.canonical_normalization_identity} gives $\mcl P_{M,t,q}(p,b)=-M\sP_{t,q}(p)$, which proves~\eqref{e.cavity_increment_lower_bound} and completes the proof.
\end{proof}

\section{Upper bound at a critical path}
\label{s.crit_pt_bdd}

We now prove Proposition~\ref{p.crit_pt_bdd_spherical}. We work under the standing assumptions of Section~\ref{s.cavity}: $\lambda=\lambda_\infty\in\Q^{\sS}$, the cavity block $(M_s,\msf M_s)_{s\in\sS}$ satisfies~\eqref{e.lambda_compatible_block} with the integers $M_s=M\lambda_s$ of the proposition, and the system sizes are those of the sequence~\eqref{e.compatible_sequence}, with the identification~\eqref{e.N_compatible_block}, so that the passage from $N_n$ to $N_n+M=N_{n+1}$ adds $M_s$ coordinates of species $s$. Since the free energy depends on the partition $(I_{N,s})_{s\in\sS}$ only through the cardinalities $|I_{N,s}|$, by permutation invariance within each species, the identification~\eqref{e.N_compatible_block} is a labeling convention and entails no loss of generality. We fix $t>0$ throughout the section.

The bound~\eqref{e.spherical_critical_bounds} alone, for a path $p$ describing the limiting overlaps and for $q=0$, is essentially what the cavity argument of~\cite{bates2022free,chen2013aizenman} provides, in the limit of a large number of cavity coordinates; combined with Guerra's bound, it gives the Parisi formula in the convex case. The critical relation~\eqref{e.spherical_critical_relation} requires the finer information on the cavity coordinates provided by Theorem~\ref{t.spherical_cavity_lim}.

The argument follows~\cite[Section~6]{chen2026ising}, see also~\cite[Section~7]{chen2025free}, with the spherical cavity computation of Section~\ref{s.cavity} in place of the Ising cavity computation. Let us describe its structure. By the definition~\eqref{e.ASS_increment_def} of the cavity increment, the free energy at size $N_n$ is, up to a negligible error, the Ces\`aro average $-\frac1{nM}\sum_{j<n}A_{N_j}$ of the cavity increments, so that an upper bound on $\bar F_{N_n}$ follows from a lower bound on $\liminf_jA_{N_j}$. Theorem~\ref{t.spherical_cavity_lim} provides such a lower bound along sequences of perturbation parameters $(x,y)$ for which the Ghirlanda--Guerra identities hold and the normalized radial derivatives self-average. Both properties hold on average over $(x,y)$, by Lemmas~\ref{l.overlap_identities_radial_spherical} and~\ref{l.radial_self_average}, and the averaging argument of~\cite[Lemma~3.3]{pan} allows us to select, for each $n$, parameters $(x_n,y_n)$ for which they hold and for which $A_{N_n}(x_n,y_n)$ is not larger than its average. This gives~\eqref{e.spherical_critical_bounds}, with $p$ the path describing the limiting overlaps.

To prove the critical relation~\eqref{e.spherical_critical_relation}, we also randomize the path $q$, replacing it by $q+q_N(z)$ where $q_N(z)$ is a small random perturbation. The free energies of the bulk system and of the $(N+M)$-spin system are within $O(N^{-1})$ of each other by Lemma~\ref{l.bulk_true_comparison}, and semi-concave in $q$ by Proposition~\ref{p.semi-concave}; a one-dimensional argument then shows that their derivatives in $q$ are close for most values of $z$. Along the selected sequence, the derivative of the bulk free energy converges to $p$ by part~\eqref{i.cav_overlaps} of Theorem~\ref{t.spherical_cavity_lim}, while the derivative of the free energy of the $(N+M)$-spin system converges to $\partial_q\psi(q+t\nabla\xi(p))$ by part~\eqref{i.cav_marginal} of the same theorem and Lemma~\ref{l.derivative_initial}, which applies since part~\eqref{i.cav_normalization} gives $b=b_\pi$. Compared with the Ising case, the new ingredients are the additional radial parameter $y$, which is averaged along with $x$; the endpoint term that the radial perturbation produces in the dependence of the perturbed free energies on the path, because of which we work with directional derivatives, see Lemma~\ref{l.q_dependence}; and the identification of the quadratic coefficient $b$ through the spherical normalization.

Throughout the section, the parameters $(x,y)$ range over the set
\begin{equation*}
    \mfk X:=[1,2]^{\N^4}\times[1,2]^{\sS},
\end{equation*}
and we denote by $\E_{x,y}$ the expectation with respect to the product uniform measure on $\mfk X$. The parameter $z$ ranges over $[0,3]^{\N}$ in the statements and over $[1,2]^\N$ in the averaging arguments; $\E_z$ denotes the expectation with respect to the product uniform measure on $[1,2]^\N$, and $\E_{x,y,z}$ the product of these expectations, with associated probability measure $\P_{x,y,z}$. Recall $\bar F_N^{x,y}$ and $\la\cdot\ra_N^{x,y,q}$ from~\eqref{e.ZNx_def}, and $\widetilde F_N^{x,y}$ and $\la\cdot\ra_N^{\mathord{\sim},x,y,q}$ from~\eqref{e.ZcircMN_def}.

\subsection{Perturbation of the path and derivatives of the free energies}

\subsubsection*{A separating family of paths}
Let $(e_s)_{s\in\sS}$ denote the standard basis of $\R^{\sS}$. Let $\varphi:\R\to[0,\infty)$ be smooth, with $\int_\R\varphi=1$ and $\supp\varphi\subset(0,1)$, and set $\varphi_\eps(v):=\eps^{-1}\varphi(v/\eps)$. Let $(\kappa_i)_{i\ge1}$ enumerate the paths
\begin{equation*}
    \Ll\{(\one_{[a,\infty)}*\varphi_\eps)\,e_s:a\in\Q\cap[0,1),\ s\in\sS,\ \eps\in\Q\cap(0,1]\Rr\}
\end{equation*}
restricted to $[0,1]$, where $(\one_{[a,\infty)}*\varphi_\eps)(u):=\int_\R\one_{[a,\infty)}(u-v)\varphi_\eps(v)\,\d v$. Each $\kappa_i$ is a smooth path in $\mcl Q_\infty^{\sS}$ with values in $[0,1]^{\sS}$, and for every $p,p'\in\mcl Q_2^{\sS}$,
\begin{equation}
    p=p'
    \quad\Longleftrightarrow\quad
    \la\kappa_i,p\ra_{L^2}=\la\kappa_i,p'\ra_{L^2}\quad\text{for every }i\ge1.
    \label{e.separating_family}
\end{equation}
Indeed, as $\eps\downarrow0$, the paths $(\one_{[a,\infty)}*\varphi_\eps)e_s$ converge in $L^2$ to $\one_{[a,1)}e_s$, so that equality of the pairings with every $\kappa_i$ implies $\int_a^1p^s=\int_a^1p'^s$ for every $a\in\Q\cap[0,1)$ and $s\in\sS$, hence $p=p'$ almost everywhere; the converse is immediate. This is the family used in~\cite[Section~6]{chen2026ising}. Set
\begin{equation}
    a_i:=3^{-1}2^{-i}\max\{1,|\dot\kappa_i|_{L^\infty}\}^{-1},
    \qquad i\ge1,
    \label{e.a_i_spherical}
\end{equation}
and, for $z=(z_i)_{i\ge1}\in[0,3]^\N$ and $N\in\N$, define the perturbed path
\begin{equation}
    q_N(z):=\eta_N^{1/2}\sum_{i=1}^\infty a_iz_i\kappa_i,
    \label{e.q_N_y_spherical}
\end{equation}
with $\eta_N=N^{-1/16}$ as in~\eqref{e.eta_N_def}. By~\eqref{e.a_i_spherical} and since $0\le\kappa_i\le1$, the series in~\eqref{e.q_N_y_spherical} and the series of the derivatives $a_iz_i\dot\kappa_i$ converge uniformly on $[0,1]$, so that $q_N(z)$ is a nondecreasing, continuously differentiable path, with
\begin{equation}
    \sup_{z\in[0,3]^\N}\Ll(|q_N(z)|_{L^\infty}+|\dot q_N(z)|_{L^\infty}\Rr)\le2\eta_N^{1/2}.
    \label{e.q_N_small_spherical}
\end{equation}
For $i\ge1$, $z\in[0,3]^\N$, and $r\in[0,3]$, we write $z^{i,r}$ for the element of $[0,3]^\N$ obtained from $z$ by replacing its $i$-th coordinate by $r$.

\subsubsection*{Dependence of the perturbed free energies on the path}
The perturbed Hamiltonians depend on the path $q$ both through the cascade field $W^q$ and through the radial perturbation, which contains radial derivatives of $W^q$. The following lemma records that this dependence is again that of a cascade field, with a rescaled path, up to a deterministic term that involves the endpoint value $q(1)$. Because of this endpoint term, the perturbed free energies are not differentiable in $q$ with respect to the $L^2$ norm. We therefore work with directional derivatives. Given $q\in\mcl Q_\infty^{\sS}$, we say that a direction $\kappa\in L^\infty([0,1);\R^{\sS})$ is admissible at $q$ if $q+\eps\kappa\in\mcl Q_\infty^{\sS}$ for every $\eps$ in a neighborhood of $0$. An admissible direction is a multiple of the difference of two elements of $\mcl Q_\infty^{\sS}$, so that the limit $\kappa(1):=\lim_{u\nearrow1}\kappa(u)$ exists, in agreement with the endpoint convention~\eqref{e.continuous_endpoint_convention}, and $|\kappa^s(1)|\le|\kappa|_{L^\infty}$ for every $s\in\sS$. For a function $F$ of $(t,q)$, we set
\begin{equation}
    D_\kappa F(t,q):=\frac{\d}{\d\eps}F(t,q+\eps\kappa)\Big|_{\eps=0}
    \label{e.directional_derivative_def}
\end{equation}
whenever this derivative exists. When $F$ is Gateaux differentiable with an $L^2$ gradient, as is the case for $\bar F_N$ by Proposition~\ref{p.F_N_smooth}, we simply have $D_\kappa F(t,q)=\la\kappa,\dr_qF(t,q)\ra_{L^2}$. We use the notation~\eqref{e.directional_derivative_def} only for the perturbed free energies.

\begin{lemma}[Dependence on the path]
\label{l.q_dependence}
Let $K\in\{N,N+M\}$, $(x,y)\in\mfk X$, and $q\in\mcl Q_\infty^{\sS}$, and define the rescaled path
\begin{equation}
    T_{K,y}q:=\Ll((1+\eta_Ky^s)^2q^s\Rr)_{s\in\sS}\in\mcl Q_\infty^{\sS}.
    \label{e.T_Ky_def}
\end{equation}
Then $H^{x,y,q}_K(\rho,\alpha)$ is the sum of three terms: a term that does not depend on $q$; the cascade field $\sqrt2W_K^{T_{K,y}q}(\rho,\alpha)$, together with its self-overlap correction $-K(T_{K,y}q)(1)\cdot\lambda$; and the deterministic term
\begin{equation}
    K\eta_K^2\sum_{s\in\sS}(y^s)^2\lambda_sq^s(1),
    \label{e.endpoint_term}
\end{equation}
which is affine in $q$. The same holds for $\widetilde H^{x,y,q}_N$, with $K=N$. Consequently, for every direction $\kappa$ admissible at $q$, the functions $\eps\mapsto\widetilde F_N^{x,y}(t,q+\eps\kappa)$ and $\eps\mapsto\bar F_K^{x,y}(t,q+\eps\kappa)$ are twice continuously differentiable in a neighborhood of $0$, and the directional derivatives~\eqref{e.directional_derivative_def} are given by
\begin{align}
    D_\kappa\widetilde F_N^{x,y}(t,q)&=\sum_{s\in\sS}(1+\eta_Ny^s)^2\,\E\la\kappa^s(\alpha\wedge\alpha')R_{N,s}(\sigma,\sigma')\ra_N^{\mathord{\sim},x,y,q}-\eta_N^2\sum_{s\in\sS}(y^s)^2\lambda_s\kappa^s(1),
    \label{e.derivative_tildeF_exact}\\
    D_\kappa\bar F_{K}^{x,y}(t,q)&=\sum_{s\in\sS}(1+\eta_Ky^s)^2\,\E\la\kappa^s(\alpha\wedge\alpha')R_{K,s}(\rho,\rho')\ra_{K}^{x,y,q}-\eta_K^2\sum_{s\in\sS}(y^s)^2\lambda_s\kappa^s(1).
    \label{e.derivative_trueF_exact}
\end{align}
In particular, there exists a constant $C<\infty$, independent of $N$, $x$, $y$, $q$, and $\kappa$, such that
\begin{align}
    \Ll|D_\kappa\widetilde F_N^{x,y}(t,q)-\E\la\kappa(\alpha\wedge\alpha')\cdot R_N(\sigma,\sigma')\ra_N^{\mathord{\sim},x,y,q}\Rr|&\le C\eta_N|\kappa|_{L^\infty},
    \label{e.derivative_tildeF}\\
    \Ll|D_\kappa\bar F_{N+M}^{x,y}(t,q)-\E\la\kappa(\alpha\wedge\alpha')\cdot R_{N+M}(\rho,\rho')\ra_{N+M}^{x,y,q}\Rr|&\le C\eta_{N}|\kappa|_{L^\infty},
    \label{e.derivative_trueF}
\end{align}
and such that, for every $q,q'\in\mcl Q_\infty^{\sS}$ and $F\in\{\widetilde F_N^{x,y},\bar F_K^{x,y}\}$,
\begin{equation}
    \Ll|F(t,q)-F(t,q')\Rr|\le C|q-q'|_{L^1}+C\eta_K^2\sum_{s\in\sS}|q^s(1)-q'^s(1)|.
    \label{e.perturbed_path_lipschitz}
\end{equation}
Moreover, the semi-concavity estimate~\eqref{e.semi_concave_q_spherical} of Proposition~\ref{p.semi-concave} holds for $\widetilde F_N^{x,y}(t,\cdot)$ and $\bar F_{N+M}^{x,y}(t,\cdot)$ in place of $\bar F_N(t,\cdot)$, with a constant $C$ independent of $N,x,y$.
\end{lemma}

The estimates~\eqref{e.derivative_tildeF}--\eqref{e.derivative_trueF} say that, up to an error of order $\eta_N$, the directional derivatives of the perturbed free energies are the overlap pairings appearing in~\eqref{e.def.der.FN}; the endpoint term only contributes an error of order $\eta_N^2$. The Lipschitz estimate~\eqref{e.perturbed_path_lipschitz} is the analogue of~\eqref{e.F_N_Lipschitz}, with the endpoint term kept explicit; it is used in the telescoping argument of Lemma~\ref{l.choose_perturbations_spherical}.

\begin{proof}[Proof of Lemma~\ref{l.q_dependence}]
By~\eqref{e.full_radially_evaluated_H}--\eqref{e.full_radial_perturbation} and~\eqref{e.WNq_def}, the $q$-dependent part of $H^{x,y,q}_K(\rho,\alpha)$ is
\begin{equation*}
    \sqrt2W_K^q(\rho,\alpha)-\sum_{s\in\sS}q^s(1)|\rho^{(s)}|^2
    +2\eta_K\sum_{s\in\sS}y^s\partial_{r^s}\Ll(\sqrt2W_K^q(S_r\rho,\alpha)-\sum_{s'\in\sS}q^{s'}(1)|(S_r\rho)^{(s')}|^2\Rr)\Big|_{r=\vecone}.
\end{equation*}
Since $W^q_K(S_r\rho,\alpha)=\sum_{s'}\sqrt{r^{s'}}\sum_{i\in I_{K,s'}}w_i^{q^{s'}}(\alpha)\rho_i$ and $|(S_r\rho)^{(s')}|^2=r^{s'}|\rho^{(s')}|^2$, the radial derivative in $r^s$ at $r=\vecone$ equals $\frac12\sqrt2\sum_{i\in I_{K,s}}w_i^{q^s}(\alpha)\rho_i-q^s(1)|\rho^{(s)}|^2$. Collecting terms and using $|\rho^{(s)}|^2=|I_{K,s}|=K\lambda_s$ on $\Sigma_K$, the $q$-dependent part is
\begin{equation*}
    \sqrt2\sum_{s\in\sS}(1+\eta_Ky^s)\sum_{i\in I_{K,s}}w_i^{q^s}(\alpha)\rho_i-K\sum_{s\in\sS}(1+2\eta_Ky^s)q^s(1)\lambda_s.
\end{equation*}
The first sum is a cascade field with path $T_{K,y}q$, since $(1+\eta_Ky^s)w^{q^s}_i$ has covariance $(1+\eta_Ky^s)^2q^s(\alpha\wedge\alpha')$, and since $(1+2\eta_Ky^s)=(1+\eta_Ky^s)^2-\eta_K^2(y^s)^2$, the second sum is the self-overlap correction for this field plus the term~\eqref{e.endpoint_term}. The argument for $\widetilde H^{x,y,q}_N$ is identical.

We now compute the directional derivatives. Let $\kappa$ be admissible at $q$, and let $\eps_0>0$ be such that $q_\pm:=q\pm\eps_0\kappa\in\mcl Q_\infty^{\sS}$. For $|\eps|<\eps_0$, the path $q+\eps\kappa$ is the convex combination $\frac12(1-\eps/\eps_0)q_-+\frac12(1+\eps/\eps_0)q_+$, and the same is true after applying the linear map $T_{K,y}$. As in the proof of~\cite[Proposition~5.1]{chen2025free}, we realize the cascade fields with paths $T_{K,y}(q+\eps\kappa)$ jointly for all $|\eps|<\eps_0$ as $\sqrt{\tfrac12(1-\eps/\eps_0)}\,w^{T_{K,y}q_-}+\sqrt{\tfrac12(1+\eps/\eps_0)}\,w^{T_{K,y}q_+}$, with $w^{T_{K,y}q_-}$ and $w^{T_{K,y}q_+}$ independent. The Gaussian integration by parts leading to~\eqref{e.def.der.FN} applies without change to a Hamiltonian which is the sum of a cascade field, of its self-overlap correction, and of Gaussian and deterministic terms whose laws do not depend on the path of the cascade field: the derivative in $\eps$ of the free energy, at the path $T_{K,y}(q+\eps\kappa)$ and in the direction $T_{K,y}\kappa$, is $\E\la(T_{K,y}\kappa)(\alpha\wedge\alpha')\cdot R_K(\rho,\rho')\ra_K^{x,y,q+\eps\kappa}$. The deterministic term~\eqref{e.endpoint_term} enters the free energy $-\frac1K\E\log Z$ with a minus sign, and its derivative in $\eps$ is the constant $K\eta_K^2\sum_s(y^s)^2\lambda_s\kappa^s(1)$. Since $(T_{K,y}\kappa)^s=(1+\eta_Ky^s)^2\kappa^s$, we obtain~\eqref{e.derivative_trueF_exact} at $\eps=0$; the same argument gives~\eqref{e.derivative_tildeF_exact}. Moreover, a second Gaussian integration by parts, as in the proof of~\cite[Proposition~5.1]{chen2025free}, shows that the derivative in $\eps$ computed above is itself continuously differentiable on $(-\eps_0,\eps_0)$, so that the free energies are twice continuously differentiable in $\eps$ there.

For~\eqref{e.derivative_tildeF} and~\eqref{e.derivative_trueF}, we use $0\le y^s\le3$, $\eta_{N+M}\le\eta_N\le1$, and $|R_{K,s}|\le1$: the factors $(1+\eta_Ky^s)^2$ differ from $1$ by at most $15\eta_N$, and the endpoint term is bounded by $9\eta_N^2|\kappa|_{L^\infty}$.

For~\eqref{e.perturbed_path_lipschitz}, let $q,q'\in\mcl Q_\infty^{\sS}$ and $\kappa:=q'-q$. The segment $\eps\mapsto q+\eps\kappa$, $\eps\in[0,1]$, stays in $\mcl Q_\infty^{\sS}$, and, realizing the cascade fields along it as in the proof of~\cite[Proposition~5.1]{chen2025free}, the free energy is a continuous function of $\eps\in[0,1]$, continuously differentiable on $(0,1)$, with derivative given by the right-hand side of~\eqref{e.derivative_trueF_exact} (or~\eqref{e.derivative_tildeF_exact}) evaluated at the path $q+\eps\kappa$. By the invariance~\eqref{e.invariance_cascade}, the overlap $\alpha\wedge\alpha'$ is uniformly distributed on $[0,1]$ under $\E\la\cdot\ra_K^{x,y,q+\eps\kappa}$, so that, using $|R_{K,s}|\le\lambda_s$ and $\sum_s\lambda_s|a_s|\le|a|$ for $a\in\R^{\sS}$, the first term of~\eqref{e.derivative_trueF_exact} is bounded by $16\sum_s\lambda_s\E\la|\kappa^s(\alpha\wedge\alpha')|\ra_K^{x,y,q+\eps\kappa}\le16|\kappa|_{L^1}$, while the endpoint term is bounded by $9\eta_K^2\sum_s|\kappa^s(1)|$. Integrating over $\eps\in[0,1]$ gives~\eqref{e.perturbed_path_lipschitz}.

Finally, the last assertion of Proposition~\ref{p.semi-concave} allows for the addition of Gaussian fields independent of the cascade field and of affine deterministic terms, such as~\eqref{e.endpoint_term}; applying it with the path $T_{K,y}q$, whose derivative differs from that of $q$ by bounded factors, gives the semi-concavity estimate.
\end{proof}

\subsection{Matching the derivatives}

The starting point is that the free energies of the bulk and of the $(N+M)$-spin systems are close: by Lemma~\ref{l.bulk_true_comparison}, for every bounded subset $Q\subset\mcl Q_\infty^{\sS}$,
\begin{equation}
    \varepsilon_N(Q)=\sup_{x,y,q}\Ll|\widetilde F_N^{x,y}(t,q)-\bar F_{N+M}^{x,y}(t,q)\Rr|\le\frac{C_Q}{N},
    \qquad\text{so that}\qquad
    \lim_{N\to\infty}\eta_N^{-1/2}\varepsilon_N(Q)=0.
    \label{e.eps_N_small}
\end{equation}

\begin{lemma}[One-dimensional bounds]
\label{l.radial_semiconcavity_directional_bounds}
Fix $c\in(0,1]$. There exists $C_c<\infty$ such that the following holds for every $i\ge1$, $N\in\N$, $(x,y)\in\mfk X$, $q\in\mcl Q_{\uparrow,c}^{\sS}\cap\mcl Q_\infty^{\sS}$, and $z\in[0,3]^\N$. Define, for $r\in[0,3]$,
\begin{equation*}
    \msf f(r):=\widetilde F_N^{x,y}\Ll(t,q+q_N(z^{i,r})\Rr),
    \qquad
    \msf g(r):=\bar F_{N+M}^{x,y}\Ll(t,q+q_N(z^{i,r})\Rr).
\end{equation*}
Then $\msf f$ and $\msf g$ are twice continuously differentiable on $[0,3]$, with
\begin{equation}
    \msf f'(r)=\eta_N^{1/2}a_i\,D_{\kappa_i}\widetilde F_N^{x,y}\Ll(t,q+q_N(z^{i,r})\Rr)
    \qquad\text{and}\qquad
    \msf g'(r)=\eta_N^{1/2}a_i\,D_{\kappa_i}\bar F_{N+M}^{x,y}\Ll(t,q+q_N(z^{i,r})\Rr),
    \label{e.one_dim_chain_rule}
\end{equation}
and
\begin{equation}
    |\msf f'|+|\msf g'|\le C_c\,\eta_N^{1/2},
    \qquad
    \msf f''\le C_c\,\eta_N,
    \qquad
    \msf g''\le C_c\,\eta_N
    \label{e.radial_directional_semiconcavity_bounds}
\end{equation}
on $[0,3]$.
\end{lemma}

\begin{proof}
By~\eqref{e.q_N_y_spherical}, the map $r\mapsto q+q_N(z^{i,r})$ is affine, with derivative $\eta_N^{1/2}a_i\kappa_i$, and it takes values in $\mcl Q_{\uparrow,c}^{\sS}$ because $\kappa_i$ is nonnegative and nondecreasing. Since $q$ has slope at least $c>0$ while $\kappa_i$ vanishes at the origin and is nondecreasing and Lipschitz, the direction $\kappa_i$ is admissible at $q+q_N(z^{i,r})$ for every $r\in[0,3]$, in the sense of Lemma~\ref{l.q_dependence}. For $r\in[0,3]$, the function $\eps\mapsto\widetilde F_N^{x,y}(t,q+q_N(z^{i,r})+\eps\eta_N^{1/2}a_i\kappa_i)$ coincides with $\msf f(r+\eps)$ when $r+\eps\in[0,3]$, and it is twice continuously differentiable in a neighborhood of $0$ by Lemma~\ref{l.q_dependence}. Hence $\msf f$ is twice continuously differentiable on $[0,3]$, and~\eqref{e.one_dim_chain_rule} follows from the definition~\eqref{e.directional_derivative_def} of the directional derivative; the same applies to $\msf g$. By~\eqref{e.derivative_tildeF}, \eqref{e.derivative_trueF}, and $|R_{N,s}|\le1$,
\begin{equation*}
    |\msf f'(r)|+|\msf g'(r)|\le2\eta_N^{1/2}a_i\Ll(1+C\eta_N\Rr)|\kappa_i|_{L^\infty}.
\end{equation*}
The semi-concavity estimate in Lemma~\ref{l.q_dependence}, applied to two points of the segment $r\mapsto q+q_N(z^{i,r})$, says that $r\mapsto\msf f(r)-Cc^{-2}\eta_Na_i^2|\dot\kappa_i|_{L^2}^2\,r^2$ is concave on $[0,3]$; since $\msf f$ is twice differentiable, $\msf f''\le2Cc^{-2}\eta_Na_i^2|\dot\kappa_i|_{L^2}^2$, and the same holds for $\msf g''$. Since $a_i|\kappa_i|_{L^\infty}\le1$ and $a_i|\dot\kappa_i|_{L^2}\le a_i|\dot\kappa_i|_{L^\infty}\le1$ by~\eqref{e.a_i_spherical}, the bounds~\eqref{e.radial_directional_semiconcavity_bounds} follow, with a constant that does not depend on $i$.
\end{proof}

\begin{lemma}[One-dimensional semi-concavity estimate]
\label{l.semi_concavity_one_dim_spherical}
There exists $C>0$ with the following property. Let $\msf f,\msf g:[0,3]\to\R$ be twice continuously differentiable, with $|\msf f'|,|\msf g'|\le a$ and $\msf f'',\msf g''\le b$ on $[0,3]$, for some $a,b>0$. We have
\begin{equation*}
    \int_1^2|\msf f'(r)-\msf g'(r)|^2\,\d r\le C(a+b)\|\msf f-\msf g\|_{L^\infty[0,3]}.
\end{equation*}
\end{lemma}

\begin{proof}
This is~\cite[Lemma~6.2]{chen2026ising}; we recall the argument. Choose a smooth function $\zeta:[0,3]\to[0,1]$ with compact support in $(0,3)$ and such that $\zeta=1$ on $[1,2]$. Integration by parts gives
\begin{align*}
    \int_1^2|\msf f'-\msf g'|^2\le\int_0^3\zeta(\msf f'-\msf g')^2=-\int_0^3\zeta'(\msf f-\msf g)(\msf f'-\msf g')-\int_0^3\zeta(\msf f-\msf g)(\msf f''-\msf g'').
\end{align*}
The first term on the right-hand side is bounded by $6a|\zeta'|_{L^\infty}\|\msf f-\msf g\|_{L^\infty}$. Since $b-\msf f''\ge0$ and $b-\msf g''\ge0$,
\begin{align*}
    \int_0^3|\msf f''-\msf g''|\le\int_0^3\Ll((b-\msf f'')+(b-\msf g'')\Rr)=6b-\msf f'(3)-\msf g'(3)+\msf f'(0)+\msf g'(0)\le6b+4a,
\end{align*}
which bounds the second term.
\end{proof}

Given $q\in\mcl Q_{\infty,\uparrow}^{\sS}$, we set, for $i\ge1$, $N\in\N$, $(x,y)\in\mfk X$, and $z\in[0,3]^\N$,
\begin{equation}\label{e.D_N,i=}
    \mfk D_{N,i}(x,y,z):=\Ll|D_{\kappa_i}\widetilde F_N^{x,y}\Ll(t,q+q_N(z)\Rr)-D_{\kappa_i}\bar F_{N+M}^{x,y}\Ll(t,q+q_N(z)\Rr)\Rr|^2,
\end{equation}
where the directional derivatives are well defined by the observation on admissibility made in the proof of Lemma~\ref{l.radial_semiconcavity_directional_bounds}.

\begin{lemma}[Derivative matching]
\label{l.derivative_matching_spherical}
For every $q\in\mcl Q_{\infty,\uparrow}^{\sS}$ and every $i\ge1$,
\begin{equation}
    \lim_{N\to\infty}\E_{x,y,z}\mfk D_{N,i}(x,y,z)=0.
    \label{e.derivative_matching_average_spherical}
\end{equation}
\end{lemma}

\begin{proof}
Choose $c\in(0,1]$ such that $q\in\mcl Q_{\uparrow,c}^{\sS}$, and set $Q:=\{q+q_{N}(z):N\in\N,\ z\in[0,3]^\N\}$, which is a bounded subset of $\mcl Q_\infty^{\sS}$ by~\eqref{e.q_N_small_spherical}. Fix $i\ge1$, $(x,y)\in\mfk X$, and $z\in[1,2]^\N$, and let $\msf f$ and $\msf g$ be as in Lemma~\ref{l.radial_semiconcavity_directional_bounds}. By that lemma, the assumptions of Lemma~\ref{l.semi_concavity_one_dim_spherical} hold with $a=C_c\eta_N^{1/2}$ and $b=C_c\eta_N$, and $\|\msf f-\msf g\|_{L^\infty[0,3]}\le\varepsilon_N(Q)$ by the definition of $\varepsilon_N(Q)$ in~\eqref{e.eps_N_small}. Lemma~\ref{l.semi_concavity_one_dim_spherical} therefore gives
\begin{equation*}
    \int_1^2|\msf f'(r)-\msf g'(r)|^2\,\d r\le C_c\Ll(\eta_N^{1/2}+\eta_N\Rr)\varepsilon_N(Q).
\end{equation*}
By~\eqref{e.one_dim_chain_rule} and~\eqref{e.D_N,i=}, we have $|\msf f'(r)-\msf g'(r)|^2=\eta_Na_i^2\,\mfk D_{N,i}(x,y,z^{i,r})$, so that the left-hand side equals $\eta_Na_i^2\int_1^2\mfk D_{N,i}(x,y,z^{i,r})\,\d r$. Integrating over the remaining coordinates $(z_j)_{j\ne i}\in[1,2]^{\N\setminus\{i\}}$, we obtain
\begin{align*}
    \E_z\mfk D_{N,i}(x,y,z)\le C_c\,a_i^{-2}\,\eta_N^{-1/2}\Ll(1+\eta_N^{1/2}\Rr)\varepsilon_N(Q),
\end{align*}
uniformly over $(x,y)\in\mfk X$. The right-hand side tends to zero by~\eqref{e.eps_N_small}, which proves~\eqref{e.derivative_matching_average_spherical}.
\end{proof}

\subsection{Ghirlanda--Guerra identities and choice of the parameters}

Recall from~\eqref{e.Delta_MN_def} that
\begin{equation*}
    \Delta_N(x,y,q)=\sum_{j=1}^\infty2^{-j}\Delta_N^{x,y,q}(\bff_j,n_j,\msf h_j),
\end{equation*}
where $\Delta_N^{x,y,q}(\bff,n,\msf h)$, defined in~\eqref{e.Delta_r_def}, measures the failure of the Ghirlanda--Guerra identity for the bulk system $\E\la\cdot\ra_N^{\mathord{\sim},x,y,q}$, tested against a function $\bff$ of the overlaps of $n$ replicas and the kernel $\msf C_{\msf h}$, and where the triples $(\bff_j,n_j,\msf h_j)$ enumerate the monomial test functions, normalized so that~\eqref{e.Delta<1} holds.

\begin{lemma}[Ghirlanda--Guerra identities on average]
\label{l.overlap_identities_radial_spherical}
For every bounded subset $Q\subset\mcl Q_\infty^{\sS}$,
\begin{equation*}
    \lim_{N\to\infty}\sup_{q\in Q}\E_{x,y}\Delta_N(x,y,q)=0.
\end{equation*}
\end{lemma}

\begin{proof}
Fix $\msf h\in\N^4$, and set
\begin{equation*}
    H_N^{\msf h,y}(\sigma,\alpha):=H_N^{\msf h}(\sigma,\alpha)+2\eta_N\sum_{s\in\sS}y^s\dr_{r^s}H_N^{\msf h}(S_r\sigma,\alpha)\big|_{r=\vecone}.
\end{equation*}
By~\eqref{e.perturbed_enriched_H}, \eqref{e.radially_evaluated_H}, and~\eqref{e.radial_generic_perturbation}, the dependence of $\widetilde H_N^{x,y,q}$ on the coordinate $x_{\msf h}$ is
\begin{equation*}
    \widetilde H_N^{x,y,q}(\sigma,\alpha)=\widetilde H_{N,\msf h}^{x,y,q}(\sigma,\alpha)+\eta_Nx_{\msf h}c_{\msf h}H_N^{\msf h,y}(\sigma,\alpha),
\end{equation*}
where $\widetilde H_{N,\msf h}^{x,y,q}$ does not depend on $x_{\msf h}$ and is independent of $H^{\msf h,y}_N$. By Lemma~\ref{l.comput_cov_radial},
\begin{equation*}
    \E H_N^{\msf h,y}(\sigma,\alpha)H_N^{\msf h,y}(\sigma',\alpha')=N(\msf C_{\msf h})_{\eta_N,y}\Ll(R_N(\sigma,\sigma'),\alpha\wedge\alpha'\Rr),
\end{equation*}
and by~\eqref{e.f_eta_N,y=} and~\eqref{e.Lambda_h_def},
\begin{equation*}
    \sup_{y\in[1,2]^{\sS}}\sup_{v\in[-2,2]^{\sS},\ u\in[0,1]}\Ll|(\msf C_{\msf h})_{\eta_N,y}(v,u)-\msf C_{\msf h}(v,u)\Rr|\le C\Lambda_{\msf h}(\eta_N+\eta_N^2).
\end{equation*}
The perturbation $\eta_Nx_{\msf h}c_{\msf h}H_N^{\msf h,y}$ is thus of exactly the form treated in the proof of~\cite[Proposition~6.8]{chen2025free}, with the same coefficient $\eta_N=N^{-1/16}$, except that the covariance kernel $\msf C_{\msf h}$ is replaced by $(\msf C_{\msf h})_{\eta_N,y}$. That proof uses the differentiation of the free energy in $x_{\msf h}$ followed by Gaussian integration by parts, the convexity of $\log Z_N^{\mathord{\sim},x,y,q}$ in $x_{\msf h}$, and the concentration of $\log Z_N^{\mathord{\sim},x,y,q}$ around its expectation. The latter is provided by Lemma~\ref{l.cascade_concentration}: the Hamiltonian $\widetilde H^{x,y,q}_N$ is of the form~\eqref{e.linear_cascade_hamiltonian}, see Remark~\ref{r.linear_form_hamiltonians}, and by~\eqref{e.covariance_perturbed_H}, \eqref{e.second_radial_bound}, and~\eqref{e.radial_energy_covariance_bound}, its variance is bounded by $CN$ uniformly over $q\in Q$, $x$, and $y\in[1,2]^{\sS}$. Hence $\sup_{x,y,q}\E|\log Z_N^{\mathord{\sim},x,y,q}-\E\log Z_N^{\mathord{\sim},x,y,q}|\le CN^{1/2}$, which is the concentration bound on which the argument of~\cite{chen2025free} rests. That argument then applies verbatim, and the estimates it yields are uniform over $(q,y)$. The Gaussian integration by parts produces the Ghirlanda--Guerra identities for the array $(\msf C_{\msf h})_{\eta_N,y}(R^{\ell,\ell'}_N,R^{\ell,\ell'}_\alpha)$, and replacing this kernel by $\msf C_{\msf h}$ changes each term of~\eqref{e.Delta_r_def} by at most $C\Lambda_{\msf h}(\eta_N+\eta_N^2)\|\bff\|_\infty$. Therefore, for every $j\in\N$,
\begin{equation*}
    \lim_{N\to\infty}\sup_{q\in Q}\E_{x,y}\Delta_N^{x,y,q}(\bff_j,n_j,\msf h_j)=0.
\end{equation*}
Since $0\le\Delta_N^{x,y,q}(\bff_j,n_j,\msf h_j)\le1$ by~\eqref{e.Delta<1}, the definition~\eqref{e.Delta_MN_def} and dominated convergence complete the proof.
\end{proof}

We now combine Lemmas~\ref{l.overlap_identities_radial_spherical}, \ref{l.radial_self_average}, and~\ref{l.derivative_matching_spherical} to select the parameters. Given $q\in\mcl Q_{\infty,\uparrow}^{\sS}$, we set, with $\mfk D_{N,i}$ as in~\eqref{e.D_N,i=},
\begin{equation*}
    \mfk D_N(x,y,z):=\sum_{i=1}^\infty2^{-i}\mfk D_{N,i}(x,y,z),
\end{equation*}
and we abbreviate
\begin{gather*}
    A_N(x,y,z):=A_N(x,y,q+q_N(z)),
    \qquad
    \Delta_N(x,y,z):=\Delta_N(x,y,q+q_N(z)),
    \\
    \operatorname{Rad}_{N,L}(x,y,z):=\operatorname{Rad}_{N,L}(x,y,q+q_N(z)),
\end{gather*}
where $A_N$ and $\operatorname{Rad}_{N,L}$ are defined in~\eqref{e.ASS_increment_def} and~\eqref{e.radial_error_def}.

\begin{lemma}[Choice of the parameters]
\label{l.choose_perturbations_spherical}
For every $q\in\mcl Q_{\infty,\uparrow}^{\sS}$, there exist $(x_n,y_n)\in\mfk X$ and $z_n\in[1,2]^\N$ such that, with $q_n:=q+q_{N_n}(z_n)$,
\begin{gather}
    \lim_{n\to\infty}\Delta_{N_n}(x_n,y_n,z_n)=0,
    \label{e.chosen_Delta_zero}\\
    \lim_{n\to\infty}\operatorname{Rad}_{N_n,L}(x_n,y_n,z_n)=0
    \quad\text{for every }L\in\N,
    \label{e.chosen_Rad_zero}\\
    \lim_{n\to\infty}\Ll(D_{\kappa_i}\widetilde F_{N_n}^{x_n,y_n}(t,q_n)-D_{\kappa_i}\bar F_{N_n+M}^{x_n,y_n}(t,q_n)\Rr)=0
    \quad\text{for every }i\ge1,
    \label{e.chosen_derivative_match}\\
    \limsup_{n\to\infty}\bar F_{N_n}(t,q)\le\limsup_{n\to\infty}\Ll(-\frac1M A_{N_n}(x_n,y_n,z_n)\Rr).
    \label{e.upper_selected_A}
\end{gather}
\end{lemma}

\begin{proof}
Set $Q:=\{q+q_N(z):N\in\N,\ z\in[0,3]^\N\}$, a bounded subset of $\mcl Q_\infty^{\sS}$ by~\eqref{e.q_N_small_spherical}. All the constants and the sets in~\eqref{e.xyqbR_uniform} are understood with this $Q$.

\smallskip
\noindent\emph{Step 1: Averaged bounds.}\par
Lemma~\ref{l.overlap_identities_radial_spherical} gives $\lim_{N\to\infty}\E_{x,y,z}\Delta_N(x,y,z)=0$. Lemma~\ref{l.radial_self_average}, applied with the set $Q$, gives $\lim_{N\to\infty}\E_{x,y,z}\operatorname{Rad}_{N,L}(x,y,z)=0$ for every $L\in\N$. Moreover, by Lemma~\ref{l.fixed_tilt_uniform_bounds} and Jensen's inequality, for every $L\in\N$ there is $K_L<\infty$ such that $\sup_{N,x,y,z}\operatorname{Rad}_{N,L}(x,y,z)\le K_L$. Since $0\le\frac{2^{-L}}{1+K_L}\operatorname{Rad}_{N,L}\le2^{-L}$, dominated convergence gives
\begin{equation*}
    \lim_{N\to\infty}\E_{x,y,z}\sum_{L=1}^\infty \frac{2^{-L}}{1+K_L} \operatorname{Rad}_{N,L}(x,y,z)=0.
\end{equation*}
Finally, \eqref{e.derivative_tildeF}, \eqref{e.derivative_trueF}, and $0\le\kappa_i\le1$ give $\mfk D_{N,i}\le C$ uniformly in $i,N,x,y,z$; Lemma~\ref{l.derivative_matching_spherical} and dominated convergence therefore yield $\lim_{N\to\infty}\E_{x,y,z}\mfk D_{N}(x,y,z)=0$.

\smallskip
\noindent\emph{Step 2: Selection of the parameters.}\par
Set
\begin{equation*}
    \mfk E_N(x,y,z):=\Delta_N(x,y,z)+\mfk D_N(x,y,z)+\sum_{L=1}^\infty\frac{2^{-L}}{1+K_L}\operatorname{Rad}_{N,L}(x,y,z),
\end{equation*}
a nonnegative quantity with $\lim_{N\to\infty}\E_{x,y,z}\mfk E_N(x,y,z)=0$ by Step~1. By Lemma~\ref{l.fixed_block_increment_bound}, there is $C_A<\infty$ such that $|A_{N}(x,y,z)|\le C_A$ for every sufficiently large $N$ and every $(x,y,z)\in\mfk X\times[1,2]^\N$. For $\eps\in(0,C_A]$, consider the event
\begin{equation*}
    \Omega_{\eps,N}:=\{(x,y,z):A_{N}(x,y,z)\le\E_{x,y,z}A_{N}(x,y,z)+\eps\}.
\end{equation*}
Since $A_N-\E_{x,y,z}A_N$ is centered, bounded below by $-2C_A$, and larger than $\eps$ on the complement of $\Omega_{\eps,N}$, we have $0\ge\eps\,\P_{x,y,z}(\Omega_{\eps,N}^c)-2C_A\P_{x,y,z}(\Omega_{\eps,N})$, hence
\begin{equation*}
    \P_{x,y,z}(\Omega_{\eps,N})\ge\frac{\eps}{2C_A+\eps}\ge\frac{\eps}{3C_A},
\end{equation*}
which is the averaging argument of~\cite[Lemma~3.3]{pan}. On the other hand, Markov's inequality gives $\P_{x,y,z}(\mfk E_N\le\eps)\ge1-\eps^{-1}\E_{x,y,z}\mfk E_N$. Take $\eps_N:=2\Ll(C_A\Ll(\E_{x,y,z}\mfk E_N+N^{-1}\Rr)\Rr)^{1/2}$, so that $\eps_N\to0$ and $\eps_N^2>3C_A\E_{x,y,z}\mfk E_N$. For every sufficiently large $N$, we have $\eps_N\le C_A$, and the sum of the two probabilities above exceeds one, so that $\Omega_{\eps_N,N}\cap\{\mfk E_N\le\eps_N\}\ne\emptyset$. We choose $(x_n,y_n,z_n)$ in this intersection for $N=N_n$. Then
\begin{equation*}
    \lim_{n\to\infty}\mfk E_{N_n}(x_n,y_n,z_n)=0
    \qquad\text{and}\qquad
    A_{N_n}(x_n,y_n,z_n)\le\E_{x,y,z}A_{N_n}(x,y,z)+\eps_{N_n}.
\end{equation*}
Since all the terms in $\mfk E_N$ are nonnegative with positive weights, the first limit gives~\eqref{e.chosen_Delta_zero}, \eqref{e.chosen_Rad_zero}, and $\lim_n\mfk D_{N_n}(x_n,y_n,z_n)=0$; the latter implies $\lim_n\mfk D_{N_n,i}(x_n,y_n,z_n)=0$ for every $i$, which is~\eqref{e.chosen_derivative_match}.

\smallskip
\noindent\emph{Step 3: Free energy as an average of cavity increments.}\par
By~\eqref{e.ASS_increment_def} and the definition of $A_N(x,y,z)$,
\begin{multline*}
    -N_{j+1}\bar F_{N_{j+1}}^{x,y}(t,q+q_{N_{j+1}}(z))+N_j\bar F_{N_j}^{x,y}(t,q+q_{N_j}(z))
    \\
    =A_{N_j}(x,y,z)-N_{j+1}\Ll(\bar F_{N_{j+1}}^{x,y}(t,q+q_{N_{j+1}}(z))-\bar F_{N_{j+1}}^{x,y}(t,q+q_{N_j}(z))\Rr).
\end{multline*}
By~\eqref{e.q_N_y_spherical}, we have $q_{N_{j+1}}(z)-q_{N_j}(z)=(\eta_{N_{j+1}}^{1/2}-\eta_{N_j}^{1/2})\sum_ia_iz_i\kappa_i$, and the path $\sum_ia_iz_i\kappa_i$ is bounded by $1$ in $L^\infty$ by~\eqref{e.a_i_spherical}. Since $\eta_{N_j}^{1/2}=(jM)^{-1/32}$, both $|q_{N_{j+1}}(z)-q_{N_j}(z)|_{L^1}$ and $\sum_s|q^s_{N_{j+1}}(z)(1)-q^s_{N_j}(z)(1)|$ are bounded by $Cj^{-1-1/32}$. The Lipschitz estimate~\eqref{e.perturbed_path_lipschitz}, which applies to the perturbed free energy and accounts for the endpoint term in Lemma~\ref{l.q_dependence}, therefore bounds the last term by $CN_{j+1}j^{-1-1/32}\le Cj^{-1/32}$. Summing over $1\le j<n$ and dividing by $N_n=nM$, we obtain, uniformly over $(x,y,z)\in\mfk X\times[1,2]^\N$,
\begin{equation}
    \bar F_{N_n}^{x,y}(t,q+q_{N_n}(z))=-\frac1{nM}\sum_{j=1}^{n-1}A_{N_j}(x,y,z)+o(1),
    \label{e.telescoping_identity}
\end{equation}
where we also used that the term $N_1\bar F^{x,y}_{N_1}(t,q+q_{N_1}(z))$ is bounded. Moreover, removing the perturbations at the path $q+q_{N_n}(z)$ costs at most $C\eta_{N_n}$ by~\eqref{e.full_radial_free_energy_cost} and~\eqref{e.GG_free_energy_cost}, and then~\eqref{e.F_N_Lipschitz} and~\eqref{e.q_N_small_spherical} give $|\bar F_{N_n}(t,q+q_{N_n}(z))-\bar F_{N_n}(t,q)|\le C\eta_{N_n}^{1/2}$; hence
\begin{equation*}
    \sup_{(x,y,z)\in\mfk X\times[1,2]^\N}\Ll|\bar F_{N_n}^{x,y}(t,q+q_{N_n}(z))-\bar F_{N_n}(t,q)\Rr|=o(1).
\end{equation*}
Averaging~\eqref{e.telescoping_identity} over $(x,y,z)$ therefore gives
\begin{equation}
    \bar F_{N_n}(t,q)=-\frac1{nM}\sum_{j=1}^{n-1}\E_{x,y,z}A_{N_j}(x,y,z)+o(1).
    \label{e.average_telescoping_identity}
\end{equation}
Since $\limsup_{n\to\infty}\frac1n\sum_{j=1}^{n-1}v_j\le\limsup_{j\to\infty}v_j$ for every bounded sequence $(v_j)$ of real numbers, we deduce from~\eqref{e.average_telescoping_identity} and the choice of $(x_n,y_n,z_n)$ in Step~2 that
\begin{equation*}
    \limsup_{n\to\infty}\bar F_{N_n}(t,q)\le\limsup_{n\to\infty}\Ll(-\frac1M\E_{x,y,z}A_{N_n}(x,y,z)\Rr)\le\limsup_{n\to\infty}\Ll(-\frac1M A_{N_n}(x_n,y_n,z_n)\Rr),
\end{equation*}
which is~\eqref{e.upper_selected_A}.
\end{proof}

\subsection{Proof of Proposition~\ref{p.crit_pt_bdd_spherical}}
\label{s.crit_pt_conclusion}

\begin{proof}
\smallskip
\noindent\emph{Step 1: The upper bound.}\par
Let $(x_n,y_n,z_n,q_n)$ be given by Lemma~\ref{l.choose_perturbations_spherical}. By~\eqref{e.q_N_small_spherical}, the paths $q_n$ are uniformly bounded and converge to $q$ uniformly on $[0,1]$, so that~\eqref{e.cavity_theorem_path_convergence} holds. Pass to a subsequence along which $-\frac1MA_{N_n}(x_n,y_n,z_n)$ converges to its $\limsup$. By~\eqref{e.chosen_Delta_zero} and~\eqref{e.chosen_Rad_zero}, Theorem~\ref{t.spherical_cavity_lim} applies along this subsequence, and provides a further subsequence, a path $p\in\mcl Q_{\infty,\le\lambda}^{\sS}$, and a vector $b$ admissible for $\pi:=q+t\nabla\xi(p)$ such that, by~\eqref{e.cavity_increment_lower_bound},
\begin{equation}
    \limsup_{n\to\infty}\Ll(-\frac1M A_{N_n}(x_n,y_n,z_n)\Rr)\le\sP_{t,q}(p).
    \label{e.critical_A_limit}
\end{equation}
Combining~\eqref{e.upper_selected_A} and~\eqref{e.critical_A_limit} gives~\eqref{e.spherical_critical_bounds}.

\smallskip
\noindent\emph{Step 2: The critical relation.}\par
We keep the subsequence, the path $p$, and the vector $b$ of Step~1, and we show that $p$ satisfies~\eqref{e.spherical_critical_relation}. Fix $i\ge1$. As observed in the proof of Lemma~\ref{l.radial_semiconcavity_directional_bounds}, the direction $\kappa_i$ is admissible at $q_n$, and by~\eqref{e.derivative_tildeF} with $\kappa=\kappa_i$, the directional derivative $D_{\kappa_i}\widetilde F_{N_n}^{x_n,y_n}(t,q_n)$ differs from $\E\la\kappa_i(\alpha\wedge\alpha')\cdot R_{N_n}(\sigma,\sigma')\ra_{N_n}^{\mathord{\sim},x_n,y_n,q_n}$ by at most $C\eta_{N_n}$. By part~\eqref{i.cav_overlaps} of Theorem~\ref{t.spherical_cavity_lim}, the continuity and boundedness of $\kappa_i$, and the fact that $\alpha\wedge\alpha'$ is uniformly distributed on $[0,1]$ under $\E\la\cdot\ra_\fR$, we deduce that
\begin{equation}
    \lim_{n\to\infty}D_{\kappa_i}\widetilde F_{N_n}^{x_n,y_n}(t,q_n)=\E\la\kappa_i(\alpha\wedge\alpha')\cdot p(\alpha\wedge\alpha')\ra_\fR=\la\kappa_i,p\ra_{L^2}.
    \label{e.tilde_derivative_limit}
\end{equation}
For the $(N+M)$-spin system, we first use the exchangeability of the coordinates within each species. By~\eqref{e.mean_perturbed_H} and~\eqref{e.covariance_perturbed_H}, the law of $H^{x,y,q}_{N+M}$ is invariant under permutations of the coordinates in $I_{N+M,s}$; so is the reference measure $P_{N+M}$, and therefore so is $\E\la\cdot\ra_{N+M}^{x,y,q}$. Hence, for $j_0\in\msf M_s$, using $|I_{N_n+M,s}|/(N_n+M)=\lambda_s=M_s/M$ by~\eqref{e.compatible_sequence} and~\eqref{e.lambda_compatible_block},
\begin{align}
    \E\la\kappa_i^s(\alpha\wedge\alpha')R_{N_n+M,s}(\rho,\rho')\ra_{N_n+M}^{x_n,y_n,q_n}
    &=\frac{|I_{N_n+M,s}|}{N_n+M}\E\la\kappa_i^s(\alpha\wedge\alpha')\rho_{N_n+j_0}\rho'_{N_n+j_0}\ra_{N_n+M}^{x_n,y_n,q_n}\notag\\
    &=\frac{M_s}{M}\E\la\kappa_i^s(\alpha\wedge\alpha')\tau_{j_0}\tau'_{j_0}\ra_{N_n+M}^{x_n,y_n,q_n}\notag\\
    &=\E\la\kappa_i^s(\alpha\wedge\alpha')\widetilde R_{M,s}(\tau,\tau')\ra_{N_n+M}^{x_n,y_n,q_n}.
    \label{e.permutation_replace_overlap}
\end{align}
Now apply part~\eqref{i.cav_marginal} of Theorem~\ref{t.spherical_cavity_lim} with two replicas and $g((a^{\ell,\ell'},r^{\ell,\ell'})_{\ell,\ell'\le2}):=\sum_{s\in\sS}\kappa_i^s(r^{1,2})a_s^{1,2}$, which is continuous with linear growth. Together with~\eqref{e.derivative_trueF}, \eqref{e.permutation_replace_overlap}, the identity $b=b_\pi$ from part~\eqref{i.cav_normalization}, and Lemma~\ref{l.derivative_initial} applied with $\kappa=\kappa_i\in\mcl Q_\infty^{\sS}$, this gives
\begin{equation}
    \lim_{n\to\infty}D_{\kappa_i}\bar F_{N_n+M}^{x_n,y_n}(t,q_n)=\sum_{s\in\sS}\E\la\kappa_i^s(\alpha\wedge\alpha')\widetilde R_{M,s}(\tau,\tau')\ra^{\pi,\fR}_{M,b_\pi}=\la\kappa_i,\dr_q\psi(\pi)\ra_{L^2}.
    \label{e.true_derivative_limit}
\end{equation}
Combining~\eqref{e.chosen_derivative_match}, \eqref{e.tilde_derivative_limit}, and~\eqref{e.true_derivative_limit}, we obtain
\begin{equation*}
    \la\kappa_i,p\ra_{L^2}=\la\kappa_i,\dr_q\psi\Ll(q+t\nabla\xi(p)\Rr)\ra_{L^2},
    \qquad i\ge1.
\end{equation*}
Since both $p$ and $\dr_q\psi(q+t\nabla\xi(p))$ belong to $\mcl Q_{\infty,\le\lambda}^{\sS}$, by~\eqref{e.derivative_psi_lambda}, the separating property~\eqref{e.separating_family} yields $p=\dr_q\psi(q+t\nabla\xi(p))$, which is~\eqref{e.spherical_critical_relation}.
\end{proof}

\section{Hamilton--Jacobi equation and one-sided bound}
\label{s.hj}

In this section, we show that the Hamilton--Jacobi equation~\eqref{e.main.hj} has a unique Lipschitz viscosity solution $f$ with initial condition $\psi$, and that this solution is given by the Hopf formula as displayed on the right side of~\eqref{e.main.1} (Theorem~\ref{t.vis_sol}). This relies on the well-posedness theory of~\cite{chen2022hamilton}, and on the convexity and regularity of $\psi$ established in Section~\ref{s.enriched_properties}. We also show that $f$ is a lower bound for the limit free energy (Proposition~\ref{p.HJ_bdd_spherical}). This is an application of the free energy bound of~\cite{mourrat2023free}, in the form restated in~\cite{chen2022hamilton}. The arguments in this section do not require the species proportions to be rational, in contrast with the upper bound of Section~\ref{s.crit_pt_bdd}; the rational approximation of Section~\ref{s.conclusion} is only needed for the latter.

\subsection{The equation on the cone}
\label{s.hj.cone}

Let $\cH:=L^2([0,1);\R^{\sS})$, with inner product $\la\cdot,\cdot\ra_\cH:=\la\cdot,\cdot\ra_{L^2}$, and view $\mcl Q_2^{\sS}$ as a closed convex cone in $\cH$. This cone has empty interior in $\cH$, and this is the main source of difficulty in the definition of viscosity solutions of~\eqref{e.main.hj}. Its dual cone is
\begin{equation}
\label{e.Cstar}
\Ll(\mcl Q_2^{\sS}\Rr)^*:=\{\kappa\in\cH:\la\kappa,q\ra_\cH\ge0\text{ for every }q\in\mcl Q_2^{\sS}\}.
\end{equation}
It has a concrete description, as in~\cite[Lemma~3.4~(2)]{chen2022hamilton} and~\cite[Lemma~3.5]{chen2025free}: for $\kappa\in\cH$,
\begin{equation}
\label{e.dual_cone_characterization}
    \kappa\in\Ll(\mcl Q_2^{\sS}\Rr)^*
    \qquad\Longleftrightarrow\qquad
    \int_r^1\kappa^s(v)\d v\ge0\quad\text{for every }s\in\sS\text{ and }r\in[0,1).
\end{equation}
Indeed, the condition on the right side is obtained by testing $\kappa$ against the paths $\one_{[r,1)}e_s\in\mcl Q_2^{\sS}$, where $(e_s)_{s\in\sS}$ is the canonical basis of $\R^{\sS}$; conversely, every finite-step path in $\mcl Q_2^{\sS}$ is a linear combination with nonnegative coefficients of such paths, and finite-step paths are dense in $\mcl Q_2^{\sS}$. In particular, $\mcl Q_2^{\sS}\subset(\mcl Q_2^{\sS})^*$. Let $g$ be a real-valued function defined on a subset $G$ of $\cH$, respectively of $\R^{\sS}$. We say that $g$ is $\Ll(\mcl Q_2^{\sS}\Rr)^*$-increasing, respectively $\R_+^{\sS}$-increasing, if $g(a)\ge g(b)$ whenever $a,b\in G$ satisfy $a-b\in\Ll(\mcl Q_2^{\sS}\Rr)^*$, respectively $a-b\in\R_+^{\sS}$. As for the free energy in~\cite[Proposition~3.8]{mourrat2023free}, the initial condition has this monotonicity, because its derivative is a path in the cone.

\begin{lemma}[Monotonicity for the dual-cone order]
\label{l.dual_cone_monotone}
For every probability vector $\lambda$ with positive entries, the function $\psi_\lambda$ is $(\mcl Q_2^{\sS})^*$-increasing on $\mcl Q_2^{\sS}$; in particular, so is $\psi$.
\end{lemma}

\begin{proof}
Let $q,q'\in\mcl Q_2^{\sS}$ be such that $q'-q\in(\mcl Q_2^{\sS})^*$. Since $\mcl Q_2^{\sS}$ is convex, the segment $q_u:=q+u(q'-q)$, $u\in[0,1]$, stays in $\mcl Q_2^{\sS}$. By~\eqref{e.psi_quadratic_remainder}, the function $u\mapsto\psi_\lambda(q_u)$ is differentiable on $[0,1]$, with derivative $\la q'-q,\dr_q\psi_\lambda(q_u)\ra_{L^2}$. By~\eqref{e.derivative_psi_lambda}, $\dr_q\psi_\lambda(q_u)$ is a path in $\mcl Q_\infty^{\sS}\subset\mcl Q_2^{\sS}$, so this derivative is nonnegative, by the definition~\eqref{e.Cstar} of the dual cone. Hence $u\mapsto\psi_\lambda(q_u)$ is nondecreasing, and $\psi_\lambda(q')\ge\psi_\lambda(q)$.
\end{proof}

Two issues arise in giving a meaning to~\eqref{e.main.hj} in the viscosity sense. The nonlinearity $p\mapsto\int_0^1\xi(p)$ is naturally defined on the cone, whereas the gradient of a test function at a point of the cone need not belong to the cone; and the function $\xi$ is only locally Lipschitz, whereas the well-posedness theory requires a Lipschitz nonlinearity. The following definitions, taken from~\cite[Section~4]{chen2022hamilton}, address these two issues; the regularization of $\xi$ and a finite-dimensional version of the extension~\eqref{e.def_H_spin_glass} already appear in~\cite{mourrat2023free}.

\begin{definition}[Regularization]
\label{d.regularization}
A function $\bar\xi:\R_+^{\sS}\to\R$ is called a regularization of $\xi$ if the following conditions hold.
\begin{enumerate}
    \item The function $\bar\xi$ agrees with $\xi$ on $\{a\in\R_+^{\sS}:|a|\le1\}$.
    \item The function $\bar\xi$ is Lipschitz and proper, in the sense that $\bar\xi$ is $\R_+^{\sS}$-increasing and, for every $b\in\R_+^{\sS}$, the map $\R_+^{\sS}\ni a\mapsto\bar\xi(a+b)-\bar\xi(a)$ is also $\R_+^{\sS}$-increasing.
    \item If $\xi$ is convex on $\R_+^{\sS}$, then $\bar\xi$ is convex.
\end{enumerate}
\end{definition}

This is~\cite[Definition~4.2]{chen2022hamilton} in the present setting. The last condition plays no role here, since we do not assume $\xi$ to be convex; it is included to match the definition in~\cite{chen2022hamilton}. The function $\xi$ itself is proper on $\R_+^{\sS}$: since the series~\eqref{e.xi_power_series_def} has nonnegative coefficients and converges on all of $\R^{\sS}$, the partial derivatives of $\xi$ are nonnegative and nondecreasing in each coordinate on $\R_+^{\sS}$, so that $\xi$ is $\R_+^{\sS}$-increasing, and so is $a\mapsto\xi(a+b)-\xi(a)=\int_0^1b\cdot\nabla\xi(a+ub)\d u$ for every $b\in\R_+^{\sS}$. The function $\xi$ is locally Lipschitz, but not globally Lipschitz in general, and the existence of a regularization follows from~\cite[Lemma~4.4]{chen2022hamilton}, which builds on~\cite[Proposition~6.8]{mourrat2023free}. Given a regularization $\bar\xi$, we define $\H:\cH\to\R$ by
\begin{equation}
\label{e.def_H_spin_glass}
\H(\kappa):=\inf\left\{\int_0^1\bar\xi(p(r))\d r:
    p\in\mcl Q_2^{\sS}\cap\Ll(\kappa+\Ll(\mcl Q_2^{\sS}\Rr)^*\Rr)\right\},
    \qquad\kappa\in\cH.
\end{equation}
In words, $\H(\kappa)$ is the smallest value of $\int_0^1\bar\xi(p)$ over the paths $p$ in the cone that dominate $\kappa$ for the dual-cone order. By~\cite[Lemma~4.6]{chen2022hamilton}, the function $\H$ is Lipschitz, bounded below, and $\Ll(\mcl Q_2^{\sS}\Rr)^*$-increasing, and
\begin{equation}
\label{e.H_on_cone}
\H(p)=\int_0^1\bar\xi(p(r))\d r,
\qquad p\in\mcl Q_2^{\sS}.
\end{equation}
Thus, replacing $\xi$ by $\bar\xi$ makes the nonlinearity Lipschitz, while the infimum in~\eqref{e.def_H_spin_glass} extends the functional $p\mapsto\int_0^1\bar\xi(p)$ from the cone to all of $\cH$, in a way that preserves its monotonicity for the dual-cone order; this monotonicity is what makes a comparison principle available for the equation on the cone.

A function $\phi:(0,\infty)\times\mcl Q_2^{\sS}\to\R$ is called smooth if the following two conditions hold; this is~\cite[Definition~1.2]{chen2022hamilton}.
\begin{enumerate}
    \item For every $(t,q)\in(0,\infty)\times\mcl Q_2^{\sS}$, there exists a unique element of $\R\times\cH$, denoted by $(\partial_t\phi(t,q),\partial_q\phi(t,q))$, such that, as $(s,y)$ tends to $(t,q)$ in $\R\times\cH$ with $y\in\mcl Q_2^{\sS}$,
    \begin{equation*}
        \phi(s,y)-\phi(t,q)=\partial_t\phi(t,q)(s-t)+\la\partial_q\phi(t,q),y-q\ra_\cH+O\Ll(|s-t|^2+|y-q|_\cH^2\Rr).
    \end{equation*}
    \item The map $(t,q)\mapsto(\partial_t\phi(t,q),\partial_q\phi(t,q))$ is continuous from $(0,\infty)\times\mcl Q_2^{\sS}$ to $\R\times\cH$.
\end{enumerate}
The uniqueness required in the first condition is automatic, since $\mcl Q_2^{\sS}-\mcl Q_2^{\sS}$ is dense in $\cH$.

\begin{definition}[Viscosity solutions]
\label{d.vs}
Let $\bar\xi$ be a regularization of $\xi$, and let $\H$ be defined by~\eqref{e.def_H_spin_glass}. Consider the equation
\begin{equation}
\label{e.hj_H}
\partial_t f-\H(\partial_q f)=0
\qquad\text{on }\R_+\times\mcl Q_2^{\sS}.
\end{equation}
\begin{enumerate}
    \item A continuous function $f:\R_+\times\mcl Q_2^{\sS}\to\R$ is a viscosity subsolution of~\eqref{e.hj_H} if, for every $(t,q)\in(0,\infty)\times\mcl Q_2^{\sS}$ and every smooth function $\phi$ such that $f-\phi$ has a local maximum at $(t,q)$, relative to $(0,\infty)\times\mcl Q_2^{\sS}$, we have
    \begin{equation*}
        \Ll(\partial_t\phi-\H(\partial_q\phi)\Rr)(t,q)\le0.
    \end{equation*}
    \item A continuous function $f:\R_+\times\mcl Q_2^{\sS}\to\R$ is a viscosity supersolution of~\eqref{e.hj_H} if, for every $(t,q)\in(0,\infty)\times\mcl Q_2^{\sS}$ and every smooth function $\phi$ such that $f-\phi$ has a local minimum at $(t,q)$, relative to $(0,\infty)\times\mcl Q_2^{\sS}$, we have
    \begin{equation*}
        \Ll(\partial_t\phi-\H(\partial_q\phi)\Rr)(t,q)\ge0.
    \end{equation*}
    \item A continuous function $f:\R_+\times\mcl Q_2^{\sS}\to\R$ is a viscosity solution of~\eqref{e.hj_H} if it is both a viscosity subsolution and a viscosity supersolution.
\end{enumerate}
Finally, a continuous function $f:\R_+\times\mcl Q_2^{\sS}\to\R$ is called a viscosity solution of the equation in~\eqref{e.main.hj} if it is a viscosity solution of~\eqref{e.hj_H} for some regularization $\bar\xi$ of $\xi$.
\end{definition}

The first part of the definition is~\cite[Definition~1.4]{chen2022hamilton}, and the last part is~\cite[Definition~4.3]{chen2022hamilton}. The results of~\cite{chen2022hamilton} are stated for paths with values in the cone of positive semidefinite matrices, which is the natural setting for vector spin glasses; in the present multi-species setting, the paths are $\R_+^{\sS}$-valued, which corresponds to diagonal matrices, and the arguments of~\cite{chen2022hamilton} apply verbatim to this cone, as in~\cite{chen2024ms,chen2026ising}. In the proof of Proposition~\ref{p.HJ_bdd_spherical}, we nevertheless embed the spherical model into the matrix setting and identify the corresponding solution with $f$ through the Hopf formula, so that the lower bound is a direct application of the results of~\cite{chen2022hamilton}.

\subsection{The solution and its Hopf representation}
\label{s.hj.hopf}

Recall the functional $\mcl J_{t,q}$ from~\eqref{e.mcJ}.

\begin{theorem}[Viscosity solution and Hopf formula]
\label{t.vis_sol}
Let $\psi$ be the spherical initial condition defined in~\eqref{e.multi_species_initial}, restricted to $\mcl Q_2^{\sS}$.
\begin{enumerate}
    \item \label{i.vis_sol_existence} For every regularization $\bar\xi$ of $\xi$, there exists a viscosity solution $f$ of~\eqref{e.hj_H} with $f(0,\cdot)=\psi$, unique in the class of Lipschitz functions on $\R_+\times\mcl Q_2^{\sS}$, and $f$ does not depend on the choice of $\bar\xi$. In particular, $f$ is the unique Lipschitz viscosity solution of~\eqref{e.main.hj}.
    \item \label{i.vis_sol_hopf} For every $(t,q)\in\R_+\times\mcl Q_2^{\sS}$,
    \begin{equation}
    \label{e.Hopf_spin_glass}
        f(t,q)=\sup_{p\in\mcl Q_\infty^{\sS}}\ \inf_{q'\in\mcl Q_\infty^{\sS}}\ \mcl J_{t,q}(q',p)=\sup_{p\in\mcl Q_{\infty,\le\lambda_\infty}^{\sS}}\ \inf_{q'\in\mcl Q_\infty^{\sS}}\ \mcl J_{t,q}(q',p).
    \end{equation}
\end{enumerate}
\end{theorem}

At the level of the variational formula, the restriction of the supremum to $\mcl Q^{\sS}_{\infty,\le\lambda_\infty}$ is a consequence of the bound $\dr_q\psi(\pi)\in\mcl Q^{\sS}_{\infty,\le\lambda_\infty}$ on the derivative of the initial condition.

\begin{proof}[Proof of Theorem~\ref{t.vis_sol}]
We verify the hypotheses of~\cite[Theorem~4.7]{chen2022hamilton}. The function $\xi$ is locally Lipschitz, and it is proper on $\R_+^{\sS}$, as shown after Definition~\ref{d.regularization}. By Lemma~\ref{l.dual_cone_monotone} and Corollary~\ref{c.psi_multi_species}, the function $\psi$ is $(\mcl Q_2^{\sS})^*$-increasing on $\mcl Q_2^{\sS}$ and satisfies
\begin{equation*}
    |\psi(q)-\psi(q')|\le|q-q'|_{L^1},
    \qquad q,q'\in\mcl Q_2^{\sS},
\end{equation*}
and $\psi$ is convex by Lemma~\ref{l.psi_convex}. Part~\eqref{i.vis_sol_existence} is therefore~\cite[Theorem~4.7]{chen2022hamilton}, in the diagonal setting discussed after Definition~\ref{d.vs}, together with~\cite[Remark~4.10]{chen2022hamilton} for the independence of $f$ from the choice of $\bar\xi$; the last sentence of part~\eqref{i.vis_sol_existence} is a rephrasing, since a viscosity solution of the equation in~\eqref{e.main.hj} is by definition a viscosity solution of~\eqref{e.hj_H} for some regularization. Since $\psi$ is convex, \cite[Theorem~4.7]{chen2022hamilton} also gives the Hopf formula, which is the first equality in~\eqref{e.Hopf_spin_glass}: the supremum and the infimum there are over bounded paths in the cone, that is, over $\mcl Q_\infty^{\sS}$.

It remains to show that the supremum can be restricted to $\mcl Q_{\infty,\le\lambda_\infty}^{\sS}$, and for this it suffices to show that, for $p\in\mcl Q_\infty^{\sS}\setminus\mcl Q_{\infty,\le\lambda_\infty}^{\sS}$, the infimum over $q'$ in~\eqref{e.Hopf_spin_glass} equals $-\infty$. By the endpoint convention~\eqref{e.continuous_endpoint_convention}, such a path has a coordinate $s\in\sS$ and a point $r_0\in[0,1)$ with $p^s(r_0)>\lambda_{\infty,s}$, so that, by monotonicity,
\begin{equation*}
    \int_{r_0}^1\Ll(p^s(v)-\lambda_{\infty,s}\Rr)\d v>0.
\end{equation*}
For $a>0$, let $q'_a\in\mcl Q_\infty^{\sS}$ have $s$-component $a\one_{[r_0,1)}$ and all other components equal to zero. Since $\psi$ is convex, with derivative $\dr_q\psi(q'_a)$ at $q'_a$ by~\eqref{e.psi_quadratic_remainder}, we have $\psi(0)\ge\psi(q'_a)-\la q'_a,\dr_q\psi(q'_a)\ra_{L^2}$, and the bound $\dr_q\psi(q'_a)\in\mcl Q^{\sS}_{\infty,\le\lambda_\infty}$ in~\eqref{e.bound_derivative_spherical_psi} gives $\la q'_a,\dr_q\psi(q'_a)\ra_{L^2}\le a\lambda_{\infty,s}(1-r_0)$. Therefore
\begin{equation*}
    \mcl J_{t,q}(q'_a,p)\le\psi(0)-a\int_{r_0}^1\Ll(p^s(v)-\lambda_{\infty,s}\Rr)\d v+\la q,p\ra_{L^2}+t\int_0^1\xi(p(r))\d r
    \xrightarrow[a\to\infty]{}-\infty.
\end{equation*}
Hence such paths $p$ do not contribute to the supremum, and the second equality in~\eqref{e.Hopf_spin_glass} follows.
\end{proof}

Two comments are in order. First, the properties of $\psi$ used in the proof are stated in Lemma~\ref{l.psi_convex}, Corollary~\ref{c.psi_multi_species}, Lemma~\ref{l.dual_cone_monotone}, and~\eqref{e.derivative_psi_lambda} for the initial condition $\psi_\lambda$ in~\eqref{e.psi_lambda_def} of a model with limiting proportions $\lambda$, for every probability vector $\lambda$ with positive entries. The theorem thus applies as well with $(\psi_\lambda,\lambda)$ in place of $(\psi,\lambda_\infty)$, and we denote by $f_\lambda$ the corresponding solution, which is given by the Hopf formula
\begin{equation}
\label{e.f_lambda_def}
    f_\lambda(t,q)=\sup_{p\in\mcl Q^{\sS}_{\infty,\le\lambda}}\ \inf_{q'\in\mcl Q_\infty^{\sS}}\ \mcl J^\lambda_{t,q}(q',p),
    \qquad (t,q)\in\R_+\times\mcl Q_2^{\sS},
\end{equation}
where $\mcl J^\lambda_{t,q}$ is defined as in~\eqref{e.mcJ} with $\psi_\lambda$ in place of $\psi$; the supremum may equivalently be taken over $\mcl Q_\infty^{\sS}$, and $f=f_{\lambda_\infty}$. This is used in Section~\ref{s.conclusion}. Second, since every $p\in\mcl Q^{\sS}_{\infty,\le\lambda}$ satisfies $|p(u)|\le|\lambda|\le1$ for every $u$, the only term of $\mcl J^\lambda_{t,q}(q',p)$ that depends on $q$, namely $\la q,p\ra_{L^2}$, is $1$-Lipschitz in $q$ for the $L^1$ norm, uniformly over $p$ and $q'$. Taking the infimum over $q'$ and then the supremum over $p$, we obtain
\begin{equation}
\label{e.f_L1_lipschitz}
    |f_\lambda(t,q)-f_\lambda(t,\tilde q)|\le|q-\tilde q|_{L^1},
    \qquad t\ge0,\quad q,\tilde q\in\mcl Q_2^{\sS}.
\end{equation}
This is the form of the Lipschitz continuity of $f$ that is used in Section~\ref{s.conclusion}, where the identification of the limit is first obtained for regular bounded paths and then extended by density.

\subsection{The lower bound}
\label{s.hj.lower_bound}

\begin{proposition}[Hamilton--Jacobi lower bound]
\label{p.HJ_bdd_spherical}
Let $f$ be the viscosity solution of Theorem~\ref{t.vis_sol}. Then, for every $(t,q)\in\R_+\times\mcl Q_2^{\sS}$,
\begin{equation}
\label{e.HJ_lower_spherical}
    f(t,q)\le\liminf_{N\to\infty}\bar F_N(t,q).
\end{equation}
\end{proposition}

\begin{proof}
The lower bound is an application of~\cite[Theorem~4.14]{chen2022hamilton}, which is a restatement of~\cite[Theorem~3.4]{mourrat2023free}. This result concerns vector spin glasses in which the reference measure is an arbitrary probability measure supported in the closed ball of radius $\sqrt N$ of $(\R^N)^D$, for some $D\in\N$; in particular, no product structure of the reference measure is required, so that the spherical constraint is allowed. We first embed the spherical model into this setting, then verify the hypotheses of the theorem, and finally identify the solution of the matrix-valued equation appearing in its conclusion with $f$ along diagonal paths.

\smallskip
\noindent\emph{Step 1: Embedding.}\par
Let $D:=|\sS|$, and, as in~\cite[Section~6.2]{mourrat2023free}, associate with $\sigma\in\Sigma_N$ the vector-spin configuration $\widehat\sigma=(\widehat\sigma^s)_{s\in\sS}\in(\R^N)^{\sS}$ defined by
\begin{equation*}
    \widehat\sigma^s_i:=\sigma_i\one_{\{i\in I_{N,s}\}},
    \qquad i\in[N],\quad s\in\sS.
\end{equation*}
Then, for $\sigma,\sigma'\in\Sigma_N$,
\begin{equation*}
    \frac1N\widehat\sigma^s\cdot\widehat\sigma'^{s'}=R_{N,s}(\sigma,\sigma')\one_{\{s=s'\}},
    \qquad\text{and}\qquad
    \sum_{s\in\sS}|\widehat\sigma^s|^2=N.
\end{equation*}
In other words, the overlap matrix $(\frac1N\widehat\sigma^s\cdot\widehat\sigma'^{s'})_{s,s'\in\sS}$ is the diagonal matrix with diagonal $R_N(\sigma,\sigma')$, and the embedded configurations lie on the sphere of radius $\sqrt N$ of $(\R^N)^{\sS}$. We absorb the external field into the reference measure: let $P_N^h$ be the probability measure on $\Sigma_N$ defined by
\begin{equation*}
    \d P_N^h(\sigma):=\exp\Ll(N\bar F_N(0,0)+\sum_{s\in\sS}h^s\sum_{i\in I_{N,s}}\sigma_i\Rr)\d P_N(\sigma),
\end{equation*}
which is indeed a probability measure by~\eqref{e.def.FN.delta0} at $t=0$. The reference measure of the vector-spin model is the image of $P_N^h$ under the embedding.

\smallskip
\noindent\emph{Step 2: Covariance function and free energy.}\par
For a matrix $A\in\R^{\sS\times\sS}$, we set $\Xi(A):=\xi((A_{ss})_{s\in\sS})$, which is well defined since the series~\eqref{e.xi_power_series_def} converges on all of $\R^{\sS}$. The function $\Xi$ is smooth, hence locally Lipschitz, and it is proper on the cone of positive semidefinite matrices in the sense of~\cite[Definition~4.2]{chen2022hamilton}: the diagonal entries of a positive semidefinite matrix are nonnegative, and if $B-A$ is positive semidefinite, then $A_{ss}\le B_{ss}$ for every $s\in\sS$, so that the properness of $\xi$ on $\R_+^{\sS}$, shown after Definition~\ref{d.regularization}, transfers to $\Xi$. The setting of~\cite{chen2022hamilton} requires a centered Gaussian field indexed by all of $(\R^N)^{\sS}$, with covariance $N\Xi$ evaluated at the overlap matrix, that is, $N\xi((\frac1N\widehat\sigma^s\cdot\widehat\sigma'^s)_{s\in\sS})$. Such a field exists: each kernel $(\widehat\sigma,\widehat\sigma')\mapsto\widehat\sigma^s\cdot\widehat\sigma'^s$ is nonnegative definite, hence so is every monomial $\prod_s(\widehat\sigma^s\cdot\widehat\sigma'^s)^{k_s}$ by the Schur product theorem, and the covariance is a convergent sum of such monomials with the nonnegative coefficients of~\eqref{e.xi_power_series_def}. On the embedded configurations, this field has the same law as $H_N$, by~\eqref{e.def_H_N}. Likewise, the external field of~\cite{chen2022hamilton} associated with a bounded matrix-valued path $\boldsymbol q$, nondecreasing for the positive semidefinite order, has covariance $\sum_{s,s'\in\sS}\boldsymbol q_{ss'}(\alpha\wedge\alpha')\,\widehat\sigma^s\cdot\widehat\sigma'^{s'}$; on the embedded configurations, only the diagonal $\diag\boldsymbol q:=(\boldsymbol q_{ss})_{s\in\sS}$ contributes, and the field has the law of $W_N^{\diag\boldsymbol q}$ in~\eqref{e.WNq_covariance}. Note that $\diag\boldsymbol q\in\mcl Q_\infty^{\sS}$, since the diagonal entries of a nondecreasing path of positive semidefinite matrices are nonnegative and nondecreasing. Conversely, for an $\R^{\sS}$-valued path $p$, we write $\diag(p)$ for the path of diagonal matrices with diagonal $p$. Finally, the deterministic centerings of~\cite{chen2022hamilton} are $Nt\,\Xi$ evaluated at the self-overlap matrix, which is $Nt\xi(\lambda_N)$ since $R_N(\sigma,\sigma)=\lambda_N$, and $\sum_{s,s'}\boldsymbol q_{ss'}(1)\widehat\sigma^s\cdot\widehat\sigma^{s'}=N\,\diag\boldsymbol q(1)\cdot\lambda_N$. Comparing with~\eqref{e.enriched_H}, we conclude that the free energy of~\cite{chen2022hamilton} for the embedded model, with reference measure $P_N^h$, at time $t$ and path $\boldsymbol q$, is
\begin{equation}
\label{e.matrix_free_energy_diagonal}
    \bar F_N(t,\diag\boldsymbol q)-\bar F_N(0,0).
\end{equation}
This identity extends from bounded to square-integrable matrix paths, since both sides are continuous for the $L^1$ norm; see~\eqref{e.F_N_Lipschitz} for the right-hand side.

\smallskip
\noindent\emph{Step 3: Initial condition and matrix-valued solution.}\par
By~\eqref{e.matrix_free_energy_diagonal} and Proposition~\ref{l.initial_condition}, the initial condition of the embedded model converges: for every square-integrable matrix path $\boldsymbol q$ as above,
\begin{equation*}
    \lim_{N\to\infty}\Ll(\bar F_N(0,\diag\boldsymbol q)-\bar F_N(0,0)\Rr)=\boldsymbol\psi(\boldsymbol q):=\psi(\diag\boldsymbol q)-\psi(0).
\end{equation*}
The function $\boldsymbol\psi$ is convex, since $\psi$ is convex and $\boldsymbol q\mapsto\diag\boldsymbol q$ is linear. It also satisfies the hypotheses of~\cite[Theorem~4.7]{chen2022hamilton}: it is $1$-Lipschitz for the $L^1$ norm, by~\eqref{e.psi_lambda_lipschitz} and since $\Ll|\diag A\Rr|\le|A|$ for the Frobenius norm; and it is increasing for the dual of the matrix cone, since if $\boldsymbol q-\boldsymbol q'$ belongs to this dual cone, then testing against the diagonal paths $\diag(p)$ with $p\in\mcl Q_2^{\sS}$, which belong to the matrix cone, shows that $\diag\boldsymbol q-\diag\boldsymbol q'\in(\mcl Q_2^{\sS})^*$, and Lemma~\ref{l.dual_cone_monotone} gives $\psi(\diag\boldsymbol q)\ge\psi(\diag\boldsymbol q')$. Let $\boldsymbol f$ be the unique Lipschitz viscosity solution of the matrix-valued Hamilton--Jacobi equation with nonlinearity $\Xi$ and initial condition $\boldsymbol\psi$, given by~\cite[Theorem~4.7]{chen2022hamilton}. Since $\boldsymbol\psi$ is convex, the same theorem gives the Hopf formula
\begin{equation}
\label{e.matrix_Hopf}
    \boldsymbol f(t,\boldsymbol q)=\sup_{\boldsymbol p}\inf_{\boldsymbol q'}\Ll\{\boldsymbol\psi(\boldsymbol q')+\la\boldsymbol q-\boldsymbol q',\boldsymbol p\ra+t\int_0^1\Xi(\boldsymbol p(r))\d r\Rr\},
\end{equation}
where $\boldsymbol p$ and $\boldsymbol q'$ range over bounded paths that are nondecreasing for the positive semidefinite order, and $\la\cdot,\cdot\ra$ is the $L^2$ pairing associated with the Frobenius inner product.

\smallskip
\noindent\emph{Step 4: Identification along diagonal paths.}\par
We claim that, for every $q\in\mcl Q_\infty^{\sS}$,
\begin{equation}
\label{e.matrix_diagonal_identification}
    \boldsymbol f(t,\diag(q))=f(t,q)-\psi(0).
\end{equation}
The restriction of a viscosity solution to a subcone is not, in general, a viscosity solution of the restricted equation, so we argue through the two Hopf formulas. Let $\boldsymbol q:=\diag(q)$. For every admissible $\boldsymbol p$, the path $p:=\diag\boldsymbol p$ belongs to $\mcl Q_\infty^{\sS}$, and restricting the infimum in~\eqref{e.matrix_Hopf} to the diagonal paths $\boldsymbol q'=\diag(q')$ with $q'\in\mcl Q_\infty^{\sS}$, for which $\la\boldsymbol q-\boldsymbol q',\boldsymbol p\ra=\la q-q',p\ra_{L^2}$ and $\Xi(\boldsymbol p)=\xi(p)$, gives
\begin{equation*}
    \inf_{\boldsymbol q'}\Ll\{\boldsymbol\psi(\boldsymbol q')+\la\boldsymbol q-\boldsymbol q',\boldsymbol p\ra+t\int_0^1\Xi(\boldsymbol p(r))\d r\Rr\}\le\inf_{q'\in\mcl Q_\infty^{\sS}}\mcl J_{t,q}(q',p)-\psi(0)\le f(t,q)-\psi(0),
\end{equation*}
by~\eqref{e.Hopf_spin_glass}. Taking the supremum over $\boldsymbol p$ proves that $\boldsymbol f(t,\diag(q))\le f(t,q)-\psi(0)$. Conversely, we restrict the supremum in~\eqref{e.matrix_Hopf} to the diagonal paths $\boldsymbol p=\diag(p)$ with $p\in\mcl Q_\infty^{\sS}$. For such paths, $\la\boldsymbol q-\boldsymbol q',\boldsymbol p\ra=\la q-\diag\boldsymbol q',p\ra_{L^2}$ and $\Xi(\boldsymbol p)=\xi(p)$, so that every term in the braces depends on $\boldsymbol q'$ only through $\diag\boldsymbol q'$, which ranges over all of $\mcl Q_\infty^{\sS}$ as $\boldsymbol q'$ varies. The infimum over $\boldsymbol q'$ is therefore $\inf_{q'\in\mcl Q_\infty^{\sS}}\mcl J_{t,q}(q',p)-\psi(0)$, and taking the supremum over $p$ gives $\boldsymbol f(t,\diag(q))\ge f(t,q)-\psi(0)$, again by~\eqref{e.Hopf_spin_glass}. This proves~\eqref{e.matrix_diagonal_identification}.

\smallskip
\noindent\emph{Step 5: Conclusion.}\par
Let $q\in\mcl Q_\infty^{\sS}$. By Steps~1 to~3, the embedded model satisfies the hypotheses of~\cite[Theorem~4.14]{chen2022hamilton}, and the conclusion of this theorem at the path $\diag(q)$ reads, by~\eqref{e.matrix_free_energy_diagonal} and~\eqref{e.matrix_diagonal_identification},
\begin{equation*}
    f(t,q)-\psi(0)=\boldsymbol f(t,\diag(q))\le\liminf_{N\to\infty}\Ll(\bar F_N(t,q)-\bar F_N(0,0)\Rr)=\liminf_{N\to\infty}\bar F_N(t,q)-\psi(0),
\end{equation*}
where we used~\eqref{e.initial_condition_limit} at $q=0$ in the last step. This proves~\eqref{e.HJ_lower_spherical} for bounded $q$. For $q\in\mcl Q_2^{\sS}$, set $q_m^s:=q^s\wedge m$ for every $s\in\sS$ and $m\in\N$. Then $q_m\in\mcl Q_\infty^{\sS}$ and $|q_m-q|_{L^1}\to0$ as $m\to\infty$. The estimates~\eqref{e.f_L1_lipschitz} and~\eqref{e.F_N_Lipschitz} give
\begin{equation*}
    f(t,q)\le f(t,q_m)+|q_m-q|_{L^1}\le\liminf_{N\to\infty}\bar F_N(t,q)+2|q_m-q|_{L^1},
\end{equation*}
and letting $m\to\infty$ proves~\eqref{e.HJ_lower_spherical}.
\end{proof}

\section{Proof of the main result}
\label{s.conclusion}

Let us summarize what has been achieved so far. Proposition~\ref{p.crit_pt_bdd_spherical}, proved in Section~\ref{s.crit_pt_bdd}, bounds the limit free energy from above by the value $\sP_{t,q}(p)$ of the Parisi functional at a critical path $p$, but only when the limiting proportions are rational, along the sequence of sizes $N\in M\N$ compatible with the cavity block, and for regular bounded paths $q$. Section~\ref{s.hj} identifies the candidate limit: the function $f$ given by the Hopf formula~\eqref{e.Hopf_spin_glass} is the Lipschitz viscosity solution of~\eqref{e.main.hj}, and Proposition~\ref{p.HJ_bdd_spherical} shows that it is a lower bound for the limit free energy, for arbitrary species proportions, along the full sequence of sizes, and for every $q\in\mcl Q_2^{\sS}$. In this section, we first combine these two results into the identification of the limit free energy for rational proportions, Corollary~\ref{c.rational_identification}, and then remove the restrictions on the sizes and on the proportions, which completes the proof of Theorem~\ref{t.main}.

\subsection{Identification of the limit for rational proportions}
\label{s.conclusion.rational}

\begin{corollary}[Identification of the limit for rational proportions]
\label{c.rational_identification}
Assume that $\lambda_\infty\in\Q^{\sS}$, and let $M$ and the sequence of sizes~\eqref{e.compatible_sequence} be as in Proposition~\ref{p.crit_pt_bdd_spherical}. Let $f$ be the viscosity solution of Theorem~\ref{t.vis_sol}. For every $t>0$ and $q\in\mcl Q_\infty^{\sS}$, we have
\begin{equation}
\label{e.rational_identification}
    \lim_{n\to\infty}\bar F_{N_n}(t,q)=f(t,q).
\end{equation}
\end{corollary}

\begin{proof}
\noindent\emph{Step 1: Regular paths.}\par
Let $q\in\mcl Q_{\infty,\uparrow}^{\sS}$. Propositions~\ref{p.HJ_bdd_spherical} and~\ref{p.crit_pt_bdd_spherical} being now proved, the chain~\eqref{e.identification_chain} holds, with $f(t,q)$ equal to the right-hand side of~\eqref{e.main.1} by the Hopf formula~\eqref{e.Hopf_spin_glass}. This proves~\eqref{e.rational_identification} for $q\in\mcl Q_{\infty,\uparrow}^{\sS}$.

\smallskip
\noindent\emph{Step 2: Bounded paths.}\par
Let $q\in\mcl Q_\infty^{\sS}$, and, for every integer $m\ge1$, define $q_m$ by
\begin{equation*}
    q_m^s(u):=q^s(u)\wedge(mu)+\frac um,
    \qquad u\in[0,1),\quad s\in\sS.
\end{equation*}
Each component of $q_m$ is right-continuous, bounded by $|q|_{L^\infty}+1$, vanishes at $0$, and satisfies $q_m^s(v)-q_m^s(u)\ge(v-u)/m$ for $0\le u\le v<1$, so $q_m\in\mcl Q_{\infty,\uparrow}^{\sS}$. Moreover, $q^s(u)\wedge(mu)$ and $q^s(u)$ both belong to $[0,|q|_{L^\infty}]$, and they coincide for $u\ge|q|_{L^\infty}/m$, so that
\begin{equation*}
    |q_m-q|_{L^1}\le\frac{|\sS|}{m}\Ll(|q|_{L^\infty}^2+1\Rr)\xrightarrow[m\to\infty]{}0.
\end{equation*}
Both $\bar F_{N_n}(t,\cdot)$ and $f(t,\cdot)$ are $1$-Lipschitz for the $L^1$ norm, by~\eqref{e.F_N_Lipschitz} and~\eqref{e.f_L1_lipschitz}. Hence, by Step~1,
\begin{equation*}
    \limsup_{n\to\infty}\Ll|\bar F_{N_n}(t,q)-f(t,q)\Rr|\le2|q_m-q|_{L^1}+\limsup_{n\to\infty}\Ll|\bar F_{N_n}(t,q_m)-f(t,q_m)\Rr|=2|q_m-q|_{L^1},
\end{equation*}
and letting $m\to\infty$ proves~\eqref{e.rational_identification}.
\end{proof}

It remains to extend~\eqref{e.rational_identification} to arbitrary sizes and proportions. Lemma~\ref{l.finite_size_comparison} shows that the free energy is not sensitive to the addition of a bounded number of coordinates, nor to small changes of the species proportions; this allows us to compare the actual system with a system with rational proportions $\lambda$ close to $\lambda_\infty$, whose limit free energy is $f_\lambda$ by Corollary~\ref{c.rational_identification}. Lemma~\ref{l.hopf_continuity_lambda} then shows that $f_\lambda(t,q)$ is continuous in $\lambda$, which lets us pass from rational proportions to $\lambda_\infty$.

\subsection{Finite-size comparison}
\label{s.conclusion.finite_size}

The argument of Lemma~\ref{l.fixed_block_increment_bound} in Section~\ref{s.cavity} adapts to the comparison of systems of different sizes and species proportions: a single coordinate is added to one species, the Hamiltonian is unperturbed, and the dependence of the constant on $t$ and $q$ is made explicit. For an integer vector $\mathbf n=(n_s)_{s\in\sS}$ with positive entries, we write $|\mathbf n|_1:=\sum_{s\in\sS}n_s$, and we let $Z_{\mathbf n}(t,q)$ and $\bar F_{\mathbf n}(t,q)$ denote the partition function~\eqref{e.ZN_def} and the free energy~\eqref{e.F_N_spherical} of the system with $n_s$ coordinates in species $s$, for every $s\in\sS$; by permutation invariance within each species, these quantities do not depend on the choice of the partition $(I_{N,s})_{s\in\sS}$ with $|I_{N,s}|=n_s$. We recall that $(e_s)_{s\in\sS}$ denotes the canonical basis of $\R^{\sS}$.

\begin{lemma}[Adding one coordinate]
\label{l.one_coordinate_comparison}
There exists a constant $C<\infty$, depending only on $\xi$ and~$h$, such that, for every integer vector $\mathbf n$ with positive entries, $s\in\sS$, $t\ge0$, and $q\in\mcl Q_\infty^{\sS}$,
\begin{gather}
\label{e.log_partition_uniform_bound}
    \Ll|\E\log Z_{\mathbf n}(t,q)\Rr|\le C|\mathbf n|_1\Ll(1+t+|q|_{L^\infty}\Rr),\\
\label{e.one_coordinate_log_partition}
    \Ll|\E\log Z_{\mathbf n+e_s}(t,q)-\E\log Z_{\mathbf n}(t,q)\Rr|\le C\Ll(1+t+|q|_{L^\infty}\Rr).
\end{gather}
\end{lemma}

\begin{proof}
Write $K:=|\mathbf n|_1$ and $\lambda:=\mathbf n/K$, and let $(I_{K,u})_{u\in\sS}$ denote the species partition of the system with $n_u$ coordinates in species $u$. We first prove~\eqref{e.log_partition_uniform_bound}. By~\eqref{e.F_N_Lipschitz}, we have $|\bar F_{\mathbf n}(t,q)-\bar F_{\mathbf n}(0,0)|\le|q|_{L^1}+t\sup_{[-1,1]^{\sS}}|\xi|$. Moreover, $|\bar F_{\mathbf n}(0,0)|\le\max_u|h^u|$, since $|\sum_uh^u\sum_{i\in I_{K,u}}\sigma_i|\le K\max_u|h^u|$ on $\Sigma_K$, by the Cauchy--Schwarz inequality. Since $\E\log Z_{\mathbf n}=-K\bar F_{\mathbf n}$ and $|q|_{L^1}\le|q|_{L^\infty}$, this proves~\eqref{e.log_partition_uniform_bound}.

We now turn to~\eqref{e.one_coordinate_log_partition}. The larger system has size $K+1$ and proportions $\lambda':=(\mathbf n+e_s)/(K+1)$. We apply the disintegration~\eqref{e.block_disintegration} with $m=1$ and $n=n_s$ to the species-$s$ sphere of the larger system, leaving the spheres of the other species unchanged: writing $\tau\in\mcl B_{1,n_s}$ for the added coordinate, $r^s(\tau):=1+(1-\tau^2)/n_s$, and $r^u(\tau):=1$ for $u\ne s$, a configuration of the larger system is $(S_{r(\tau)}\sigma,\tau)$ with $\sigma\in\Sigma_K$, and its reference measure is $p_{1,n_s}(\tau)\d\tau\,P_K(\d\sigma)$. This is the chart of Subsection~\ref{s.spherical_chart} for a cavity block consisting of one coordinate of species $s$, and we follow Steps~1 and~2 of the proof of Lemma~\ref{l.fixed_block_increment_bound}, with $x=y=0$. By~\eqref{e.overlap_dilation}, the analogue of~\eqref{e.full_and_bulk_overlap_identity} holds, and since $n_s|r^s(\tau)-1|\le1+\tau^2$ and $K|R_{K,s}(\sigma,\sigma')|\le n_s$, the analogue of~\eqref{e.full_and_bulk_overlap_bound} reads
\begin{equation*}
    (K+1)\Ll|R_{K+1}\Ll((S_{r(\tau)}\sigma,\tau),(S_{r(\tau')}\sigma',\tau')\Rr)-R_K(\sigma,\sigma')\Rr|\le C\Ll(1+\tau^2+\tau'^2\Rr),
\end{equation*}
for a constant $C$ depending only on $|\sS|$. We compare the enriched Hamiltonians~\eqref{e.enriched_H} of the two systems, viewed as Gaussian fields indexed by $(\sigma,\alpha,\tau)$; recall that their means and covariances depend on the configurations only through the overlaps and the sums $\sum_{i\in I_{K,u}}\sigma_i$. Both overlap vectors lie in the overlap domain $[-1,1]^{\sS}$, on which $\xi$ is Lipschitz, so the covariance $(K+1)\xi(R_{K+1})$ of the interaction field of the larger system differs from the covariance $K\xi(R_K)$ of the smaller one by at most $C(1+\tau^2+\tau'^2)$; and the deterministic centerings $-(K+1)t\xi(\lambda')$ and $-Kt\xi(\lambda)$ differ by at most $Ct$, since $|\lambda'-\lambda|\le2/K$ and $\xi$ is Lipschitz on $[0,1]^{\sS}$. For the cascade field, the covariance of $\sqrt2W^q_{K+1}$ in~\eqref{e.WNq_covariance} is $2\sum_{u\in\sS}q^u(\alpha\wedge\alpha')(K+1)R_{K+1,u}$, which differs from the covariance of $\sqrt2W_K^q$ by at most $2|q|_{L^\infty}\sum_u|(K+1)R_{K+1,u}-KR_{K,u}|\le C|q|_{L^\infty}(1+\tau^2+\tau'^2)$, while the self-overlap corrections $-(K+1)q(1)\cdot\lambda'$ and $-Kq(1)\cdot\lambda$ differ by the constant $q^s(1)$. The external-field terms differ by at most $C(1+\tau^2)$, as in~\eqref{e.frozen_chart_mean_bound}. Altogether, the means and covariances of the two Hamiltonians differ by at most $C(1+t+|q|_{L^\infty})(1+\tau^2+\tau'^2)$.

We now apply Lemma~\ref{l.quadratically_confined_Gaussian_comparison} with $M=1$, $\mathbb X^0:=\Sigma_K\times\mfk U$, $\mu^0:=P_K\otimes\fR$, $B:=\mcl B_{1,n_s}$, $\nu:=p_{1,n_s}(\tau)\d\tau$, and $C_0:=C(1+t+|q|_{L^\infty})$. The set $\mcl B_{1,n_s}$ is bounded, $\nu$ is a probability measure on it, and $\int\tau^2p_{1,n_s}(\tau)\d\tau=1$ by~\eqref{e.block_disintegration}, since each coordinate of a point of $\Sph_{n_s+1}$ has second moment $1$ under the uniform measure. The condition~\eqref{e.quadratic_comparison_second_moments} holds for the same reason as in the proof of Lemma~\ref{l.fixed_block_increment_bound}: the law of the Hamiltonian of the larger system is invariant under permutations of the coordinates within species $s$, and $\sum_{i\in I_{K+1,s}}\rho_i^2=n_s+1$ on its configuration space, so that $\E\la\tau^2\ra_{\mathbf n+e_s}=1$, where $\la\cdot\ra_{\mathbf n+e_s}$ denotes the Gibbs bracket associated with $Z_{\mathbf n+e_s}(t,q)$. The lemma yields~\eqref{e.one_coordinate_log_partition}.
\end{proof}

We now translate this estimate into a comparison of free energies. If $\lambda$ is a probability vector and $N\in\N$ is such that $N\lambda_s\in\N$ for every $s\in\sS$, we write
\begin{equation}
\label{e.F_N_lambda_def}
    \bar F_{N,\lambda}(t,q):=\bar F_{\mathbf n}(t,q),
    \qquad\text{where }\mathbf n:=(N\lambda_s)_{s\in\sS},
\end{equation}
for the free energy of a system of size $N$ with species proportions $\lambda$. In particular, $\bar F_N=\bar F_{N,\lambda_N}$.

\begin{lemma}[Finite-size comparison]
\label{l.finite_size_comparison}
There exists a constant $C<\infty$, depending only on $\xi$ and $h$, such that the following holds. Let $N,N'\in\N$ satisfy $N\le N'\le2N$, and let $\lambda,\lambda'$ be probability vectors such that $N\lambda_s,N'\lambda'_s\in\N$ for every $s\in\sS$. For every $t\ge0$ and $q\in\mcl Q_\infty^{\sS}$, we have
\begin{equation}
    \Ll|\bar F_{N,\lambda}(t,q)-\bar F_{N',\lambda'}(t,q)\Rr|\le C\Ll(1+t+|q|_{L^\infty}\Rr)\Ll(|\lambda-\lambda'|+\frac{N'-N}{N}\Rr).
    \label{e.finite_size_comparison}
\end{equation}
\end{lemma}

The constant involves $|q|_{L^\infty}$ rather than $|q|_{L^1}$, which is what the Gaussian comparison in Lemma~\ref{l.one_coordinate_comparison} directly gives; this suffices, since the lemma is only applied at fixed bounded paths.

\begin{proof}[Proof of Lemma~\ref{l.finite_size_comparison}]
We first translate the comparison of $\E\log Z_{\mathbf n}$ in Lemma~\ref{l.one_coordinate_comparison} into a comparison of free energies. Let $\mathbf n$ be an integer vector with positive entries, and let $s\in\sS$. Since $\bar F_{\mathbf n}(t,q)=-|\mathbf n|_1^{-1}\E\log Z_{\mathbf n}(t,q)$, we can write
\begin{equation*}
    \bar F_{\mathbf n+e_s}(t,q)-\bar F_{\mathbf n}(t,q)=\Ll(\frac1{|\mathbf n|_1}-\frac1{|\mathbf n|_1+1}\Rr)\E\log Z_{\mathbf n+e_s}(t,q)-\frac{1}{|\mathbf n|_1}\Ll(\E\log Z_{\mathbf n+e_s}(t,q)-\E\log Z_{\mathbf n}(t,q)\Rr).
\end{equation*}
By the two estimates of Lemma~\ref{l.one_coordinate_comparison}, the first term is bounded by $C(1+t+|q|_{L^\infty})/|\mathbf n|_1$ in absolute value, and so is the second. Hence
\begin{equation}
    \Ll|\bar F_{\mathbf n+e_s}(t,q)-\bar F_{\mathbf n}(t,q)\Rr|\le\frac{C}{|\mathbf n|_1}\Ll(1+t+|q|_{L^\infty}\Rr),
    \qquad t\ge0,\quad q\in\mcl Q_\infty^{\sS},
    \label{e.one_coordinate_free_energy_comparison}
\end{equation}
for a constant $C$ depending only on $\xi$ and $h$.

We now prove~\eqref{e.finite_size_comparison}. Put $\mathbf n:=(N\lambda_s)_{s\in\sS}$ and $\mathbf n':=(N'\lambda'_s)_{s\in\sS}$, and choose integer vectors $\mathbf m_0,\ldots,\mathbf m_d$ with $\mathbf m_0=\mathbf n$, $\mathbf m_d=\mathbf n'$, and $\mathbf m_{j+1}-\mathbf m_j\in\{e_s,-e_s:s\in\sS\}$ for every $j<d$, such that each coordinate moves monotonically from $n_s$ to $n'_s$ along the path, and such that all the increments $+e_s$ come before all the decrements $-e_s$. Then
\begin{equation*}
    d=\sum_{s\in\sS}|N'\lambda'_s-N\lambda_s|\le N'\sum_{s\in\sS}|\lambda'_s-\lambda_s|+(N'-N)\le CN\Ll(|\lambda-\lambda'|+\frac{N'-N}{N}\Rr).
\end{equation*}
Moreover, for every $j$ and $s$, the coordinate $(\mathbf m_j)_s$ lies between $n_s$ and $n'_s$, and is therefore positive; and since $|\mathbf m_j|_1$ first increases from $N$ and then decreases to $N'\ge N$, we have $|\mathbf m_j|_1\ge N$ for every $j$. The estimate~\eqref{e.one_coordinate_free_energy_comparison} applies to each pair $(\mathbf m_j,\mathbf m_{j+1})$, possibly after exchanging the roles of $\mathbf n$ and $\mathbf n+e_s$ there. Summing over $j\in\{0,\ldots,d-1\}$ gives
\begin{equation*}
    \Ll|\bar F_{\mathbf n'}(t,q)-\bar F_{\mathbf n}(t,q)\Rr|\le\frac{Cd}{N}\Ll(1+t+|q|_{L^\infty}\Rr),
\end{equation*}
and the bound on $d$ proves~\eqref{e.finite_size_comparison}.
\end{proof}

\subsection{Continuity of the limit free energy in the species proportions}
\label{s.conclusion.continuity}

Recall from~\eqref{e.f_lambda_def} the function $f_\lambda$ associated with a probability vector $\lambda$ with positive entries. To make its dependence on $\lambda$ explicit, we write $\psi_s:=\psi^\circ_{h^s}$ for $s\in\sS$, and we let $\psi_s^*$ be the convex conjugate of the restriction of $\psi_s$ to $\mcl Q_\infty$,
\begin{equation}
\label{e.psi_conjugate_def}
    \psi_s^*(a):=\sup_{\rho\in\mcl Q_\infty}\Ll\{\la\rho,a\ra_{L^2}-\psi_s(\rho)\Rr\}\in[-\psi_s(0),+\infty],
    \qquad a\in\mcl Q_{\infty,\le1}.
\end{equation}
The lower bound is obtained by taking $\rho=0$. Moreover, $\psi_s^*(0)=-\psi_s(0)$: for $\rho\in\mcl Q_\infty$, the path $\rho e_s$ belongs to $\mcl Q_2^{\sS}\subset(\mcl Q_2^{\sS})^*$, so that Lemma~\ref{l.dual_cone_monotone} gives $\psi_\lambda(\rho e_s)\ge\psi_\lambda(0)$, that is, $\lambda_s\psi_s(\rho)\ge\lambda_s\psi_s(0)$ by~\eqref{e.psi_lambda_def}. We set
\begin{equation}
\label{e.cost_c_s_def}
    c_s(a):=\psi_s^*(a)-\psi_s^*(0)=\psi_s^*(a)+\psi_s(0)\in[0,+\infty],
    \qquad\text{so that}\qquad c_s(0)=0.
\end{equation}

\begin{lemma}[Normalized form of the Hopf formula]
\label{l.hopf_normalized}
For every probability vector $\lambda$ with positive entries, $t\ge0$, and $q\in\mcl Q_2^{\sS}$, we have
\begin{equation}
\label{e.hopf_normalized}
\begin{gathered}
    f_\lambda(t,q)=\sum_{s\in\sS}\lambda_s\psi_s(0)+\sup_{a\in(\mcl Q_{\infty,\le1})^{\sS}}\Ll\{A_\lambda(a)-\sum_{s\in\sS}\lambda_sc_s(a_s)\Rr\},\\
    \text{where}\qquad
    A_\lambda(a):=\sum_{s\in\sS}\lambda_s\la q^s,a_s\ra_{L^2}+t\int_0^1\xi\Ll((\lambda_sa_s(r))_{s\in\sS}\Rr)\d r.
\end{gathered}
\end{equation}
Moreover, for every $a\in(\mcl Q_{\infty,\le1})^{\sS}$,
\begin{equation}
\label{e.A_lambda_bounds}
    0\le A_\lambda(a)\le C_{t,q}:=\sum_{s\in\sS}|q^s|_{L^1}+t\sup_{v\in[0,1]^{\sS}}\xi(v),
\end{equation}
and thus the supremum in~\eqref{e.hopf_normalized} belongs to $[0,C_{t,q}]$.
\end{lemma}

\begin{proof}
The map $p\mapsto a:=(p^s/\lambda_s)_{s\in\sS}$ is a bijection from $\mcl Q^{\sS}_{\infty,\le\lambda}$ onto $(\mcl Q_{\infty,\le1})^{\sS}$, and $\la q',p\ra_{L^2}=\sum_s\lambda_s\la q'^s,a_s\ra_{L^2}$ for every $q'\in\mcl Q_\infty^{\sS}$, while $\psi_\lambda(q')=\sum_s\lambda_s\psi_s(q'^s)$ by~\eqref{e.psi_lambda_def}. Since the components $q'^s$ of $q'\in\mcl Q_\infty^{\sS}$ vary independently in $\mcl Q_\infty$, we have
\begin{equation*}
    \sup_{q'\in\mcl Q_\infty^{\sS}}\Ll\{\la q',p\ra_{L^2}-\psi_\lambda(q')\Rr\}=\sum_{s\in\sS}\lambda_s\psi_s^*(a_s),
\end{equation*}
and therefore $\inf_{q'}\mcl J^\lambda_{t,q}(q',p)=A_\lambda(a)-\sum_s\lambda_s\psi_s^*(a_s)$, so that, by~\eqref{e.f_lambda_def},
\begin{equation}
\label{e.hopf_inner_inf_computed}
    f_\lambda(t,q)=\sup_{a\in(\mcl Q_{\infty,\le1})^{\sS}}\Ll\{A_\lambda(a)-\sum_{s\in\sS}\lambda_s\psi_s^*(a_s)\Rr\};
\end{equation}
the conjugates are bounded below and $A_\lambda(a)$ is finite, so that no indeterminate expression appears. Inserting $\psi_s^*=c_s-\psi_s(0)$ gives~\eqref{e.hopf_normalized}. The bounds~\eqref{e.A_lambda_bounds} follow from $0\le a_s\le1$, $q^s\ge0$, $0\le\lambda_s\le1$, and the nonnegativity of $\xi$ on $[0,1]^{\sS}$, which holds since its coefficients are nonnegative. Finally, since $c_s\ge0$, the expression in braces is at most $C_{t,q}$; and it vanishes at $a=0$, since $c_s(0)=0$ and $A_\lambda(0)=t\xi(0)=0$.
\end{proof}

The formula~\eqref{e.hopf_normalized} isolates the dependence on $\lambda$: the proportions enter through explicit linear weights and through the argument of $\xi$, while the cost functions $c_s$ do not depend on $\lambda$. The next lemma exploits this. It only uses near-maximizers of~\eqref{e.hopf_normalized}, for which the costs $c_s(a_s)$ are bounded; nothing is assumed on $\psi_s^*$ away from these. For $\eta\in(0,1]$, we set
\begin{equation}
\label{e.Lambda_eta_def}
    \Lambda_\eta:=\Ll\{\lambda\in[\eta,1]^{\sS}:\sum_{s\in\sS}\lambda_s=1\Rr\}.
\end{equation}

\begin{lemma}[Continuity in the species proportions]
\label{l.hopf_continuity_lambda}
Fix $\eta\in(0,1]$. There exists $C<\infty$, depending only on $\xi$, $h$, and $\eta$, such that, for every $\lambda,\lambda'\in\Lambda_\eta$, $t\ge0$, and $q\in\mcl Q_2^{\sS}$,
\begin{equation}
\label{e.hopf_continuity_lambda}
    |f_\lambda(t,q)-f_{\lambda'}(t,q)|\le C\Ll(1+t+|q|_{L^1}\Rr)|\lambda-\lambda'|.
\end{equation}
\end{lemma}

\begin{proof}
Fix $t$ and $q$, and let $\msf V_\lambda$ denote the supremum in~\eqref{e.hopf_normalized}. Let $\eps\in(0,1)$, and let $a\in(\mcl Q_{\infty,\le1})^{\sS}$ satisfy $A_\lambda(a)-\sum_s\lambda_sc_s(a_s)\ge\msf V_\lambda-\eps$. Since $\msf V_\lambda\ge0$ and $A_\lambda(a)\le C_{t,q}$, we get $\sum_s\lambda_sc_s(a_s)\le C_{t,q}+1$, hence $c_s(a_s)\le(C_{t,q}+1)/\eta$ for every $s\in\sS$, as $\lambda_s\ge\eta$. We now evaluate the same $a$ at $\lambda'$. Since $\xi$ is Lipschitz on $[0,1]^{\sS}$, with a constant $L_\xi$ depending only on $\xi$, and since $|(\lambda_sa_s(r))_s-(\lambda'_sa_s(r))_s|\le|\lambda-\lambda'|$ for every $r$, we have
\begin{equation*}
    |A_\lambda(a)-A_{\lambda'}(a)|\le\sum_{s\in\sS}|\lambda_s-\lambda'_s|\,|q^s|_{L^1}+tL_\xi|\lambda-\lambda'|
    \quad\text{and}\quad
    \Big|\sum_{s\in\sS}(\lambda_s-\lambda'_s)c_s(a_s)\Big|\le\frac{C_{t,q}+1}{\eta}\sum_{s\in\sS}|\lambda_s-\lambda'_s|.
\end{equation*}
Using $\sum_s|x_s|\le|\sS|^{1/2}|x|$ for $x\in\R^{\sS}$, and $\sum_s|q^s|_{L^1}\le|\sS|^{1/2}|q|_{L^1}$, we deduce that
\begin{equation*}
    \msf V_{\lambda'}\ge A_{\lambda'}(a)-\sum_{s\in\sS}\lambda'_sc_s(a_s)\ge\msf V_\lambda-\eps-C\Ll(1+t+|q|_{L^1}\Rr)|\lambda-\lambda'|,
\end{equation*}
for a constant $C$ depending only on $\xi$ and $\eta$. Exchanging the roles of $\lambda$ and $\lambda'$ and letting $\eps\downarrow0$ bounds $|\msf V_\lambda-\msf V_{\lambda'}|$ by the last term. The first term in~\eqref{e.hopf_normalized} is Lipschitz in $\lambda$ with constant $\max_s|\psi_s(0)|$, and $\psi_s(0)=\psi^\circ_{h^s}(0)$ depends only on $h^s$. This proves~\eqref{e.hopf_continuity_lambda}.
\end{proof}

\subsection{Conclusion}
\label{s.conclusion.proof}

\begin{proof}[Proof of Theorem~\ref{t.main}]
Let $f=f_{\lambda_\infty}$ be the Lipschitz viscosity solution of~\eqref{e.main.hj} given by Theorem~\ref{t.vis_sol}. By~\eqref{e.Hopf_spin_glass}, $f(t,q)$ is the right-hand side of~\eqref{e.main.1}, so it remains to prove that $\bar F_N(t,q)$ converges to $f(t,q)$ for every $t\ge0$ and $q\in\mcl Q_2^{\sS}$. At $t=0$, this is the convergence of the initial condition in Proposition~\ref{l.initial_condition}: $\bar F_N(0,q)\to\psi(q)=f(0,q)$ for every $q\in\mcl Q_2^{\sS}$. We fix $t>0$ from now on.

\smallskip
\noindent\emph{Step 1: Bounded paths.}\par
Let $q\in\mcl Q_\infty^{\sS}$. Set $\eta:=\frac12\min_{s\in\sS}\lambda_{\infty,s}>0$, and choose rational probability vectors $\lambda^{(m)}\in\Lambda_\eta\cap\Q^{\sS}$, $m\in\N$, such that $\lambda^{(m)}\to\lambda_\infty$ as $m\to\infty$, together with integers $M_m\in\N$ such that $M_m\lambda^{(m)}_s\in\N$ for every $s\in\sS$. Fix $m$. The systems of sizes $N\in M_m\N$ with proportions $\lambda^{(m)}$ form a model of the type considered in Section~\ref{s.introduction}, with the same $\xi$ and $h$, whose limiting proportions $\lambda^{(m)}$ are rational, and whose initial condition is $\psi_{\lambda^{(m)}}$; the sizes $N\notin M_m\N$ play no role, and can be assigned arbitrary proportions converging to $\lambda^{(m)}$. Corollary~\ref{c.rational_identification} applies to this model and gives
\begin{equation}
\label{e.rational_limit_m}
    \lim_{k\to\infty}\bar F_{kM_m,\lambda^{(m)}}(t,q)=f_{\lambda^{(m)}}(t,q).
\end{equation}
For $N\ge M_m$, let $K_N:=M_m\lfloor N/M_m\rfloor$, so that $N-M_m<K_N\le N\le2K_N$ and $K_N\lambda^{(m)}_s\in\N$ for every $s$. Lemma~\ref{l.finite_size_comparison}, applied with the sizes $K_N\le N$ and the proportions $\lambda^{(m)}$ and $\lambda_N$, gives
\begin{equation*}
    \Ll|\bar F_N(t,q)-\bar F_{K_N,\lambda^{(m)}}(t,q)\Rr|\le C\Ll(1+t+|q|_{L^\infty}\Rr)\Ll(|\lambda_N-\lambda^{(m)}|+\frac{M_m}{K_N}\Rr),
\end{equation*}
with $C$ depending only on $\xi$ and $h$. Letting $N\to\infty$ at fixed $m$, using~\eqref{e.rational_limit_m} and $\lambda_N\to\lambda_\infty$, and then applying Lemma~\ref{l.hopf_continuity_lambda} to $\lambda^{(m)},\lambda_\infty\in\Lambda_\eta$, we obtain, since $|q|_{L^1}\le|q|_{L^\infty}$,
\begin{equation*}
    \limsup_{N\to\infty}\Ll|\bar F_N(t,q)-f(t,q)\Rr|\le C\Ll(1+t+|q|_{L^\infty}\Rr)|\lambda^{(m)}-\lambda_\infty|,
\end{equation*}
with $C$ depending only on $\xi$, $h$, and $\eta$.
Letting $m\to\infty$ proves that $\bar F_N(t,q)\to f(t,q)$.

\smallskip
\noindent\emph{Step 2: General paths.}\par
Let $q\in\mcl Q_2^{\sS}$, and set $q_j^s:=q^s\wedge j$ for $j\in\N$ and $s\in\sS$, so that $q_j\in\mcl Q_\infty^{\sS}$ and $|q_j-q|_{L^1}\to0$ as $j\to\infty$, by monotone convergence. By~\eqref{e.F_N_Lipschitz}, \eqref{e.f_L1_lipschitz}, and Step~1,
\begin{equation*}
    \limsup_{N\to\infty}\Ll|\bar F_N(t,q)-f(t,q)\Rr|\le2|q_j-q|_{L^1}+\limsup_{N\to\infty}\Ll|\bar F_N(t,q_j)-f(t,q_j)\Rr|=2|q_j-q|_{L^1}.
\end{equation*}
Letting $j\to\infty$ completes the proof.
\end{proof}

\section{Balanced models}
\label{s.balanced}

We show that, for balanced models, the formula of Theorem~\ref{t.main} is bounded from below and from above by formulas for a one-species spherical model, and that the two bounds coincide when all the components of $q$ are equal, in particular at $q=0$. Balanced models were introduced in~\cite{bates2025balanced}, see also~\cite{issa2024existence}; we follow the argument given for Ising spins in~\cite[Section~9]{chen2026ising}, see also~\cite{ho2026concavity}. In this section, we write $\lambda=(\lambda_s)_{s\in\sS}$ for $\lambda_\infty$, and $r\lambda:=(r\lambda_s)_{s\in\sS}$ for $r\in\R$.

Throughout this section, we assume that the external field does not depend on the species: there is $h_\star\in\R$ such that $h^s=h_\star$ for every $s\in\sS$, and we write $\psi_\star^\circ:=\psi_{h_\star}^\circ$, so that $\psi(q)=\sum_{s\in\sS}\lambda_s\psi_\star^\circ(q^s)$. We define the effective one-species covariance function
\begin{equation}
\label{e.xi_star_def}
    \xi_\star(r):=\xi(r\lambda),\qquad r\in\R.
\end{equation}
Grouping the terms of~\eqref{e.xi_power_series_def} according to their total degree, we can write
\begin{equation}
\label{e.xi_star_series}
    \xi_\star(r)=\sum_{n\ge1}\Big(\sum_{s_1,\ldots,s_n\in\sS}\beta_{s_1,\ldots,s_n}^2\prod_{j=1}^n\lambda_{s_j}\Big)r^n,
    \qquad r\in\R,
\end{equation}
an absolutely convergent series with nonnegative coefficients. Thus, $\xi_\star$ satisfies the assumptions of Section~\ref{s.introduction} for a model with a single species.

\begin{definition}[Balanced models]
\label{d.balanced_comparison_structure}
We say that $\xi$ is balanced with respect to $\lambda$ if
\begin{equation}
\label{e.balanced_inequality}
    \xi\Ll((\lambda_sx_s)_{s\in\sS}\Rr)\le\sum_{s\in\sS}\lambda_s\,\xi_\star(x_s),
    \qquad x=(x_s)_{s\in\sS}\in[0,1]^{\sS}.
\end{equation}
\end{definition}

In~\eqref{e.balanced_inequality}, the variable $x_s$ stands for the overlap of species $s$ normalized by the number of coordinates in this species, so that $\lambda_sx_s$ is this overlap normalized by the total number of coordinates, as in~\eqref{e.R_Ns_def}. The inequality holds with equality when all the normalized overlaps $x_s$ agree, by~\eqref{e.xi_star_def}. The simplest example is the bipartite model with two species of equal sizes,
\begin{equation*}
    \sS=\{1,2\},
    \qquad
    \lambda=\Ll(\tfrac12,\tfrac12\Rr),
    \qquad
    \xi(a_1,a_2)=a_1a_2,
    \qquad
    \xi_\star(r)=\frac{r^2}4,
\end{equation*}
for which~\eqref{e.balanced_inequality} reads $x_1x_2\le\frac12(x_1^2+x_2^2)$. More generally, a pure monomial $\xi(a)=\beta^2\prod_{s\in\sS}a_s^{k_s}$ with $k_s\ge1$ is balanced with respect to $\lambda_s:=k_s/n$, where $n:=|k|_1$: in this case, $\xi_\star(r)=cr^n$ with $c:=\beta^2\prod_s\lambda_s^{k_s}$, and the inequality~\eqref{e.balanced_inequality} becomes $\prod_s(x_s^n)^{\lambda_s}\le\sum_s\lambda_sx_s^n$, which is the weighted arithmetic-geometric mean inequality. A general criterion on the coefficients in~\eqref{e.xi_power_series_def} ensuring~\eqref{e.balanced_inequality} is given in~\cite[Lemma~9.2]{chen2026ising}, which covers the balanced models of~\cite{bates2025balanced}.

For $(t,\rho)\in\R_+\times\mcl Q_2$, $\rho'\in\mcl Q_\infty$, and $a\in\mcl Q_{\infty,\le1}$, we define the one-species analogue of the functional $\mcl J_{t,q}$ in~\eqref{e.mcJ}
\begin{equation}
\label{e.balanced_J_star}
    \mcl J^\star_{t,\rho}(\rho',a):=\psi^\circ_\star(\rho')+\la\rho-\rho',a\ra_{L^2}+t\int_0^1\xi_\star(a(v))\d v,
\end{equation}
and we set
\begin{equation}
\label{e.balanced_u_def}
    u(t,\rho):=\sup_{a\in\mcl Q_{\infty,\le1}}\ \inf_{\rho'\in\mcl Q_\infty}\ \mcl J^\star_{t,\rho}(\rho',a).
\end{equation}
This is the right-hand side of~\eqref{e.main.1} for the one-species spherical model with covariance function $\xi_\star$ and external field $h_\star$: for this model, the set of species is a singleton, $\lambda_\infty=1$, and the initial condition~\eqref{e.multi_species_initial} is $\psi_\star^\circ$. By Theorem~\ref{t.main} applied to this model, the function $u$ is thus the limit of its enriched free energy, and it is the Lipschitz viscosity solution of~\eqref{e.main.hj} with $\xi_\star$ and $\psi_\star^\circ$ in place of $\xi$ and $\psi$.

For $q=(q^s)_{s\in\sS}\in\mcl Q_2^{\sS}$, we write $\bar q:=\sum_{s\in\sS}\lambda_sq^s\in\mcl Q_2$.

\begin{proposition}[Balanced reduction]
\label{p.balanced_reduction}
Assume that $h^s=h_\star$ for every $s\in\sS$ and that $\xi$ is balanced with respect to $\lambda$. Let $f$ be the limit free energy in Theorem~\ref{t.main}. For every $t\ge0$ and $q\in\mcl Q_2^{\sS}$, we have
\begin{equation}
\label{e.balanced_comparison}
    u(t,\bar q)\le f(t,q)\le\sum_{s\in\sS}\lambda_su(t,q^s).
\end{equation}
In particular, if $q^s=\rho$ for every $s\in\sS$, for some $\rho\in\mcl Q_2$, then $f(t,q)=u(t,\rho)$. Moreover, the first inequality in~\eqref{e.balanced_comparison} holds even if $\xi$ is not balanced.
\end{proposition}

\begin{proof}
Since $h^s=h_\star$ for every $s\in\sS$, the conjugate $\psi_s^*$ in~\eqref{e.psi_conjugate_def} does not depend on $s$: it is the conjugate of $\psi^\circ_\star$, and we denote it by $\psi_\star^*$. By Theorem~\ref{t.main} and~\eqref{e.hopf_inner_inf_computed} with $\lambda=\lambda_\infty$, and by the definition~\eqref{e.balanced_u_def} of $u$, we have, for every $\rho\in\mcl Q_2$,
\begin{align}
\label{e.balanced_f_hopf}
    f(t,q)&=\sup_{a\in(\mcl Q_{\infty,\le1})^{\sS}}\Ll\{\sum_{s\in\sS}\lambda_s\la q^s,a_s\ra_{L^2}-\sum_{s\in\sS}\lambda_s\psi_\star^*(a_s)+t\int_0^1\xi\Ll((\lambda_sa_s(v))_{s\in\sS}\Rr)\d v\Rr\},
    \\
\label{e.balanced_u_hopf}
    u(t,\rho)&=\sup_{a\in\mcl Q_{\infty,\le1}}\Ll\{\la\rho,a\ra_{L^2}-\psi_\star^*(a)+t\int_0^1\xi_\star(a(v))\d v\Rr\}.
\end{align}
For the lower bound, we restrict the supremum in~\eqref{e.balanced_f_hopf} to the families with $a_s=a$ for every $s\in\sS$, for some $a\in\mcl Q_{\infty,\le1}$. For such a family, the expression in braces in~\eqref{e.balanced_f_hopf} is that in~\eqref{e.balanced_u_hopf} with $\rho=\bar q$, since $\sum_s\lambda_s=1$ and $\xi(a(v)\lambda)=\xi_\star(a(v))$ by~\eqref{e.xi_star_def}. We thus have $f(t,q)\ge u(t,\bar q)$.

For the upper bound, let $a\in(\mcl Q_{\infty,\le1})^{\sS}$. For every $v\in[0,1)$, the vector $(a_s(v))_{s\in\sS}$ belongs to $[0,1]^{\sS}$, so that~\eqref{e.balanced_inequality} gives $\xi((\lambda_sa_s(v))_{s\in\sS})\le\sum_s\lambda_s\xi_\star(a_s(v))$. The expression in braces in~\eqref{e.balanced_f_hopf} is therefore at most
\begin{equation*}
    \sum_{s\in\sS}\lambda_s\Ll\{\la q^s,a_s\ra_{L^2}-\psi_\star^*(a_s)+t\int_0^1\xi_\star(a_s(v))\d v\Rr\}\le\sum_{s\in\sS}\lambda_su(t,q^s),
\end{equation*}
by~\eqref{e.balanced_u_hopf}. Taking the supremum over $a$ proves the upper bound in~\eqref{e.balanced_comparison}.

Finally, if $q^s=\rho$ for every $s\in\sS$, then $\bar q=\rho$ and $\sum_s\lambda_su(t,q^s)=u(t,\rho)$, so that both sides of~\eqref{e.balanced_comparison} are equal to $u(t,\rho)$.
\end{proof}

\begin{corollary}[One-species formula for balanced models]
\label{c.balanced_one_species_formula}
Assume that $h^s=h_\star$ for every $s\in\sS$ and that $\xi$ is balanced with respect to $\lambda$. Let $(H^\star_N(\sigma))_{\sigma\in\R^N}$ be a centered Gaussian field with covariance
\begin{equation*}
    \E\Ll[H^\star_N(\sigma)H^\star_N(\sigma')\Rr]=N\xi_\star\Ll(\frac1N\sum_{i=1}^N\sigma_i\sigma_i'\Rr),
    \qquad\sigma,\sigma'\in\R^N,
\end{equation*}
and let
\begin{equation*}
    \bar F^\star_N(t):=-\frac1N\E\log\int_{\Sph_N}\exp\Ll(\sqrt{2t}\,H^\star_N(\sigma)-Nt\xi_\star(1)+h_\star\sum_{i=1}^N\sigma_i\Rr)\mu_N(\d\sigma)
\end{equation*}
be the free energy~\eqref{e.def.FN.delta0} of the one-species spherical model with covariance function $\xi_\star$ and external field $h_\star$. For every $t\ge0$, we have
\begin{equation}
\label{e.balanced_simple_formula}
    \lim_{N\to\infty}\bar F_N(t,0)=\lim_{N\to\infty}\bar F_N^\star(t)=\sup_{a\in\mcl Q_{\infty,\le1}}\ \inf_{\rho'\in\mcl Q_\infty}\Ll\{\psi_\star^\circ(\rho')-\la\rho',a\ra_{L^2}+t\int_0^1\xi_\star(a(v))\d v\Rr\}.
\end{equation}
\end{corollary}

\begin{proof}
The right-hand side of~\eqref{e.balanced_simple_formula} is $u(t,0)$, which is equal to $f(t,0)=\lim_N\bar F_N(t,0)$ by Proposition~\ref{p.balanced_reduction} applied with $q=0$ and Theorem~\ref{t.main}. As observed below~\eqref{e.balanced_u_def}, Theorem~\ref{t.main} applied to the one-species model shows that $\lim_N\bar F^\star_N(t)$ is also equal to $u(t,0)$.
\end{proof}

\appendix

\section{Tilted discrete cascades}
\label{s.app_cascade}

This appendix collects the properties of discrete Poisson--Dirichlet cascades under a tilt that are used in Section~\ref{s.enriched_properties} and in Appendix~\ref{s.app_concentration}: the recursive formula for the free energy and the variance bound, used in the concentration Lemma~\ref{l.cascade_concentration}, and the law of the marks along a sampled leaf, used for the Gaussian model with $n$ coordinates in Proposition~\ref{l.initial_condition}. They all follow from an invariance property of tilted cascades due to Panchenko and Talagrand, \cite[Lemma~3]{pantal07}, see also~\cite[Lemma~3.1]{pantal07note}, which we recall in the form we use: tilting the weights of the cascade by the exponential of a function of the marks has the same effect, in law, as changing the law of the marks along the tree, level by level, while leaving the weights untouched. The recursive formula of~\cite[Theorem~2.9]{pan} or \cite[Theorem~5.25]{HJbook} is the identity in expectation which corresponds to this invariance.

\subsubsection*{Setting}
We use the discrete cascade of Subsection~\ref{s.enriched_def}: the tree $\mcl A$ of depth $k\ge0$, the parameters $0=m_0<m_1<\cdots<m_k<m_{k+1}=1$, the unnormalized weights $w_\alpha=\prod_{j=1}^ku_{\alpha_{|j}}$, and the normalized weights $v_\alpha$. We attach to every node $\beta\in\mcl A$ a standard Gaussian vector $z_\beta:=(z_{\beta,i})_{i\le n}\in\R^n$, independently over the nodes and of the weights. For a node $\beta$ of depth $j$, for a leaf $\alpha$, and for $i\le n$, we write
\begin{equation*}
    \Omega_\beta:=(z_{\beta_{|0}},\ldots,z_{\beta_{|j}})\in(\R^n)^{j+1}
    \qquad\text{and}\qquad
    \Omega_{\alpha,i}:=(z_{\alpha_{|0},i},\ldots,z_{\alpha_{|k},i})\in\R^{k+1}
\end{equation*}
for the marks along the path from the root to $\beta$, and for the marks of coordinate $i$ along the path from the root to $\alpha$. Given a measurable $X:(\R^n)^{k+1}\to\R$, we set $X_k:=X$ and define recursively, for $j$ from $k$ down to $1$,
\begin{equation}
\label{e.tilted_recursion}
    X_{j-1}(x_0,\ldots,x_{j-1}):=\frac1{m_j}\log\E\exp\Ll(m_jX_j(x_0,\ldots,x_{j-1},z_j)\Rr),
    \qquad X_{-1}:=\E X_0(z_0),
\end{equation}
where $z_j$ denotes a standard Gaussian vector in $\R^n$. We say that $X$ is admissible if all the expectations in~\eqref{e.tilted_recursion} are finite, that is, if $\E\exp(m_jX_j(x_0,\ldots,x_{j-1},z_j))<\infty$ for every $j\in\{1,\ldots,k\}$ and $(x_0,\ldots,x_{j-1})\in(\R^n)^j$, and if $X_0(z_0)$ is integrable. Every Lipschitz function is admissible. The recursive formula of~\cite[Theorem~2.9]{pan} or \cite[Theorem~5.25]{HJbook} is stated under the stronger assumption $\E\exp(m_kX(z_0,\ldots,z_k))<\infty$, in which all the marks are integrated at once; the conditional formulation is needed for the quadratic tilts of the Gaussian block, for which this stronger assumption can fail even when the multiplier $b$ is admissible for the path, that is, when $b>2\mcl K(\rho)$, while the recursion~\eqref{e.tilted_recursion} is finite at every level exactly under this condition, see Step~1 of the proof of Lemma~\ref{l.scalar_initial_continuous_derivative}. For admissible $X$, the tilted path law $\mu_X$ is the probability measure on $(\R^n)^{k+1}$ given by
\begin{equation}
\label{e.tilted_path_law}
    \mu_X(\d x):=\gamma^{\otimes n}(\d x_0)\prod_{j=1}^k\exp\Ll(m_jX_j(x_0,\ldots,x_j)-m_jX_{j-1}(x_0,\ldots,x_{j-1})\Rr)\gamma^{\otimes n}(\d x_j);
\end{equation}
by~\eqref{e.tilted_recursion}, each factor integrates to one in $x_j$, so that $\mu_X$ is indeed a probability measure, under which the mark at the root is not tilted and the mark at level $j$ is tilted by $\exp(m_jX_j)$ given the marks of the previous levels. Finally, $\la\cdot\ra_X$ denotes the average with respect to the tilted cascade $v^X_\alpha\propto v_\alpha\exp(X(\Omega_\alpha))$; that its normalization $\sum_\alpha v_\alpha\exp(X(\Omega_\alpha))$ is almost surely finite for admissible $X$ is part of the next lemma.

\subsubsection*{Tilted marks}
Let $X$ be admissible. We define a second family of marks $(z'_\beta)_{\beta\in\mcl A}$, generated along the tree: $z'_\emptyset:=z_\emptyset$, and, recursively over the depth $j\in\{1,\ldots,k\}$, conditionally on the marks of depth at most $j-1$, the marks $(z'_\beta)_{\beta\in\N^j}$ are independent, the mark of the node $\beta=\beta'i$ with parent $\beta'\in\N^{j-1}$ having the law
\begin{equation}
\label{e.tilted_node_law}
    \exp\Ll(m_jX_j(\Omega'_{\beta'},z)-m_jX_{j-1}(\Omega'_{\beta'})\Rr)\gamma^{\otimes n}(\d z),
\end{equation}
where $\Omega'_\beta:=(z'_{\beta_{|0}},\ldots,z'_{\beta_{|j}})$ denotes the new marks along the path from the root to a node $\beta$ of depth $j$. Comparing~\eqref{e.tilted_node_law} with~\eqref{e.tilted_path_law}, we see that $\Omega'_\alpha$ has the law $\mu_X$, for every leaf $\alpha$. We also set, for every node $\beta$ of depth $j\ge1$,
\begin{equation}
\label{e.tilt_factor_def}
    e_\beta:=\exp\Ll(X_j(\Omega_\beta)-X_{j-1}(\Omega_{\beta_{|j-1}})\Rr),
    \qquad\text{so that}\qquad
    \prod_{j=1}^ke_{\alpha_{|j}}=\exp\Ll(X(\Omega_\alpha)-X_0(z_\emptyset)\Rr)
\end{equation}
for every leaf $\alpha$, by telescoping.

\begin{lemma}[Invariance of tilted cascades]
\label{l.tilted_invariance}
Let $X$ be admissible. Conditionally on $z_\emptyset$, the point processes
\begin{equation}
\label{e.tilted_invariance}
    \Ll(u_{\alpha_{|1}}e_{\alpha_{|1}},\ldots,u_{\alpha_{|k}}e_{\alpha_{|k}},\Omega_\alpha\Rr)_{\alpha\in\N^k}
    \qquad\text{and}\qquad
    \Ll(u_{\alpha_{|1}},\ldots,u_{\alpha_{|k}},\Omega'_\alpha\Rr)_{\alpha\in\N^k}
\end{equation}
on $(0,\infty)^k\times(\R^n)^{k+1}$ have the same law, the marks $(z'_\beta)_{\beta\in\mcl A}$ in the second process being generated by~\eqref{e.tilted_node_law} independently of the weights. In particular, by~\eqref{e.tilt_factor_def}, conditionally on $z_\emptyset$,
\begin{equation}
\label{e.tilted_invariance_weights}
    \Ll(w_\alpha\exp\Ll(X(\Omega_\alpha)-X_0(z_\emptyset)\Rr),\Omega_\alpha\Rr)_{\alpha\in\N^k}
    \stackrel{\d}{=}
    \Ll(w_\alpha,\Omega'_\alpha\Rr)_{\alpha\in\N^k}
\end{equation}
as point processes on $(0,\infty)\times(\R^n)^{k+1}$, and $\sum_{\alpha\in\N^k}w_\alpha\exp(X(\Omega_\alpha))$ is almost surely finite.
\end{lemma}

\begin{proof}
This is~\cite[Lemma~3]{pantal07}, for independent standard Gaussian marks and with the root mark $z_\emptyset$ frozen: the recursion~\eqref{e.tilted_recursion}, the laws~\eqref{e.tilted_node_law}, and the factors~\eqref{e.tilt_factor_def} are respectively (3.4), (3.6), and (3.7) there. The result is stated in~\cite{pantal07} under the assumption $\E\exp X(\Omega_\alpha)<\infty$, but its proof is an induction on the depth of the tree, in which the invariance of Poisson--Dirichlet point processes with i.i.d.\ marks, \cite[Lemma~2]{pantal07} or~\cite[Theorem~5.19]{HJbook}, is applied to the children of each node $\beta$ of depth $j-1$, conditionally on the marks of depth at most $j-1$, with the function $z\mapsto\exp(X_j(\Omega_\beta,z))$, whose moment of order $m_j$ is $\exp(m_jX_{j-1}(\Omega_\beta))$. It therefore only uses the finiteness of the conditional expectations in~\eqref{e.tilted_recursion}, that is, the admissibility of $X$. The last assertion follows from~\eqref{e.tilt_factor_def} and from the almost sure finiteness of $\sum_\alpha w_\alpha$, \cite[Lemma~5.23]{HJbook}.
\end{proof}

We now derive the three consequences that we need. The first one is the recursive formula of~\cite[Theorem~2.9]{pan} or \cite[Theorem~5.25]{HJbook}, under the conditional integrability assumption, and the second one is~\cite[Theorem~6]{pantal07}.

\begin{corollary}[Free energy, marks along the sampled leaf, and variance]
\label{c.tilted_consequences}
Let $X$ be admissible. We have
\begin{equation}
\label{e.recursive_formula}
    \E\log\sum_{\alpha\in\N^k}v_\alpha\exp\Ll(X(\Omega_\alpha)\Rr)=X_{-1},
\end{equation}
and, for every bounded measurable $\Phi$ on $(\R^n)^{k+1}$,
\begin{equation}
\label{e.tilted_marks_one}
    \E\la\Phi(\Omega_\alpha)\ra_X=\int\Phi\,\d\mu_X.
\end{equation}
If moreover $k\ge1$ and $X_0(z_\emptyset)$ is square integrable, and if $S$ denotes the sum of the points of a Poisson--Dirichlet point process with parameter $m_1$, then $\Var(\log S)<\infty$ and
\begin{equation}
\label{e.variance_bound_general}
    \Var\Ll(\log\sum_{\alpha\in\N^k}v_\alpha\exp\Ll(X(\Omega_\alpha)\Rr)\Rr)\le2\Var\Ll(X_0(z_\emptyset)\Rr)+4\Var(\log S).
\end{equation}
In words, under $\E\la\cdot\ra_X$, the marks along the sampled leaf have law $\mu_X$; and, apart from the contribution of the untilted root mark, the fluctuations of the tilted free energy are those of the weights at the first level, whatever the tilt.
\end{corollary}

\begin{proof}
When $k=0$, the tree is reduced to its root, $v_\emptyset=1$, $\la\Phi\ra_X=\Phi(z_\emptyset)$, $\mu_X$ is the law of $z_\emptyset$, and $X_{-1}=\E X(z_\emptyset)$, so that~\eqref{e.recursive_formula} and~\eqref{e.tilted_marks_one} are immediate. Let $k\ge1$, and let $T:=\sum_{\alpha\in\N^k}w_\alpha$, which is independent of $z_\emptyset$. We first check that $T$ has the law of $cS$ for some constant $c\in(0,\infty)$, and that $\log S$ is square integrable. For $k=1$, $T=S$. For $k\ge2$, we write $T=\sum_{i\ge1}u_iT_i$, where $T_i:=\sum_{\alpha:\alpha_{|1}=i}w_\alpha/u_i$ is the sum of the weights of the subtree rooted at the node $i$ of depth one. The random variables $(T_i)_{i\ge1}$ are i.i.d., independent of $(u_i)_{i\ge1}$, and distributed as the sum of the weights of the cascade of depth $k-1$ with parameters $(m_2,\ldots,m_k)$, which by induction has the law of $c'S'$, with $S'$ the sum of a Poisson--Dirichlet point process with parameter $m_2>m_1$; hence $\E T_1^{m_1}<\infty$ by~\cite[Proposition~5.18]{HJbook}, and~\cite[Theorem~5.19]{HJbook} gives $T\stackrel{\d}{=}(\E T_1^{m_1})^{1/m_1}S$. Moreover, $S$ is at least the largest point $u_1$ of the process, whose law is $\P(u_1\le x)=\exp(-x^{-m_1})$, so that $S$ has finite negative moments, and $S$ has finite positive moments of order $p<m_1$ by~\cite[Proposition~5.18]{HJbook}; hence $\log S$ is square integrable, and so is $\log T$, with $\Var(\log T)=\Var(\log S)$.

Set $A:=\log\sum_\alpha w_\alpha\exp(X(\Omega_\alpha))$, so that $\log\sum_\alpha v_\alpha\exp(X(\Omega_\alpha))=A-\log T$. By~\eqref{e.tilted_invariance_weights}, conditionally on $z_\emptyset$, the random variable $\sum_\alpha w_\alpha\exp(X(\Omega_\alpha)-X_0(z_\emptyset))$ has the same law as $T$; since $T$ is independent of $z_\emptyset$, the pair $(z_\emptyset,A)$ therefore has the same law as $(z_\emptyset,X_0(z_\emptyset)+\log T')$, with $T'\stackrel{\d}{=}T$ independent of $z_\emptyset$. Since $X_0(z_\emptyset)$ and $\log T$ are integrable, taking expectations gives $\E A=X_{-1}+\E\log T$, which is~\eqref{e.recursive_formula}. If $X_0(z_\emptyset)$ is square integrable, then $\Var(A)=\Var(X_0(z_\emptyset))+\Var(\log T)$ by independence, and $\Var(A-\log T)\le2\Var(A)+2\Var(\log T)$ gives~\eqref{e.variance_bound_general}. Finally, for~\eqref{e.tilted_marks_one}, conditionally on $z_\emptyset$, \eqref{e.tilted_invariance_weights} gives
\begin{equation*}
    \la\Phi(\Omega_\alpha)\ra_X=\frac{\sum_\alpha w_\alpha\exp\Ll(X(\Omega_\alpha)-X_0(z_\emptyset)\Rr)\Phi(\Omega_\alpha)}{\sum_\alpha w_\alpha\exp\Ll(X(\Omega_\alpha)-X_0(z_\emptyset)\Rr)}
    \stackrel{\d}{=}\frac{\sum_\alpha w_\alpha\Phi(\Omega'_\alpha)}{\sum_\alpha w_\alpha}=\sum_\alpha v_\alpha\Phi(\Omega'_\alpha).
\end{equation*}
Conditionally on $z_\emptyset$ and on the weights, each $\Omega'_\alpha$ has the conditional law of $\mu_X$ given that its first coordinate is $z_\emptyset$, and $\sum_\alpha v_\alpha=1$; taking expectations, first over the marks, then over the weights, and finally over $z_\emptyset$, gives~\eqref{e.tilted_marks_one}.
\end{proof}

We also record a quantitative bound on the last term in~\eqref{e.variance_bound_general}, which is used in Lemma~\ref{l.cascade_concentration} with a parameter $m_1$ tending to zero.

\begin{lemma}[Variance of the logarithm of a Poisson--Dirichlet sum]
\label{l.log_S_variance}
There exists $C<\infty$ such that, for every $m\in(0,\frac12]$, the sum $S$ of the points of a Poisson--Dirichlet point process with parameter $m$ satisfies $\Var(\log S)\le Cm^{-2}$.
\end{lemma}

\begin{proof}
Let $u_1>u_2>\cdots$ be the points of the process, ordered decreasingly, and write $\log S=\log u_1+\log(S/u_1)$. The random variable $E:=u_1^{-m}$ is exponentially distributed with parameter one, so $\Var(\log u_1)=m^{-2}\Var(\log E)$. Conditionally on $u_1$, the remaining points form a Poisson point process on $(0,u_1)$ with intensity $mx^{-m-1}\d x$, so that
\begin{align*}
    \E\Ll[\frac{S}{u_1}-1\,\Big|\,u_1\Rr] & =\int_0^{u_1}\frac{x}{u_1}\,mx^{-m-1}\d x=\frac{m}{1-m}E, \\
    \Var\Ll(\frac S{u_1}\,\Big|\,u_1\Rr) & =\int_0^{u_1}\frac{x^2}{u_1^2}\,mx^{-m-1}\d x=\frac{m}{2-m}E.
\end{align*}
Since $0\le\log(S/u_1)\le S/u_1-1$, we get $\E[\log(S/u_1)^2]\le\E[(S/u_1-1)^2]\le Cm$ for $m\le\frac12$, and $\Var(\log S)\le2\Var(\log u_1)+2\E[\log(S/u_1)^2]\le Cm^{-2}$.
\end{proof}

\begin{lemma}[Tilts with a product structure]
\label{l.tilted_cascade}
Let $\{1,\ldots,n\}=\bigsqcup_{g\in\mcl G}I_g$ be a partition of the coordinates into groups, and for a leaf $\alpha$ and $g\in\mcl G$, let $\Omega_{\alpha,g}:=(\Omega_{\alpha,i})_{i\in I_g}$ denote the marks of the coordinates of the group $g$. For every $g\in\mcl G$, let $X^{(g)}$ be a measurable function of the marks of the group $g$ which is admissible for the cascade with $|I_g|$ marks per node, and set $X(\Omega_\alpha):=\sum_{g\in\mcl G}X^{(g)}(\Omega_{\alpha,g})$.
\begin{enumerate}
    \item \label{i.tilted_additivity} The function $X$ is admissible, the recursion~\eqref{e.tilted_recursion} is additive: $X_j=\sum_gX^{(g)}_j$ for every $j$, the tilted path law is the product $\mu_X=\bigotimes_{g\in\mcl G}\mu_{X^{(g)}}$, and
    \begin{equation}
    \label{e.additive_recursion}
        \E\log\sum_{\alpha\in\N^k}v_\alpha\exp\Ll(X(\Omega_\alpha)\Rr)=\sum_{g\in\mcl G}X^{(g)}_{-1}.
    \end{equation}
    \item \label{i.tilted_marks} For every bounded measurable $\Phi$,
    \begin{equation*}
        \E\la\Phi\Ll((\Omega_{\alpha,g})_{g\in\mcl G}\Rr)\ra_X=\int\Phi\,\d\bigotimes_{g\in\mcl G}\mu_{X^{(g)}}.
    \end{equation*}
    In words, under $\E\la\cdot\ra_X$, the marks of the groups along the sampled leaf are independent, with laws $\mu_{X^{(g)}}$.
    \item \label{i.tilted_variance} Assume moreover that $X^{(g)}_0(z_0)$ is square integrable for every $g$, and let $S$ be as in Corollary~\ref{c.tilted_consequences} if $k\ge1$, with the convention $\Var(\log S):=0$ if $k=0$. Then
    \begin{equation*}
        \Var\Ll(\log\sum_{\alpha\in\N^k}v_\alpha\exp\Ll(X(\Omega_\alpha)\Rr)\Rr)\le2\sum_{g\in\mcl G}\Var\Ll(X^{(g)}_0(z_0)\Rr)+4\Var(\log S).
    \end{equation*}
\end{enumerate}
\end{lemma}

\begin{proof}
Since the marks of different groups are independent at every level, the expectation in~\eqref{e.tilted_recursion} factorizes over the groups, and an induction from $j=k$ down to $j=0$ gives $X_j=\sum_gX^{(g)}_j$, each term being finite by assumption; in particular $X$ is admissible, $X_{-1}=\sum_gX^{(g)}_{-1}$, and the density in~\eqref{e.tilted_path_law} is the product of the corresponding densities of the groups. Part~\eqref{i.tilted_additivity} then follows from~\eqref{e.recursive_formula}, and part~\eqref{i.tilted_marks} from~\eqref{e.tilted_marks_one}. This is the additivity of~\cite[Corollary~5.26]{HJbook}, under the conditional integrability assumption. For part~\eqref{i.tilted_variance}, the case $k=0$ is immediate, since then $\log\sum_\alpha v_\alpha\exp(X(\Omega_\alpha))=X(z_\emptyset)=\sum_gX^{(g)}_0(z_{\emptyset,g})$, whose variance is $\sum_g\Var(X^{(g)}_0(z_0))$ by independence. When $k\ge1$, it follows from~\eqref{e.variance_bound_general} and from $\Var(X_0(z_\emptyset))=\sum_g\Var(X^{(g)}_0(z_0))$, again by independence.
\end{proof}

\section{Concentration for Hamiltonians linear in the cascade fields}
\label{s.app_concentration}

The Hamiltonians used in the cavity computation of Section~\ref{s.cavity} involve, besides the field $W^q_N$, a countable family of additional cascade fields. Conditionally on the cascade, the free energy of such a Hamiltonian is a Lipschitz function of Gaussian fields, and the Gaussian concentration inequality of~\cite[Theorem~1.2]{pan} controls its fluctuations around its conditional expectation. The fluctuations of the cascade itself, which is not a Gaussian object, require a separate argument. This second source of fluctuations was overlooked in~\cite[Proposition~6.8]{chen2025free}, where the Gaussian rate is invoked for the full free energy. We consider Hamiltonians of the form
\begin{equation}
\label{e.linear_cascade_hamiltonian}
    G(\sigma,\alpha):=G_0(\sigma)+\msf m(\sigma)+\sum_{j\in J}w_j(\alpha)\phi_j(\sigma),
    \qquad \sigma\in\Sigma_N,\quad\alpha\in\mfk U,
\end{equation}
where $J$ is a finite or countable set, $(G_0(\sigma))_{\sigma\in\Sigma_N}$ is a centered Gaussian field independent of the cascade, $\msf m:\Sigma_N\to\R$ is bounded and measurable, the functions $\phi_j:\Sigma_N\to\R$ are bounded and measurable, and, conditionally on $\fR$, the fields $(w_j)_{j\in J}$ are independent centered Gaussian fields on $\mfk U$ with covariances $\rho_j(\alpha\wedge\alpha')$, for paths $\rho_j\in\mcl Q_\infty$, independent of $G_0$, and jointly measurable as in the construction of $w^q$ in Subsection~\ref{s.enriched_def}. We assume that
\begin{equation}
\label{e.linear_cascade_variance}
    \msf V:=\sup_{\sigma\in\Sigma_N}\Ll(\Var G_0(\sigma)+\sum_{j\in J}\rho_j(1)\phi_j(\sigma)^2\Rr)<\infty,
\end{equation}
which is the supremum of the variance of $G(\sigma,\alpha)$, and that there exists an increasing sequence of finite sets $J_n\uparrow J$ such that
\begin{equation}
\label{e.linear_cascade_tail}
    \lim_{n\to\infty}\sup_{\sigma\in\Sigma_N}\sum_{j\in J\setminus J_n}\rho_j(1)\phi_j(\sigma)^2=0,
\end{equation}
a condition which is void when $J$ is finite. Under these assumptions, the series in~\eqref{e.linear_cascade_hamiltonian} converges in $L^2$, uniformly over $(\sigma,\alpha)$, and we let
\begin{equation*}
    Z:=\iint\exp\Ll(G(\sigma,\alpha)\Rr)P_N(\d\sigma)\fR(\d\alpha)
\end{equation*}
be the associated partition function, with Gibbs bracket $\la\cdot\ra_G$. The enriched Hamiltonian $H^{t,q}_N$ in~\eqref{e.enriched_H} is of this form, with $J=[N]$, $\rho_i=q^s$ and $\phi_i(\sigma)=\sqrt2\sigma_i$ for $i\in I_{N,s}$, $G_0=\sqrt{2t}H_N$, and $\msf V\le2Nt\xi(\lambda_N)+2N\max_sq^s(1)$; the perturbed Hamiltonians of Section~\ref{s.cavity} are of this form as well, see Remark~\ref{r.linear_form_hamiltonians} there.

\begin{lemma}[Concentration of the free energy]
\label{l.cascade_concentration}
There exists a universal constant $C<\infty$ such that, under the assumptions above,
\begin{equation}
\label{e.cascade_concentration}
    \E\Ll|\log Z-\E\log Z\Rr|\le C\,(1+\msf V)^{1/2}.
\end{equation}
\end{lemma}

The bound~\eqref{e.cascade_concentration} is the one given by the Gaussian concentration inequality when the cascade is frozen; the content of the lemma is that the fluctuations of the cascade do not degrade it. For the discrete cascades of Appendix~\ref{s.app_cascade}, the variance bound~\eqref{e.variance_bound_general} shows that, apart from the contribution of the root mark, the fluctuations of a tilted free energy are bounded by those of the weights of the first level, uniformly in the tilt. The continuous cascade has levels arbitrarily close to the root, for which this bound degenerates. We therefore first cut off the levels below a small threshold $\delta$, by modifying the paths $\rho_j$ on $[0,\delta)$, at a cost controlled by a Gaussian interpolation performed conditionally on the cascade, and then apply the discrete bound to the modified model, whose first level is at height $\delta$.

\begin{proof}[Proof of Lemma~\ref{l.cascade_concentration}]
Throughout the proof, $C$ denotes a universal constant.

\smallskip
\noindent\emph{Step 1: Integrability, and reduction to a finite family.}\par
By Jensen's inequality under $P_N\otimes\fR$, we have $\log Z\ge\iint G\,\d P_N\,\d\fR$, and, conditionally on $\fR$, the right-hand side is a Gaussian random variable with mean in $[-\|\msf m\|_\infty,\|\msf m\|_\infty]$ and variance at most $\msf V$; moreover, $\E[Z\mid\fR]=\iint\exp(\msf m(\sigma)+\frac12\Var G(\sigma,\alpha))\,\d P_N\,\d\fR\le\exp(\|\msf m\|_\infty+\msf V/2)$. Hence $\log Z$ is integrable, and, almost surely,
\begin{equation}
\label{e.log_Z_mean_bounds}
    -\|\msf m\|_\infty\le\E\Ll[\log Z\,\Big|\,\fR\Rr]\le\|\msf m\|_\infty+\frac{\msf V}2,
\end{equation}
with the same bounds for $\E\log Z$. When $J$ is finite, conditionally on $\fR$, the random variable $\log Z$ is the free energy of the Gaussian field $G$ on $\Sigma_N\times\mfk U$, whose variance is at most $\msf V$ by~\eqref{e.linear_cascade_variance}, and the Gaussian concentration inequality, in the form used for free energies, see~\cite[Theorem~1.2]{pan} and~\cite[Theorem~4.7]{HJbook}, gives Gaussian tails for $\log Z-\E[\log Z\mid\fR]$, uniformly over the realization of $\fR$; in particular,
\begin{equation}
\label{e.conditional_gaussian_concentration}
    \E\Ll[\Ll(\log Z-\E[\log Z\mid\fR]\Rr)^2\,\Big|\,\fR\Rr]\le C\msf V
\end{equation}
almost surely, and, in view of the paragraph preceding~\eqref{e.log_Z_mean_bounds}, $\log Z$ has finite moments of all orders.

We now reduce to the case of a finite family. Let $Z^{(n)}$ be the partition function of the Hamiltonian $G^{(n)}$ obtained by restricting the sum in~\eqref{e.linear_cascade_hamiltonian} to $J_n$, with Gibbs bracket $\la\cdot\ra_{(n)}$, and let $T_n:=G-G^{(n)}=\sum_{j\notin J_n}w_j(\alpha)\phi_j(\sigma)$, so that
\begin{equation*}
    \log Z-\log Z^{(n)}=\log\la\exp(T_n)\ra_{(n)}.
\end{equation*}
Conditionally on $\fR$, the field $T_n$ is centered Gaussian, independent of $G^{(n)}$, with variance at most $\eps_n:=\sup_\sigma\sum_{j\notin J_n}\rho_j(1)\phi_j(\sigma)^2$, which tends to zero by~\eqref{e.linear_cascade_tail}. By Jensen's inequality and the inequality $\log u\le u-1$, we have $\la T_n\ra_{(n)}\le\log Z-\log Z^{(n)}\le\la\exp(T_n)\ra_{(n)}-1$. Averaging first over $T_n$, with the bracket $\la\cdot\ra_{(n)}$ frozen, we get $\E\la T_n\ra_{(n)}^2\le\E\la T_n^2\ra_{(n)}\le\eps_n$ and
\begin{equation*}
    \E\Ll[\Ll(\la\exp(T_n)\ra_{(n)}-1\Rr)^2\Rr]=\E\la\exp\Ll(T_n(z)+T_n(z')\Rr)\ra_{(n)}-2\,\E\la\exp(T_n)\ra_{(n)}+1\le e^{2\eps_n}-1,
\end{equation*}
where $z=(\sigma,\alpha)$ and $z'=(\sigma',\alpha')$ denote two replicas. Hence $\E|\log Z-\log Z^{(n)}|\le\eps_n^{1/2}+(e^{2\eps_n}-1)^{1/2}$, which tends to zero as $n\to\infty$, and
\begin{equation*}
    \E|\log Z-\E\log Z|\le2\,\E\Ll|\log Z-\log Z^{(n)}\Rr|+\E\Ll|\log Z^{(n)}-\E\log Z^{(n)}\Rr|.
\end{equation*}
Since $G^{(n)}$ is again of the form~\eqref{e.linear_cascade_hamiltonian} and satisfies~\eqref{e.linear_cascade_variance} with the same $\msf V$, it therefore suffices to prove the lemma when $J$ is finite, which we assume from now on. We also fix $\delta\in(0,\frac12]$, to be chosen at the end.

\smallskip
\noindent\emph{Step 2: Flattening of the paths on $[0,\delta)$.}\par
For $j\in J$, let $\tilde\rho_j(u):=\rho_j(u\vee\delta)$, which belongs to $\mcl Q_\infty$, coincides with $\rho_j$ on $[\delta,1]$, and satisfies $0\le\tilde\rho_j-\rho_j\le\rho_j(1)\one_{[0,\delta)}$. Let $(\tilde w_j)_{j\in J}$ be independent cascade fields with covariances $\tilde\rho_j(\alpha\wedge\alpha')$, independent of $(w_j)_{j\in J}$ and of $G_0$ conditionally on $\fR$, and, for $a\in[0,1]$, let $G^a$ be defined by~\eqref{e.linear_cascade_hamiltonian} with $w^a_j:=\sqrt{1-a}\,w_j+\sqrt a\,\tilde w_j$ in place of $w_j$, with partition function $Z_a$ and bracket $\la\cdot\ra_a$. The field $w^a_j$ has covariance $((1-a)\rho_j+a\tilde\rho_j)(\alpha\wedge\alpha')$, so that $G^a$ is again of the form~\eqref{e.linear_cascade_hamiltonian}, with the same~$\msf V$, and $Z_0=Z$, while $\widetilde Z:=Z_1$ is the partition function of the flattened model. The derivative in~$a$ of the covariance of $G^a$ is
\begin{equation*}
    \Delta(\sigma,\alpha;\sigma',\alpha'):=\sum_{j\in J}\Ll(\tilde\rho_j-\rho_j\Rr)(\alpha\wedge\alpha')\,\phi_j(\sigma)\phi_j(\sigma'),
\end{equation*}
which vanishes when $\alpha=\alpha'$, since $\tilde\rho_j(1)=\rho_j(1)$, and satisfies $|\Delta(\sigma,\alpha;\sigma',\alpha')|\le\msf V\one_{\{\alpha\wedge\alpha'<\delta\}}$ by the Cauchy--Schwarz inequality and~\eqref{e.linear_cascade_variance}. We compare the conditional expectations
\begin{equation*}
    \mu:=\E\Ll[\log Z\,\Big|\,\fR\Rr]
    \qquad\text{and}\qquad
    \tilde\mu:=\E\Ll[\log\widetilde Z\,\Big|\,\fR\Rr],
\end{equation*}
which integrate the Gaussian fields with the cascade frozen. Conditionally on $\fR$, the family $(G^a)_{a\in[0,1]}$ is a Gaussian interpolation, and Gaussian integration by parts, as in the proof of Lemma~\ref{l.gaussian_comparison_covariance_norm}, gives
\begin{equation}
\label{e.conditional_interpolation}
    \frac{\d}{\d a}\E\Ll[\log Z_a\,\Big|\,\fR\Rr]=-\frac12\E\Ll[\la\Delta(z,z')\ra_a\,\Big|\,\fR\Rr],
\end{equation}
where $z=(\sigma,\alpha)$ and $z'=(\sigma',\alpha')$ denote two replicas, the diagonal term of the integration by parts vanishing since $\Delta(z,z)=0$. By the invariance~\eqref{e.invariance_cascade}, applied conditionally on $G_0$, the overlap $\alpha\wedge\alpha'$ is uniformly distributed on $[0,1]$ under $\E\la\cdot\ra_a$, so that $\E\la|\Delta(z,z')|\ra_a\le\msf V\delta$ for every $a\in[0,1]$. Integrating~\eqref{e.conditional_interpolation} over $a\in[0,1]$ and taking expectations, we obtain
\begin{equation}
\label{e.coarsening_interpolation}
    \E|\mu-\tilde\mu|\le\frac12\msf V\delta.
\end{equation}
Since $\E\mu=\E\log Z$ and $\E\tilde\mu=\E\log\widetilde Z$, the triangle inequality gives
\begin{equation*}
    \E|\log Z-\E\log Z|\le\E|\log Z-\mu|+2\,\E|\mu-\tilde\mu|+\E|\tilde\mu-\E\tilde\mu|.
\end{equation*}
We bound the first term by~\eqref{e.conditional_gaussian_concentration}, the second one by~\eqref{e.coarsening_interpolation}, and the last one by the conditional Jensen inequality $\E|\tilde\mu-\E\tilde\mu|\le\E|\log\widetilde Z-\E\log\widetilde Z|\le \Var(\log\widetilde Z)^{1/2}$, and we get
\begin{equation}
\label{e.coarsening_reduction}
    \E|\log Z-\E\log Z|\le C\msf V^{1/2}+\msf V\delta+\Var\Ll(\log\widetilde Z\Rr)^{1/2}.
\end{equation}
It remains to bound the variance of the free energy of the flattened model.

\smallskip
\noindent\emph{Step 3: Discrete approximation of the flattened model.}\par
The flattened Hamiltonian depends on the cascade only through $(\alpha\wedge\alpha')\vee\delta$, since $\tilde\rho_j(u)=\tilde\rho_j(u\vee\delta)$. For $K\in\N$, consider the discrete cascade of depth $K$ with parameters $m_1:=\delta$ and $m_{l+1}:=\delta+l(1-\delta)/K$ for $1\le l\le K$, so that $m_{K+1}=1$, and and let $\hat\rho_j$ be such that $\hat\rho_j=\rho_j(\delta)$ on $[0,\delta)$ and $\hat\rho_j=\rho_j(m_l)$ on $[m_l,m_{l+1})$ for $l\ge1$. Let $\widetilde Z^{(K)}$ be the partition sum obtained by replacing, in the flattened model, the cascade $\fR$ by this discrete cascade and the fields $\tilde w_j$ by the discrete fields $\msf w^{\hat\rho_j}_j$ of~\eqref{e.discrete_field}, built from independent families of marks. As $K\to\infty$, the array $(m_{\alpha^\ell\wedge\alpha^{\ell'}}\vee\delta)_{\ell,\ell'\le n'}$ of the discrete cascade converges in law to $((\alpha^\ell\wedge\alpha^{\ell'})\vee\delta)_{\ell,\ell'\le n'}$ under $\E\la\cdot\ra_\fR$, for every $n'$, by~\cite[Corollary~5.32]{HJbook}, and since the off-diagonal overlaps are uniformly distributed while the paths $\tilde\rho_j$ have countably many discontinuities, the arrays $(\hat\rho_j(m_{\alpha^\ell\wedge\alpha^{\ell'}}))_{j,\ell,\ell'}$ converge in law to $(\tilde\rho_j(\alpha^\ell\wedge\alpha^{\ell'}))_{j,\ell,\ell'}$. Let $A>\|\msf m\|_\infty+\msf V$, and let $\log_A:=(-A)\vee(\log\wedge A)$. The argument of the proof of~\cite[Proposition~4.8]{chen2025free}, based on~\cite[Propositions~4.4 and~4.6]{chen2025free}, shows that $\E\log_A\widetilde Z$ and $\E(\log_A\widetilde Z)^2$ are, up to an error which can be made arbitrarily small uniformly over the cascades under consideration, expectations of bounded continuous functions of these arrays: one truncates the Gaussian fields, using their Gaussian tails, approximates the bounded continuous functions $\log_A$ and $(\log_A)^2$ by polynomials on the compact range of the truncated partition function, and observes that the expectation of a polynomial in the partition function is, after integrating the Gaussian fields, a continuous function of the overlap array of finitely many replicas. Consequently, $\Var(\log_A\widetilde Z)=\lim_{K\to\infty}\Var(\log_A\widetilde Z^{(K)})$. Since $\log_A$ is the projection of $\log$ onto $[-A,A]$, and since $\E\log\widetilde Z^{(K)}$ lies in $[-\|\msf m\|_\infty,\|\msf m\|_\infty+\msf V/2]\subset[-A,A]$, by~\eqref{e.log_Z_mean_bounds}, whose proof applies verbatim to the discrete model, we have $\Var(\log_A\widetilde Z^{(K)})\le\E[(\log_A\widetilde Z^{(K)}-\E\log\widetilde Z^{(K)})^2]\le \Var(\log\widetilde Z^{(K)})$. Letting $A\to\infty$, and using that $\log\widetilde Z$ is square integrable, by Step~1 applied to the flattened model, we conclude that
\begin{equation}
\label{e.coarsened_discrete_variance}
    \Var(\log\widetilde Z)\le\liminf_{K\to\infty}\Var\Ll(\log\widetilde Z^{(K)}\Rr).
\end{equation}

\smallskip

\noindent\emph{Step 4: The discrete bound.}\par
We fix $K$. In the notation of Appendix~\ref{s.app_cascade}, with $n=|J|$ marks per node, we have 
\begin{equation*}
\log\widetilde Z^{(K)}=\log\sum_{\alpha\in\N^K}v_\alpha\exp(X(\Omega_\alpha)),
\end{equation*}
where
\begin{align*}
    X(\Omega_\alpha)&:=\log\int\exp\Big(G_0(\sigma)+\msf m(\sigma)+\sum_{j\in J}\msf w^{\hat\rho_j}_j(\alpha)\phi_j(\sigma)\Big)P_N(\d\sigma),
    \\
    \msf w^{\hat\rho_j}_j(\alpha)&=\sum_{l=0}^K\Ll(\hat\rho_{j,l}-\hat\rho_{j,l-1}\Rr)^{1/2}z_{\alpha_{|l},j},
\end{align*}
The gradient of $X$ with respect to the marks is a Gibbs average of $((\hat\rho_{j,l}-\hat\rho_{j,l-1})^{1/2}\phi_j(\sigma))_{j,l}$, whose norm is at most $(\sum_j\hat\rho_j(1)\phi_j(\sigma)^2)^{1/2}\le\msf V^{1/2}$; hence $X$ is Lipschitz with respect to these variables, and admissible. Conditionally on $G_0$, the variance bound~\eqref{e.variance_bound_general} and the recursive formula~\eqref{e.recursive_formula} give
\begin{align*}
    \Var\Ll(\log\widetilde Z^{(K)}\,\Big|\,G_0\Rr)&\le2\Var\Ll(X_0(z_\emptyset)\,\Big|\,G_0\Rr)+4\Var(\log S)
    \\
    \text{and}\qquad
    \E\Ll[\log\widetilde Z^{(K)}\,\Big|\,G_0\Rr]&=\E\Ll[X_0(z_\emptyset)\,\Big|\,G_0\Rr],
\end{align*}
where $S$ is the sum of the points of a Poisson--Dirichlet point process with parameter $m_1=\delta$, so that $\Var(\log S)\le C\delta^{-2}$ by Lemma~\ref{l.log_S_variance}. Combining the two identities through the decomposition of the variance into conditional variance and variance of the conditional expectation, we obtain
\begin{equation*}
    \Var\Ll(\log\widetilde Z^{(K)}\Rr)\le2\Var\Ll(X_0(z_\emptyset)\Rr)+C\delta^{-2},
\end{equation*}
where the variance of $X_0(z_\emptyset)$ is now taken over $G_0$ and $z_\emptyset$ jointly. The random variable $X_0(z_\emptyset)$ depends on $G_0$ and $z_\emptyset$ only through the centered Gaussian field $\msf G(\sigma):=G_0(\sigma)+\sum_{j\in J}\hat\rho_{j,0}^{1/2}z_{\emptyset,j}\phi_j(\sigma)$ on $\Sigma_N$, whose variance is at most $\msf V$ by~\eqref{e.linear_cascade_variance}, since $\hat\rho_{j,0}=\rho_j(\delta)\le\rho_j(1)$. Moreover, $X_K$ is of the form $\log\int\exp(\msf G(\sigma)+\cdots)P_N(\d\sigma)$, whose derivative with respect to a bounded perturbation $h$ of the field is $\int h\,\d\nu$ for a probability measure $\nu$ on $\Sigma_N$, and the recursion~\eqref{e.tilted_recursion} averages such derivatives with nonnegative weights summing to one; the same is therefore true of $X_0$. The Gaussian concentration inequality, in the form used for free energies, see~\cite[Theorem~1.2]{pan} and~\cite[Theorem~4.7]{HJbook}, thus applies to $X_0$ as a functional of $\msf G$, and gives $\Var(X_0(z_\emptyset))\le C\msf V$. Hence $\Var(\log\widetilde Z^{(K)})\le C\msf V+C\delta^{-2}$, uniformly in $K$.

\smallskip

\noindent\emph{Step 5: Conclusion.}\par
By~\eqref{e.coarsening_reduction}, \eqref{e.coarsened_discrete_variance}, and Step~4,
\begin{equation*}
    \E|\log Z-\E\log Z|\le C\msf V^{1/2}+\msf V\delta+C\delta^{-1}.
\end{equation*}
Choosing $\delta:=\min\{\frac12,(1+\msf V)^{-1/2}\}$ gives~\eqref{e.cascade_concentration}.
\end{proof}

\small
\bibliographystyle{plain}
\bibliography{ref}
\end{document}